\documentclass[9pt]{extarticle}
\pdfoutput=1
\usepackage{graphicx}
\usepackage{hyperref}
\usepackage[utf8]{inputenc}
\usepackage{array,epsfig}
\usepackage{amsmath}
\usepackage{amsfonts}
\usepackage{amssymb}
\usepackage{dsfont}
\usepackage{amsxtra}
\usepackage{amsthm}
\usepackage{leftidx}
\usepackage{mathrsfs}
\usepackage{color}
\let\originallhook\lhook
\usepackage{MnSymbol}
\let\lhook\originallhook
\usepackage{tikz-cd}
\usepackage{enumerate}
\usepackage{todonotes}
\usepackage[backend=biber,style=alphabetic]{biblatex}
\theoremstyle{definition}
\newtheorem{defn}{Definition}[subsection]
\theoremstyle{definition}

\theoremstyle{definition}

\theoremstyle{definition}
\newtheorem{cons}[defn]{Construction}
\theoremstyle{definition}

\theoremstyle{definition}

\theoremstyle{definition}

\theoremstyle{definition}
\newtheorem{warn}[defn]{Warning}
\theoremstyle{definition}

\theoremstyle{plain}
\newtheorem{thm}[defn]{Theorem}
\newtheorem*{thm*}{Theorem}
\theoremstyle{plain}

\theoremstyle{definition}

\theoremstyle{plain}
\newtheorem{cor}[defn]{Corollary}
\theoremstyle{definition}
\newtheorem{rmk}[defn]{Remark}
\theoremstyle{definition}

\theoremstyle{plain}
\newtheorem{lem}[defn]{Lemma}
\theoremstyle{plain}

\theoremstyle{definition}
\newtheorem{ex}[defn]{Example}
\theoremstyle{plain}

\theoremstyle{plain}
\newtheorem{prop}[defn]{Proposition}
\theoremstyle{plain}

\newcommand{\bb}[1]{\mathbb{#1}}
\newcommand{\Z}{\bb{Z}}
\newcommand{\Q}{\bb{Q}}
\newcommand{\R}{\bb{R}}

\newcommand{\C}{\bb{C}}

\newcommand{\LS}{\mathcal{L}}

\newcommand{\einfty}{\bb{E}_{\infty}}
\newcommand{\lieg}{\mathfrak{g}}

\newcommand{\Of}{\mathcal{O}}

\newcommand{\del}{\partial}

\newcommand{\hooklongrightarrow}{\lhook\joinrel\longrightarrow}

\newcommand{\simp}{\boldsymbol{\Delta}}
\newcommand{\sset}{\mathsf{Set}_{\simp}}
\newcommand{\set}{\mathsf{Set}}
\newcommand{\spa}{\mathcal{S}}
\newcommand{\map}{\mathsf{Map}}
\newcommand{\Hom}{\mathrm{Hom}}

\newcommand{\adj}{\dashv}
\newcommand{\colim}{\mathrm{colim}\,}

\newcommand{\spec}{\mathsf{Spec}}

\newcommand{\loc}{\mathrm{loc}}

\newcommand{\sring}{\mathsf{d}C^{\infty}\mathsf{Alg}}

\newcommand{\scring}{\mathsf{dCAlg}}
\newcommand{\Mod}{\mathsf{Mod}}

\newcommand{\cotan}{\mathbb{L}}

\newcommand{\icat}{\mathcal{C}}
\newcommand{\icatd}{\mathcal{D}}
\newcommand{\icate}{\mathcal{E}}

\newcommand{\oinfty}{\otimes^{\infty}}
\newcommand{\cinfty}{C^{\infty}}

\newcommand{\fun}{\mathrm{Fun}}

\newcommand{\simpop}{\boldsymbol{\Delta}^{op}}
\newcommand{\ner}{\mathbf{N}}
\newcommand{\simpopplus}{\simp^{op}_+}

\newcommand{\xtop}{\mathcal{X}}
\newcommand{\ytop}{\mathcal{Y}}

\newcommand{\sch}{\mathsf{Sch}}
\newcommand{\dsch}{\mathsf{d}\cinfty\mathsf{Sch}}

\newcommand{\pshv}{\mathsf{PShv}}
\newcommand{\shv}{\mathsf{Shv}}
\newcommand{\pr}{\mathsf{Pr}}

\newcommand{\prr}{\mathsf{Pr}^{\mathrm{R}}}
\newcommand{\prl}{\mathsf{Pr}^{\mathrm{L}}}

\newcommand{\infcat}{$\infty$-category }

\newcommand{\infcats}{$\infty$-categories }

\newcommand{\infcatt}{$\infty$-category}

\newcommand{\infcatst}{$\infty$-categories}

\newcommand{\inftop}{$\infty$-topos }
\newcommand{\inftopt}{$\infty$-topos}
\newcommand{\inftopoi}{$\infty$-topoi }
\newcommand{\inftopoit}{$\infty$-topoi}

\newcommand{\catinf}{\mathsf{Cat}_{\infty}}
\newcommand{\catinfh}{\widehat{\mathsf{Cat}}_{\infty}}

\newcommand{\topo}{\mathsf{Top}}
\newcommand{\cartsp}{\mathsf{CartSp}}
\newcommand{\alg}{\mathsf{Alg}}
\newcommand{\calg}{\mathsf{CAlg}}

\newcommand{\sym}{\mathrm{Sym}}

\newcommand{\daff}{\mathsf{d}C^{\infty}\mathsf{Aff}}

\newcommand{\dstack}{\mathsf{d}C^{\infty}\mathsf{St}}

\newcommand{\smst}{\mathsf{SmSt}}

\newcommand{\ev}{\mathrm{ev}}
\newcommand{\et}{\mathrm{\acute{e}t}}

\newcommand{\fp}{\mathrm{fp}}

\newcommand{\gmt}{\mathrm{gmt}}

\newcommand{\dR}{\mathrm{dR}}
\newcommand{\rmalg}{\mathrm{alg}}
\usepackage{subfiles}

\title{Representability of Elliptic Moduli Problems in Derived $\cinfty$-Geometry}
\author{Pelle Steffens\thanks{pelle.steffens@tum.de} \\ Technical University of Munich}
\date{April 11, 2024}

\begin{document}
\setcounter{tocdepth}{2}
\setcounter{page}{1}

\maketitle
\begin{abstract}
We study moduli spaces of solutions of nonlinear Partial Differential Equations on manifolds in the framework of derived $\cinfty$-geometry. For an arbitrary smooth stack $S$, we define $S$-families of nonlinear PDEs acting between $S$-families of submersions over an $S$-family of manifolds and show that in case the family of PDEs is elliptic and the base family of manifolds is proper over $S$, then the moduli stack of solutions is relatively representable by quasi-smooth derived $\cinfty$-schemes over $S$. Along the way, we develop tools to analyse the local structure of families of mapping stacks between manifolds and explain how to compare mapping stacks in smooth and in derived geometry. To access the notion of a family of PDEs over an arbitrary smooth base stack, we introduce a formalism of stacks of relative jets. Finally, we show how natural ideas from (higher) topos theory can be leveraged to facilitate the application of nonlinear Fredholm analysis to derived stacks of solutions of elliptic PDEs.
\end{abstract}

\tableofcontents
\newpage

\section{Introduction}
The purpose of this paper is to study moduli spaces of solutions of nonlinear Partial Differential Equations in the framework of derived $\cinfty$-geometry. Our primary goal is twofold.
\begin{enumerate}[$(1)$]
    \item Show that the \emph{topological space} $\mathsf{Sol}(P)$ of solutions of an arbitrary nonlinear partial differential operator $P$ acting between sections of submersions over a fixed manifold can be given in a natural manner the structure of a \emph{derived $\cinfty$-stack}, a functor 
    \[  \mathsf{Sol}(P):\daff^{op}\longrightarrow \spa \]
    from the \infcat of affine derived $\cinfty$-schemes introduced in, for instance, \cite{cinftyI}, to the \infcat of \emph{spaces} (or \emph{$\infty$-groupoids}) satisfying a sheaf condition. In this article, we let $\dstack$ denote the \infcat of derived $\cinfty$-stacks. As an \infcat of sheaves, it is an \emph{\inftopt}.
    \item Prove the \emph{elliptic representability theorem}, a result that was announced in our antecedent work \cite{cinftyI}: if the underlying manifold is compact and the operator $P$ is elliptic, then the stack $\mathsf{Sol}(P)$ is representable by a (possibly non-affine) derived $\cinfty$-scheme which is locally of finite presentation and quasi-smooth.
\end{enumerate}
\begin{rmk}
An affine derived $\cinfty$-scheme $(X,\Of_X)$ (which we will here represent as a ringed space) is \emph{of finite presentation} and \emph{quasi-smooth} precisely if we can find a triple 
   \[(M,p:E\rightarrow M,s)\]
   of a smooth manifold $M$, a finite rank vector bundle $p:E\rightarrow M$ and a smooth section $s$ of $p$ such that $(X,\Of_X)$ fits into a pullback diagram 
    \[
    \begin{tikzcd}
    (X,\Of_X) \ar[d]\ar[r] & (M,\cinfty_M)\ar[d,"s"] \\
    (M,\cinfty_M)\ar[r,"0"] & (E,\cinfty_E)
    \end{tikzcd}
    \]
    in the \infcat $\daff$, where the lower horizontal map is the zero section of $p$. For those readers familiar with this terminology, a triple $(M,p:E\rightarrow M,s)$ is known in the literature on moduli spaces of elliptic PDEs as a \emph{local} or \emph{affine Kuranishi model} (without isotropy, that is, we do not allow $M$ to be an orbifold and $E$ an orbifold vector bundle). A \emph{non}-affine quasi-smooth derived $\cinfty$-scheme locally of finite presentation then, is a \emph{homotopy coherent gluing together} of such local Kuranishi models; that is, a \emph{Kuranishi space} (in a sense close to, but more general than, that of Fukaya-Ono).      
\end{rmk}
To set the stage, let $M$ be a compact manifold and let $X\rightarrow M$ and $Y\rightarrow M$ be submersions. Recall that for $k\geq 1$ an integer, a \emph{$k$'th order partial differential equation} acting between sections of $Y$ and $X$ is smooth map $J^k_M(Y)\rightarrow X$ over $M$, where $J^k_M(Y)$ is the $k$'th order jet space of $Y$. The jet space comes equipped with a map of derived $\cinfty$-stacks
\[  j^k:\map_M(M,Y) \longrightarrow\map_M(M,J_M^k(Y)), \]
the \emph{jet prolongation map}. Here $\map_M(M,Y)$ denotes the stack of sections of $Y\rightarrow M$. Thus, composing the PDE with the jet prolongation determines a map $P:\map_M(M,Y)\longrightarrow \map_M(M,X)$ of derived $\cinfty$-stacks. A \emph{differential moduli problem} over $M$ is a quintuple $\icate=(X,Y_1,Y_2,P_1,P_2)$ where $X$, $Y_1$ and $Y_2$ denote submersions over $M$ and $P_1$ and $P_2$ are nonlinear PDEs
\[ P_1:\map_M(M,Y_1)\longrightarrow \map_M(M,X),\quad\quad  P_2:\map_M(M,Y_2)\longrightarrow \map_M(M,X) \]
acting between sections of $Y_1$ and $X$ and $Y_2$ and $X$. We will say that $\icate$ is \emph{elliptic} if at each solution $(\sigma_1,\sigma_2)$ (that is, a pair $(\sigma_1,\sigma_2)\in \map_M(M,Y_1\times_MY_2)$ such that $P_1(\sigma_1)\simeq \sigma\simeq P_2(\sigma_2)$) the linearizations of $P_1$ and $P_2$ determine a linear elliptic PDE
\[  J_M^k(\sigma_1^*TY_1/M\times_M \sigma_2^*TY_2/M) \longrightarrow \sigma^*TX/M . \]
Here, $TX/M$ (and similarly $TY_1/M$ and $TY_2/M$) denotes the vector bundle of vertical tangent vectors with respect to the submersion $Y\rightarrow M$. 
\begin{thm}[Elliptic representability]
Let $M$ be a compact manifold and let $\icate=(X,Y_1,Y_2,,P_1,P_2)$ be an elliptic moduli problem over $M$. Then the solution stack $\mathsf{Sol}(\icate)$ defined as the cone in the pullback diagram
\[
\begin{tikzcd}
\mathsf{Sol}(\icate)\ar[d]\ar[r] & \map_M(M,Y_1)\ar[d,"P_1"]\\
\map_M(M,Y_2)\ar[r,"P_2"] \ar[r] & \map_M(M,X)
\end{tikzcd}
\]
among derived $\cinfty$-stacks, is representable by a 1-quasi-smooth derived $\cinfty$-scheme locally of finite presentation. 
\end{thm}

In fact, we will prove a more general \emph{relative} representability result. If $S$ is a smooth stack, $M\rightarrow S$ is an \emph{$S$-family of manifolds} and $Y\rightarrow M$ is an \emph{$S$-family of submersions}, we have for each integer $k\geq 1$ an $S$-family $J^k_{M/S}(Y)$ of $k$'th order jet spaces together with a jet prolongation map
\[ j^k:\mathsf{Sections}_{/S}(Y\rightarrow M)\longrightarrow  \mathsf{Sections}_{/S}(J^k_{M/S}(Y)\rightarrow M). \]
Here $\mathsf{Sections}_{/S}(Y\rightarrow M)$ is the stack of \emph{$S$-parametrized sections of $Y\rightarrow M$}. All these objects will be introduced and explained in the body of this article; for now, we only remark that these notions are stable under base change along arbitrary maps $S\rightarrow S'$ of smooth stacks, and return the usual notions for $S=*$. An \emph{$S$-family of elliptic moduli problems} over an $S$-family of manifolds $M\rightarrow S$ is a quintuple $\icate=(X,Y_1,Y_2,P_1,P_2)$ for $X\rightarrow M$, $Y_1\rightarrow M$ and $Y_1\rightarrow M$ all $S$-families of submersions and $P_1$ and $P_2$ are $S$-families of nonlinear PDEs acting between families of sections of $Y_1$ and $X$ and $Y_2$ and $X$ respectively, that are elliptic for \emph{each point $s:*\rightarrow S$}. The derived $\cinfty$-stack of solutions of $\icate$ is the cone in the pullback diagram 
\[
\begin{tikzcd}
\mathsf{Sol}(\icate)\ar[d]\ar[r] & \mathsf{Sections}_{/S}(Y_1\rightarrow M) \ar[d,"P_1"] \\
\mathsf{Sections}_{/S}(Y_2\rightarrow M)\ar[r,"P_2"] &  \mathsf{Sections}_{/S}(X \rightarrow M)
\end{tikzcd}
\]
in the \inftop $\dstack_{/S}$. The main result of this work may be stated as follows.
\begin{thm}[Relative elliptic representability]\label{thm:ellipticrepresentabilityfamily}
Let $S$ be a smooth stack and let $M\rightarrow S$ be a \emph{proper} $S$-family of manifolds. Let $\icate=(X,Y_1,Y_2,P_1,P_2)$ be an $S$-family of elliptic moduli problems over $M$. Then the solution morphism 
\[ \mathsf{Sol}(\icate) \longrightarrow S \]
is representable by quasi-smooth derived $\cinfty$-schemes locally of finite presentation. 
\end{thm}
This theorem states that for any manifold $N$ and any map $N\rightarrow S$, the pullback $N\times_S\mathsf{Sol}(\icate)$ is (representable by) a quasi-smooth derived $\cinfty$-scheme locally of finite presentation.
\begin{rmk}[Relation to other work]
The study of moduli spaces of geometric elliptic PDEs is a vast and multifaceted field, and we could not possibly hope to provide the reader with a representative sample of all the literature that has relevant intersection with this work. Let us instead focus narrowly on approaches to PDE moduli spaces that, like ours, aim to describe these via universal constructions involving sheaves on a category of `derived $\cinfty$ objects'.\\ Joyce has developed a very substantial theory of derived differential geometry \cite{Joy1,Joykura,Joykurasymp1,Joykurasymp2}, which one could roughly think of as 2-categorical truncation of the formalism of this article. On his website and in various talks (see for instance \cite{JoyTalkSimons}), Joyce has explained his intent to describe elliptic PDE moduli spaces using only universal constructions as we do here, by writing down a functor
\[  \mathsf{Sol}_{\mathrm{Joyce}}(\icate):\mathsf{Kur}_{\mathrm{Joyce}}\longrightarrow \tau_{\leq 2}\spa \]
on his $2$-category of Kuranishi spaces (which is equivalent to his category of $d$-orbifolds) which presumably takes values in the $3$-category of 2-groupoids, and showing it is representable. As far as we understand, Joyce's approach differs from ours in that his 2-functor $\mathsf{Sol}_{\mathrm{Joyce}}(\icate)$ is supposed to arise as the unstraightening of a right fibration $\mathsf{Kur}_{\mathrm{Joyce}}(\icate)\rightarrow \mathsf{Kur}_{\mathrm{Joyce}}$ whose fibre over a Kuranishi space $\mathbf{K}$ is a groupoid of families of solutions parametrized by $\mathbf{K}$ which one writes down explicitly. While this may well be feasible in Joyce's 2-categorical setting, writing down the homotopy type of the fully derived solution stack at each affine derived $\cinfty$-scheme by hand is impractical, which is why we work with families of mapping stacks and jet stacks instead.\\
Very recently, Pardon \cite{PardonFredholm} has independently given an outline of an argument for the representability of the moduli stack of nonsingular pseudo-holomorphic curves in an almost complex manifold as a sheaf on an \infcat of derived manifolds which coincides with our \infcat of derived $\cinfty$-schemes locally of finite presentation. We thank him for making us aware of his work and for his interest in ours. 
\end{rmk}

\subsubsection{A sketch of the proof of elliptic representability}
Let us give an indication of the proof of Theorem \ref{thm:ellipticrepresentabilityfamily}. Firstly, by some general yoga of higher topos theory, the notion of an $S$-family of elliptic moduli problems satisfies \emph{descent}, which in so many words means that it suffices to prove the result for $S$ an ordinary smooth manifold instead of an arbitrary smooth stack (see also the appendix). It follows from the results of Section 3.2 that the stack $\mathsf{Sections}_{/S}(Y\rightarrow M)$ locally around a section $\sigma:M_s\rightarrow Y_s$ (here $M_s$ and $Y_s$ are the fibres $M\times_S\{s\}$ and $Y\times_S\{s\}$) looks like the stack $S\times \map_{M_s}(M_s,\sigma^*TY_s/M_s)$, the product of $S$ with a stack of sections of a vector bundle over the manifold $M_s$ (specifically, the pullback along the section $\sigma$ of the vertical tangent bundle of $Y_s$ with respect to the projection $Y_s\rightarrow M_s$). We deduce this, and other results of this kind, by methods very similar to the ones used to endow infinite dimensional manifolds of mappings with smooth (Fr\'{e}chet) manifold structures. By this maneuver, we are reduced to studying pullback diagrams
\[
\begin{tikzcd}
\mathsf{Sol}(\icate)\ar[d]\ar[r] & S\times\map_M(M,F) \ar[d,"P"] \\
S\ar[r,"\mathrm{id}\times 0"] & S\times \map_M(M,E)
\end{tikzcd}
\]
where $P$ is an $S$-parametrized nonlinear elliptic PDE (of order $k$, say) acting between sections of vector bundles $E$ and $F$ over a compact manifold $M$, and the lower horizontal map is the $S$-parametrized zero section. Now we wish to apply Kuranishi's method of local finite dimensional reduction of solution spaces of nonlinear elliptic PDEs. To employ analytic tools, we need to know how our derived mapping stacks relate to the infinite dimensional manifolds of mappings studied by global analysis. In Sections 3.1 and 3.2, we embed the category of convenient manifolds of Fr\"{o}licher-Kriegl-Michor \cite{krieglmichor,frolicherkriegl} into the \infcat of derived $\cinfty$-stacks and show that the derived stack $\map_M(M,Y)$ of sections of a (surjective) submersion $Y\rightarrow M$ is represented by the manifold $\Gamma(Y;M)$ of sections (which is an infinite dimensional manifold modelled on nuclear Fr\'{e}chet spaces). For the \emph{smooth} mapping stack, this is immediate from the Fr\"{o}licher-Kriegl-Michor calculus; for derived mapping stacks, a bit more work is required. For $(s,\sigma)$ a solution, let $T_{(s,\sigma)}P$ be the linearization of $P$ at $(s,\sigma)$ and choose a section $\iota$ of the projection $\Gamma(E;M)\rightarrow \mathrm{coker}\,T_{(s,\sigma)}P$. By ellipticity, $\mathrm{coker}\,T_{(s,\sigma)}P$ is a finite dimensional vector space. Consider the pullback 
\[
\begin{tikzcd}
Z\ar[d]\ar[r] & S\times\map_M(M,F)\times\mathrm{coker}\,T_{(s,\sigma)}P \ar[d,"P+\iota"] \\
S\ar[r,"\mathrm{id}\times 0"] & S\times \map_M(M,E),
\end{tikzcd}
\]
then the pullback $\mathsf{Sol}(\icate)$ can locally around $(s,\sigma)$ be identified with the pullback $Z\times_{S\times \mathrm{coker}\,T_{(s,\sigma)}P}S$, so if $Z$ is (representable by) an ordinary finite dimensional manifold, we are done. Since the differential of $P+\iota$ is surjective at $(s,\sigma)$ and has finite dimensional kernel (by ellipticity), this should follow from an infinite dimensional implicit function theorem. At this point, there are two problems to overcome.
\begin{enumerate}
\item[$(P1)$] The map $P+\iota:S\times\Gamma(F;M)\times \mathrm{coker}\,T_{(s,\sigma)}P\rightarrow S\times \Gamma(E;M)$ (which represents the corresponding map between derived $\cinfty$-stacks) is a map between Fr\'{e}chet spaces, so the Banach space implicit function theorem does not apply.
\end{enumerate}
This issue is dealt with by standard nonlinear elliptic theory: the map $P+\iota$ is a limit (in the \infcat of smooth, but not derived, stacks) of a tower $\{P_l+\iota: S\times H^{k+l}(F;M)\times\mathrm{coker}\,T_{(s,\sigma)}P\rightarrow S\times H^l(E;M)\}_{l>>1}$ of completions of $P+\iota$ with respect to the various of Sobolev-Hilbert norms. At each finite stage of this tower, we can take our pullback of $P_l+\iota$ with the zero section $S\rightarrow S\times H^l(E;M)$ in the category of Banach manifolds and obtain a finite dimensional manifold $Z_l$ equipped with a submersion $Z_l\rightarrow S$. By an easy elliptic bootstrapping argument, the tower $\{Z_l\}_{l>>1}$ becomes stationary for $l$ large enough (alternatively, one could apply the hard inverse function theorem of Nash and Moser directly to the map $P+\iota$). However, this does not yet conclude the proof.
\begin{enumerate}
\item[$(P2)$] The argument sketched above only shows that the pullback
\[
\begin{tikzcd}
Z_{\mathsf{Sm}}\ar[d]\ar[r] & S\times\map_M(M,F)\times\mathrm{coker}\,T_{(s,\sigma)}P \ar[d,"P+\iota"] \\
S\ar[r,"\mathrm{id}\times 0"] & S\times \map_M(M,E),
\end{tikzcd}
\]
\emph{taken in the \infcat of smooth stacks} is a manifold, that is, $Z_{\mathsf{Sm}}=Z_l$ for $l$ large enough.    
\end{enumerate}  
We circumvent this problem with a few basic features of (higher) topos theory we expose in Section 2.2: if a map of smooth stacks $X\rightarrow Y$ has the property that for \emph{any} manifold $N$ and any map $N\rightarrow Y$, the map $N\times_YX\rightarrow N$ is (representable by) a submersion, then any pullback in $\smst$ along $X\rightarrow Y$ is preserved by the inclusion of smooth stacks into derived $\cinfty$-stacks. The map $P+\iota$ evidently has this property: we just repeat the argument below $(P1)$ for an arbitrary map $N\rightarrow S\times H^l(E;M)$ from an ordinary manifold instead of the zero section $S\rightarrow S\times H^l(E;M)$. 
\begin{rmk}
While the framework of derived $\cinfty$-geometry necessitates the use of a lot of relatively sophisticated modern homotopy theory, the argument sketched above shows that we use a rather modest collection of tools from global analysis and elliptic theory, especially compared to the polyfold approach \cite{HWZbook}. There's only a few places in this work where we employ nontrivial analysis (apart from a few results from \cite{cinftyI} which also do not rely on analysis more advanced than, say, Taylor expansions), such as they are:
\begin{enumerate}[$(1)$]
\item Somewhat elaborate tubular neighbourhood theorems, to describe the local structure of mapping stacks of sections.
\item Linear elliptic theory on closed manifolds, that is, the package of the standard pseudodifferential calculus acting acting on $L^2$ Sobolev scales.  
\item The implicit function theorem for Banach manifolds (see e.g. \cite{Lang}).
\item The \emph{regularizing property of nonlinear elliptic PDEs} (a.k.a elliptic bootstrapping): if $u\in H^{k+l}(F;M)$ is a solution of the nonlinear elliptic equation $Pu=f$ with $P:H^{k+l}(F;M)\rightarrow H^l(E;M)$, $f$ is smooth and $u$ is sufficiently regular (i.e. $l$ is sufficiently large), then $u$ is smooth. There are many versions of such results; the less regularity assumed on $u$, the harder the analysis. Since we may assume arbitrarily high regularity on $u$, the simplest regularity results suffice for our purposes (see for instance Theorem 1.2.D of \cite{taylorpseudopde}).
\end{enumerate}
\end{rmk}

\subsubsection{Moduli spaces as Kuranishi spaces}
Since there are very many elliptic moduli problems, the elliptic representability theorem establishes the existence of a huge variety of interesting moduli spaces as derived $\cinfty$-schemes and stacks. In this work, we give only one example: the moduli spaces of nonsingular genus $g$ pseudo-holomorphic curves in some closed symplectic manifold. This family of moduli spaces (rather, their Deligne-Mumford compactifications) is of central importance to symplectic topology, particularly Gromov-Witten theory. Since the seminal work of Fukaya-Ono \cite{FO}, traditionally, these pseudo-holomorphic curve moduli spaces are represented as \emph{Kuranishi spaces}. The original definition has been reworked and improved upon by by many authors (see e.g. \cite{fooo1,mcduffwehrheim,Joykura,Yangkura,foookura} and references therein). To give at this point a technical overview of all the variants of the notion of a Kuranihsi structure that exist in the literature would be distracting and unproductive. For definiteness, we include here the definition given in the introduction of \cite{cinftyI} (for the sake of exposition, we only treat here Kuranishi spaces without isotropy). 
\begin{defn}[Kuranishi atlas]\label{defn:specialkuranishi}
An \emph{affine Kuranishi model} is a triple $\mathbf{K}=(M,p:E\rightarrow M,s)$ with $M$ a smooth manifold (the \emph{domain}), $p:E\rightarrow M$ a finite rank vector bundle (the \emph{obstruction bundle}) and $s:M\rightarrow E$ a section of $p$. Given two affine Kuranishi models $(M,E,s)$ and $(Y,F,t)$, a morphism $f:(M,E,s)\rightarrow (Y,F,t)$ is a commuting diagram 
\[
\begin{tikzcd}
E\ar[d,"p"] \ar[r,"f_v"] & F \ar[d,"q"] \\
M\ar[r,"f_b"] & Y 
\end{tikzcd}
\]
where $f_v$ is fibrewise linear such that $f_v\circ s=t\circ f_b$. Affine Kuranishi models and morphisms between them form a category that we denote $\mathsf{AffKur}$. We will write $Z(\mathbf{K})\subset M$ for the zero set $Z(s)$ of $s$ endowed with the subspace topology; this determines a functor $Z:\mathsf{AffKur}\rightarrow \mathsf{Top}$. A morphism $f:(M,E,s)\rightarrow (Y,F,t)$ is a \emph{weak equivalence} if $f$ induces a homeomorphism $Z(s)\cong Z(t)$ and for each $x\in Z(s)$, $f$ induces a quasi-isomorphism between the complexes 
\[ \ldots\longrightarrow 0\longrightarrow T_xM\overset{Ts_x}{\longrightarrow} E_x \longrightarrow 0\ldots,\quad\text{and}\quad   \ldots\longrightarrow 0\longrightarrow T_{f_b(x)}Y\overset{Tt_{f_b(x)}}{\longrightarrow} F_{f_b(x)} \longrightarrow 0\ldots. \]
Let $X$ be a paracompact Hausdorff topological space. A \emph{Kuranishi atlas on $X$} consists of the following data.
\begin{enumerate}[$(a)$]
\item An open cover $\{U_i\subset X\}_i$.
\item For each nonempty finite subset $J\subset I$, an affine Kuranishi model $\mathbf{K}_J$.
\item A collection of homeomorphisms $\psi_J:Z(\mathbf{K}_J)\rightarrow U_J$ with $U_J=\cap_{j\in J}U_j$.
\item For each inclusion $J\subset J'$, there is an open set $V_{J'}$ of the domain $M$ of $\mathbf{K}_J=(M,E,s)$ such that $Z(\mathbf{K}_J)\cap V_{J'}=\psi^{-1}_J(U_{J}')$ and a weak equivalence $\phi_{JJ'}:\mathbf{K}_J|_{V_{J'}}\rightarrow \mathbf{K}_{J'}$, where the restriction $\mathbf{K}_J|_{V_{J'}}$ is the affine Kuranishi model $(V_{J'},E|_{V_{J'}},s|_{V_{J'}})$.
\end{enumerate}
These data are required to satisfy the following conditions.
\begin{enumerate}[$(1)$]
\item The transition maps are compatible with the footprint maps: for $J\subset J'$, the diagram
\[
\begin{tikzcd}
Z(\mathbf{K}_J|_{V_{J'}}) = \psi^{-1}_J(U_{J}') \ar[rr,"Z(\phi_{JJ'})"] \ar[dr,"\psi_J|_{Z(\mathbf{K}_J)\cap V_{J'}}"'] && Z(\mathbf{K}_{J'}) \ar[dl,"\psi_{J'}"] \\
& U_{J'}
\end{tikzcd}
\]
commutes.
\item The cocycle condition holds: for every composition $J\subset J'\subset J''$, we have \[\phi_{J'J''}\circ \phi_{JJ'}=\phi_{JJ''}\]
on the open set $\phi_{J'J''}^{-1} (V_{J''})\cap V_{J'}$.
\end{enumerate}
\end{defn}
This definition is almost identical to one proposed by McDuff-Wehrheim \cite[Definition 6.2.1]{mcduffwehrheim} (McDuff and Wehrheim as well as as most other authors require the transition maps $\phi_{JJ'}$ to not only be weak equivalences but also to be embeddings on the domain and on the obstruction bundle. These requirements are automatically satisfied for Kuranishi atlases constructed via the aforementioned references) and is very close to what Fukaya-Oh-Ohta-Ono call a `good coordinate system'; when one attempts to glue together local Kuranishi models on elliptic PDE moduli spaces `by hand' it is precisely this structure one encounters naturally (see for instance \cite[Theorem 4.3.1]{mcduffwehrheim} and \cite[Construction 1.1.18]{cinftyI}). In \cite{cinftyI}, we call Definition \ref{defn:specialkuranishi} a `special Kuranishi atlas' to contrast with a more general notion of a Kuranishi atlas we introduced there to improve on Definition \ref{defn:specialkuranishi}. To see why this is necessary, observe that the data of an ordinary $\cinfty$-atlas on a paracompact Hausdorff space $X$ may be packaged as follows: suppose we are given a collection of maps $\{U_i\rightarrow X\}_{i\in I}$ of open embeddings that cover $X$ with each $U_i$ an open subset of some Cartesian space, then this cover induces a functor $f:P_I\rightarrow\mathsf{Top}^{\mathsf{Open}}$ from the poset $P_I$ of finite nonempty subsets of the indexing set $I$ (partially ordered by \emph{reverse} inclusion) which carries a finite $J\subset I$ to the intersection of all $U_i$ in $X$ for $i\in J$. Then a $\cinfty$-atlas for this cover is a dotted lift $\tilde{f}$ as in the diagram 
\[
\begin{tikzcd}
& \mathsf{CartSp}^{\mathsf{Open}}\ar[d] \\
P_I\ar[r,"f"]\ar[ur,dotted,"\tilde{f}"] & \mathsf{Top}^{\mathsf{Open}},
\end{tikzcd}
\]
where $\mathsf{CartSp}^{\mathsf{Open}}$ is the category of Cartesian spaces and smooth open embeddings among them. An (overly) abstract definition of the category of smooth manifolds then runs as follows: every such atlas $\tilde{f}$ determines by composition a functor
\[   P_I\longrightarrow \mathsf{CartSp} \overset{j}{\longrightarrow}\shv(\mathsf{CartSp}) \]
to the category of sheaves on Cartesian spaces and the colimit of this diagram determines the smooth manifold structure on $X$. The full subcategory of $\shv(\mathsf{CartSp})$ spanned by all such colimits for all $\cinfty$-atlases on paracompact Hausdorff spaces is precisely the category of manifolds. In particular, maps between manifolds are maps between sheaves on Cartesian spaces. Now let $\mathsf{AffKur}^{\mathsf{Open}}\subset \mathsf{AffKur}$ be the subcategory whose morphisms are open embeddings. In contrast to a $\cinfty$-atlas giving a space $X$ the structure of a manifold, a Kuranishi atlas cannot be written as a functor 
\begin{align}\label{eq:noatlas}
P_I\longrightarrow \mathsf{AffKur}^{\mathsf{Open}}
\end{align}  
from the poset of subsets of $I$, so it is not clear how a Kuranishi atlas determines a sheaf on $\mathsf{AffKur}$, nor what morphisms between spaces equipped with Kuranishi atlases should be. If we wish to define moduli spaces via universal construction instead of by ad-hoc formulae, we certainly need to be able to access morphism sets (or spaces) of Kuranishi spaces. Parsing Definition \ref{defn:specialkuranishi} we observe that the reason for the nonexistence of a functor \eqref{eq:noatlas} is that the transition maps $\phi_{IJ}$ in a Kuranishi atlas `go in the wrong direction': from a restriction of an affine Kuranishi model on the larger open to an affine Kuranishi model on the smaller one (simply defining a Kuranishi atlas as a functor \eqref{eq:noatlas} does not work; moduli spaces cannot be given such a structure in general). In the introduction to \cite{cinftyI}, we argue that this problem can be fixed by replacing Kuranishi atlases with a weaker notion whereby the cocycle condition is not satisfied on the nose but \emph{up to coherent homotopy}. These generalized atlases will still satisfy all desiderata of enumerative geometry and have much better formal properties. Passing to homotopy coherent atlases amounts to \emph{inverting} the weak equivalences in $\mathsf{AffKur}$ (in the $\infty$-categorical sense) so that, morally speaking, the transition maps $\phi_{IJ}$ of Definition \ref{defn:specialkuranishi} will point in the right direction after all and we can expect to describe such atlases via a functor of type \eqref{eq:noatlas}. This requires some technology to set up: replace the ordinary category $\mathsf{AffKur}$ with a Kan complex-enriched version $\mathsf{AffKur}_{\simp}$ (this implements the simplicial, or $\infty$-categorical, localization at the weak equivalences of affine Kuranishi models) and the poset $P_I$ with its simplicial `thickening' $\mathfrak{C}[P_I]$ (see \cite[Section 1.1.5]{HTT}), then a \emph{homotopy coherent Kuranishi atlas} (which is just called a `Kuranishi atlas' in \cite{cinftyI}) is a functor of simplicial categories $\tilde{f}:\mathfrak{C}[P_I]\rightarrow \mathsf{AffKur}^{\mathsf{Open}}$ fitting into a commuting diagram  
\[
\begin{tikzcd}
& \mathsf{AffKur}^{\mathsf{Open}}_{\simp}\ar[d] \\
\mathfrak{C}[P_I]\ar[r]\ar[ur,dotted,"\tilde{f}"] & \mathsf{Top}^{\mathsf{Open}},
\end{tikzcd}
\]
of simplicial categories (here the lower horizontal map factors as $\mathfrak{C}[P_I]\rightarrow P_I\overset{f}\rightarrow \mathsf{Top}^{\mathsf{Open}}$ where the first functor `forgets the thickening'). By design (see \cite[Section 1.1.8]{cinftyI}, the coherent nerve $\mathbf{N}(\mathsf{AffKur})$ can be identified with the \infcat of quasi-smooth affine derived $\cinfty$-schemes of finite presentation (see Definition \ref{defn:affinederived}), so passing to \infcatst, a homotopy coherent Kuranishi atlaes is simply a functor $P_I \rightarrow\daff$ of \infcats taking values in the full subcategory spanned by quasi-smooth objects, and the colimit of the composition $P_I\rightarrow \daff\overset{j}{\hookrightarrow}\dstack$ is representable by a quasi-smooth derived $\cinfty$-scheme (we prove this below in Proposition 
\ref{prop:representablebyscheme}), so every homotopy coherent Kuranishi atlas determines a derived $\cinfty$-scheme. Conversely, choosing a cover by affines for a derived $\cinfty$-scheme $(X,\Of_X)$ of course yields a functor $P_I\rightarrow \daff$, that is, a homotopy coherent Kuranishi atlas. Thus we deduce from the elliptic representability theorem that homotopy coherent Kuranishi atlases exist on PDE moduli spaces. In this work, we of course take the point of view that a derived $\cinfty$-scheme or stack is the correct notion for the geometry of differential geometric moduli spaces, however for those interested in classical Kuranishi atlases (which satisfy the cocycle condition strictly) the question remains whether Theorem \ref{thm:ellipticrepresentabilityfamily} implies the existence of a Kuranishi atlas on a PDE moduli space: can a homotopy coherent Kuranishi atlas on a PDE moduli space be strictified? The answer to this question is affirmative; in fact it is possible to show the following.
\begin{enumerate}[$(1)$]
\item Suppose on a space $X$ with cover $\{U_i\subset X\}_{i\in I}$ a Kuranishi atlas $\{\mathbf{K}_J\}_{J\subset I\,\mathrm{finite}}$ on a topological space $X$ is given, then one can construct a homotopy coherent Kuranishi atlas on $X$ that assigns $\mathbf{K}_J$ to the open set $U_J$ (so that homotopy coherent Kuranishi atlases indeed constitute a generalization). In particular, a Kuranishi atlas determines a derived $\cinfty$-scheme, which must be locally of finite presentation and quasi-smooth.
\item Suppose that two Kuranishi atlases on a space are contained in a third (are \emph{commensurate} in the terminology of \cite{mcduffwehrheim}), then the derived $\cinfty$-schemes they determine are equivalent. 
\item Let $(X,\Of_X)$ be a quasi-smooth derived $\cinfty$-scheme locally of finite present and suppose that $X$ is a paracompact Hausdorff space. Then there exists a Kuranishi atlas on $X$ that determines (by $(1)$) a derived $\cinfty$-scheme equivalent to $(X,\Of_X)$.
\end{enumerate}
We will not prove these assertions in this work.

\subsubsection{Differential moduli problems in context}
Let $(X,E\rightarrow X,s)$ be a Kuranishi model. Such a triple specifies a $\cinfty$-structure on the zero set $Z(s)$: we should remember that $Z(s)$ was obtained by intersecting a smooth section of a vector bundle with the zero section. Since we only care about $(X,E\rightarrow X,s)$ up to weak equivalence of Kuranishi models, we conclude that derived zero locus of $s$ as an object of the \infcat $\daff$ captures precisely the amount of intersection theoretic information that the triple $(X,E\rightarrow X,s)$ endows the topological space $Z(s)$ with, in an intrinsic manner. Our notion of a differential moduli problem is similar to that of a Kuranishi model in that it specifies an intersection problem, which overdetermines the derived geometry of the solution space. A more satisfying approach would be to view PDEs themselves as geometric objects in their own right. This calls for a derived version of Vinogradov's `secondary geometry' \cite{PDEvinogradov}, or a differential geometric version of Beilinson and Drinfeld's `D-geometry' \cite{BDchiralalgebras}. Such a theory broadly has the following features (an algebro-geometric variant is also under development, see \cite{Kryczka}).
\begin{enumerate}[$(1)$]
\item Let $M$ be a manifold, then we will consider PDEs over $M$. The derived geometry of such is controlled by the \infcat of algebras $\mathsf{PDE}_M$ for a Lawvere theory whose finitely generated free objects are (smooth functions on) \emph{infinite jet bundles} $\{J^{\infty}_M(V)\}_{V\rightarrow M}$ over finite rank vector bundles $V\rightarrow M$, and whose morphisms are infinitely prolonged PDEs. 
\item Just as the \infcat $\sring$ of derived $\cinfty$-rings has an underlying algebra $(\_)^{\rmalg}:\sring\rightarrow\scring_{\R}$, so too does the \infcat $\mathsf{PDE}_M$ of derived PDEs over $M$ admit a forgetful underlying algebra functor to the \infcat $\mathsf{CAlg}(\mathcal{D}\Mod^{\geq 0}_{\cinfty(M)})$ of commutative algebra objects in the \infcat of connective \emph{left D-modules} on $M$. 
\item As is done by Beilinson-Drinfeld, scheme theory can be developed for objects in the \infcat $\mathsf{PDE})_M$ so that we have \infcats $\mathsf{d}\mathcal{D}\cinfty\mathsf{Sch}_M$ and $\mathsf{d}\mathcal{D}\cinfty\mathsf{St}_M$ of \emph{derived D$\cinfty$-schemes} and \emph{derived D$\cinfty$-stacks}. Now a differential moduli problem $(X,Y_1,Y_2,P_1,P_2)$ over $M$ determines a quasi-smooth derived D$\cinfty$-scheme $\icate$ in the same manner that a Kuranishi model $(X,E\rightarrow X,s)$ determines a derived $\cinfty$-scheme. Consequently, the \infcat $\mathsf{d}\mathcal{D}\cinfty\mathsf{Sch}_M$ captures precisely the relevant sense in which two quintuples $(X,Y_1,Y_2,P_1,P_2)$, $(X',Y'_1,Y'_2,P'_1,P'_2)$ of differential moduli problems are equivalent.
\item Ellipticity for differential moduli problems can be intrinsically formulated (that is, as objects in the \infcat $\mathsf{d}\mathcal{D}\cinfty\mathsf{Sch}_M$, without reference to a representing quintuple $(X,Y_1,Y_2,P_1,P_2)$): objects of $\mathsf{d}\mathcal{D}\cinfty\mathsf{Sch}_M$ have a \emph{cotangent complex} (see \cite[Section 7.3]{HA}). For a `free PDE', that is, a jet bundle $J^{\infty}_M(V)$ on a vector bundle, this cotangent complex $\cotan_{J^{\infty}_M(V)}$ is the \emph{variational bicomplex} of \cite{Andersonbicomplex}. At a given solution $\sigma$, the cotangent complex $\cotan_{\icate,\sigma}$ of a differential moduli problem $\icate=(X,Y_1,Y_2,P_1,P_2)$ at $\sigma$ is an object of Tor-amplitude $\leq 1$ in the \infcat of D-modules on $M$. This $D$-module is represented by the linear PDE obtained from linearizing $P_1$ and $P_2$ at $\sigma$, and it makes sense to ask this D-module to be elliptic.
\item There is a solution stack functor
\[ \mathsf{Sol}: \mathsf{d}\mathcal{D}\cinfty\mathsf{St}_M \longrightarrow \dstack \]
which carries quasi-smooth D$\cinfty$-scheme defined by a quintuple $\icate:(X,Y_1,Y_2,P_1,P_2)$ to $\mathsf{Sol}(\icate)$. In particular, the elliptic representability theorem states that for $M$ compact, quasi-smooth derived \emph{elliptic} D$\cinfty$-schemes are carried to quasi-smooth derived $\cinfty$-schemes.
\item The theory gestured at here may be described by means of the \emph{de Rham sheaf} of Simpson \cite{Simpsonderham}. As in algebraic geometry, one can establish an equivalence $\icatd\Mod_M\simeq \Mod_{M_{\dR}}$ between the \infcat of left D-modules on $M$ and the \infcat of sheaves of modules on the de Rham sheaf of $M$. A nonlinear variant demonstrates an equivalence $\mathsf{d}\mathcal{D}\cinfty\mathsf{St}_M\simeq \dstack_{/M_{\dR}}$, so that we may interpret differential moduli problems as derived stacks $\icate\rightarrow M_{\dR}$ over the de Rham sheaf. Under these identifications, the solution stack functor above is simply the functor $\map_{M_{\dR}}(M_{\dR},\_)$ taking sections over the de Rham sheaf.
\end{enumerate}
The program sketched above is meant to aid the investigation of more sophisticated geometric data often present on PDE moduli spaces. For instance, in situations pertaining to mathematical physics PDEs are obtained via a variational principle. In such cases, one expects the moduli space of solutions to be equipped with a \emph{derived $(-1)$-shifted symplectic structure} \cite{PTVV1}, which one could hope to quantize. The \emph{formal} geometry of this situation at a given solution is studied by the mathematical BV-formalism \cite{CMRBV,Costello,factalg1,factalg2}. In the global setting of the framework above, a variational elliptic moduli problem endows $\icate$ with a $(-1)$-shifted symplectic structure $\omega$ \emph{as an object of $\mathsf{d}\mathcal{D}\cinfty\mathsf{Sch}_M$}. Given an orientation on $M$, this symplectic structure determines a symplectic structure on the solution stack by a version of the \emph{transgression} procedure invented in \cite{AKSZ} and implemented in derived algebraic geometry by \cite{PTVV1}: we pull-push the symplectic structure $\omega$ on $\icate$ along the span 
\[
\begin{tikzcd}
& \map_{M_{\dR}}(M_{\dR},\icate) \times M_{\dR} \ar[dl,"\pi_1"']\ar[dr,"\ev"] \\
\map_{M_{\dR}}(M_{\dR},\icate) && \icate.
\end{tikzcd}
\]  
In this sense, every `classical field theory' on $M$ interpreted as a derived D$\cinfty$-scheme or stack equipped with a $(-1)$-shifted symplectic structure is an AKSZ theory, over the de Rham sheaf of $M$.

\subsubsection{Generalizations}
The representability theorem proven in this work is somewhat bare-bones and does not cover most moduli spaces of elliptic PDEs on manifolds of interest in the literature. Nevertheless, we posit that the methods developed in this work can serve as a blueprint for handling the interaction between derived geometry and nonlinear global analysis in geometric elliptic PDE moduli problems broadly. To illustrate this point, we offer two examples that generalize our setup in differing directions. These are relatively minor extensions and could have been included in this paper if not for lack of space.
\begin{ex}[Local elliptic boundary conditions]
Let $M$ be a compact manifold with boundary $\del M$. We will for simplicity only be considering PDEs among vector bundles on $M$ but this example goes through in the same general setting that we consider for moduli problems without boundary. Let $E,F$ be vector bundles on $M$ and let $H$ be a vector bundle on $\del M$. Suppose we are given a PDE $\mathcal{P}:J^k_M(F)\rightarrow E$ over $M$ and a boundary condition $\mathcal{B}:J^r_M(F)\times_M\del M\rightarrow H$ over $\del M$ with $r\leq k$. This determines a pullback diagram
\[
\begin{tikzcd}
\mathsf{Sol}\ar[d] \ar[r] & \map_M(M,F) \ar[d,"P\times B"] \\
0\ar[r] & \map_M(M,E)\times\map_{\del M}(\del M,E)
\end{tikzcd}
\]
among derived $\cinfty$-stacks. If this \emph{differential moduli problem with boundary} is elliptic in the sense that its linearization at each solution determines an elliptic boundary problem (see \cite[Definition 20.1.1]{HormanderIII}, basic cases include Dirichlet and Neumann boundary conditions, while a more elaborate example is observed in the Lagrangian boundary conditions for pseudo-holomorphic strips in Lagrangian Floer theory), then $\mathsf{Sol}$ is representable by a quasi-smooth derived $\cinfty$-scheme.
\end{ex}
\begin{ex}[Gauge theory]
Let $M$ be a compact manifold and let $G$ be a compact Lie group. Let $P$ be a principal $G$-bundle on $M$ and let $\mathsf{Conn}_P(M)$ be the stack of principal $G$-connections on $P$; this is the stack of sections of the first principal jet bundle of $P$. The gauge group $\mathsf{Gauge}(P)$ acts on $\mathsf{Conn}_P(M)$; let $[\mathsf{Conn}_MP(M)/\mathsf{Gauge}(P)]$ be its quotient stack. The curvature map $\mathrm{curv}:\mathsf{Conn}_P(M)\rightarrow \Omega^2(\lieg;M)=\map_M(M,\Lambda^2T^{\vee}M\otimes \lieg)$, where $\lieg$ is the bundle associated to the adjoint representation of $G$ on its Lie algebra, is $\mathsf{Gauge}(P)$-equivariant so we may consider the pullback diagram 
\[
\begin{tikzcd}
\mathsf{Flat}_P(M) \ar[d] \ar[r] & {[}\mathsf{Conn}_MP(M)/\mathsf{Gauge}(P){]}\ar[d,"\mathrm{curv}"] \\
0\ar[r] & \Omega^2(\lieg;M)
\end{tikzcd}
\]
among derived $\cinfty$-stacks. Then $\mathsf{Flat}_P(M)$ is a quasi-smooth derived \emph{Artin} $\cinfty$-stack. The curvature map is only elliptic transversally to the orbits of the gauge group, so a proof of the representability of the moduli stack of flat connections involves constructing a slice for the gauge action. 
\end{ex}
\begin{rmk}
We will note that specifically in dimension 3, the stack $\mathsf{Flat}_P(M)$ is not the most interesting derived enhancement of the ordinary moduli space of flat connections to consider. In case $M$ is 3-dimensional, we should expect $\mathsf{Flat}_P(M)$ to admit a $(-1)$-shifted symplectic structure as this space is supposed to be the derived critical locus of the Chern-Simons functional, but this is not so; if $P$ is the trivial bundle, then at the trivial flat connection, such a $(-1)$-shifted symplectic structure should yield in particular a nondegenerate pairing between $H^0_{\dR}(M)$ and $H^3_{\dR}(M)$, but the stack $\mathsf{Flat}_P(M)$ knows nothing about 3-forms on $M$. The stack we should be entertaining (in dimension 3) is the pullback 
\[
\begin{tikzcd}
\mathsf{Flat}'_P(M) \ar[d] \ar[r] & {[}\mathsf{Conn}_MP(M)/\mathsf{Gauge}(P){]}\ar[d,"dS_{CS}"] \\
{[}\mathsf{Conn}_MP(M)/\mathsf{Gauge}(P){]}\ar[r,"0"] & T^{\vee}{[}\mathsf{Conn}_MP(M)/\mathsf{Gauge}(P){]}
\end{tikzcd}
\]
where the right vertical map is the differential of the Chern-Simons functional and the lower horizontal map is the zero section. This version of the moduli stack of flat connections is `derived enough' to admit a shifted symplectic pairing. Note that it is not a priori clear what the cotangent stack to the space of connections moduli gauge equivalences should be. To give meaning to this object, prove the representability of $\mathsf{Flat}'_P(M)$ and endow it with its shifted symplectic structure, we have to resort to the `global BV formalism' sketched in the previous subsection (this is a common feature of all gauge theories of which the moduli problem considered here is a paradigmatic but relatively simple example) which is beyond the scope of the current work.
\end{rmk}
In applications to Morse and Floer theory, one is urged to compactify moduli spaces of gradient flow lines, or moduli spaces of pseudo-holomorphic curves/strips/polygons to encompass analytic limiting phenomena like `bubbling', 'node formation' and `trajectory/strip-breaking' from which extraordinarily rich algebraic structures may be extracted. Endowing these compactified moduli spaces with their appropriate derived geometric $\cinfty$-structures (often \emph{with corners}) is a subtle undertaking, both geometrically and analytically. We will refrain from commenting further on this problem for the moment, except for noting that in order to give the relevant moduli spaces derived $\cinfty$-structures in a canonical fashion (as is done for moduli spaces without corners in this work), \emph{gluing} of elliptic PDE solutions dictates that one must consider a nonstandard notion of smoothness in the presence of boundaries for which a function $f:(0,\infty]\rightarrow\R$ is required to decompose into $f=\alpha+g$ with $\alpha\in\R$ and $g$ a function on $(0,\infty)$ all iterated derivatives of which decay uniformly on compact sets to $0$ as $x\rightarrow \infty$. We refer to the work of Parker \cite{ParkerExploded,ParkerLog} and more recent work of Joyce and Pardon \cite{Joyanalyticcorners,PardonFredholm} for an indication of how such a theory might be set up.

\subsection{Overview of contents}
We give a section-by-section overview of the material covered in this article.
\subsubsection*{Section 2}
The first section starts with a recollection of the affine theory of derived $\cinfty$-geometry as developed in \cite{cinftyI} and assemble the tools for gluing affines along equivalences. We prove a general gluing theorem for derived locally $\cinfty$-ringed spaces (whose precise derived algebro-geometric analogue seems curiously absent in the literature), which we use to show that the restricted Yoneda functor
\[  \mathsf{d}\cinfty\mathsf{Sch} \longrightarrow \dstack \]
is fully faithful. Another corollary of the fact that derived locally $\cinfty$-ringed spaces glue along equivalences is a useful representability criterion for derived $\cinfty$-stacks: if a derived stack $X$ admits an open cover by derived $\cinfty$-schemes, then it is (representable by) a derived $\cinfty$-scheme. We show that colimits in $\dstack$ can be detected at the level of what are sometimes called the \emph{petit} topoi of $\dstack$: the \infcats of sheaves on all spaces $X$ for $(X,\Of_X)$ an affine derived $\cinfty$-scheme. We use this observation to prove that the functor $\shv(\mathsf{Mfd})\rightarrow\dstack$ is a fully faithful embedding. At the end of the first subsection, several local properties of morphisms of derived $\cinfty$-schemes are introduced, including the properties of being an \'{e}tale, submersive and proper map. By general principles, these determine corresponding local properties of morphisms of stacks. \\
The second subsection deals with classes of generalized \'{e}tale and submersive maps between derived stacks, which can be detected locally. In fact, because pullbacks along these maps are preserved by the embedding $\shv(\mathsf{Mfd})\rightarrow\dstack$, one can recognize a map of smooth stacks as belonging to one of these classes (as a map of derived stacks) \emph{locally in the \infcat of smooth stacks}. A consequence of this observation is the following fact on which the representability theorem relies: a nonlinear elliptic PDE with surjective differential determines a submersive map of \emph{derived} stacks. 
\subsubsection*{Section 3}
This section develops tools for setting up differential moduli problems, recollects background from global analysis and proves the representability theorem. We identify derived stacks of families of sections as Weil restrictions. Using some hand-crafted tubular neighbourhoods, we show that derived stacks of families of submersions are actually smooth, and that maps between submersions of finite dimensional manifolds induce submersions of (infinite dimensional) stacks of sections. To recognize PDEs, we construct stacks of jets over an arbitrary smooth family of manifolds, which satisfy many natural functorialities. In the last subsection, we assemble these ingredients to prove the main theorem.
\subsubsection*{Section 4}
In this very short final section we demonstrate that our methods efficiently black-box the representability of the moduli stack of nonsingular pseudo-holomorphic curves in an almost complex manifold over the smooth moduli stack of surfaces equipped with complex structures. Since the moduli stack of surfaces with complex structures $\mathsf{Surf}_{\C}$ is a \emph{$1$-stack} (that is, for any manifold, the space $\mathsf{Surf}_{\C}(M)$ is a 1-type, an ordinary groupoid), constructing the relevant elliptic moduli problem can easily be done by hand. 
\subsubsection*{Appendix}
The appendix is but a minor elaboration of certain results from \cite[Section 6.1.3]{HTT} on local properties of morphisms. We show that the sheafification of a property that is merely stable under pullbacks is easy to construct explicitly and extend the characterization of local properties of morphisms in terms of Cartesian transformations to local properties of arbitrary diagrams.

\subsection*{Notation and conventions}
\addcontentsline{toc}{subsection}{\protect\numberline{}Notation and conventions}
\begin{itemize}
    \item We handle the interplay between small and large categories via the usual device of Grothendieck universes, i.e. we assume Tarski-Grothendieck set theory. For any cardinal $\kappa$, we denote by $\mathcal{U}(\kappa)$ the collection of sets of rank $<\kappa$. We fix once and for all three strongly inaccessible cardinals $\kappa_s<\kappa_l<\kappa_{vl}$; then we call the sets in $\mathcal{U}(\kappa_s)$ \emph{small}, those in $\mathcal{U}(\kappa_l)$ \emph{large}, and those in $\mathcal{U}(\kappa_{vl})$ \emph{very large}.
    \item The ordinary category of (small) sets is denoted as $\set$. The ordinary category of (small) simplicial sets is denoted as $\sset$.
    \item An $\infty$-category or $(\infty,1)$-category is a \emph{weak Kan complex}, also known as a \emph{quasi-category}. Our reference on the foundations of such higher categories is Lurie's book \emph{Higher Topos Theory} \cite{HTT}. The large \infcat of small \infcats is denoted $\catinf$; it is the homotopy coherent nerve of the simplicial category of fibrant cofibrant objects of the simplicial model category of marked simplicial sets. The very large \infcat of large \infcats is denoted $\catinfh$. 
     \item For $\icat$ an \infcat and $f:K\rightarrow \icat$ a diagram, we let $\icat_{/f}$ and $\icat_{f/}$ denote the slice \infcats defined by the existence of a bijection between the set of maps $S\rightarrow \icat_{/f}$ of simplicial sets and the set of maps $S\star K\rightarrow \icat$, where $\star$ is the join of simplicial sets, and $\icat_{f/}$ is defined similarly. If $K=\Delta^0$ and $f$ classifies an object $C\in \icat$, we have slices $\icat_{/C}$ and $\icat_{C/}$. For $f:\Delta^1\rightarrow\icat$ a morphism $C\rightarrow C'$, we will also write $\icat_{/C\rightarrow C'}$ instead of $\icat_{/f}$. The joins $S\star\Delta^0$ and $\Delta^0\star S$ are denoted $S^{\rhd}$ and $S^{\lhd}$ respectively.
    \item For $\icat$ an $\infty$-category, the Kan complex of morphisms between two objects $X$ and $Y$ is denoted by $\Hom_{\icat}(X,Y)$.
    	\item We will not distinguish between an ordinary category $\icat$ and its nerve $\ner(\icat)$, that is, we view every category as a 1-category. 
     \item For $\icat$ a simplicial set (usually an $\infty$-category) and $\icatd$ an $\infty$-category, the simplicial set of morphism from $\icat$ to $\icatd$ is denoted as $\fun(\icat,\icatd)$. It is an $\infty$-category and it is called \emph{the $\infty$-category of functors from $\icat$ to $\icatd$}. When $\icat=\spa$, the \infcat of spaces, we write $\pshv(\icatd)$ for $\fun(\icatd^{op},\spa)$, and call it the \emph{\infcat of presheaves on $\icatd$}.   
     \item For $\icat$ an \infcatt, we will freely use the straightening-unstraightening equivalences 
     \[ \mathsf{coCart}_{\icat}\simeq\fun(\icat,\catinf),\quad  \mathsf{Cart}_{\icat}\simeq\fun(\icat^{op},\catinf) \]
     and 
      \[ \mathsf{LFib}_{\icat}\simeq\fun(\icat,\spa),\quad  \mathsf{RFib}_{\icat}\simeq\fun(\icat^{op},\spa)=\pshv(\icat) \]
   of Sections 2.3 and 3.2 of \cite{HTT}. Here $\mathsf{coCart}_{\icat}$ is the \infcat of coCartesian fibrations over $\icat$ whose morphisms are functor preserving coCartesian morphisms and $\mathsf{LFib}\subset \mathsf{coCart}_{\icat}$ is the full subcategory spanned by those $p:\icate\rightarrow\icat$ for which every edge of $\icate$ is $p$-coCartesian; $\mathsf{Cart}_{\icat}$ and $\mathsf{RFib}_{\icat}$ are defined similarly. 
\item We let $\prl\subset\catinf$ denote the subcategory of presentable \infcats and functors admitting right adjoints among them. Similarly, we let $\prr\subset\catinf$ denote the subcategory of presentable \infcats and functors admitting left adjoints between them. 
    \item We let $\topo_{\infty}\subset\catinfh$ denote the subcategory whose objects are \inftopoit, and whose morphisms are functors that admit a left exact left adjoint. Morphisms in $\topo_{\infty}$ will be called \emph{geometric morphisms}, while their left adjoints will be called \emph{algebraic morphisms}. An algebraic morphism $f^*:\xtop\rightarrow \ytop$ is \'{e}tale if $f^*$ is equivalent to the right adjoint of a projection $\xtop_{/X}\rightarrow \xtop$ for some $X\in\xtop$. 
\item If $\icat$ is an \infcat equipped with a Grothendieck topology, we will write $\shv(\icat)\subset\pshv(\icat)$ for the \inftop of sheaves on $\icat$. If this topology is \emph{subcanonical}, that is, the Yoneda embedding $j:\icat\hookrightarrow\pshv(\icat)$ factors through $\shv(\icat)$, we will generally abuse notation by conflating objects of $\icat$ with their image under $j$.
\item For $\icat$ an \infcat that admit finite products and $h:X\rightarrow Y$ a morphism in $\icat$, we will write $\check{C}(h)_{\bullet}^+:\simpopplus\rightarrow\icat$ for the augmented \v{C}ech nerve of $h$ and $\check{C}(h)_{\bullet}$ for its restriction to $\simpop$.
 \item A \emph{manifold} is a second countable, Hausdorff topological manifold without boundary whose topological dimension is globally bounded, equipped with a maximal $C^{\infty}$-atlas. The category of manifolds is denoted $\mathsf{Mfd}$. A manifold in our sense may have connected components of differing dimensions, as long as there is not a (countable) sequence of connected components whose dimensions grow to infinity. An \emph{$n$-manifold} is a manifold each connected component of which has dimension $n$.
\end{itemize}

\subsection*{Acknowledgements}
\addcontentsline{toc}{subsection}{\protect\numberline{}Acknowledgements}
This paper is based in large part on results obtained during my PhD under Damien Calaque's supervision. It is gratifying to finally see the many hours we spent talking about derived geometry and PDE moduli crystallized in this work. I'd like to thank John Pardon and Alexander Schmeding for productive discussions on the representability theorem and Kuranishi structures, and on local additions compatible with a submersion respectively. Lastly, I'm very grateful to Claudia Scheimbauer for her support and advice during the writing of this article.\\
\emph{This work was supported by the Simons Collaboration on Global Categorical Symmetries (1013836). This project has also received funding from the European Research Council (ERC) under the European Union’s Horizon 2020 research and innovation programme (grant agreement No 768679).}
\newpage
\section{Preliminaries on derived $\cinfty$-geometry}
While this section is named `preliminaries', most of the theory we develop here has not appeared in the literature before (at least as far as we are aware); rather the material below is either well known to experts or is familiar from derived algebraic geometry. Specifically, in Section 2.1 we set up the basic theory of derived $\cinfty$-schemes and stacks, and in Section 2.2 we introduce classes of generalized \'{e}tale and submersive maps of stacks and demonstrate how such classes of morphisms facilitate the interaction between smooth and derived stacks.

\subsection{Derived $\cinfty$-schemes}
In this section we introduce and recall the basic geometric objects in derived $\cinfty$-geometry, drawing heavily on results in \cite{cinftyI}. Our theory is paradigmatically algebro-geometric, that is, we work with \emph{structured spaces}: pairs $(X,\Of_X)$ where $X$ is a topological space and $\Of_X$ is a sheaf of `rings' of some kind. Here we specify the kind of rings that we will work with.
\begin{defn}
Let $\cartsp\subset \mathsf{Mfd}$ be the full subcategory spanned by objects of the form $\R^n$ for some $n\geq 0$. The \infcat of \emph{derived $\cinfty$-rings} is the \infcat of finite product preserving functors
\[ \cartsp\longrightarrow \spa. \]
We will let $\sring\subset\fun(\cartsp,\spa)$ denote the full subcategory spanned by derived $\cinfty$-rings. 
\end{defn}
In much greater generality, if $\mathrm{T}$ is a (small) \infcat category that admits finite limits, the \infcat $\mathsf{d}\mathrm{T}\alg$ of \emph{(derived) $\mathrm{T}$-algebras} is the full subcategory $\fun^{\pi}(\mathrm{T},\spa)\subset\fun(\mathrm{T},\spa)$ spanned by functors preserving finite products. Such \infcats $\mathrm{T}$ are \emph{algebraic theories}, or \emph{Lawvere theories}. The assignment $\mathrm{T}\mapsto \mathsf{d}\mathrm{T}\alg$ determines an equivalence of \infcats (really, of $(\infty,2)$-categories) between the \infcat $\mathsf{LawThy}$ of (idempotent complete) Lawvere theories whose morphisms are product preserving functors and the \infcat $\prl_{\mathrm{Proj}}$ of projectively generated presentable \infcats whose morphism are left adjoints whose right adjoint preserves sifted colimits. We will refer to \cite[Section 5.5.6]{HTT} and \cite[Section 2.4]{cinftyI} for more background on algebraic theories and projective generation. The \infcat of \emph{derived commutative $\R$-algebras} (also known as the \infcat of \emph{animated commutative $\R$-algebras}) is the \infcat of derived algebras for the algebraic theory $\mathsf{Poly}_{\R}\subset\mathsf{CartSp}$ on the polynomial maps. This inclusion then yields an adjunction 
\[ \begin{tikzcd}\scring_{\R} \ar[r,shift left,"F^{\cinfty}"] &[2em] \sring.\ar[l,shift left,"(\_)^{\rmalg}"]\end{tikzcd} \]
Here, $(\_)^{\rmalg}$ is the underlying algebra functor and its left adjoint is the free derived $\cinfty$-ring functor (see \cite[Section 4]{cinftyI} for further discussion on the free $\cinfty$-ring functor). The \infcat $\scring_{\R}$ admits the following operadic description: let $\calg$ be the \infcat of commutative algebras in spectra, that is, algebras for the $\einfty$-operad. The \infcat of $\calg_{\R}^{\geq 0}:=\calg^{\geq 0}_{\R/}$ of connective commutative $\R$-algebras is projectively generated (as all \infcats of operad algebras in projectively generated presentably symmetric monoidal \infcats are), and because $\Q\subset\R$, we can identify the full subcategory of $\calg^{\geq 0}_{\R/}$ spanned by compact projective objects with the opposite of the category $\mathsf{Poly}_{\R}$, which determines an equivalence $\scring_{\R}\simeq \calg_{\R}^{\geq 0}$. \\
The \infcat $\sring$ was not arbitrarily chosen. Let $\sring_{\fp}
\subset \sring$ spanned by compact objects, that is, those $A\in \sring$ for which $\Hom_{\sring}(A,\_):\sring\rightarrow\spa$ preserves filtered colimits. We will call such objects \emph{finitely presented}. Then it can be shown (\cite[Section 3.1]{cinftyI}) that the functor $\cinfty(\_):\mathsf{Mfd}\rightarrow \sring^{op}$ factors through $\sring_{\fp}^{op}$ and exhibits the \infcat $\sring_{\fp}^{op}$ as the \emph{universal derived geometry} associated with the category $\mathsf{Mfd}$ (see \cite{univprop} and \cite{cinftyI}; the latter reference establishes a similar result for manifolds with corners). We regard the objects of $\sring_{\fp}$ and more generally those of $\sring$ as \emph{affine}. We will glue such objects together using a derived version of the theory of $\cinfty$-schemes, due originally to Dubuc \cite{dubucschemes} (see also the work of Joyce \cite{Joy2}). For a comprehensive account of algebro-geometric variants, we refer to Lurie's work \cite{dagv,sag}.
\begin{defn}[Derived $\cinfty$-ringed spaces]
Let $\mathsf{Top}(\sring)$ be the \infcat defined as follows.
\begin{enumerate}
    \item[$(O)$] Objects are pairs $(X,\Of_X)$ for $X$ a topological space and $\Of_X$ a Postnikov complete (see the remark below) sheaf of derived $\cinfty$-rings on $X$. We will call such objects \emph{derived $\cinfty$-ringed spaces}.
    \item[$(M)$] For any pair $(X,\Of_X)$, $(Y,\Of_Y)$ of $\cinfty$-ringed spaces, the space of morphisms between them is given by 
    \[ \coprod_{f\in \Hom_{\mathsf{Top}}(X,Y)}\Hom_{\widehat{\shv}_{\sring}(X)}(f^*\Of_Y,\Of_X).\]
\end{enumerate}
More precisely, let $\mathsf{Top}_{\infty}^{\mathrm{Pc}}\subset \mathsf{Top}_{\infty}$ be the full subcategory spanned by Postnikov complete \inftopoi (see the remark below) and consider the functor 
\[ \begin{tikzcd}\widehat{\shv}_{\sring}(\_)^{op}: \mathsf{Top}\ar[r,"X\mapsto \widehat{\shv}(X)"] &[3em]\mathsf{Top}_{\infty}^{\mathrm{Pc}}\subset \topo_{\infty}\subset(\prl)^{op}\ar[r,"\fun^{\mathrm{R}}((\_)^{op}{,}\sring)"]&[6em] (\prl)^{op}\subset \catinfh^{op}\ar[r,"(\_)^{op}"]&\catinfh^{op}\end{tikzcd}\]
carrying a topological space $X$ to the opposite of the presentable \infcat of Postnikov complete $\sring$-valued sheaves on $X$. Then we define the \infcat $\mathsf{Top}(\sring)$ as a Cartesian fibration
\[ \mathsf{Top}(\sring)\longrightarrow \mathsf{Top} \]
associated to the functor $\widehat{\shv}_{\sring}(\_)^{op}$.
\end{defn}
\begin{rmk}
An \inftop is \emph{Postnikov complete} if every object $X\in \xtop$ is a limit of its Postnikov tower; equivalently, if $\xtop$ is a limit (in the \infcat of large \infcats) of the tower 
\[ \xtop\longrightarrow \ldots\longrightarrow \tau_{\leq n}\xtop\longrightarrow \tau_{\leq (n-1)}\xtop\longrightarrow \ldots\longrightarrow \tau_{\leq 0}\xtop. \]
The inclusion $\mathsf{Top}_{\infty}^{\mathrm{Pc}}\subset \mathsf{Top}_{\infty}$ of the full subcategory spanned by Postnikov complete \inftopoi admits a right adjoint, the  \emph{Posntikov completion}, which takes an \inftop $\xtop$ that is not necessarily Postnikov complete to the limit of the tower above (one has to prove this is again an \inftop and that the map $\xtop\rightarrow \lim_n\tau_{\leq n}\xtop$ is an algebraic morphism; see \cite[A.2.7]{sag}). We let $\widehat{\xtop}$ denote the Postnikov completion. By passing to the Postnikov completion, we consider only those sheaves which are \emph{well approximated by their truncations}. In the \infcat $(X,\Of_X)$ we thus consider only sheaves $\Of_X:\mathrm{Open}(X)^{op}\rightarrow \sring$ which are limits of their Postnikov towers, that is, $\Of_X\simeq \lim_n\tau_{\leq n}\Of_X$. For most applications, this is not a restriction: whenever the topological space $X$ is finite dimensional in a suitable sense, $\shv(X)$ will automatically be Postnikov complete. 
\end{rmk}
\begin{rmk}
Let $\icat$ be an $n$-category for $n<\infty$ equipped with a Grothendieck topology that is subcanonical (such as the $(-1)$-category of opens of a topological space). The Yoneda embedding $j:\icat\hookrightarrow\pshv(\icat)$ factors fully faithfully through $\tau_{\leq n}\shv(\icat)$. Since for each $k\in \Z_{\geq -1}$, the Postnikov completion $\shv(\icat)\rightarrow\widehat{\shv}(\icat)$ determines an equivalence $
\tau_{\leq k}\shv(\icat)\simeq\tau_{\leq k}\widehat{\shv}(\icat)$, the composition $\icat\hookrightarrow\shv(\icat)\rightarrow\widehat{\shv}(\icat)$ is also fully faithful. We will abusively denote this Postnikov completed Yoneda embedding also by $j$.  
\end{rmk}
\begin{rmk}
The functor $q_{\sring}:\mathsf{Top}(\sring)\rightarrow\mathsf{Top}$ is by construction a Cartesian fibration. Since for each map of spaces $f:X\rightarrow Y$, the functor $f^*:\widehat{\shv}(Y)\rightarrow\widehat{\shv}(X)$ admits right adjoint, $q$ is also a coCartesian fibration; in fact $q$ is the opposite of a presentable fibration. This property is useful for computing limits and colimits in $\mathsf{Top}(\sring)$ (see for instance \cite[Corollary 4.3.1.11]{HTT}. To compute the limit of a diagram $f:K\rightarrow \mathsf{Top}(\sring)$, take a limit of the composition $q_{\sring}f:K\rightarrow\mathsf{Top}$ and then take the limit in the fibre $q_{\sring}^{-1}(\lim_Kq_{\sring}f)\simeq\widehat{\shv}(\lim_Kq_{\sring}f)^{op}$ of the `Cartesian pullback' of $f$ to this fibre (colimits are computed via the formally dual procedure). For example, a diagram 
\[
\begin{tikzcd}
(Q,\Of_Q)\ar[d,"k"]\ar[r,"h"] & (Y,\Of_Y)\ar[d,"f"] \\
(Z,\Of_Z)\ar[r,"g"] & (X,\Of_X)
\end{tikzcd}
\]
is a pullback of derived $\cinfty$-ringed spaces if and only if the underlying diagram of spaces is a pullback and the diagram 
\[
\begin{tikzcd}
(fh)^*\Of_X\simeq (gk)^*\Of_X\ar[d]\ar[r] & h^*\Of_Y\ar[d] \\
k^*\Of_Z\ar[r] & \Of_Q
\end{tikzcd}
\]
of (Postnikov complete) sheaves of derived $\cinfty$-rings is a pushout. In particular, the functor $q_{\sring}$ preserves limits and colimits.
\end{rmk}
\begin{defn}[Derived locally $\cinfty$-ringed spaces]
A pair $(X,\Of_X)$ is a \emph{derived locally $\cinfty$-ringed space} if for each $x\in X$, the stalk $x^*\Of_X\in\sring$ is an \emph{Archimedean local} derived $\cinfty$-ring: there is a (unique) morphism of derived $\cinfty$-rings $x^*\Of_X\rightarrow \R$ and the underlying commutative $\R$-algebra $\pi_0(x^*\Of_X)^{\rmalg}$ is a local ring (in particular, the map $\pi_0(x^*\Of_X)^{\rmalg}\rightarrow\R$ is the projection onto the residue field). We denote by $\mathsf{Top}^{\mathrm{loc}}(\sring)\subset \mathsf{Top}(\sring)$ the full subcategory spanned by derived locally $\cinfty$-ringed spaces. Abusing notation, we will also write $q_{\sring}$ for the forgetful functor
\[ \mathsf{Top}^{\mathrm{loc}}(\sring)\subset \mathsf{Top}(\sring)\overset{q_{\sring}}{\longrightarrow} \mathsf{Top}.\]
\end{defn}
\begin{rmk}
In general, the subcategory of locally ringed spaces excludes morphisms of ringed spaces that do not preserve the maximal ideals at all stalks. Since the residue field of an Archimedean local $\cinfty$-ring is $\R$, every morphism between such preserves maximal ideals.    
\end{rmk}
\begin{rmk}
Let $(f,\alpha):(X,\Of_X)\rightarrow (Y,\Of_Y)$ be a morphism of derived $\cinfty$-ringed spaces. Suppose that $\alpha$ is an equivalence, that is, $(f,\alpha)$ is $q_{\sring}$-Cartesian, then $(X,\Of_X)$ is also locally $\cinfty$-ringed. It follows that the functor $q_{\sring}:\mathsf{Top}^{\loc}(\sring)\rightarrow \mathsf{Top}$ is also a Cartesian fibration and the full subcategory inclusion $\mathsf{Top}^{\loc}(\sring)\subset \mathsf{Top}(\sring)$ preserves Cartesian morphisms. Note however that the functor $\mathsf{Top}^{\loc}(\sring)\rightarrow \mathsf{Top}$ is \emph{not} a coCartesian fibration: if $(X,\Of_X)$ is locally $\cinfty$-ringed, the pushforward $f_!\Of_X$ is usually not locally $\cinfty$-ringed.
\end{rmk}
\begin{rmk}
One can show (see \cite[Proposition 3.2.11]{cinftyI}) that the full subcategory inclusion $\mathsf{Top}^{\mathrm{loc}}(\sring)\subset \mathsf{Top}(\sring)$ is stable under limits. In particular, $\mathsf{Top}^{\mathrm{loc}}(\sring)$ admits all limits, and the forgetful functor $q_{\sring}:\mathsf{Top}^{\mathrm{loc}}(\sring)\rightarrow\mathsf{Top}$ preserves them.    
\end{rmk}
The \infcat of derived $\cinfty$-ringed spaces admits a natural topology 

\begin{defn}\label{defn:etalemap}
A morphism $(f,\alpha):(X,\Of_X)\rightarrow (Y,\Of_Y)$ of derived $\cinfty$-ringed spaces is \emph{\'{e}tale} if $f$ is a local homeomorphism and $\alpha:f^*\Of_Y\rightarrow \Of_X$ is an equivalence of sheaves of derived $\cinfty$-rings. If $f$ is a also injective, we say that $(f,\alpha)$ is an \emph{open embedding}. We let $\mathsf{Top}(\sring)^{\mathsf{Open}}$ and $\mathsf{Top}(\sring)^{\text{\'{e}t}}$ be the subcategories of $\mathsf{Top}(\sring)$ on the open embeddings and the \'{e}tale morphisms respectively. Similarly, we have evident subcategories $\mathsf{Top}^{\mathrm{loc}}(\sring)^{\mathsf{Open}}$ and $\mathsf{Top}^{\mathrm{loc}}(\sring)^{\text{\'{e}t}}$.
\end{defn}
\begin{rmk}
Let $\mathsf{Top}^{\mathsf{Open}}$ and $\mathsf{Top}^{\text{\'{e}t}}$ be the subcategories of $\mathsf{Top}$ on the open topological embeddings and the local homeomorphisms respectively, then $q_{\sring}$ restricts to right fibrations $\mathsf{Top}(\sring)^{\mathsf{Open}}\rightarrow  \mathsf{Top}^{\mathsf{Open}}$ and $\mathsf{Top}(\sring)^{\text{\'{e}t}}\rightarrow \mathsf{Top}^{\text{\'{e}t}}$. Since a right fibration $p:\icat\rightarrow\icatd$ of \infcats determines for each $C\in\icat$ a trivial fibration $\icat_{/C}\rightarrow \icatd_{/p(C)}$, we have in particular for each derived locally $\cinfty$-ringed space an equivalence $\mathsf{Top}(\sring)^{\mathsf{Open}}_{/(X,\Of_X)}\rightarrow \mathsf{Open}(X)$. Similarly, the Cartesian fibration $\mathsf{Top}^{\loc}(\sring)\rightarrow \mathsf{Top}$ determines for each derived locally $\cinfty$-ringed space $(X,\Of_X)$ an equivalence $\mathsf{Top}^{\mathrm{loc}}(\sring)^{\mathsf{Open}}_{/(X,\Of_X)}\rightarrow \mathsf{Open}(X)$.
\end{rmk}
\begin{rmk}
The classes of \'{e}tale morphisms and open embeddings are stable under pullback by any morphism of derived $\cinfty$-ringed spaces. In particular, the subcategory inclusions $\topo^{\loc}(\sring)^{\et}\subset \topo^{\loc}(\sring)$ preserve pullbacks.  
\end{rmk}
\begin{defn}[\'{E}tale topology]
Let $(X,\Of_X)\in\mathsf{Top}(\sring)$, then a small collection $\{(f_i,\alpha_i):(U_i,\Of_{U_i})\rightarrow(X,\Of_X)\}_{i\in I}$ of morphisms in $\mathsf{Top}^{\mathrm{loc}}(\sring)$ is a covering family just in case $(f_i,\alpha_i)$ is an open embedding for all $i\in I$ and the collection $\{U_i\rightarrow X\}_{i\in I}$ of maps of topological spaces is jointly surjective. We call the Grothendieck topology determined by this collection of covering families the \emph{\'{e}tale topology} on $\mathsf{Top}(\sring)$.
\end{defn}
\begin{rmk}
We may replace $(1)$ above with the condition that $(f_i,\alpha_i)$ is \'{e}tale.
\end{rmk}
\begin{rmk}\label{rmk:subsite}
Let $\LS\subset \mathsf{Top}(\sring)$ be a full subcategory such that the following condition is satisfied.
\begin{enumerate}
    \item[$(*)$] If $(U,\Of_U)\rightarrow (X,\Of_X)$ is an open embedding in $\LS$, then for any map $(Y,\Of_Y)\rightarrow (X,\Of_X)$ in $\LS$, the derived $\cinfty$-ringed space $(U,\Of_X|_U)\times_{(X,\Of_X)}(Y,\Of_Y)$ lies in $\LS$.
\end{enumerate} 
Then the \'{e}tale topology restricts to a topology on $\LS$ (that we also call the \emph{\'{e}tale topology}). The subcategory inclusion $\topo^{\loc}(\sring)\subset\topo(\sring)$ evidently has this property.
\end{rmk}
\begin{defn}
Let $\LS\subset \mathsf{Top}(\sring)$ be a property of objects of derived $\cinfty$-ringed spaces.
\begin{enumerate}[$(1)$]
    \item If $\LS$ satisfies condition $(*)$ of Remark \ref{rmk:subsite}, we say that $\LS$ is \emph{stable under open pullbacks}.   
    \item If $\LS$ satisfies the condition that for any $(X,\Of_X)\in \LS$ and any open subset $U\subset X$, the open embedding $(U,\Of_X|_U)\rightarrow (X,\Of_X)$ lies in $\LS$, then we say that $\LS$ is \emph{stable under open subspaces}. Note that this implies that $\LS$ is stable under open pullbacks.
    \item Suppose that $\LS$ is stable under open subspaces, then we say that $\LS$ is \emph{local} if for any $(X,\Of_X)\in \mathsf{Top}(\sring)$, should there exists an open cover $\{U_i\subset X\}_{i}$ such that $(U_i,\Of_X|_{U_i})$ lies in $\LS$, then $(X,\Of_X)\in \LS$. 
\end{enumerate}
\end{defn}
\begin{ex}
Being a derived \emph{locally} $\cinfty$-ringed space is a local property of derived $\cinfty$-ringed spaces.
\end{ex}
The preceding definition provides the structure of a \emph{site} on each full subcategory $\LS\subset \mathsf{Top}(\sring)$ stable under open pullbacks and we will investigate the \infcats of sheaves they give rise to shortly. We now show that derived locally $\cinfty$-ringed spaces may be glued along equivalences.
\begin{prop}\label{prop:gluinglocringsp}
The following hold true.
\begin{enumerate}[$(1)$]
\item The \infcat $\mathsf{Top}^{\mathrm{loc}}(\sring)^{\text{\'{e}t}}$ admits small coproducts.
\item Let $U_{\bullet}$ be a groupoid object of $\mathsf{Top}^{\mathrm{loc}}(\sring)^{\text{\'{e}t}}$ such that the map
\[ q_{\sring}(U_1)\rightarrow q_{\sring}(U_0)\times q_{\sring}(U_0)  \]
of topological spaces is injective. Then the diagram $U_{\bullet}$ admits a colimit in $\mathsf{Top}^{\mathrm{loc}}(\sring)^{\text{\'{e}t}}$.
\end{enumerate}
Furthermore, the subcategory inclusions $\mathsf{Top}^{\mathrm{loc}}(\sring)^{\text{\'{e}t}}\subset \mathsf{Top}^{\mathrm{loc}}(\sring)\subset\topo(\sring)$ preserve the colimits specified in $(1)$ and $(2)$ above.
\end{prop}
\begin{rmk}
We explain how the preceding result allows one to glue derived locally $\cinfty$-ringed spaces by equivalences. Suppose we are given a collection $(X_i,\Of_{X_i})$ of derived locally $\cinfty$-ringed spaces together with, for each pair $i,j$, a pair of open embeddings $(U_{ij},\Of_{U_{ij}})\rightarrow (X_i,\Of_{X_i})$ and $(U_{ij},\Of_{U_{ij}})\rightarrow (X_i,\Of_{X_j})$. This determines (by $(1)$ of Proposition \ref{prop:gluinglocringsp}) a diagram $C_{\bullet}^{1}:\simpop_{\leq 1}\rightarrow \topo^{\loc}(\sring)$ like so
\[  \begin{tikzcd}
 \coprod_{ij} (U_{ij},\Of_{U_{ij}}) \ar[r,shift left]\ar[r,shift right] & \coprod_i (X_i,\Of_{X_i}).
\end{tikzcd} \]
where both maps are \'{e}tale. Note that the map $C^1_1\rightarrow C^1_0\times C^1_0$ is injective on the underlying topological spaces so that a groupoid object $C_{\bullet}$ extending $C^1_{\bullet}$ satisfies the condition of Proposition \ref{prop:gluinglocringsp}. Write $V_{ij}\subset X_i$ and $W_{ij}\subset X_j$ for the images of the open embeddings $U_{ij}\rightarrow X_i$ and $U_{ij}\rightarrow X_j$ respectively, then to extend $C^1_{\bullet}$ to a diagram $C^2_{\bullet}:\simpop_{\leq 2}\rightarrow \topo^{\loc}(\sring)$ with $C^2_{\bullet}= U^1_0\times_{U^1_0}U^1_0$ (a \emph{1-groupoid} in $\topo^{\loc}(\sring)$) is to provide homotopies $\phi_{ik}\overset{\alpha_{ijk}}{\simeq} \phi_{jk}\circ\phi_{ij}$ where $\phi_{ij}$ is the map 
\[ (V_{ij},\Of_{X_i}|_{V_{ij}}) \overset{\simeq}{\longrightarrow} (U_{ij},\Of_{U_{ij}}) \overset{\simeq}{\longrightarrow} (W_{ij},\Of_{X_j}|_{W_{ij}}), \]
that is, homotopies $\alpha_{ijk}$ witnessing a cocycle condition for the transition functions $\{\phi_{ij}\}$. To extend $C^2_{\bullet}$ to a \emph{$2$-groupoid} $C^3_{\bullet}:\simpop_{\leq 3}\rightarrow\topo^{\loc}(\sring)$ is to provide a collection of two dimensional homotopies witnessing coherences among the homotopies $\{\alpha_{ijk}\}$. In general, we see that to give an extension of $C^1_{\bullet}$ to a groupoid object $C_{\bullet}$ is to provide a collection of homotopies witnessing a cocycle condition together with an infinite tower of higher coherences for these homotopies. Proposition \ref{prop:gluinglocringsp} guarantees that in case such an extension $C_{\bullet}$ is provided, we can glue the collection $\{(X_i,\Of_{X_i})\}_i$ to obtain a derived locally $\cinfty$-ringed space $(X,\Of_X)$ such that
\begin{enumerate}[$(a)$]
\item The space $X=\coprod_i X_i/\sim $ is a gluing of the spaces $\{X_i\}$ so that all maps $f_i:X_i\hookrightarrow X$ are open topological embeddings.
\item The maps $(f_i,\alpha_i):(X_i,\Of_{X_i})\hookrightarrow (X,\Of_X)$ are open embeddings of derived locally $\cinfty$-ringed spaces.
\item The structure sheaf $\Of_X$ is a limit of the cosimplicial diagram obtained by `pushing forward' the diagram $C_{\bullet}$ to the fibre over $X$.
\end{enumerate} 
It may seem prohibitively complicated to construct in examples of interest a groupoid object $C_{\bullet}$ as above encoding all coherences, but this is not the case; the machinery of higher topos theory will do this job for us.
\end{rmk}
In the 1-categorical setting, (locally) ringed spaces may be glued  by constructing the sheaf $\Of_X$ on the space $X=\coprod_i X_i/\sim$ by hand. As we are in the derived setting, the proof of Proposition \ref{prop:gluinglocringsp} is somewhat more elaborate, involving colimits of \inftopoit. We recommend the reader unfamiliar with the arguments involved skip it on first reading. We first need an elementary lemma about tensor products of presentable \infcatst.
\begin{lem}\label{lem:tensoringprl}
Let $\pr^{\mathrm{L},\mathrm{R}}=\prl\cap\prr\subset\catinfh$ be the subcategory whose objects are presentable \infcats and whose morphisms are left \emph{and} right adjoints, then the subcategory inclusions $\pr^{\mathrm{L},\mathrm{R}}\subset\prl$ and $\pr^{\mathrm{L},\mathrm{R}}$ both preserve and reflect limits and colimits. Moreover, if $\icat$ be a presentable \infcatt, then the composition
\[\begin{tikzcd}\pr^{\mathrm{L},\mathrm{R}}\subset \prl\ar[r,"\_\otimes \icat"]&\prl \end{tikzcd}\]
preserves limits, where $\otimes$ denotes the Lurie tensor product of presentable \infcats of \cite[Section 4.8]{HA} which admits an equivalence $\icat\otimes\icatd=\fun^{\mathrm{R}}(\icat^{op},\icatd)$ natural in $\icat$ (using the obvious functoriality of $\fun^{\mathrm{R}}(\icat^{op},\icatd)$ in the domain by composing with the \emph{opposites} of left adjoint functors between presentable \infcatst).
\end{lem}
\begin{proof}
It follows right away from the fact that both inclusions $\prl\subset\catinfh$ and $\prr\subset\catinfh$ preserve and reflect limits that both inclusions $\pr^{\mathrm{L},\mathrm{R}}\subset \prl$ and $\pr^{\mathrm{L},\mathrm{R}}\subset \prr$ preserve and reflect limits. Since the inclusion $\pr^{\mathrm{L},\mathrm{R}}\subset \prl$ is obtained from the inclusion $\pr^{\mathrm{L},\mathrm{R}}\subset \prr$ by taking opposites, we deduce that the inclusion $\pr^{\mathrm{L},\mathrm{R}}\subset \prl$ also preserves and reflects colimits, and similarly for $\pr^{\mathrm{L},\mathrm{R}}\subset \prr$. For the second assertion, we note that the composition 
\[\begin{tikzcd}\pr^{\mathrm{L},\mathrm{R}}\subset \prl\ar[r,"\_\otimes \icat"]&\prl \end{tikzcd}\]
factors through $\pr^{\mathrm{L},\mathrm{R}}$. By what we have just shown, it suffices to show that the opposite of the functor $\_\otimes\icat:\pr^{\mathrm{L},\mathrm{R}}\rightarrow \pr^{\mathrm{L},\mathrm{R}}$ preserves colimits, but this is the same functor $\_\otimes\icat$ (the equivalence $\pr^{\mathrm{L},\mathrm{R}}\simeq (\pr^{\mathrm{L},\mathrm{R}})^{op}$ is symmetric monoidal). Using the first assertion just proven again and the fact that the Lurie tensor product $\prl\times\prl\rightarrow\prl$ preserves colimits separately in both variables, we conclude. 
\end{proof}
We also require the following result which asserts that the construction $X\mapsto\widehat{\shv}(X)$ is sufficiently well behaved.
\begin{prop}\label{prop:colimcomparetop}
The following hold true.
\begin{enumerate}[$(1)$]
\item Suppose that $X\rightarrow Y$ is a local homeomorphism of topological spaces. Then the induced geometric morphism $\widehat{\shv}(X)\rightarrow\widehat{\shv}(Y)$ is \'{e}tale.
\item Suppose that 
\[
\begin{tikzcd}
Q\ar[d]\ar[r,"g"] & Y\ar[d] \\
X\ar[r,"f"] & Z
\end{tikzcd}
\]  
is a pullback diagram of topological spaces where $f$ and $g$ are local homeomorphisms. Then the diagram
\[
\begin{tikzcd}
\widehat{\shv}(Q)\ar[d]\ar[r] & \widehat{\shv}(Y)\ar[d] \\
\widehat{\shv}(X)\ar[r] & \widehat{\shv}(Z)
\end{tikzcd}
\]
of \inftopoi and geometric morphisms between them is a pullback diagram. 
\item Let $\{X_i\}_i$ be a small collection of topological spaces, then the collection $\{\widehat{\shv}(X_i)\rightarrow \widehat{\shv}(\coprod_i X_i)\}_i$ exhibits $\widehat{\shv}(\coprod_i X_i)$ as a coproduct of the collection $\{\widehat{\shv}(X_i)\}_i$ in the \infcat of (Postnikov complete) \inftopoi (and geometric morphisms between them).
\item Let $U^+_{\bullet}$ be an augmented simplicial colimit diagram in $\topo^{\text{\'{e}t}}$ such that $U_1\rightarrow U_0\times U_0$ is an injection and $U^+_{\bullet}|_{\simpop}$ is a groupoid object, then $\widehat{\shv}(U^+_{\bullet})$ is a colimit diagram of (Postnikov complete) \inftopoit. 
\end{enumerate}
\end{prop}
For the following proof, recall that a groupoid object $U_{\bullet}$ in an \infcat $\icat$ that admits finite limits is \emph{$n$-efficient} for $n\geq 0$ an integer if the map $U_1\rightarrow U_0\times U_0$ is $(n-2)$-truncated. It is a consequence of the $n$-categorical version of Giraud's theorem (\cite[Proposition 6.4.3.6, Proposition 6.4.4.9]{HTT}) that in an $n$-topos, $n$-efficient groupoids are effective.

\begin{proof}
Let $\topo_{\infty}^{\text{\'{e}t}}\subset\topo_{\infty}$ be the subcategory on the \'{e}tale geometric morphisms and note that by \cite[Corollary 6.3.5.9]{HTT} the functor $(\topo_{\infty}^{\text{\'{e}t}})_{/\xtop}\subset (\topo_{\infty})_{/\xtop}$ is a full subcategory for any \inftop $\xtop$. It is a consequence of \cite[Remark 6.3.5.10, Proposition 6.3.5.14]{HTT} that the unstraightening of the codomain evaluation $\fun(\Delta^1,\xtop)\rightarrow\xtop$ determines an equivalence $\xtop\simeq (\topo^{\text{\'{e}t}}_{\infty})_{/\xtop}$ and that the corresponding fully faithful functor $\xtop \hookrightarrow (\topo_{\infty})_{/\xtop}$ is stable under colimits. We now prove the assertions in the proposition.
\begin{enumerate}[$(1)$]
\item We first note that in case $f:X\rightarrow Y$ is an open embedding of topological spaces, the functor $\mathsf{Open}(Y)\simeq \mathsf{Open}(X)_{/f(Y)}\rightarrow\mathsf{Open}(X)$ left adjoint to the pullback along $f$ induces the left adjoint $\widehat{\shv}(Y)\simeq\widehat{\shv}(X)_{/jf(Y)}\rightarrow\widehat{\shv}(X)$ to the algebraic morphism $\widehat{\shv}(Y)\rightarrow \widehat{\shv}(X)$. In the general case, choose an open cover $\{U_i\rightarrow X\}_i$ such that each map $U_i\rightarrow Y$ is an open embedding of topological spaces and consider the \v{C}ech nerve of the map $h:\coprod_i\widehat{\shv}(U_i)\rightarrow\widehat{\shv}(X)$ in $\topo_{\infty}$. Since \'{e}tale geometric morphisms are stable under pullbacks in $\topo_{\infty}$, this \v{C}ech nerve lies in $(\topo_{\infty}^{\text{\'{e}t}})_{/\widehat{\shv}(X)}$. Under the equivalence $\widehat{\shv}(X)\simeq  (\topo_{\infty}^{\text{\'{e}t}})_{/\widehat{\shv}(X)}$, the map $\coprod_i\widehat{\shv}(U_i)\rightarrow\widehat{\shv}(X)$ corresponds to the map $\coprod_i j(U_i)\rightarrow 1_{\widehat{\shv}(X)}$ which is an effective epimorphism so we conclude that $\check{C}(h)_{\bullet}^+$ is a colimit diagram. Since the diagram $\check{C}(h)_{\bullet}$ lies in $(\topo_{\infty}^{\text{\'{e}t}})_{/\widehat{\shv}(Y)}$, we conclude that $\widehat{\shv}(X)\rightarrow\widehat{\shv}(Y)$ is \'{e}tale.
\item First assume that $X\rightarrow Z$ is an open embedding, then it follows from \cite[Remark 6.3.5.8]{HTT} that the diagram 
\[
\begin{tikzcd}
\widehat{\shv}(Q)\simeq \widehat{\shv}(Y)_{/jf(Q)}\ar[d]\ar[r] & \widehat{\shv}(Y)\ar[d] \\
\widehat{\shv}(X)\simeq \widehat{\shv}(Z)_{/jf(X)}\ar[r] & \widehat{\shv}(Z)
\end{tikzcd}
\]
is a pullback. Now choose an open cover $\{U_i\rightarrow X\}_{i\in I}$ such that $U_i\rightarrow Z$ is an open embedding for each $i$. Since we have already proven the result for pullbacks along open embeddings, we deduce that the \v{C}ech nerve of the map $\coprod_i\widehat{\shv}(U_i)\rightarrow \widehat{\shv}(X)$ is in each level a coproduct of \inftopoi of the form $\widehat{\shv}(U_{i_1}\cap \ldots\cap U_{i_n})$ for $i_1,\ldots,i_n \in I$. For each such each finite tuple, all squares in the diagram
\[
\begin{tikzcd}
 \widehat{\shv}(U_{i_1}\cap \ldots\cap U_{i_n}\times_ZQ) \ar[d,equal]\ar[r] & \widehat{\shv}(Q)\ar[d] \\
 \widehat{\shv}(U_{i_1}\cap \ldots\cap U_{i_n}\times_ZQ) \ar[d]\ar[r] & \widehat{\shv}(X)\times_{\widehat{\shv}(Z)}\widehat{\shv}(Y) \ar[d] \ar[r] & \widehat{\shv}(Y) \ar[d] \\
 \widehat{\shv}(U_{i_1}\cap \ldots\cap U_{i_n})\ar[r] & \widehat{\shv}(X) \ar[r]& \widehat{\shv}(Z) 
\end{tikzcd}
\]  
of \inftopoi are pullbacks by virtue of \cite[Remark 6.3.5.8]{HTT}. It follows that amalgamating the \v{C}ech nerve of the map $h:\coprod_i \widehat{\shv}(U_i\times_XQ)\rightarrow\widehat{\shv}(Q)$ with the map $\widehat{\shv}(Q)\rightarrow\widehat{\shv}(X)\times_{\widehat{\shv}(Z)}\widehat{\shv}(Y)$ is again a \v{C}ech nerve. During the proof of $(1)$ we verified that the \v{C}ech nerve of the map $h$ is a colimit diagram, so it suffices to show that under the equivalence $\widehat{\shv}(X)\times_{\widehat{\shv}(Z)}\widehat{\shv}(Y)\simeq (\topo_{\infty}^{\text{\'{e}t}})_{/\widehat{\shv}(X)\times_{\widehat{\shv}(Z)}\widehat{\shv}(Y)}$, the \'{e}tale geometric morphism $h':\coprod_i \widehat{\shv}(U_i\times_XQ)\rightarrow \widehat{\shv}(X)\times_{\widehat{\shv}(Z)}\widehat{\shv}(Y)$ corresponds to an effective epimorphism $V\rightarrow 1$ to the final object. This follows because $h'$ is a pullback of the \'{e}tale geometric morphims $\coprod_i \widehat{\shv}(U_i)\rightarrow\widehat{\shv}(X)$ which has this property as the $U_i$ cover $X$.
\item We have to show that the open sets $\{X_i\subset \coprod_i X_i\}$ determine an equivalence $\coprod_i X_i \simeq 1_{\widehat{\shv}(\coprod_iX_i)}$ (where we identify $X_i$ with its image under the Postnikov completed Yoneda embedding $\mathsf{Open}(\coprod_i X_i)\hookrightarrow\widehat{\shv}(\coprod_iX_i)$) which is easy to see as the sets $X_i$ cover the coproduct and are pairwise disjoint.
\item We first show that the diagram $\widehat{\shv}(U^+_{\bullet})$ is a \v{C}ech nerve in $(\topo^{\text{\'{e}t}}_{\infty})_{/\widehat{\shv}(U^+_{-1})}$. In view of $(2)$ it suffices to show that the diagram $U_{\bullet}^+$ is a \v{C}ech nerve in the category of topological spaces. It is standard that the subcategory inclusion $\topo^{\text{\'{e}t}}\subset\topo$ preserves pullbacks and colimits and the forgetful functor $\mathsf{Top}\rightarrow\mathsf{Set}$ is conservative and preserves limits and colimits, so it follows from the fact that 1-efficient groupoid objects are effective in $\mathsf{Set}$ that $U_{\bullet}^+$ is a \v{C}ech nerve (note that the injectivity of the map $U_1\rightarrow U_0\times U_0$ simply asserts the 1-efficiency of the groupoid of sets $U_{\bullet}$). As the diagram $\widehat{\shv}(U^{+}_{\bullet})$ of \inftopoi and \'{e}tale geometric morphisms between them is a \v{C}ech nerve, it suffices to show that the functor $f_!:\widehat{\shv}(U_{0})\rightarrow\widehat{\shv}(U_{-1})$ left adjoint to the \'{e}tale algebraic morphism $\widehat{\shv}(U_{-1})\rightarrow \widehat{\shv}(U_{0})$ determines an effective epimorphism $f_!(1_{\widehat{\shv}(U_{0})})\rightarrow 1_{\widehat{\shv}(U_{-1})}$. This follows because $U_0\rightarrow U_{-1}$ is an \'{e}tale covering.
\end{enumerate}
\end{proof}
\begin{proof}[Proof of Proposition \ref{prop:gluinglocringsp}]
We will first establish the following assertion.
\begin{enumerate}
\item[$(*)$] Let $K$ be a (small) simplicial set and let $f:K\rightarrow \mathsf{Top}^{\mathrm{loc}}(\sring)$ be a diagram that carries each edge of $K$ to an \'{e}tale morphism. Consider a colimit diagram $K^{\rhd}\rightarrow \mathsf{Top}$ extending the composition $q_{\sring}f$ and  suppose that the diagram 
\[  (K^{op})^{\lhd}\longrightarrow \mathsf{Top}^{op}\longrightarrow (\mathsf{Top}^{\mathrm{Pc}}_{\infty})^{op}\subset\catinfh \]
is a limit diagram of \infcats (and thus of Postnikov complete \inftopoi as the inclusions $(\topo_{\infty}^{\mathrm{Pc}})^{op}\subset\topo_{\infty}^{op}\subset\prl\subset \catinfh$ all preserve limits). Then there is an extension $\overline{f}:K^{\rhd}\rightarrow \mathsf{Top}^{\mathrm{loc}}(\sring)$ fitting into a commuting diagram 
\[
\begin{tikzcd}
K\ar[d,hook]\ar[r,"f"] & \mathsf{Top}^{\mathrm{loc}}(\sring)\ar[d,"q_{\sring}"] \\
K^{\rhd} \ar[r]\ar[ur,"\overline{f}"]& \mathsf{Top} 
\end{tikzcd}
\]
of simplicial sets satsifying the following conditions.
\begin{enumerate}[$(a)$]
\item The functor $\overline{f}$ carries each edge of $K^{\rhd}$ to an \'{e}tale morphism (in particular, $\overline{f}$ carries each edge to a $q_{\sring}$-Cartesian morphism).
\item The diagram $\overline{f}$ is a colimit diagram. 
\end{enumerate}
\end{enumerate}
We prove $(*)$. The functor $q_{\sring}:\mathsf{Top}(\sring)\rightarrow\mathsf{Top}$ is a Cartesian fibration associated to the functor 
\[ \begin{tikzcd}\widehat{\shv}_{\sring}(\_)^{op}: \mathsf{Top}^{op}\ar[r,"X\mapsto \widehat{\shv}(X)"] &[3em](\mathsf{Top}_{\infty}^{\mathrm{Pc}})^{op}\subset \mathsf{Top}_{\infty}^{op}\subset \prl \ar[r,"\fun^{\mathrm{R}}((\_)^{op}{,}\sring)"]&[6em] \prl\subset\catinfh\ar[r,"(\_)^{op}"]&\catinfh\end{tikzcd}\]
We will first show that the composition
\[\begin{tikzcd} (K^{op})^{\lhd}\ar[r,"f"]& \topo^{\loc}(\sring)^{op}\ar[r,"q_{\sring}"]&[2em]  \topo^{op}\ar[r,"X\mapsto \widehat{\shv}_{\sring}(X)"]&[5em]\catinfh\end{tikzcd}\]  
is a limit diagram. By assumption on the functor $f$, the composition $(K^{op})^{\lhd}\rightarrow \prl$ carrying $k\in K^{\rhd}$ to $\widehat{\shv}(q_{\sring}f(k))$ is a limit diagram. Since $f$ carries each edge $k\rightarrow k'$ to an \'{e}tale morphism by $(1)$ of Proposition \ref{prop:colimcomparetop}, the algebraic morphism of \inftopoi $\widehat{\shv}(q_{\sring}f(k'))\rightarrow\widehat{\shv}(q_{\sring}f(k))$ is \'{e}tale and therefore admits a left adjoint. It follows that the functor $(K^{op})^{\lhd}\rightarrow \prl$ factors through $\pr^{\mathrm{L},\mathrm{R}}$. Since for any presentable \infcat $\icat$ the functor
\[ \begin{tikzcd}\pr^{\mathrm{L},\mathrm{R}}\subset\prl \ar[r,"\_\otimes\icat"] & \prl  \subset\catinfh \end{tikzcd}\]
preserves limits by Lemma \ref{lem:tensoringprl}, we conclude that the diagram $(K^{op})^{\lhd}\rightarrow \catinfh$ carrying $k\in K^{\rhd}$ to $\widehat{\shv}_{\sring}(q_{\sring}f(k))$ is a limit diagram. Combining this result with the fact that taking opposites is an equivalence of $\catinfh$ (and in particular preserves limits), it follows from \cite[Proposition 5.2.2.36]{HA} that there exists an extension $K^{\rhd}\rightarrow \mathsf{Top}(\sring)$ of the composition $K\overset{f}{\rightarrow}\mathsf{Top}^{\mathrm{loc}}(\sring)\subset\mathsf{Top}(\sring)$ fitting into a commuting diagram
\[
\begin{tikzcd}
K\ar[d,hook]\ar[r,"f"] & \mathsf{Top}(\sring)\ar[d,"q_{\sring}"] \\
K^{\rhd} \ar[r]\ar[ur,"\overline{f}"]& \mathsf{Top}, 
\end{tikzcd}
\] 
and that $\overline{f}$ is a $q_{\sring}$-colimit diagram and carries every edge of $K^{\rhd}$ to a $q_{\sring}$-Cartesian edge. Since the diagram $K^{\rhd}\rightarrow\mathsf{Top}$ is a colimit diagram, $\overline{f}$ is a colimit diagram (\cite[Lemma 4.3.1.5]{HTT}). Recall that the subcategory $\mathsf{Top}^{\text{\'{e}t}}\subset\topo$ on the local homeomorphisms preserves colimits, so to prove $(*)$, it suffices to show that the diagram $\overline{f}$ lies in $\mathsf{Top}^{\loc}(\sring)$, the full subcategory of derived \emph{locally} $\cinfty$-ringed spaces. To see this, we note that we have an \'{e}tale covering $\{f(k)\rightarrow \overline{f}(\infty)\}_{k\in K}$ of the colimit $\overline{f}(\infty)$ so we conclude as being locally ringed is a local property of derived $\cinfty$-ringed spaces. It follows from Proposition \ref{prop:colimcomparetop} that in case $f:K\rightarrow \topo^{\loc}(\sring)$ is a diagram of the kind specified in $(1)$ and $(2)$ above (so that $K$ is either a discrete simplicial set or the category $\simpop$), then $f$ satisfies the condition given in $(*)$. To complete the proof, it suffices to show that the resulting extension $\overline{f}$ satisfying $(a)$ and $(b)$ above is also a colimit diagram in $\topo^{\loc}(\sring)^{\text{\'{e}t}}$, in other words, that for a morphism $g:\colim_K f\rightarrow (Y,\Of_Y)$ of locally ringed spaces, if for each $k\in K$ the composition $f(k)\rightarrow (Y,\Of_Y)$ is \'{e}tale, then $g$ is \'{e}tale. A map $Q\rightarrow Z$ of topological spaces such that $Q$ admits a cover $\{Q_i\rightarrow Q\}_i$ by local homeomorphisms such that the compositions $Q_i\rightarrow  Z$ are local homeomorphisms is itself a local homeomorphism, so it suffices to show that the map $g$ is $q_{\sring}$-Cartesian. Write $(X,\Of_X)$ for $\colim_Kf$ and $(X_k,\Of_{X_k})$ for $f(k)$ so that we have a collection of maps $\{h_k:X_k\rightarrow X\}_k$ of topological spaces. Choose a $q_{\sring}$-Cartesian lift $\tilde{g}:(X,\Of_X')\rightarrow (Y,\Of_Y)$ of $q_{\sring}(g)$ terminating at $(Y,\Of_Y)$, then we wish to show that the map $(X,\Of_X)\rightarrow (X,\Of_X')$ is an equivalence. Since the functors
\[ \fun^{\mathrm{Cart}}_{\mathsf{Top}}(K,\topo(\sring)) \longleftarrow \fun_{\topo}^{\mathrm{Cart}}(K^{\rhd},\topo(\sring))\overset{\ev_{\infty}}{\longrightarrow} \widehat{\shv}(X)^{op} \]
(where `Cart' denotes we are taking full subcategories spanned by Cartesian sections) are equivalences, it suffices to argue that for every $k\in K$, the left vertical map in the diagram 
\[
\begin{tikzcd}
(X_k,\Of_{X_k}) \ar[r]\ar[d] & (X,\Of_X) \ar[d] \ar[dr,"g"]& \\
(X_k,h_k^*\Of_X') \ar[r] & (X,\Of_X') \ar[r,"\tilde{g}"'] & (Y,\Of_Y)
\end{tikzcd}
\]
where both horizontal maps in the square are $q_{\sring}$-Cartesian, is an equivalence. This follows since by assumption the upper composition is $q_{\sring}$-Cartesian and by construction the lower composition is also $q_{\sring}$-Cartesian.
\end{proof}

\begin{rmk}
The preceding result only establishes the existence of a certain class of colimit diagrams, which corresponds to the fact that in the \infcat of derived locally $\cinfty$-ringed spaces, we are only allowed to glue along \emph{equivalences}. This inconvenient feature (the lack of \'{e}tale colimits) indicates that the \infcat $\mathsf{Top}^{\loc}(\sring)$ is a somewhat artificial construct; a much more natural (albeit less familiar) \infcat of structured spaces is the \infcat of derived locally $\cinfty$-ringed \emph{\inftopoit,} where we can take arbitrary colimits of diagrams with \'{e}tale transition maps.   
\end{rmk}
\begin{rmk}(\cite[Remark 3.2.14]{cinftyI}).
Every derived locally $\cinfty$-ringed space $(X,\Of_X)$ has an \emph{underlying} locally $\cinfty$-ringed space $(X,\pi_0(X))$ obtained as follows: view $\Of_X$ as a product preserving functor
\[ \Of_X:\cartsp\longrightarrow \widehat{\shv}(X),  \]
then $\pi_0(\Of_X)$ is the composition
\[ \cartsp\overset{\Of_X}{\longrightarrow} \widehat{\shv}(X)\overset{\tau_{\leq 0}}{\longrightarrow}  \widehat{\shv}(X),  \]
which preserves finite products as $\tau_{\leq 0}$ does. Since the $\pi_0(\Of_X)$ takes $0$-truncated values, the adjoint functor $\widehat{\shv}(X)\rightarrow \sring$ takes $0$-truncated values (\cite[Lemma 2.2.36]{cinftyI}). The operation 
\[ (X,\Of_X)\longmapsto (X,\pi_0(\Of_X)) \]
preserves locality and determines a right adjoint to the inclusion $\mathsf{Top}^{\mathrm{loc}}(\tau_{\leq 0}\sring)\subset\mathsf{Top}^{\mathrm{loc}}(\sring)$ (\cite[Proposition 2.2.44]{cinftyI}). More generally, for every $n\geq 0$, we have a truncation functor
\[ (X,\Of_X)\longmapsto (X,\tau_{\leq n}(\Of_X)) \]
right adjoint to the inclusion $\mathsf{Top}^{\mathrm{loc}}(\tau_{\leq n}\sring)\subset\mathsf{Top}^{\mathrm{loc}}(\sring)$.
\end{rmk}
We now summarize some results of \cite{cinftyI}, which are derived, $\infty$-categorical versions of similar results of Joyce \cite{Joy2}.
\begin{thm}\label{thm:spectrumglobalsections} The global sections functor
\[ \Gamma:\mathsf{Top}(\sring)\longrightarrow \sring^{op}\]
admits a right adjoint. Let $\spec$ denote this right adjoint, then the following hold true.
\begin{enumerate}[$(1)$]
    \item Let $M$ be a manifold (possibly with boundary or corners), then $\spec$ carries the $\cinfty$-ring of functions on $M$ to the pair $(M,\cinfty_M)$.
    \item There exists a large full subcategory $\sring_{\mathsf{gmt}}\subset\sring$ of \emph{geometric} derived $\cinfty$-rings (see \cite[Defintion 3.2.16]{cinftyI}) on which $\spec$ is fully faithful. The full subcategory $\sring_{\mathsf{gmt}}$ contains $\sring_{\fp}$, the finitely presented derived $\cinfty$-rings (but much more general objects as well, see below).
    \item An object $(X,\Of_X)\in \mathsf{Top}^{\mathrm{loc}}(\sring)$ is equivalent to an object of the form $\spec\,A$ for $A$ geometric if and only if the following conditions are satisfied.
    \begin{enumerate}[$(i)$]
        \item The space $X$ is Hausdorff and Lindel\"{o}f (that is, every open cover of $X$ admits a countable refinement).
        \item The pair $(X,\Of_X)$ is \emph{$\cinfty$-regular}: consider the set $S\subset \Hom_{\mathsf{Top}}(X,\R) $ defined as the image of the map 
        \[ \Hom_{\mathsf{Top}^{\mathrm{loc}}(\sring)}((X,\Of_X),(\R,\cinfty_{\R}))\longrightarrow\Hom_{\mathsf{Top}}(X,\R). \]
        Then the collection $\{f^{-1}(U)\}_{f\in S,U\in \mathsf{Open}(\R)}\subset\mathsf{Open}(X)$
        is a basis for the topology on $X$.
    \end{enumerate}
\end{enumerate}
\end{thm}
\begin{rmk}\label{rmk:lindelofunity}
Conditions $(i)$ and $(ii)$ above guarantee in particular that the sheaf of $\cinfty$-rings $\pi_0(\Of_X)$ admits partitions of unity, which is crucial in order to control the homotopy theory of sheaves of modules over $\Of_X$.     
\end{rmk}
It now stands to reason to make the following definition
\begin{enumerate}
\item[$(\bullet)$] A derived locally $\cinfty$-ringed space $(X,\Of_X)$ is an \emph{affine derived $\cinfty$-scheme} if there is a derived $\cinfty$-ring $A$ and an equivalence $\spec\,A\simeq (X,\Of_X)$. 
\end{enumerate}
However, Theorem \ref{thm:spectrumglobalsections} and Remark \ref{rmk:lindelofunity} indicate that the class of \emph{all} such affine derived $\cinfty$-schemes is not as well behaved as we would like; it is better to restrict to objects of the form $\spec\,A$ for $A$ geometric. This still comprises a very large subcategory containing all the examples of geometric interest. In fact, it's convenient to impose even further restrictions on the collection of affine objects to guarantee better behaviour: note that the property of being Lindel\"{o}f is not stable under the formation of products or open subspaces; following Joyce (\cite[Corollary 4.42]{Joy2}), we will instead consider the smaller class affine derived $\cinfty$-schemes $(X,\Of_X)$ for which $X$ is \emph{second countable} (since $X$ is regular -as implied by $\cinfty$-regularity- this assumption implies that $X$ must be metrizable). The main \infcat of interest in this paper is the \infcat of sheaves on $\daff$. Since the \infcat of objects of the form $(X,\Of_X)=\spec\,A$ with $X$ second countable is not small, we will also impose a cardinality bound on the size of our affines; this is a technical issue the reader may safely ignore if she or he so wishes.
\begin{defn}\label{defn:affinederived}
Let $\kappa$ be an uncountable regular cardinal (sufficiently large for all future purposes). Let $(X,\Of_X)$ be a derived locally $\cinfty$-ringed space.
\begin{enumerate}[$(1)$]
    \item $(X,\Of_X)$ is an \emph{affine derived $\cinfty$-scheme} if $X$ is second countable and there is a $\kappa$-compact geometric derived $\cinfty$-ring $A$ and an equivalence $\spec\,A\simeq (X,\Of_X)$; equivalently, if $X$ is second countable, $(X,\Of_X)$ satisfies conditions $(i)$ and $(ii)$ of $(3)$ of Theorem \ref{thm:spectrumglobalsections} and $\Gamma(\Of_X)$ is $\kappa$-compact. We let $\daff$ denote the full subcategory spanned by affine derived $\cinfty$-schemes. 
    \item $(X,\Of_X)$ is an \emph{affine derived $\cinfty$-scheme of finite presentation} if there is a derived $\cinfty$-ring $A$ of finite presentation and an equivalence $\spec\,A\simeq (X,\Of_X)$. We let $\daff_{\fp}$ denote the full subcategory spanned by affine derived $\cinfty$-schemes of finite presentation. 
\end{enumerate}
We denote by $\iota_{\mathsf{Sm}}$ and $\iota_{\fp}$ the full subcategory inclusions $\mathsf{Mfd}\subset\daff$ and $\daff_{\fp}\subset\daff$ respectively.
\begin{rmk}
Note that the following properties of derived locally $\cinfty$-ringed spaces $(X,\Of_X)$ are stable under open subspaces.
\begin{enumerate}[$(1)$]
\item The property that $X$ is second countable.
\item The property that $X$ is Hausdorff.
\item The property that $(X,\Of_X)$ is $\cinfty$-regular. 
\end{enumerate}
It follows that the full subcategory of $\topo^{\loc}(\sring)$ spanned by objects $(X,\Of_X)$ that satisfy $(1)$ through $(3)$ is stable under open subspaces. Choosing a regular cardinal $\kappa>>\omega_1$ (where $\omega_1$ is the first uncountable regular cardinal) one can show that the property that $\Gamma(\Of_X)$ is $\kappa$-compact is also stable under open subspaces, so that the property of being an affine derived $\cinfty$-scheme is stable under open subspaces for a sufficiently large regular cardinal $\kappa$.
\end{rmk}
\end{defn}
\begin{rmk}
We make a terminological comment. In the preceding definition, we have not opted to call any subclass of affine derived $\cinfty$-schemes `derived manifolds'. Even though the term `derived manifold' has been standardized to an extent, we would prefer to avoid it since
\begin{enumerate}[$(1)$]
    \item taken literally, it is oxymoronic (or trivial at best) in the same way `derived smooth variety' is. 
    \item the term does not transparently convey the class of geometric objects under consideration, and different authors mean different things by `derived manifolds'. For Spivak \cite{spivak}, a derived manifold is what we would call a quasi-smooth derived $\cinfty$-scheme locally of finite presentation, whereas for Behrend-Liao-Xu \cite{BLX} a derived manifold is what we would call an \emph{affine} derived $\cinfty$-scheme of finite presentation (which need not be quasi-smooth). 
\end{enumerate}
\end{rmk}
\begin{prop}
The full subcategory $\daff\subset \mathsf{Top}^{\mathrm{loc}}(\sring)$ is stable under countable limits. The full subcategory $\daff_{\fp}\subset \mathsf{Top}^{\mathrm{loc}}(\sring)$ is stable under finite limits.     
\end{prop}
\begin{proof}
Using that the projection $\mathsf{Top}^{\mathrm{loc}}(\sring)\rightarrow \mathsf{Top}$ preserves limits, it is not hard to see that the condition of $\cinfty$-regularity is stable under arbitrary small limits in $\mathsf{Top}^{\mathrm{loc}}(\sring)$. Since the collection of Hausdorff spaces is similarly stable under limits in $\mathsf{Top}$, it suffices to argue that the collection of second countable spaces are stable under countable limits. This follows from the fact that the collection of second countable spaces is stable under countable products and passing to subspaces. For the claim about affine derived $\cinfty$-schemes of finite presentation, suppose that $f:K\rightarrow \daff_{\fp}$ is a finite diagram, then by Theorem \ref{thm:spectrumglobalsections}, $f$ is equivalent to a diagram $K\overset{f'}{\rightarrow} \sring_{\fp}\rightarrow \mathsf{Top}^{\mathrm{loc}}(\sring)^{op}$ where the second functor is the spectrum. Since $\spec$ (as a left adjoint) preserves colimits and $\sring_{\fp}\subset\sring$ is stable under finite colimits, we conclude. 
\end{proof}

Since our affine derived objects are \infcats of structured spaces, they inherit a natural topology from the \infcat of structured spaces.

\begin{defn}[Derived $\cinfty$-scheme]
Let $(X,\Of_X)$ be a derived locally $\cinfty$-ringed space.
\begin{enumerate}[$(1)$]
    \item $(X,\Of_X)$ is a \emph{derived $\cinfty$-scheme} if there is a covering family $\{(U_i,\Of_X|_{U_i})\rightarrow (X,\Of_X)\}_i$ such that $(X,\Of_X|_{U_i})$ is an affine derived $\cinfty$-scheme for all $i$. We let $\dsch\subset\mathsf{Top}^{\mathrm{loc}}(\sring)$ be the full subcategory spanned by derived $\cinfty$-schemes
    \item $(X,\Of_X)$ is a \emph{derived $\cinfty$-scheme locally of finite presentation} if there is a covering family $\{(U_i,\Of_X|_{U_i})\rightarrow (X,\Of_X)\}_i$ such that $(X,\Of_X|_{U_i})$ is an affine derived $\cinfty$-scheme locally of finite presentation for all $i$. We let $\dsch_{\fp}\subset \dsch$ denote the full subcategory spanned by derived $\cinfty$-schemes locally of finite presentation.
    \item More generally, suppose that $\LS\subset\daff$ is a property of affine derived $\cinfty$-schemes stable under open subspaces, then $(X,\Of_X)$ is a \emph{derived $\cinfty$-scheme locally having the property $\LS$} if there is a covering family $\{(U_i,\Of_X|_{U_i})\rightarrow (X,\Of_X)\}_i$ such that $(X,\Of_X|_{U_i})\in \LS$ for all $i$. We let $\dsch_{\LS}\subset \dsch$ denote the full subcategory spanned by derived $\cinfty$-schemes that locally have the property $\LS$.
\end{enumerate}
\end{defn}

We establish some closure properties of the class of derived $\cinfty$-schemes under limits and colimits. We will write $\dsch^{\et}_{\LS}$ for the subcategory $\dsch_{\LS}\cap \topo^{\loc}(\sring)^{\et}\subset \dsch_{\LS}$ on the \'{e}tale morphisms of derived $\cinfty$-schemes locally having the property $\LS$. Similarly, we have the subcategory $\dsch^{\mathsf{Open}}_{\LS}\subset \dsch_{\LS}$ on the open embeddings.
\begin{prop}\label{prop:schemescolimstable}
Let $\LS\subset\daff$ be a property of affine derived $\cinfty$-schemes stable under open subspaces. Then the following hold true.
\begin{enumerate}[$(1)$]
\item If $(X,\Of_X)$ is a derived locally $\cinfty$-ringed space that admits an \'{e}tale cover $\{(X_i,\Of_{X_i})\rightarrow (X,\Of_X)\}_i$ such that each $(X_i,\Of_{X_i})$ is a derived $\cinfty$-scheme locally having the property $\LS$, then $(X,\Of_X)$ is a derived $\cinfty$-scheme locally having the property $\LS$.
\item If $(X,\Of_X)$ is a derived $\cinfty$-scheme that locally has the property $\LS$ and $(Y,\Of)\rightarrow (X,\Of_X)$ is an \'{e}tale morphism of derived locally $\cinfty$-ringed spaces, then $(Y,\Of_Y)$ is a derived $\cinfty$-scheme that locally has the property $\LS$.
\item Consider a pullback diagram 
\[
\begin{tikzcd}
(Q,\Of_Q)\ar[d]\ar[r] & (Y,\Of_Y)\ar[d] \\
(Z,\Of_Z)\ar[r] & (X,\Of_X)
\end{tikzcd}
\]
of derived locally $\cinfty$-ringed spaces where the lower horizontal map is \'{e}tale and the $(Y,\Of_Y)$ is a derived $\cinfty$-scheme locally having the property $\LS$. Then $(Q,\Of_Q)$ is a derived $\cinfty$-scheme locally having the property $\LS$. In particular, the full subcategory $ \dsch_{\LS}^{\text{\'{e}t}}\subset \mathsf{Top}^{\mathrm{loc}}(\sring)^{\et}$ is stable under pullbacks.
\item The \infcats $\dsch_{\LS}^{\text{\'{e}t}}$ and $\dsch_{\LS}$ admit small coproducts and the inclusions $\dsch_{\LS}^{\text{\'{e}t}}\subset \dsch_{\LS}$ and $\dsch_{\LS}^{\text{\'{e}t}}\subset\topo^{\loc}(\sring)^{\et}$ preserve small coproducts.
\item Let $U_{\bullet}$ be a groupoid object of $\dsch_{\LS}^{\text{\'{e}t}}$. If the map 
\[  q_{\sring}(U_1)\longrightarrow q_{\sring}(U_0)\times q_{\sring}(U_0)  \]
of topological spaces is injective, then $U_{\bullet}$ admits a colimit in $\dsch_{\LS}^{\et}$ which is preserved by the inclusions $\dsch_{\LS}^{\et}\subset \dsch_{\LS}$ and $\dsch_{\LS}^{\text{\'{e}t}}\subset\topo^{\loc}(\sring)^{\et}$. Moreover, the resulting augmented simplicial object $U^+_{\bullet}$ is a \v{C}ech nerve in both $\dsch_{\LS}^{\et}$ and $\dsch_{\LS}$ (that is, $U^{+}_{\bullet}$ is an \emph{effective groupoid}). 
\item Let $U^+_{\bullet}$ be a \v{C}ech nerve in $\dsch_{\LS}^{\et}$. Suppose that the map
$q_{\sring}(U^+_1)\rightarrow q_{\sring}(U^+_0)\times q_{\sring}(U^+_0)$ is an injective map of topological spaces and that the map $q_{\sring}(U^+_0)\rightarrow q_{\sring}(U^+_{-1})$ is a surjective map of topological spaces. Then the diagram $U_{\bullet}^+$ is a colimit diagram which is preserved by the inclusions $\dsch_{\LS}^{\et}\subset \dsch_{\LS}$ and $\dsch_{\LS}^{\text{\'{e}t}}\subset\topo^{\loc}(\sring)^{\et}$.
\end{enumerate}
\end{prop}
\begin{proof}
For $(1)$, choose for each $i$ a covering family $\{(U_{ij},\Of_{X_i}|_{U_{ij}})\hookrightarrow (X_i,\Of_{X_i})\}_j$ of $(X_i,\Of_{X_i})$ such that each composition $(U_{ij},\Of_{X_i}|_{U_{ij}})\rightarrow (X,\Of_X)$ is an open embedding. Since $(X_i,\Of_{X_i})$ is a derived $\cinfty$-scheme having the property $\LS$ and $\LS$ is stable under open subspaces, we may assume that $(U_{ij},\Of_{X_i}|_{U_{ij}})$ is an affine derived $\cinfty$-scheme having the property $\LS$ for all $i,j$. For $(2)$, choose an open cover $\{(U_i,\Of_X|_{U_i})\rightarrow (X,\Of_X)\}_i$ by affine derived $\cinfty$-schemes having the property $\LS$ and a covering family $\{(V_j,\Of_Y|_{V_j})\hookrightarrow (Y,\Of_Y)\}_j$ of $(Y,\Of_Y)$ such that the composition $(V_j,\Of_Y|_{V_j})\rightarrow (X,\Of_X)$ is an open embedding for all $j$, then $(V_j,\Of_Y|_{V_j})\times_{(X,\Of_X)}(U_i,\Of_X|_{U_i})$ is an affine derived $\cinfty$-scheme locally having the property $\LS$ since $\LS$ is stable under open subspaces, and the family 
\[   \{(V_j,\Of_Y|_{V_j})\times_{(X,\Of_X)}(U_i,\Of_X|_{U_i}) \hooklongrightarrow (Y,\Of_Y)\}_{ij}\]
is a covering family exhibiting $(Y,\Of_Y)$ as a derived $\cinfty$-scheme locally having the property $\LS$. Note that $(3)$ follows immediately from $(2)$ and $(4)$ follows from $(1)$ and Proposition \ref{prop:gluinglocringsp}. The colimit described in $(5)$ exists in $\topo^{\loc}(\sring)^{\et}$ and the inclusion $\topo^{\loc}(\sring)\subset \topo^{\loc}(\sring)^{\et}$ preserves this colimit by Proposition \ref{prop:gluinglocringsp}, and the colimit lies in $\dsch_{\LS}^{\et}$ by $(1)$. It remains to argue that $U^+_{\bullet}$ is \v{C}ech nerve, but this follows from the proof of Proposition \ref{prop:gluinglocringsp}; more precisely, we just argued that the functor
\[\dsch_{\LS}^{\text{\'{e}t}}\subset \mathsf{Top}^{\mathrm{loc}}(\sring)^{\text{\'{e}t}}\overset{q_{\sring}}{\longrightarrow} \mathsf{Top}^{\text{\'{e}t}}\subset \mathsf{Top}\]
preserves our colimit. It also preserves pullbacks by $(3)$ and is conservative since $\mathsf{Top}^{\mathrm{loc}}(\sring)^{\text{\'{e}t}}\rightarrow \mathsf{Top}^{\text{\'{e}t}}$ is a right fibration, so it suffices to show that a groupoid object $V_{\bullet}$ in $\mathsf{Set}$ for which $V_1\rightarrow V_0\times V_0$ is injective is effective, which follows as $\mathsf{Set}$ is a topos (1-efficient groupoids are effective in a 1-topos). For $(6)$, the same reasoning as for $(5)$ reduces us to the assertion that a \v{C}ech nerve $V^+_{\bullet}$ in $\mathsf{Set}$ for which $V_1\rightarrow V_0\times V_0$ is injective and $V_0\rightarrow V_{-1}$ surjective is a colimit diagram. 
\end{proof}
\begin{rmk}
Note that the subcategory inclusion $\dsch_{\LS}^{\text{\'{e}t}}\subset \dsch_{\LS}$ is stable under pullbacks, but it is \emph{not} stable under products.  
\end{rmk}

\begin{defn}[Derived $\cinfty$-stacks]
Let $\LS$ be a property of affine derived $\cinfty$-schemes stable under open subspaces, then we let $\shv(\LS)\subset\pshv(\LS)$ denote the full subcategory spanned by those presheaves that are sheaves for the \'{e}tale topology; this is an \inftopt. For $\LS=\daff$, we write $\dstack$ for the \inftop $\shv(\daff)$; this is the \inftop of \emph{derived $\cinfty$-stacks}.
\end{defn}
Let $\LS\subset\daff$ be a property stable under open subspaces and let $\gamma_{\LS}:\LS\subset\dsch_{\LS}$ denote the full subcategory inclusion. Consider the composition 
\[ \dsch_{\LS}\overset{j}{\hooklongrightarrow}\pshv(\dsch_{\LS})\overset{\gamma^*_{\LS}}{\longrightarrow} \pshv(\LS), \]
then the functor $\gamma^*_{\LS}$ admits a left adjoint $\gamma_{\LS!}$ given by left Kan extension along the full subcategory inclusion $\gamma_{\LS}:\LS^{op}\subset\dsch^{op}_{\LS}$ (\cite[Proposition 4.3.2.17]{HTT}). We will let $j_{\sch_{\LS}}$ denote the composition $\gamma_{\LS}^*\circ j:\dsch_{\LS}\rightarrow \pshv(\LS)$. We let $\shv(\dsch_{\LS})\subset\pshv(\dsch_{\LS})$ denote the full subcategory spanned by sheaves for the \'{e}tale topology. Since the inclusion $\LS\subset\dsch_{\LS}$ preserves pullbacks along open embeddings and carries covering families to covering families, the functor $\gamma^*_{\LS}$ carries sheaves to sheaves.
\begin{rmk}\label{rmk:sheafificationschemes}
Since $\dsch_{\LS}$ is not a small \infcatt, the \infcat $\pshv(\dsch_{\LS})$ is not presentable and thus not an \inftopt, and we cannot apply the general theory of strongly reflective localizations of presentable \infcats (\cite[Section 5.5.4]{HTT}) to deduce that there is a sheafification functor left adjoint to the full subcategory inclusion $\shv(\dsch_{\LS})\subset\pshv(\dsch_{\LS})$ (this is nevertheless true, but we will not use or prove this here).
\end{rmk} 
\begin{prop}\label{prop:resyonedaschemes}
Let $\LS$ be a property of affine derived $\cinfty$-schemes stable under open subspaces, then the following hold true. 
\begin{enumerate}[$(1)$]
\item The functor $j_{\sch_{\LS}}$ carries $\dsch_{\LS}$ into the full subcategory of sheaves on $\LS$.
\item Let $U_{\bullet}$ be a groupoid object of $\dsch_{\LS}^{\et}$ satisfying the condition in $(2)$ of Proposition \ref{prop:gluinglocringsp} or $(5)$ of Proposition \ref{prop:schemescolimstable} (which is then also an effective groupoid object of $\dsch_{\LS}$ by $(5)$ of Proposition \ref{prop:schemescolimstable}). Then the induced functor $\dsch_{\LS}\rightarrow \shv(\LS)$ (which we also denote by $j_{\sch_{\LS}}$) preserves geometric realizations of $U_{\bullet}$.
\item The functor $j_{\sch_{\LS}}:\dsch_{\LS}\rightarrow\shv(\LS)$ preserves coproducts.
\item The functor $j_{\sch_{\LS}}:\dsch_{\LS}\rightarrow\shv(\LS)$ is fully faithful.
\end{enumerate}
\end{prop}
\begin{proof}
For $(1)$, we invoke that $\gamma^*_{\LS}$ preserves sheaves to conclude that it suffices to argue that every representable presheaf on $\dsch_{\LS}$ is a sheaf. Let $\{U_i \rightarrow X\}_i$ be an open covering of derived $\cinfty$-schemes that locally have the property $\LS$ determining a map $h':\coprod_i U_i\rightarrow X$ of derived $\cinfty$-schemes that locally have the property $\LS$. Consider the diagram 
\[
\begin{tikzcd}
\coprod_i j(U_i)\ar[rr,"f"]\ar[dr,"h"] && j(\coprod_i U_i)\ar[dl,"j(h')"] \\
&j(X)
\end{tikzcd}
\]
in the \infcat $\pshv(\dsch_{\LS})$. We wish to show that for every $Y\in\dsch_{\LS}$, the functor 
\[ \Hom_{\pshv(\dsch_{\LS})}(\_,j(Y)):\pshv(\dsch_{\LS})^{op}\longrightarrow\spa \]
represented by $j(Y)$ carries the \v{C}ech nerve $\check{C}(h)^+_{\bullet}$ to a limit diagram in $\spa$. The functor represented by $Y$ carries $f$ to an equivalence by the last part of the proof of \cite[Proposition 5.3.6.2]{HTT}. The map $f$ induces a map of augmented simplicial objects $\check{C}(h)^+_{\bullet}\rightarrow\check{C}(j(h'))^{+}_{\bullet}$ which is carried to an equivalence by $\Hom_{\pshv(\dsch_{\LS})}(\_,j(Y))$ (as $f$ is carried to an equivalence by this functor) so it suffices to show that $\Hom_{\pshv(\dsch_{\LS})}(\_,j(Y))$ carries $\check{C}(j(h'))^+_{\bullet}$ to a colimit diagram. Since the Yoneda embedding preserves limits, we are reduced to proving that the functor
\[ \Hom_{\dsch_{\LS}}(\_,Y):\dsch_{\LS}^{op}\longrightarrow\spa  \]
represented by $Y$ carries the \v{C}ech nerve of $h'$ to a limit diagram, that is, that the \v{C}ech nerve $\check{C}(h')^+_{\bullet}$ is a colimit diagram in $\dsch_{\LS}$. The map $\check{C}(h')^+_{1}\rightarrow \check{C}(h')^+_{0}\times  \check{C}(h')^+_{0}$ is given by 
\[ \coprod_{ij}  U_i\times_XU_j\longrightarrow \coprod_{ij}U_i\times U_j \]
which is an injection on the underlying topological spaces since $U_i\rightarrow X$ is an open embedding for each $i$. Invoking $(6)$ of Proposition \ref{prop:schemescolimstable}, we deduce that $\check{C}(h')^+_{\bullet}$ is a colimit diagram. We proceed with $(2)$. Let $U_{\bullet}$ be a groupoid object of $\dsch^{\text{\'{e}t}}_{\LS}$ satisfying the condition of $(5)$ Proposition \ref{prop:schemescolimstable} and let $U_{\bullet}^+$ be a colimit diagram in $\dsch^{\text{\'{e}t}}_{\LS}$ (and $\dsch_{\LS}$) extending $U_{\bullet}$, then we wish to show that $j_{\sch_{\LS}}(U^+_{\bullet})$ is a colimit diagram. Since $j_{\sch_{\LS}}$ preserves limits and $U_{\bullet}$ is an effective groupoid, the diagram $j_{\sch_{\LS}}(U^+_{\bullet})$ is a \v{C}ech nerve. Since groupoids are effective in $\dstack_{\LS}$, it suffices to argue that the map $j_{\sch_{\LS}}(U_0^+)\rightarrow j_{\sch_{\LS}}(U_{-1}^+)$ is an effective epimorphism. Choose an open covering $\{\spec\,A_i\rightarrow U^+_0\}_i$ of the derived $\cinfty$-scheme $U^+_0$ by affines having the property $\LS$ such that the composition $\spec\,A_i\rightarrow U^+_{-1}$ is an open embedding for all $i$, then it suffices to argue that the composition $\coprod_i \spec\,A_i\rightarrow j_{\sch_{\LS}}(U^+_{-1})$ is an effective epimorphism, that is, we should show that for any map $\spec\,B\rightarrow j_{\sch_{\LS}}(U^+_{-1})$ of sheaves on $\LS$, the collection $\{\spec\,B\times_{j_{\sch_{\LS}}(U^+_{-1})}\spec\,A_{i}\rightarrow \spec\,B\}_i$ determines an effective epimorphism $\coprod_i \spec\,B\times_{j_{\sch_{\LS}}(U^+_{-1})}\spec\,A_{i}\rightarrow \spec\,B$. Since we can identify the map $\spec\,B\times_{j_{\sch_{\LS}}(U^+_{-1})}\spec\,A_{i}\rightarrow \spec\,B$ of sheaves with the open embedding $\spec\,B\times_{U^+_{-1}}\spec\,A_{i}\rightarrow \spec\,B$ in $\LS$ (after applying the Yoneda embedding), we will be done once we show that the collection $\{\spec\,B\times_{U^+_{-1}}\spec\,A_{i}\rightarrow \spec\,B\}_i$ is an open covering of $\spec\,B$, but this follows from the fact that $U_0^+\rightarrow U^+_{-1}$ is an \'{e}tale covering. Now let $\{U_i\}_i$ be small collection of derived $\cinfty$-schemes having the property $\LS$, then we wish to show that the map 
\[ h:\coprod j_{\sch_{\LS}}(U_i)\longrightarrow j_{\sch_{\LS}}(\coprod_iU_i) \]
is an equivalence. To see that this map is an effective epimorphism, we note that for any map $\spec\,A\rightarrow j_{\sch_{\LS}}(\coprod_iU_i)$, the collection $\{ \spec\,A\times_{\coprod_iU_i}U_j\rightarrow\spec\,A\}_j$ is an open covering. Since $U_j\times_{\coprod_{U_i}}U_k=U_j$ if $j=k$ and is empty otherwise, the \v{C}ech nerve of the map $h$ is essentially constant, that is, $h$ is an equivalence. For $(4)$, we wish to show that for every pair $(X,\Of_X),(Y,\Of_Y)\in\dsch_{\LS}$, the map
\[\Hom_{\dsch_{\LS}}((X,\Of_X),(Y,\Of_Y))\longrightarrow\Hom_{\shv(\LS)}(j_{\sch_{\LS}}(X,\Of_X),j_{\sch_{\LS}}(Y,\Of_Y)))\]
is an equivalence. Choose an open cover $\{(U_i,\Of_X|_{U_i})\hookrightarrow (X,\Of_X))\}_i$ of $(X,\Of_X)$ by affine derived $\cinfty$-schemes that have the property $\LS$, then the \v{C}ech nerve of the map 
\[ h:\coprod_i(U_i,\Of_X|_{U_i})\longrightarrow  (X,\Of_X)  \]
in $\dsch_{\LS}$ is a groupoid object in $\dsch_{\LS}^{\text{\'{e}t}}$ satisfying the condition in $(6)$ of Proposition \ref{prop:schemescolimstable}. It follows that $\check{C}(h)_{\bullet}^+$ is a colimit diagram. It follows from $(2)$ and $(3)$ that we may assume that $(X,\Of_X)$ is an affine derived $\cinfty$-scheme having the property $\LS$, in which case the assertion is obvious.
\end{proof}
\begin{cor}
Let $\LS\subset\daff$ be a property of affine derived $\cinfty$-schemes stable under open subspaces, then the \'{e}tale topology on $\LS$ is subcanonical.
\end{cor}
\begin{defn}
In case a sheaf $X$ on affine derived $\cinfty$-schemes locally having the property $\LS$ lies in the image of $j_{\mathsf{Sch}_{\LS}}$, we will say that $X$ is \emph{representable by a derived $\cinfty$-scheme locally having the property $\LS$}.
\end{defn}
Because the sites $\LS$ for properties of affine derived $\cinfty$-schemes stable under open subspaces inherit their topology from the open cover topology on the category of topological spaces, they have some particular properties, some of which we now record.
\begin{lem}\label{lem:detectetiononspaces}
Let $\LS\subset\daff$ be a property of affine derived $\cinfty$-schemes stable under open subspaces. Let $(X,\Of_X)$ be an affine derived $\cinfty$-scheme having the property $\LS$ and let $\varphi_{(X,\Of_X)}$ denote the functor
\[ \mathsf{Open}(X) \simeq \mathsf{Top}^{\mathrm{loc}}(\sring)^{\mathsf{Open}}_{/(X,\Of_X)}= \LS^{\mathsf{Open}}_{/(X,\Of_X)} \longrightarrow \LS^{\mathsf{Open}} \subset \LS   \]
(note that the second identity follows from the fact that $\LS$ is assumed stable under open subspaces), then the following hold true.
\begin{enumerate}[$(1)$]
    \item For every $(X,\Of_X)$, the functor $\varphi_{(X,\Of_X)}^*:\pshv(\LS)\rightarrow \pshv(X)$ carries sheaves to sheaves.
    \item The collection of functors $\{\varphi^*_{(X,\Of_X)}\}_{(X,\Of_X)\subset \LS}$ detects sheaves.
    \item For each $(X,\Of_X)\in \LS$, the induced functor $\varphi_{(X,\Of_X)}^*:\shv(\LS)\rightarrow \shv(X)$ is an \emph{essential} geometric morphism (that is, $\varphi_{(X,\Of_X)}^*$ does not only preserve limits but colimits as well).
    \item The collection of functors $\{\varphi^*_{(X,\Of_X)}:\shv(\LS)\rightarrow\shv(X)\}_{(X,\Of_X)\in \LS}$ detects colimits. 
\end{enumerate}
\end{lem}
For the proof, we recall the following fact.
\begin{lem}\label{lem:unitpresheaves}
Let $f:\icat\rightarrow\icatd$ be a functor among essentially small \infcats and consider the induced adjunction $(f_!\adj f^*)$ between their \infcats of presheaves. Let $X\rightarrow Y$ be a map in $\icat$. Then the equivalence $\mathsf{RFib}_{\icat}\rightarrow\pshv(\icat)$ carries the (strictly) commuting diagram 
\[
\begin{tikzcd}
\icat_{/X\rightarrow Y} \ar[d]\ar[r] & \icat_{/Y} \ar[d] \\
\icatd_{/f(X)\rightarrow f(Y)}\times_{\icatd}\icat \ar[r] &\icatd_{/f(Y)}\times_{\icatd}\icat
\end{tikzcd}
\]
of right fibrations over $\icat$ induced by the functor $f$ to the unit transformation of the map $j(X)\rightarrow j(Y)$. 
\end{lem}
\begin{proof}
The right adjoint functor $\_\times_{
\icat}\icatd:\mathsf{RFib}_{\icat}\rightarrow\mathsf{RFib}_{\icatd}$ is induced by the simplicial Quillen adjunction 
\[ \begin{tikzcd} (\sset)_{/\icat} \ar[r,shift left,"f\circ {\_}"] &[1.5em] (\sset)_{/\icatd} \ar[l,shift left,"{\_}\times_{\icat}\icatd"] \end{tikzcd}\]
(where both categories are equipped with the contravariant models structures). According to \cite[Proposition 5.2.4.6]{HTT}, given a simplicial Quillen adjunction
\[ \begin{tikzcd} \mathbf{A} \ar[r,shift left,"L"] & \mathbf{A}' \ar[l,shift left,"R"]\end{tikzcd}\]
a unit for the adjunction on the underlying \infcats at some map $A\rightarrow B$ of fibrant-cofibrant objects of $\mathbf{A}$ is given by the total vertical composition in the diagram
\[
\begin{tikzcd}
A \ar[r]\ar[d] & B\ar[d] \\
RL(A) \ar[r] \ar[d] & RL(B)\ar[d] \\
L(A')  \ar[r] & L(B')
\end{tikzcd}
\]
where the upper vertical maps are unit transformations and the lower vertical maps are induced by fibrant replacements $L(A)\rightarrow A'$ and $L(B)\rightarrow B'$ in $\mathbf{A}'$. Since the obvious maps $\icat_{/X\rightarrow Y}\rightarrow \icatd_{/(f(X)\rightarrow f(y))}$ and $\icat_{/Y}\rightarrow\icatd_{/f(Y)}$ are fibrant replacement in the contravariant model structure on $(\sset)_{/\icatd}$, we conclude.
\end{proof}
\begin{proof}[Proof of Lemma \ref{lem:detectetiononspaces}]
The proof of $(1)$ is nothing but the observation that $\varphi_{(X,\Of_X)}:\mathsf{Open}(X)\rightarrow \LS$ carries covering families to covering families and the fact that open embeddings of derived locally $\cinfty$-ringed spaces are stable under pullbacks. To prove $(2)$, $(3)$ and $(4)$, we assert the following.
\begin{enumerate}[$(*)$]
    \item Let $F\rightarrow F'$ be a map of presheaves on $\LS$ exhibiting a sheafification, then the map $\varphi_{(X,\Of_X)}^*(F)\rightarrow \varphi_{(X,\Of_X)}^*(F')$ exhibits a sheafification.
\end{enumerate}
Suppose $(*)$ holds. Let $F\in\pshv(\LS)$ and suppose that $\varphi_{(X,\Of_X)}^*(F)$ is a sheaf for all $(X,\Of_X)$. Let $\alpha:F\rightarrow LF$ be a sheafification map, then we wish to show that $\alpha$ is an equivalence. we note that the functor
\[ \Phi^*:\pshv(\LS)\longrightarrow \prod_{(X,\Of_X)\in \LS} \pshv(X)\]
is conservative, so it suffices to argue that $\Phi^*$ carries $\alpha$ to an equivalence, but this is guaranteed by $(*)$. For $(3)$ we note that $(*)$ implies that the diagram 
\[
\begin{tikzcd}
\pshv(\LS)\ar[r,"\varphi_{(X,\Of_X)}^*"] &[3em] \pshv(X) \\
\shv(\LS)\ar[r,"\varphi_{(X,\Of_X)}^*"]\ar[u,hook] &[3em] \shv(X)\ar[u,hook]
\end{tikzcd}
\]
is vertically left adjointable. It then follows that the composition
\[ \pshv(\LS)\overset{L}{\longrightarrow}\shv(\LS)\longrightarrow\shv(X)  \]
is equivalent to the composition 
\[ \pshv(\LS)\overset{\varphi_{(X,\Of_X)}^*}{\longrightarrow}\pshv(X)\longrightarrow\shv(X)  \]
which preserves colimits, so that the functor $\shv(\LS)\rightarrow \shv(X)$ also preserves colimits by the universal property of cocontinuous localizations \cite[Proposition 5.5.4.20]{HTT}. Now $(4)$ follows from $(3)$ and the fact that the functor 
\[ \Phi^*:\shv(\LS)\longrightarrow \prod_{(X,\Of_X)\in \LS} \shv(X)\]
is conservative. We are left with the proof of the assertion $(*)$. Unwinding the definitions, we are required to show the following.
\begin{enumerate}
\item[$(*')$] Let $(Y,\Of_Y) \in \LS$ and let $\{(U_i,\Of_Y|_{U_i})\rightarrow (Y,\Of_Y)\}_i$ be an open covering family, which determines a family $\{\varphi^*_{(X,\Of_X)}(U_i,\Of_Y|_{U_i})\rightarrow \varphi^*_{(X,\Of_X)}(Y,\Of_Y)\}_i$ in $\shv(X)$ since $\varphi^*_{(X,\Of_X)}$ carries sheaves to sheaves by $(1)$ and each representable presheaf on $\LS$ is a sheaf by $(4)$ of Proposition \ref{prop:resyonedaschemes}. Then the map 
    \[h:\coprod_i \varphi^*_{(X,\Of_X)}(U_i,\Of_Y|_{U_i})\longrightarrow \varphi^*_{(X,\Of_X)}(Y,\Of_Y)\] 
    of sheaves on $X$ is an effective epimorphism.
\end{enumerate}
Consider the (strictly) commuting diagram 
\[
\begin{tikzcd}
\LS_{/(X,\Of_X)}^{\mathsf{Open}}\ar[d,"\simeq"] \ar[r,"\varphi_{(X,\Of_X)}"] & \LS \ar[r,hook]& \topo^{\loc}(\sring) \ar[d,"q_{\sring}"] \\
\mathsf{Open}(X)\ar[rr,"\psi_X"] && \topo
\end{tikzcd}
\]
of \infcatst, where $\psi_X$ is the composition $\mathsf{Open}(X)\simeq \mathsf{Top}^{\mathsf{Open}}_{/X}\rightarrow\topo$. Let $\phi_{(X,\Of_X)}$ denote the composition $\mathsf{Open}(X)\rightarrow\LS\rightarrow \topo^{\loc}(\sring)$, then the map $\varphi^*_{(X,\Of_X)}(U_i,\Of|_{U_i})\rightarrow \varphi^*_{(X,\Of_X)}(Y,\Of_Y)$ is equivalent to the map $\phi^*_{(X,\Of_X)}(U_i,\Of|_{U_i})\rightarrow \phi^*_{(X,\Of_X)}(Y,\Of_Y)$. The diagram above determines a commuting diagram 
\[
\begin{tikzcd}
& \pshv(\topo^{\loc}(\sring))\ar[dr,"\phi^*_{(X,\Of_X)}"]\\  \pshv(\topo)\ar[rr,"\psi_X^*"] \ar[ur,"q_{\sring}^*"]&&   \pshv(\mathsf{Open}(X))
\end{tikzcd}
\]
of \infcatst. Note that the functor $q_{\sring}$ has a right adjoint that carries a topological space $Z$ to the derived locally $\cinfty$-ringed space $(Z,\underline{\R})$ equipped with the locally constant sheaf $\underline{\R}$ with stalk $\R$ at all points; it follows from \cite[Proposition 3.3.5]{cinftyI} that the diagram 
\[
\begin{tikzcd}
(U_i,\Of_{Y}|_{U_i})\ar[d]\ar[r] & (Y,\Of_Y)\ar[d] \\
(U_i,\underline{\R})\ar[r] & (Y,\underline{\R})
\end{tikzcd}
\]
induced by the unit transformation of this adjunction is a pullback of derived locally $\cinfty$-ringed spaces. Applying the Yoneda embedding and the functor $\phi_{(X,\Of_X)}^*$ yields in light of the commuting diagram of \infcats above a pullback diagram
\[
\begin{tikzcd}
\varphi^*_{(X,\Of_X)}(U_i,\Of_Y|_{U_i}) \ar[d]\ar[r] & \varphi^*_{(X,\Of_X)}(Y,\Of_Y)\ar[d] \\
\psi^*_X(U_i)\ar[r] & \psi^*_X(Y)
\end{tikzcd}
\]
in the \infcat $\pshv(\mathsf{Open}(X))$, and therefore also in the \infcat $\shv(X)$. It will thus suffice to show that the map $\coprod_i \psi^*_X(U_i)\rightarrow\psi^*_X(Y)$ is an effective epimorphism in $\shv(X)$. To see this, it suffices to show that for each $V\in \mathsf{Open}(X)$ and each map $j(V)\rightarrow \psi_X^*(Y)$, the maps $j(V)\times_{\psi_X^*(Y)}\psi_X^*(U_i)\rightarrow j(V)$ determine an effective epimorphism $\coprod_i j(V)\times_{\psi_X^*(Y)}\psi_X^*(U_i)\rightarrow j(V)$. The map $j(V)\rightarrow \psi_X^*(Y)$ factors via the unit map $j(V)\rightarrow \psi_X^*(V)\rightarrow\psi_X(Y)$ where the second map corresponds to some map $V\rightarrow Y$ of topological spaces. Set $W_i=U_i\times_YV\subset V$, then the diagram
\[
\begin{tikzcd}
\psi_X^*(W_i)\ar[d]\ar[r] & \psi_X^*(V)\ar[d] \\
\psi_X^*(U_i) \ar[r] & \psi_X^*(Y)
\end{tikzcd}
\]
is a pullback. Since the map $W_i\rightarrow V$ lies in the image of the functor $\mathsf{Open}(X)\rightarrow \mathsf{Top}$ and the collection $\{W_i\subset V\}_i$ is an open cover of $V$, it suffices to argue that the diagram 
\[
\begin{tikzcd}
\mathsf{Open}(X)_{/(W_i\rightarrow V)} \ar[d]\ar[r] & \mathsf{Open}(X)_{/V} \ar[d] \\
\topo_{/(W_i\rightarrow V)}\times_{\topo}\mathsf{Open}(X)\ar[r] & \topo_{/V}\times_{\topo}\mathsf{Open}(X)
\end{tikzcd}
\]
is a homotopy pullback diagram of right fibrations over $\mathsf{Open}(X)$, in view of Lemma \ref{lem:unitpresheaves}. This is the case since the diagram is evidently an ordinary pullback, all objects are fibrant and the horizontal maps are fibrations. 
\end{proof}
\begin{prop}\label{prop:schemeproperties}
Let $\iota_{\LS}:\LS\subset \daff$ be a property of affine derived $\cinfty$-schemes stable under open subspaces, then the following hold.
\begin{enumerate}[$(1)$]
    \item The pullback functor $\iota_{\LS}^*:\pshv(\daff)\rightarrow\pshv(\LS)$ carries $\dstack$ into $\shv(\LS)$.
    \item The induced functor $\iota_{\LS}^*:\dstack\rightarrow \shv(\LS)$ is an essential geometric morphisms (that is, $\iota_{\LS}^*$ preserves not only limits but colimits as well).
    \item The left adjoint $\iota_{\LS!}:\shv(\LS)\rightarrow\dstack$ is fully faithful. 
\end{enumerate}
\end{prop}
\begin{proof}
$(1)$ follows easily from the assumption that that the inclusion $\LS\subset \daff$ carries covering families to covering families and that $\LS$ is stable under open pullbacks. To see that $\iota_{\LS}^*$ preserves colimits, consider the commuting diagram 
\[
\begin{tikzcd}
\dstack\ar[r,"\iota_{\LS}^*"] \ar[d]& \shv(\LS) \ar[d] \\
\prod_{(X,\Of_X)\in \daff} \shv(X)\ar[r] & \prod_{(X,\Of_X)\in \LS} \shv(X)
\end{tikzcd}
\]
and invoke Lemma \ref{lem:detectetiononspaces}. If $\iota_{\LS}^*$ preserves colimits, then the full subcategory $\icat\subset \shv(\LS)$ spanned by sheaves $F$ for which the unit map $F\rightarrow \iota_{\LS}^*\iota_{\LS!}F$ is an equivalence is stable under colimits, so fully faithfulness of $\iota_{\LS!}$ follows from the observation that the image of the Yoneda embedding lies in $\icat$.
\end{proof}
In view of Proposition \ref{prop:schemeproperties}, we may without loss replace the \infcats $\shv(\mathsf{\mathsf{Mfd}})$ and $\shv(\daff_{\fp})$ (and more generally $\shv(\LS)$ for any subcategory $\LS\subset\daff$ stable under open subspaces) with their essential images under the functors $\iota_{\mathsf{Sm!}}$ and $\iota_{\fp!}$.
\begin{defn}
Let $X$ be a derived $\cinfty$-stack.
\begin{enumerate}[$(1)$]
    \item We will say that $X$ is \emph{smooth} if $X$ lies in the essential image of the functor $\iota_{\mathsf{Sm}!}:\shv(\mathsf{Mfd})\rightarrow\dstack$. We let $\smst\subset \dstack$ denote the full subcategory spanned by smooth stacks.
    \item We will say that $X$ is \emph{locally of finite presentation} if $X$ lies in the essential image of the functor $\iota_{\fp!}:\shv(\daff_{\fp})\rightarrow\dstack$. We let $\dstack_{\fp}\subset\dstack$ denote the full subcategory spanned by derived $\cinfty$-stacks locally of finite presentation.
    \item More generally, for $\LS$ a property of affine derived $\cinfty$-schemes stable under open subspaces, we say that a stack $X$ \emph{locally has the property $\LS$} if $X$ lies in the essential image of the functor $\iota_{\LS!}:\shv(\LS)\rightarrow\dstack$. We will write $\dstack_{\LS}\subset\dstack$ for the full subcategory spanned by derived $\cinfty$-stacks that locally have the property $\LS$.
\end{enumerate}
Let $f:X\rightarrow Y$ be a morphism of derived $\cinfty$-stacks, then we say that $f$ is \emph{locally of finite presentation} if for any $Z\in \dstack_{\fp}$, the stack $X\times_YZ$ lies in $\dstack_{\fp}$.
\end{defn}
\begin{defn}\label{defn:representablemorphism}
Let $\LS$ be a property of affine derived $\cinfty$-schemes stable under open subspaces.
\begin{enumerate}[$(1)$]
    \item A map $X\rightarrow Y$ of derived $\cinfty$-stacks locally having the property $\LS$ is \emph{representable} if for any affine derived $\cinfty$-scheme $\spec\,A$ in $\LS$ and any map $\spec\,A\rightarrow Y$, the object $\spec\,A\times_YX$ (where the pullback is taken in $\dstack_{\LS}$ is representable.
    \item A map $X\rightarrow Y$ in $\dstack_{\LS}$ is \emph{locally representable} if it has the property defined in Proposition \ref{prop:sheafificationproperty}; that is, there exists an effective epimorphism $\coprod_i U_i\rightarrow Y$ such that $U_i\times_XY\rightarrow U_i$ is representable for all $i$.
    \item A map $X\rightarrow Y$ in $\dstack_{\LS}$ is \emph{representable by derived $\cinfty$-schemes that locally have the property $\LS$} if for any affine derived $\cinfty$-scheme $\spec\,A$ in $\LS$ and any map $\spec\,A\rightarrow Y$, the object $\spec\,A\times_YX$ is representable by a derived $\cinfty$-scheme that locally has the property $\LS$. 
\end{enumerate}  
\end{defn}
\begin{warn}\label{warn:representable}
For many properties $\LS\subset\daff$ stable under open subspaces, the inclusion $\dstack_{\LS}\subset\dstack$ is not an algebraic morphism of \inftopoi since finite limits may not be preserved. It is also generally not the case that a morphism $f:X\rightarrow Y$ in $\dstack_{\LS}$ that is representable as a morphism of $\dstack_{\LS}$ is representable as a morphism of $\dstack$ without imposing additional conditions on $f$ (like those discussed in Section 2.2 below).  
\end{warn}
\begin{rmk}
We may replace $(2)$ of Definition \ref{defn:representablemorphism} above by the following condition.
\begin{enumerate}
    \item[$(2')$] A map $X\rightarrow Y$ is locally representable if there exists an effective epimorphism $\coprod_i U_i\rightarrow Y$ where each $U_i$ is representable such that $U_i\times_XY\rightarrow U_i$ is representable for all $i$.
\end{enumerate}
This follows right away from the fact that every stack $Y$ is a canonical colimit $\colim_{\spec\,A \in\LS_{/Y}}\spec\,A\simeq Y$ and \cite[Lemma 6.2.3.13]{HTT}.
\end{rmk}
What follows is a useful criterion for being representable by a derived $\cinfty$-scheme. First, we will say that a map $f:X\rightarrow Y$ of derived $\cinfty$-stacks locally having the property $\LS$ is an \emph{open (substack) inclusion} if $f$ is representable and for any $\spec\,A\in \LS$, the map $\spec\,A\times_YX \rightarrow\spec\,A$ is (representable by) an open embedding of affine derived $\cinfty$-schemes that have the property $\LS$. Clearly, the property of being an open substack inclusion is a property of morphisms of $\dstack_{\LS}$ stable under pullbacks .
\begin{rmk}\label{rmk:opensubstackmono}
Note that an open inclusion $X\rightarrow Y$ in $\dstack_{\LS}$ is a $(-1)$-truncated morphism. To see this, note that being $n$-truncated for $n\geq -2$ is a local property of morphisms in any \inftopt, so by definition of an open substack, it suffices to show that open embeddings of objects in $\LS$ determine $(-1)$-truncated morphisms in $\dstack_{\LS}$, that is, for every $Z\in\dstack_{\LS}$ and every open embedding $\spec\,A\rightarrow\spec\,B$, the fibres of the map $\Hom_{\dstack_{\LS}}(Z,\spec\,A)\rightarrow\Hom_{\dstack_{\LS}}(Z,\spec\,B)$ are either empty or contractible. Since this last property is stable under limits of morphisms of spaces, we may suppose that $Z$ is also affine, in which case the assertion follows immediately from the definition of an open embedding.
\end{rmk}
\begin{rmk}
Note that in order for an open substack inclusion $U\rightarrow X$ to be an equivalence, it suffices to show that for each $\R$-point $*=\spec\,\R\rightarrow X$, the pullback $U\times_X*$ is not empty. If $X$ is representable, then this condition implies that the underlying map $q_{\sring}(U)\rightarrow q_{\sring}(X)$ is a homeomorphism so $U\rightarrow X$ must be an equivalence as this map is $q_{\sring}$-Cartesian. The case of a general open substack inclusion follows immediately.
\end{rmk}
\begin{rmk}
Let $f:X\rightarrow Y$ be a map of derived $\cinfty$-stacks locally having the property $\LS$. The phrase `$f$ is an open substack inclusion' is potentially ambiguous, since it could also mean that $f$ is an open substack inclusion as map of derived $\cinfty$-stacks having the property $\LS'$ for any $\LS\subset \LS'$ (see also Warning \ref{warn:representable}. In view of Proposition \ref{prop:criterion} below however, there is as it turns out no ambiguity possible.
\end{rmk}
\begin{prop}\label{prop:representablebyscheme}
Let $X$ be a derived $\cinfty$-stack locally having the property $\LS$, then the following are equivalent.
\begin{enumerate}[$(1)$]
    \item The stack $X$ is representable by a derived $\cinfty$-scheme that locally has the property $\LS$. 
    \item There exists a small collection of open inclusions $\{U_i\rightarrow X\}_i$ with each $U_i$ representable by a derived $\cinfty$-scheme that locally has the property $\LS$ which determines an effective epimorphism $\coprod_i U_i\rightarrow X$.
    \item There exists a small collection of open inclusions $\{U_i\rightarrow X\}_i$ with each $U_i$ representable by an affine derived $\cinfty$-scheme that has the property $\LS$ which determines an effective epimorphism $\coprod_i U_i\rightarrow X$. 
\end{enumerate}
\end{prop}
\begin{proof}
The implication $(1)\Rightarrow (2)$ is obvious and for $(2)\Rightarrow (3)$ it suffices to choose for each $U_i$ an open cover by affines. To prove the implication $(3)\Rightarrow(1)$, take an effective epimorphism $h:\coprod_iU_i\rightarrow X$ as in $(3)$, then $X$ is a colimit of the \v{C}ech nerve $\check{C}(h)_{\bullet}$. Since each $U_i\rightarrow X$ is an open inclusion, the \v{C}ech nerve is in each level a coproduct of affines and every map in the \v{C}ech nerve is an open inclusion. From this and $(3)$ of Proposition \ref{prop:resyonedaschemes} it follows that $\check{C}(h)_{\bullet}$ is the image of some simplicial object $C_{\bullet}$ in $\mathsf{d}\cinfty\mathsf{Sch}^{\mathsf{Open}}_{\LS}$. The map $C_1\rightarrow C_0\times C_0$ is given by 
\[ \coprod_{ij} U_{i}\times_XU_j\longrightarrow\coprod_{ij}U_i\times U_j \]
which is an injection on the underlying topological spaces since each map $U_i\times_XU_j\rightarrow U_j$ is assumed an open embedding. Invoking $(5)$ of Proposition \ref{prop:schemescolimstable}, we deduce that $C_{\bullet}$ is an effective groupoid in $\dsch_{\LS}^{\text{\'{e}t}}$. Invoking $(2)$ of Proposition \ref{prop:resyonedaschemes}, we deduce that $X$ is equivalent to the geometric realization $|C_{\bullet}|$ taken in $\dsch_{\LS}$. 
\end{proof}
The representability criterion of \cite[Theorem 1.1.72]{cinftyI} is a particular case of the result above. We will use it repeatedly during the proof of the elliptic representability theorem. The preceding result also implies that the property of morphisms of being locally representable implies representability by (non-affine) derived $\cinfty$-schemes.

\begin{cor}\label{cor:localyrepscheme}
Let $\LS$ be a property of affine derived $\cinfty$-schemes stable under open subspaces. Let $X\rightarrow Y$ be a map of derived $\cinfty$-stacks locally having the property $\LS$. Suppose that the map $X\rightarrow Y$ is locally representable, then for each map $\spec\,A\rightarrow Y$ from a representable, the pullback $\spec\,A\times_YX$ is (representable by) a (possibly non-affine) derived $\cinfty$-scheme that locally has the property $\LS$. 
\end{cor}
\begin{proof}
Suppose $X\rightarrow Y$ locally representable and let a morphism $\spec\,A\rightarrow Y$ be given. Choose an effective epimorphism $\coprod_i U_i\rightarrow Y$ such that $X\times_YU_i\rightarrow U_i$ is representable for all $i$. We may choose an open cover $\{V_j\hookrightarrow\spec\,A\}$ such that for each $j$, there exists an $i$ and a factorization 
\[
\begin{tikzcd}
& U_i\times_Y\spec\,A\ar[dr] \\
V_j\ar[ur]\ar[rr]  && \spec\,A.   
\end{tikzcd}
\]
It follows that the map $V_j\times_YX\rightarrow V_j$ is a pullback of the map $U_i\times_YX\rightarrow U_i$, which is a representable map. We conclude that there is an effective epimorphism $\coprod_j V_j\times_{Y}X\rightarrow \spec\,A\times_YX$ such that for each $j$, the map $V_j\times_{Y}X\rightarrow \spec\,A\times_YX$ is a pullback along the open inclusion $V_j\rightarrow\spec\,A$ and therefore an open inclusion itself. It follows from Proposition \ref{prop:representablebyscheme} that $\spec\,A\times_YX$ is representable by a derived $\cinfty$-scheme. 
\end{proof}

\begin{cor}\label{cor:localprop}
Let $\LS$ be a local property of objects of affine derived $\cinfty$-schemes. Let $f:X\rightarrow Y$ be a morphism in $\dstack_{\LS}$, then $f$ is locally representable if and only if $f$ is representable (in other words, if $\LS$ is a local property of objects, then the property of being representable is a local property of morphisms of $\dstack_{\LS}$).
\end{cor}
\begin{proof}
We first show this for $\LS=\daff$. It follows from Corollary \ref{cor:localyrepscheme} that it suffices to argue the following: if $(X,\Of_X)\rightarrow (Y,\Of_Y)$ is a map of derived $\cinfty$-schemes with $(Y,\Of_Y)$ affine and there is an open cover $\{U_i\subset Y\}_{i\in I}$ such that $(X\times_YU_i,\Of_{X}|_{X\times_YU_i})$ is affine, then $(X,\Of_X)$ is affine. Since $Y$ is Lindel\"{o}f, we may choose the set $I$ countable; since a countable disjoint union of second countable spaces is second countable, it follows that $X$ admits an open surjection $\coprod_i X\times_YU_i\rightarrow X$ from a second countable space and is therefore second countable. The $\cinfty$-regularity similarly follows easily from the $\cinfty$-regularity of $(Y,\Of_Y)$ and $\{(X\times_YU_i,\Of_{X}|_{X\times_YU_i})\}_i$. To see that $X$ is Hausdorff, it suffices to show that the map $X\rightarrow Y$ is separated since $Y$ is Hausdorff, but being separated is a local property of morphisms of topologocal spaces, so we conclude since $X\times_YU_i$ is Hausdorff. It remains to observe that $\Gamma(\Of_X)$ is $\kappa$-compact since we have chosen a regular cardinal $\kappa>>\omega_1$. It follows from Theorem \ref{thm:spectrumglobalsections} that $(X,\Of_X)$ is an affine derived $\cinfty$-scheme. That $(X,\Of_X)$ has the property $\LS$ if $\LS$ is a local property and $(X,\Of_X)$ and $(Y,\Of_Y)$ have the property $\LS$ is immediate.
\end{proof}
\begin{rmk}\label{rmk:localprop}
The property of being a manifold is stable under open subspaces and local. The property of being an affine derived $\cinfty$-scheme of finite presentation is stable under open subspaces, but not local.
\end{rmk}
We now turn our attention to local properties of morphisms of derived $\cinfty$-schemes.
\begin{defn}\label{defn:localonthetarget}
Let $\LS$ be a property of affine derived $\cinfty$-schemes stable under open subspaces and let $P$ be a property of morphisms of $\LS$ (which we always assumes spans a full subcategory of the arrow category).
\begin{enumerate}[$(1)$]
\item We say that $P$ is \emph{stable under pullbacks} if pullbacks along any map $(X,\Of_X)\rightarrow (Y,\Of_Y)$ having the property $P$ exist in $\LS$ and for any pullback diagram
\[
\begin{tikzcd}
(Q,\Of_Q)\ar[d,"g"]\ar[r] & (X,\Of_X)\ar[d,"f"] \\
(Z,\Of_Z)\ar[r] & (Y,\Of_Y)
\end{tikzcd}
\]
in $\LS$, if $f$ has the property $P$, then so does $g$.
\item  We say that the property $P$ of morphisms of $\LS$ is \emph{local on the target} if it is stable under pullbacks and for any map $(X,\Of_X)\rightarrow(Y,\Of_Y)$ of derived $\cinfty$-schemes, should there exist an open cover $\{U_i\subset Y\}_i$ such that for all $i$, the map $(X\times_YU_i,\Of_X|_{X\times_YU_i})\rightarrow (U_i,\Of_Y|_{U_i})$ has the property $P$, then $(X,\Of_X)\rightarrow(Y,\Of_Y)$ has the property $P$.
\end{enumerate}
We will write $\fun(\Delta^1,\LS)^P\subset \fun(\Delta^1,\LS)$ for the full subcategory spanned by morphisms that have the property $P$. By stability under pullbacks, the evaluation functor $\ev_1:\fun(\Delta^1,\LS)^P\rightarrow\LS$ is a Cartesian fibration. We will let $\fun(\Delta^1,\LS)^{P\simeq}\subset \fun(\Delta^1,\LS)^{P}$ denote the maximal right subfibration, the subcategory on the Cartesian morphisms. 
\end{defn}
Of course, there are very many such properties. We have already seen a few: being an equivalence, an open inclusion and an \'{e}tale map are local on the target for the \'{e}tale topology. We now introduce a hierarchy of `smoothness' conditions on morphisms of derived $\cinfty$-schemes which plays a central role in any derived geometry. First, consider $*\overset{0}{\hookrightarrow}\R^n$ as a pointed affine derived $\cinfty$-scheme of finite presentation. Let $\Omega^n\R^k$ be the $n$-fold looping of $\R^k$ with respect to the point $0$, so that $\Omega^n\R^k$ fits into a pullback diagram 
\[
\begin{tikzcd}
\Omega^n\R^k\ar[d]\ar[r] & *\ar[d,"0"] \\
*\ar[r,"0"] & \Omega^{n-1}\R^k
\end{tikzcd}
\]
of affine derived $\cinfty$-schemes.
\begin{defn}
Let $f:(X,\Of_X)\rightarrow (Y,\Of_Y)$ be a morphism of derived $\cinfty$-schemes. 
\begin{enumerate}[$(1)$]
\item The morphism $f$ is \emph{submersive} if there exist a small collection of commuting diagrams 
\[
\begin{tikzcd}
(U_i,\Of_{U_i})\times (V_i,\Of_{V_i}) \ar[d]\ar[r] & (X,\Of_X) \ar[d] \\
(U_i,\Of_{U_i}) \ar[r] & (Y,\Of_Y)
\end{tikzcd}
\]
in $\dsch$ determining a covering family $\{(U_i,\Of_{U_i})\times (V_i,\Of_{V_i})\rightarrow (X,\Of_X)\}_i$ in which the horizontal maps are open embeddings, the left vertical map is the projection onto the first factor and $(V_i,\Of_{V_i})$ is equivalent as a derived $\cinfty$-scheme to an open submanifold of some Cartesian space. We say that a derived $\cinfty$-scheme $(X,\Of_X)$ is \emph{smooth} if the canonical map $(X,\Of_X)\rightarrow *$ is submersive (the smooth derived $\cinfty$-schemes are precisely smooth manifolds which are not necessarily paracompact Hausdorff). 
\item The morphism $f$ is a \emph{locally trivial bundle} if there exist a small collection of pullback diagrams 
\[
\begin{tikzcd}
(U_i,\Of_{U_i})\times (V_i,\Of_{V_i}) \ar[d]\ar[r] & (X,\Of_X) \ar[d] \\
(U_i,\Of_{U_i}) \ar[r] & (Y,\Of_Y)
\end{tikzcd}
\]
in $\dsch$ determining a covering family $\{(U_i,\Of_{U_i})\times (V_i,\Of_{V_i})\rightarrow (X,\Of_X)\}_i$ in which the horizontal maps are open embeddings and the left vertical map is the projection onto the first factor.
\item The morphism $f$ is \emph{quasi-smooth} if there exists a small collection of commuting diagrams
\[
\begin{tikzcd}
(U_i,\Of_{U_i})\ar[d]\ar[r] & (X,\Of_X) \ar[d] \\
(V_i,\Of_{V_i}) \ar[r] & (Y,\Of_Y)
\end{tikzcd}
\]
in $\dsch$ determining a covering family $\{(U_i,\Of_{U_i})\rightarrow (X,\Of_X)\}_i$ in which the top horizontal map is an open embedding, the lower horizontal map is submersive and the left vertical fits into a pullback diagram 
\[
\begin{tikzcd}
(U_i,\Of_{U_i}) \ar[d] \ar[r]& *\ar[d,"0"] \\
(V_i,\Of_{V_i}) \ar[r] & \R^k
\end{tikzcd}
\]
for some $k\geq 0$. We say that a derived $\cinfty$-scheme $(X,\Of_X)$ is \emph{quasi-smooth} if the canonical map $(X,\Of_X)\rightarrow *$ is quasi-smooth.
\item The morphism $f$ is \emph{$n$-quasi-smooth} for an integer $n>1$ if there exists a small collection of commuting diagrams
\[
\begin{tikzcd}
(U_i,\Of_{U_i})\ar[d]\ar[r] & (X,\Of_X) \ar[d] \\
(V_i,\Of_{V_i}) \ar[r] & (Y,\Of_Y)
\end{tikzcd}
\]
in $\dsch$ determining a covering family $\{(U_i,\Of_{U_i})\rightarrow (X,\Of_X)\}_i$ in which the top horizontal map is an open embedding, the lower horizontal map is $(n-1)$-quasi-smooth (quasi-smooth if $n=2$) and the left vertical fits into a pullback diagram 
\[
\begin{tikzcd}
(U_i,\Of_{U_i}) \ar[d] \ar[r]& *\ar[d,"0"] \\
(V_i,\Of_{V_i}) \ar[r] & \Omega^{n-1}\R^k
\end{tikzcd}
\]
We say that a derived $\cinfty$-scheme $(X,\Of_X)$ is \emph{$n$-quasi-smooth} if the canonical map $(X,\Of_X)\rightarrow *$ is $n$-quasi-smooth.
\end{enumerate}
\end{defn}
\begin{rmk}
Submersive, quasi-smooth and $n$-quasi-smooth morphisms are locally of finite presentation.
\end{rmk}
\begin{rmk}
The properties of being quasi-smooth and $n$-quasi-smooth for $n> 1$ are properties of affine derived $\cinfty$-schemes that are stable under open subspaces and local.
\end{rmk}
For the purposes of this work, we need to introduce one more property of morphisms: we say that a map $(X,\Of_X)\rightarrow (Y,\Of_Y)$ of derived $\cinfty$-schemes is \emph{proper} if $f:X\rightarrow Y$ is a proper map of topological spaces; that is, if $f^{-1}(K)$ is compact for any compact $K\subset Y$.
\begin{prop}\label{prop:localtarget}
The following properties of morphisms of affine derived $\cinfty$-schemes are stable under pullbacks and local on the target.
\begin{enumerate}[$(1)$]
    \item The property of being an equivalence.
    \item The property of being an open inclusion.
    \item The property of being \'{e}tale.
    \item The property of being submersive.
    \item The property of being proper. 
    \item The property of being locally of finite presentation.
    \item The property of being quasi-smooth.
    \item The property of being $n$-quasi-smooth for any $n>1$.
\end{enumerate}
\end{prop}
We leave the straightforward proof to the reader. Any property of morphisms of affine derived $\cinfty$-schemes that is stable under pullbacks and local on the target determines a local property of morphisms of derived stacks.
\begin{prop}\label{prop:extendprop}
Let $\LS\subset\daff$ be a local property of affine derived $\cinfty$-schemes and let $P$ be a property of morphisms of $\LS$ that is stable under pullbacks and local on the target. Let $P$ be the following property of morphisms of $\dstack_{\LS}$.
\begin{enumerate}
    \item[$(P)$] A map $X\rightarrow Y$ of derived $\cinfty$-stacks has the property $P$ just in case $X\rightarrow Y$ is representable and for each map $\spec\,A\rightarrow Y$, the map $\spec\,A\times_YX\rightarrow\spec\,A$ has the property $P$.
\end{enumerate}
Then $P$ is a local property of morphisms in the \inftop $\dstack_{\LS}$. Moreover, the right fibration $\fun(\Delta^1,\LS)^{P\simeq}\rightarrow\LS$ corresponds under unstraightening to a sheaf $\LS^{op}\rightarrow\spa$, and this sheaf is a classifying object for the property $P$ of morphisms of $\dstack_{\LS}$ defined above.
\end{prop}
\begin{proof}
The smallness of $P$ is a consequence of the smallness of $\LS$. The locality of $P$ follows by combining Corollary \ref{cor:localprop} and Proposition \ref{prop:denselocal}.
\end{proof}
For instance, if $P$ is the property of being an open embedding, we have already defined the corresponding property of morphism of derived stacks as the \emph{open substack inclusions}. If $P$ is the property of being \'{e}tale, then we say that a map $X\rightarrow Y$ of derived stacks having the corresponding property is \'{e}tale, and so on. \\
In case the full subcategory $\LS\subset\daff$ does not determine a local property (so that the property of being representable may not be a local property of morphisms of $\dstack_{\LS}$), we have to sheafify.
\begin{prop}
Let $P$ be a \emph{small} property of morphisms of $\dsch_{\LS}$ stable under pullback. Let $P$ be the property of morphisms of $\dstack_{\LS}$ defined in Proposition \ref{prop:extendprop}. Let $\widehat{P}$ be the following property of morphisms of $\dstack_{\LS}$.
\begin{enumerate}
    \item[$(\widehat{P})$] A map $X\rightarrow Y$ has the property $\widehat{P}$ just in case for every map $\spec\,A\rightarrow Y$ from a representable, there exists an open cover $\{(U_i,\Of_{U_i})\rightarrow \spec\,A\}_i$ such that for each $i$, the map $X\times_Y(U_i,\Of_{U_i})\rightarrow (U_i,\Of_{U_i})$ has the property $P$.
\end{enumerate}
Then $\widehat{P}$ is small, stable under pullbacks and local on the target. Moreover, the natural transformation of presheaves 
\[ \Of^{(P)}\longrightarrow \Of^{(\widetilde{P})}  \]
exhibits $\Of^{(\widetilde{P})}$ as a sheafification of $\Of^{(P)}$.
\end{prop}
\begin{proof}
Combine Propositions \ref{prop:denselocal} and \ref{prop:sheafificationproperty}.
\end{proof}
\begin{rmk}
All of the properties of Proposition \ref{prop:localtarget} are evidently defined and are local for morphisms of arbitrary derived $\cinfty$-schemes, not just affine ones. However these properties, viewed as sheaves on $\daff$ via the restricted Yoneda embedding, are not small and thus do not admit classifying objects.      
\end{rmk}
We end this subsection by introducing some terminology for \emph{smooth} stacks. 
\begin{defn}
Let $S$ be a smooth stack.
\begin{enumerate}[$(1)$]
    \item We say that a map $M\rightarrow S$ of smooth stacks is an \emph{$S$-family of manifolds} if $M\rightarrow S$ is a submersive map of smooth stacks.
    \item We say that $M\rightarrow S$ is a \emph{locally trivial $S$-family of manifolds} if $M\rightarrow S$ is an $S$-family of manifolds which is also a locally trivial bundle.
    \item We say that $M\rightarrow S$ is a \emph{proper $S$-family of manifolds} if $M\rightarrow S$ is an $S$-family of manifolds and a proper map of smooth stacks.
    \item Let $f:Y\rightarrow M$ be a map over $S$, then we say that $f$ is an \emph{$S$-family of submersions} if $M\rightarrow S$ is an $S$-family of manifolds and $Y\rightarrow M$ is a submersive map of smooth stacks. More generally, if $P$ is a property of morphisms of $\mathsf{Mfd}$ stable under pullbacks and local on the target, we will say that a map $Y\rightarrow M$ of smooth stacks over $S$ is an \emph{$S$-family of $P$-morphisms} if $M\rightarrow S$ is an $S$-family of manifolds and $Y\rightarrow M$ is a $P$-morphism. 
\end{enumerate}    
\end{defn}
The property of arrows $Y\rightarrow M$ of being an $S$-family of $P$-morphisms is a local property of $\Delta^2$-diagrams in $\smst$ by Remark \ref{rmk:localdiagram}; we denote by $\Of^{\Delta^1,P}\subset\fun(\Delta^2,\smst)$ the full subcategory spanned by $S$-families of $P$-morphisms. By Proposition \ref{prop:diagramproperty}, the associated sheaf $\Of^{(\Delta^1,P)}$ coincides with the sheaf associated to the right fibration 
\[\fun(\Delta^2,\mathsf{Mfd})^{P\simeq}\overset{\ev_2}{\longrightarrow}\mathsf{Mfd},\]
the category $\fun(\Delta^2,\mathsf{Mfd})^{P}$ being the Cartesian fibration over $\mathsf{Mfd}$ whose objects are diagrams
\[
\begin{tikzcd}
Y\ar[dr]\ar[rr,"f"] && M\ar[dl,"q"] \\
& S 
\end{tikzcd}
\]
where $q$ is a submersion and $f$ has the property $P$.
\begin{prop}
Let $S$ be a smooth stack. A proper $S$-family of manifolds $M\rightarrow S$ is a locally trivial $S$-family of manifolds. 
\end{prop}
\begin{proof}
Apply Ehresmann's fibration theorem.    
\end{proof}
\subsection{Stacky \'{e}tale maps and stacky submersions}
The functor $\iota_{\fp}:\daff_{\fp}\hookrightarrow\daff$ preserves finite limits; it follows that the induced functor $\iota_{\fp!}:\shv(\daff_{\fp})\simeq \dstack_{\fp}\subset \dstack$ is an algebraic morphism. The functor $\iota_{\mathsf{Sm}!}:\shv(\mathsf{Mfd})\simeq\smst\subset\dstack$ is not an algebraic morphism because $\iota_{\mathsf{Sm}}$ only preserves \emph{transverse} pullbacks. To prove the representability of elliptic moduli problems, it will be of crucial importance to understand what kind of limits the functor $\iota_{\mathsf{Sm}!}$ \emph{does} preserve. We will show that the inclusion $\smst\subset\dstack$ has essentially the same properties as the inclusion $\iota_{\mathsf{Sm}}:\mathsf{Mfd}\subset\daff$: it preserves pullbacks along \'{e}tale (more generally submersive) maps for an appropriate notion of \emph{\'{e}tale maps} between stacks.
\begin{cons}\label{cons:fracture}
Let $\LS$ be a property of affine derived $\cinfty$-schemes stable under open subspaces and let $\LS^{\mathsf{Open}}\subset\LS$ be the nonfull subcategory on the open embeddings. Since the inclusion $y_{\LS}:\LS^{\mathsf{Open}}\subset\LS$ preserves pullbacks, declaring a family $\{(U_i,\Of_X|_{U_i})\rightarrow (X,\Of_X)\}_i$ of open embeddings to be covering just in case it is a covering family for the \'{e}tale topology on $\LS$ determines a Grothendieck topology for which the restriction functor
\[  y_{\LS}^*:\pshv(\LS)\longrightarrow\pshv(\LS^{\mathsf{Open}}) \]
carries sheaves to sheaves. 
\end{cons}
\begin{rmk}
Note that the Grothendieck topology on $\LS^{\mathsf{Open}}$ of Construction \ref{cons:fracture} is \emph{not} subcanonical. Accordingly, we will write $j(X,\Of_X)\in \shv(\LS^{\mathsf{Open}})$ for the image of $(X,\Of_X)$ under the \emph{sheafified} Yoneda embedding
\[ \LS^{\mathsf{Open}}\hooklongrightarrow\pshv(\LS^{\mathsf{Open}})\longrightarrow \shv(\LS^{\mathsf{Open}}).  \]
\end{rmk}
\begin{prop}\label{prop:fracture}
Let $\LS$ be a property of affine derived $\cinfty$-schemes stable under open subspaces, then the following hold true.
\begin{enumerate}[$(1)$]
    \item The functor $y_{\LS}^*:\shv(\LS)\rightarrow \shv(\LS^{\mathsf{Open}})$ is an essential geometric morphism. 
    \item The left adjoint $y_{\LS!}:\shv(\LS^{\mathsf{Open}})\rightarrow\shv(\LS)$ preserves pullbacks. 
    \item For $(X,\Of_X)\in \LS$ an affine derived $\cinfty$-scheme that has the property $\LS$, the functor \[\shv(\LS^{\mathsf{Open}})_{/j(X,\Of_X)}\longrightarrow\shv(\LS)_{/(X,\Of_X)}\]
    induced by $y_{\LS!}$ (note that $y_{\LS!}$ carries representable sheaves to representable sheaves) is fully faithful; moreover, a map $f:(Y,\Of_Y)\rightarrow (X,\Of_Y)$ of affine derived $\cinfty$-schemes that have the property $\LS$ lies in the image of this functor if and only if $f$ is an \'{e}tale morphism. 
    \item The left adjoint $y_{\LS!}$ is faithful and determines an equivalence of $\shv(\LS^{\mathsf{Open}})$ onto a subcategory of $\shv(\LS)$.
    \end{enumerate}
\end{prop}
\begin{proof}
Repeating the argument in the proof of Lemma \ref{lem:detectetiononspaces} for $\LS^{\mathsf{Open}}$ in place of $\LS$ grants that the collection of functors 
\[ \{\varphi^*_{(X,\Of_X)}:\pshv(\LS^{\mathsf{Open}})\longrightarrow \pshv(X)\}_{(X,\Of_X)\in \LS} \]
induced by the collection of functors
\[ \{\varphi_{(X,\Of_X)}:\mathsf{Open}(X)\longrightarrow\LS^{\mathsf{Open}}_{/(X,\Of_X)}\}_{(X,\Of_X)\in\LS} \]
preserves and jointly detects sheaves and that the induced functor
\[ \Phi^*:\shv(\LS^{\mathrm{Open}})\longrightarrow \prod_{(X,\Of_X)\in\LS} \shv(X)  \]
is conservative and preserves and detects colimits. We have a commuting diagram 
\[
\begin{tikzcd}
    \shv(\LS)\ar[rr,"y_{\LS}^*"] \ar[dr] && \shv(\LS^{\mathsf{Open}}) \ar[dl] \\
    & \prod_{(X,\Of_X)\in\LS} \shv(X) 
\end{tikzcd}
\]
so we conclude that $y_{\LS}^*$ preserves (and detects) colimits. To see that the left adjoint $y_{\LS!}:\shv(\LS^{\mathsf{Open}})\rightarrow \shv(\LS)$ preserves pullbacks, we note that for formal reasons, the composition 
\[\LS^{\mathsf{Open}} \overset{j}{\hooklongrightarrow}\pshv(\LS^{\mathsf{Open}})\overset{L}{\longrightarrow} \shv(\LS^{\mathsf{Open}}) \overset{y_{\LS!}}{\longrightarrow} \shv(\LS)\]
is equivalent to the composition 
\[ \LS^{\mathsf{Open}} \overset{y_{\LS}}{\subset}\LS\overset{j}{\longrightarrow}\shv(\LS). \]
Since $y$ and $j$ preserve pullbacks, we conclude by invoking \cite[Proposition 20.2.4.2]{sag}. We proceed with $(3)$: note that the functor 
\[ \shv(\LS^{\mathsf{Open}})_{/j(X,\Of_X)}\longrightarrow\shv(\LS)_{/(X,\Of_X)} \]
induced by $y_{\LS!}$ has a right adjoint given by the formula 
\[ F\longmapsto y_{\LS}^*F\times_{y_{\LS}^*(X,\Of_X)}j(X,\Of_X)  \]
where the pullback is taken along the unit map $j(X,\Of_X)\rightarrow y_{\LS}^*(X,\Of_X)$. Since $y_{\LS!}$ and $y_{\LS}^*$ preserve colimits as just verified and colimits are universal in \inftopoit, it suffices to show that the for all representable sheaves of the form $j(Y,\Of_Y)\rightarrow j(X,\Of_X)$ induced by an open embedding $(Y,\Of_Y)\rightarrow (X,\Of_X)$, the diagram 
\begin{equation}\label{eq:pullback}
\begin{tikzcd}
 j(Y,\Of_Y)\ar[d]\ar[r]  & y_{\LS}^*(Y,\Of_Y)\ar[d] \\
 j(X,\Of_X)\ar[r] & y_{\LS}^*(X,\Of_X)
\end{tikzcd}
\end{equation}
is a pullback. Using Lemma \ref{lem:unitpresheaves} we see that this amounts to the assertion that a map $(Z,\Of_Z)\rightarrow (Y,\Of_Y)$ is an open embedding if and only if the composition $(Z,\Of_Z)\rightarrow  (X,\Of_X)$ is an open embedding. On the level of topological spaces, this is obvious and on the level of structure sheaves, this follows from \cite[Proposition 2.4.1.7]{HTT}. Now suppose that a morphism $g:(Y,\Of_Y)\rightarrow (X,\Of_X)$ of affine derived $\cinfty$-schemes that have the property $\LS$ is the image of some morphism $j(Y,\Of_Y)\rightarrow j(X,\Of_X)$, then we wish to show that this morphism is \'{e}tale. Let $\LS^{\et}\subset\LS$ be the subcategory on the \'{e}tale morphisms, then it follows from the same arguments used to prove $(1)$ of Proposition \ref{prop:resyonedaschemes} that the \'{e}tale topology on $\LS^{\et}$ is subcanonical. It follows that the functor $\shv(\LS^{\et})\rightarrow \shv(\LS)$ left adjoint to the restriction of sheaves induced by the subcategory inclusion $\LS^{\et}\subset \LS$ carries morphisms among representable sheaves to \'{e}tale morphisms among representable sheaves. Since the functor $y_{\LS!}$ factors through this left adjoint, we conclude that $g$ is an \'{e}tale morphism. Conversely, if $g$ is \'{e}tale, then we may choose a collection of open embeddings $\{(U_i,\Of_{U_i})\rightarrow (X,\Of_X)\}_i$ that cover $(X,\Of_X)$ such that each composition $(U_i,\Of_{U_i})\rightarrow (Y,\Of_Y)$ is an open embedding. We have an effective epimorphism
\[
\begin{tikzcd}
    \coprod_i (U_i,\Of_{U_i}) \ar[d,"h"]\ar[r] & (Y,\Of_Y)\ar[d,equal] \\
    (X,\Of_X)\ar[r,"g"] & (Y,\Of_Y)
\end{tikzcd}
\]
in the arrow \inftop $\fun(\Delta^1,\dstack_{\LS})$. Notice that the \v{C}ech nerve $\check{C}(h)_{\bullet}:\simpop\rightarrow (\dstack_{\LS})_{/(Y,\Of_Y)}$ factors through $\shv(\LS^{\mathsf{Open}})_{/j(Y,\Of_{Y})}$. Since $y_{\LS!}$ preserves colimits, we conclude that $g$ belongs to the essential image of $y_{\LS!}$. To show $(4)$, we note that by descent for \inftopoi and the fact that fully faithful functors are stable under limits, the collection of all objects $F\in \shv(\LS^{\mathsf{Open}})$ for which the functor
\[ \shv(\LS^{\mathsf{Open}})_{/F} \longrightarrow \shv(\LS)_{/y_{\LS!}(F)}\]
is fully faithful is stable under colimits, so we conclude that there is a fully faithful embedding $\shv(\LS^{\mathsf{Open}})\rightarrow \shv(\LS)_{/y_{\LS!}(1^{\mathsf{Open}})}$ with $1^{\mathsf{Open}}$ the final object of $\shv(\LS^{\mathsf{Open}})$ (note that $y_{\LS!}$ does \emph{not} carry $1^{\mathsf{Open}}$ to a final object in $\shv(\LS)$). Thus, to see that the composition 
\[ \shv(\LS^{\mathsf{Open}})\longrightarrow \shv(\LS)_{/y_{\LS!}(1^{\mathsf{Open}})} \longrightarrow\shv(\LS)\]
is faithful, it suffices to argue that the unit map $1^{\mathsf{Open}}\rightarrow y^*_{\LS}y_{\LS!}(1^{\mathsf{Open}})$ is $(-1)$-truncated. This unit map is a colimit of the natural transformation $\{j(X,\Of_X)
\rightarrow y_{\LS}^*(X,\Of_X)\}_{(X,\Of_X)\in \LS^{\mathsf{Open}}}$. The fact that the square diagram \eqref{eq:pullback} is a pullback diagram implies that this natural transformation is Cartesian. Since the property of being $n$-truncated for $n\geq -2$ is a local property of morphisms in any \inftopt, it suffices to show that the map $j(X,\Of_X)
\rightarrow y_{\LS}^*(X,\Of_X)$ is $(-1)$-truncated, but this map evaluates on each $(Z,\Of_Z)$ as the inclusion of connected components corresponding to \emph{open embeddings} $(Z,\Of_Z)\rightarrow(X,\Of_X)$. Since $y_{\LS!}$ is factor as a fully faithful functor followed by a right fibration, we conclude that in case there is an equivalence between the images of $F,G\in \shv(\LS^{\mathsf{Open}})$ under $y_{\LS!}$, then $F$ and $G$ are equivalent, so that $y_{\LS!}$ is an equivalence onto its essential image as desired.
\end{proof}
\begin{rmk}
The preceding proposition can be summarized by stating that the functor $y_{\LS!}:\shv(\LS^{\mathsf{Open}})\rightarrow\shv(\LS)$ determines a \emph{fracture subtopos} of $\shv(\LS)$ in the sense of \cite[Section 20.1]{sag}; see also the work of Carchedi on \'{e}tale $\infty$-stacks \cite{Car1}.     
\end{rmk}
\begin{rmk}\label{rmk:compareffepi}
Let $f:X\rightarrow Y$ a map in $\shv(\LS^{\mathsf{Open}})$ and consider the \v{C}ech nerve $\check{C}(f)_{\bullet}$ of $f$ in $\shv(\LS^{\mathsf{Open}})$. Since $y_{\LS!}$ preserves pullbacks, the diagram $y_{\LS!}(\check{C}(f)_{\bullet})$ is a \v{C}ech nerve of $y_{\LS!}(f)$. Since $y_{\LS!}$ preserves colimits, we deduce that $y_{\LS!}$ preserves and reflects effective epimorphisms.  
\end{rmk}
\begin{defn}
Let $\LS$ be a property of affine derived $\cinfty$-schemes stable under open subspaces. We denote by $\dstack_{\LS}^{\text{\'{e}t}}$ the essential image of the composition 
\[ \shv(\LS^{\mathsf{Open}})\overset{y_{\LS!}}{\longrightarrow}\shv(\LS)\overset{\iota_{\LS!}}{\hooklongrightarrow} \dstack. \]
For $\LS=\daff$, we will simply write $\dstack^{\text{\'{e}t}}$ for the essential image of the functor
\[ y_!:\shv(\daff^{\mathsf{Open}})\longrightarrow \dstack. \]
A derived $\cinfty$-stack is an \emph{\'{e}tale stack} if it lies in $\dstack^{\text{\'{e}t}}$. Let $f:X\rightarrow Y$ be a map of derived $\cinfty$-stacks. The map $f$ is \emph{stacky \'{e}tale} if for each map $Z\rightarrow Y$ with $Z\in \dstack^{\et}$, the map $X\times_YZ\rightarrow Z$ lies in the subcategory $\dstack^{\et}$.     
\end{defn}
\begin{rmk}
Let $f:X\rightarrow Y$ be a map in $\dstack$ and suppose that $X\in \dstack^{\et}$. Then $f$ is stacky \'{e}tale if and only if $f$ lies in the subcategory $\dstack^{\et}$.
\end{rmk}
\begin{rmk}
An \'{e}tale map of stacks is stacky \'{e}tale, but the converse is false. 
\end{rmk}
\begin{lem}\label{lem:stackyetalelocal}
The property of being stacky \'{e}tale is a local property of morphisms of $\dstack$.
\end{lem}
\begin{proof}
It is clear that the property of being stacky \'{e}tale is stable under pullbacks and compositions (it follows from $(3)$ of Proposition \ref{prop:fracture} that for $f:X\rightarrow Y$ and $g:Y\rightarrow Z$ morphisms of $\dstack$, in case $g$ is stacky \'{e}tale, $f$ is stacky \'{e}tale if and only if $g\circ f$ is stacky \'{e}tale). Suppose that for a map $X\rightarrow Y$ of derived $\cinfty$-stacks, there exists a collection $\{U_i\rightarrow Y\}_{i\in I}$ determining an effective epimorphism $\coprod_iU_i\rightarrow Y$ such that $X\times_YU_i
\rightarrow U_i$ is stacky \'{e}tale. Let $Z\rightarrow Y$ with $Z\in \dstack^{\et}$ be any map, then we wish to show that $X\times_YZ\rightarrow Z$ lies in $\dstack^{\text{\'{e}t}}$. For every $i\in I$, let $V_i\in \dstack^{\text{\'{e}t}}$ be the cone in the pullback diagram 
\[
\begin{tikzcd}
 V_i\ar[d]\ar[r] & Z\ar[d] \\
 y_!y^*(U_i\times_Y Z)\ar[r] & y_!y^*(Z)
\end{tikzcd}
\]
taken in the \inftop $\dstack^{\text{\'{e}t}}$ where the right vertical map is (the image under $y_!$ of) a unit transformation, then the map $\coprod_i V_i\rightarrow Z$ is an effective epimorphism in $\dstack^{\text{\'{e}t}}$ and therefore also in $\dstack$ (Remark \ref{rmk:compareffepi}). We have for each $i\in I$ a commuting diagram 
\[
\begin{tikzcd}
 V_i\ar[d]\ar[r] & Z\ar[d] \\
 y_!y^*(U_i\times_Y Z)\ar[r]\ar[d] & y_!y^*(Z)\ar[d] \\
 U_i\times_Y Z\ar[r] & Z
\end{tikzcd}
\]
where the vertical maps in the lower square are counit transformations and the right vertical composition is an equivalence. It follows that for each $i\in I$, the map $V_i\times_Z(X\times_YZ)\rightarrow V_i$ is equivalent to the map $V_i\times_{U_i}(U_i\times_YX)\rightarrow V_i$, which lies in $\dstack^{\text{\'{e}t}}$ because $V_i\in\dstack^{\text{\'{e}t}}$ and $U_i\times_YX\rightarrow U_i$ was assumed stacky \'{e}tale. We have an effective epimorphism 
\[
\begin{tikzcd}
    \coprod_i V_i\times_Z(X\times_YZ)\ar[d,"h_0"]\ar[r,"f"] &\coprod_i V_i\ar[d,"h_1"] \\
    X\times_YZ\ar[r] & Z
\end{tikzcd}
\]
in the arrow \inftop $\fun(\Delta^1,\dstack)$ because $\coprod_iV_i\rightarrow Z$ is an effective epimorphism. Taking a \v{C}ech nerve of $h_0$ and $h_1$, we obtain a Cartesian transformation $C_{\bullet}:\simpop\times\Delta^1\rightarrow\dstack$ (see Definition \ref{defn:cartesiantrans}), by Lemma \ref{lem:cartesiantransfkanext}. Since $h_1$ lies in $\dstack^{\text{\'{e}t}}$, the diagram $C_{\bullet}|_{\simp\times\{1\}}=\check{C}(h_1)_{\bullet}$ lies in $\dstack^{\text{\'{e}t}}$ (Remark \ref{rmk:compareffepi}). Since $f$ lies in $\dstack^{\text{\'{e}t}}$, $y_!$ preserves pullbacks and $C_{\bullet}$ is a Cartesian transformation, the entire diagram $C_{\bullet}$ lies in $\dstack^{\text{\'{e}t}}$. Since $\dstack^{\text{\'{e}t}}$ is stable under colimits, we conclude.
\end{proof}

\begin{lem}\label{lem:stackyetaleproperty}
Let $\LS$ be a property of affine derived $\cinfty$-schemes stable under open subspaces. Let $f:X\rightarrow Y$ be a stacky \'{e}tale map and suppose that $Y$ locally has the property $\LS$. Then $X$ locally has the property $\LS$.    
\end{lem}
\begin{proof}
Since $Y$ locally has the property $\LS$, we may write $Y$ as a canonical colimit $\colim_{\spec\,A\in \LS_{/Y}}\spec\,A\simeq Y$. Then $f$ is a colimit of the diagram of arrows $\{\spec\,A\times_YX\rightarrow\spec\,A\}_{\spec\,A\in\LS_{/Y}
}$. Since $\dstack_{\LS}$ is stable under colimits and the class of stacky \'{e}tale maps is stable under pullbacks, we may suppose that $Y$ is representable. Since the representable stacks lie in the subcategory $\dstack^{\text{\'{e}t}}\subset \dstack$, the object $X$ also lies $\dstack^{\text{\'{e}t}}$. Write $X$ as the canonical colimit $\colim_{\spec\,A\in \daff^{\mathsf{Open}}_{/X}}\spec\,A$, then we may also assume that $X$ is affine. Since a stacky \'{e}tale map among affines is an \'{e}tale map by $(3)$ of Proposition \ref{prop:fracture}, we deduce that $X$ locally has the property $\LS$.    
\end{proof}
\begin{rmk}
Let $\LS$ be a property of affine derived $\cinfty$-schemes stable under open subspaces, then there is an equality of subcategories $\dstack_{\LS}^{\text{\'{e}t}}=\dstack_{\LS}\cap \dstack^{\text{\'{e}t}}$. Indeed, the inclusion $\dstack_{\LS}^{\text{\'{e}t}}\subset\dstack_{\LS}\cap \dstack^{\text{\'{e}t}}$ is straightforward. For the other inclusion, let $X\rightarrow Y$ be a map in the image of $y_{!}$ and suppose that $X$ and $Y$ locally have the property $\LS$. This map is the colimit of the diagram 
\[ \{f:\spec\,A\longrightarrow \spec\,B\}_{f \in\fun(\Delta^1,\daff^{\mathsf{Open}})_{/X\rightarrow Y}}\]
By Lemma \ref{lem:stackyetaleproperty}, $\spec\,A$ and $\spec\,B$ locally have the property $\LS$. Since $\dstack_{\LS}^{\text{\'{e}t}}$ is stable under colimits, we may thus assume that $X$ and $Y$ are affine. Choosing open covers of $X$ and $Y$ by affines in $\LS$ and repeating the preceding argument, we may assume that $X$ and $Y$ are in $\LS$, in which case the map $X\rightarrow Y$ clearly lies in $\dstack_{\LS}^{\text{\'{e}t}}$.
\end{rmk}
\begin{prop}\label{prop:pbstackyetalepreserve}
Let $\LS$ be a property of affine derived $\cinfty$-schemes stable under open subspaces, then the full subcategory inclusion $\dstack_{\LS}\subset\dstack$ preserves pullbacks along stacky \'{e}tale maps.
\end{prop}
\begin{proof}
Suppose $X\rightarrow Y$ is a stacky \'{e}tale map in $\LS$ and $Z\rightarrow Y$ is any morphism with $Z\in \dstack_{\LS}$, then it suffices to show that the pullback $X\times_YZ$ lies in $\dstack_{\LS}$. Since $X\times_YZ\rightarrow Z$ is stacky \'{e}tale, this follows from Lemma \ref{lem:stackyetaleproperty}.  
\end{proof}
\begin{prop}\label{prop:stackyetale}
 Let $\LS\subset\daff$ be a property of affine derived $\cinfty$-schemes stable under open subspaces and let $f:X\rightarrow Y$ a morphism of $\dstack_{\LS}$, then the following are equivalent.
 \begin{enumerate}[$(1)$]
     \item The map $f$ is stacky \'{e}tale.
     \item For any $Z\in\dstack_{\LS}^{\et}$ and any map $Z\rightarrow Y$ the map $Z\times_YX\rightarrow Z$ where the pullback is \emph{taken in $\dstack_{\LS}$} lies in $\dstack_{\LS}^{\text{\'{e}t}}$.
     \item For every representable $\spec\,A\in \LS$ and every map $\spec\,A\rightarrow Y$, the map $\spec\,A\times_YX\rightarrow \spec\,A$ where the pullback is \emph{taken in $\dstack_{\LS}$} lies in $\dstack_{\LS}^{\text{\'{e}t}}$.
\end{enumerate}
\end{prop}
\begin{proof}
For the proof of $(1)\Rightarrow (2)$ we invoke Proposition \ref{prop:pbstackyetalepreserve} to conclude that the pullbacks $X\times_YZ$ and $(X\times_YZ)_{\LS}$ (the pullback taken in $\dstack_{\LS}$) coincide. Since stacky \'{e}tale maps are stable under pullbacks in $\dstack$, we are done. The implication $(2)\Rightarrow (3)$ is obvious. For $(3)\Rightarrow (1)$, we suppose that the map $X\rightarrow Y$ has the property specified in $(3)$. Write $Y\simeq \colim_{\spec\,A\in\LS_{/Y}}\spec\,A$ as a canonical colimit, then $X\rightarrow Y$ is a colimit of the Cartesian (in $\dstack_{\LS}$) natural transformation $\{(X\times_{Y}\spec\,A)_{\LS}\rightarrow\spec\,A\}_{\spec\,A\in \LS_{/Y}}$, where $(X\times_{Y}\spec\,A)_{\LS}$ denotes the pullback taken in $\dstack_{\LS}$. It follows from $(3)$ that each of the maps $(X\times_{Y}\spec\,A)_{\LS}\rightarrow\spec\,A$ lies in $\dstack^{\text{\'{e}t}}_{\LS}$. Invoking Proposition \ref{prop:pbstackyetalepreserve}, we deduce that the natural transformation $\{(X\times_{Y}\spec\,A)_{\LS}\rightarrow\spec\,A\}_{\spec\,A\in \LS_{/Y}}$ is also a Cartesian transformation in $\dstack$. Now we conclude by invoking Proposition \ref{prop:localpropertycharacterize} and the fact that being stacky \'{e}tale is a local property of morphisms in $\dstack$.
\end{proof}
We will also require a stacky generalization of submersive maps. For this, we need to restrict our attention to properties of of affine derived $\cinfty$-schemes stable under products.
\begin{lem}\label{lem:propertyproductstable}
Let $\LS$ be a property of affine derived $\cinfty$-schemes stable under open subspaces. Suppose that $\LS\subset \daff$ is stable under finite products, then the full subcategory inclusion $\dstack_{\LS}\subset\dstack$ is stable under finite products as well. 
\end{lem}
\begin{proof}
Let $X,Y\in\dstack_{\LS}$. To see that $X\times Y$, the product in $\dstack$, lies in $\dstack_{\LS}$ it suffices to consider the case that $X$ and $Y$ are representable, by universality of colimits and the fact that $\dstack_{\LS}\subset\dstack$ is stable under colimits.   
\end{proof}
\begin{rmk}
In virtue of the preceding result, we will in the sequel freely use that products of smooth stacks are taken in derived stacks.
\end{rmk}
\begin{defn}
A map $f:X\rightarrow Y$ of derived $\cinfty$-stacks is \emph{stacky submersive} if $f$ is stacky \'{e}tale locally a projection away from a smooth stack. More precisely, $f$ is stacky submersive if there exists a small collection of commuting diagrams
\[
\begin{tikzcd}
  U_i\times E_i\ar[d]\ar[r] & X\ar[d] \\
  U_i\ar[r] & Y
\end{tikzcd}
\]
that determines an effective epimorphism $\coprod_i U_i\times E_i\rightarrow X$ in which the horizontal maps are stacky \'{e}tale, the left vertical map is the projection onto the first factor and $E_i$ is smooth. 
\end{defn}
Here are some basic properties of stacky submersive maps.
\begin{prop}\label{prop:stackysubmersive}
Let $\mathsf{Mfd}\subset \LS\subset\daff$ be a property of affine derived $\cinfty$-schemes stable under open subspaces and stable under products, then the following hold true.
\begin{enumerate}[$(1)$]
    \item The property of being stacky submersive is a local property of morphisms of $\dstack$.
    \item If $X\rightarrow Y$ is a stacky submersive map and $Y$ locally has the property $\LS$, then $X$ locally has the property $\LS$.
    \item The full subcategory inclusion $\dstack_{\LS}\subset\dstack$ preserves pullbacks along stacky submersive maps.
    \item Let $f:X\rightarrow Y$ a map in $\dstack_{\LS}$, then $f$ is stacky submersive if and only if for each $\spec\,A\in \LS$, the pullback $\spec\,A\times_YX\rightarrow\spec\,A$ \emph{taken in $\dstack_{\LS}$} is stacky submersive. 
\end{enumerate}
\end{prop}
\begin{proof}
It is easy to see that the property of being stacky submersive is stable under pullbacks and compositions. Now suppose that for $X\rightarrow Y$ a map of derived $\cinfty$-stacks, there exists a small collection $\{V_j\rightarrow Y\}_j$ determining an effective epimorphism $\coprod_jV_j\rightarrow Y$ such that $X\times_YV_j\rightarrow V_j$ is stacky submersive for each $j$. Choose for each $j$ a small collection of diagrams
\[
\begin{tikzcd}
  U_{ij}\times E_{ij}\ar[d]\ar[r] & X\times_YV_j\ar[d] \\
  U_{ij}\ar[r] & V_j
\end{tikzcd}
\]
that exhibit the map $X\times_YV_j\rightarrow V_j$ as stacky submersive. We have effective epimorphisms $h:\coprod_j V_j\rightarrow Y$ and $h':\coprod_j X\times_YV_j\rightarrow X$. Consider for each pair $j,i$ the diagrams
\[C_{\bullet}:=\check{C}(h)_{\bullet}\times_{\coprod_j V_j}U_{ij},\quad\quad C'_{\bullet}:=\check{C}(h')_{\bullet}\times_{\coprod_j X\times_YV_j} U_{ij}\times E_{ij},  \]
then we have a diagram 
\[
\begin{tikzcd}
 {|}C'_{\bullet}{|} \ar[d]\ar[r] & {|}C_{\bullet}{|} \ar[d] \\
 X\ar[r] &  Y.
\end{tikzcd}
\]
The right vertical map is a colimit of the Cartesian transformation $\check{C}(h)_{\bullet}\times_{\coprod V_j}U_{ij}\rightarrow \check{C}(h)_{\bullet}$. Since the map $U_{ij}\rightarrow \coprod_j V_j$ is stacky \'{e}tale, the map $\check{C}(h)_{n}\times_{\coprod V_j}U_{ij}\rightarrow \check{C}(h)_{n}$ is stacky \'{e}tale for all $[n]\in \simp$, so we conclude that $|C_{\bullet}|\rightarrow Y$ is stacky \'{e}tale as being stacky \'{e}tale is a local property. The same holds for the map $|C'_{\bullet}|\rightarrow X$; we now complete the proof by remarking that by universality of colimits, the map $|C'_{\bullet}|\rightarrow |C_{\bullet}|$ is equivalent to the projection $ |C_{\bullet}|\times E_{ij}\rightarrow |C_{\bullet}|$. To prove $(2)$, we assume $X\rightarrow Y$ is stacky submersive, so there exists a covering $h:\coprod_iU_i\times E_i\rightarrow X$ by stacky \'{e}tale maps such that the composition $U_i\times E_i\rightarrow Y$ factors through the projection to $U_i$ as stacky \'{e}tale map $U_i\rightarrow Y$. It follows from Lemma \ref{lem:stackyetaleproperty} that $U_i$ locally has the property $\LS$. We see from Lemma \ref{lem:propertyproductstable} that $U_i\times E_i$ locally has the property $\LS$. Since stacky \'{e}tale maps are stable under pullbacks along any morphism, every map in the \v{C}ech nerve $\check{C}(h)_{\bullet}$ is stacky \'{e}tale. Invoking Lemma \ref{lem:stackyetaleproperty} again, we deduce that the simplicial diagram $\check{C}(h)_{\bullet}$ lies in $\dstack_{\LS}$ so its geometric realization does as well. Now $(3)$ is an immediate consequence of $(2)$ as in the proof of Proposition \ref{prop:pbstackyetalepreserve}. For $(4)$, we note that the `only if' direction follows from the fact that stacky submersive maps are stable under pullbacks and $(3)$. For the `if' direction, write $Y\simeq \colim_{\spec\,A\in\LS_{/Y}}\spec\,A$ as a canonical colimit, then $X\rightarrow Y$ is a colimit of the Cartesian (in $\dstack_{\LS}$) natural transformation $\{(X\times_{Y}\spec\,A)_{\LS}\rightarrow\spec\,A\}_{\spec\,A\in \LS_{/Y}}$, where $(X\times_{Y}\spec\,A)_{\LS}$ denotes the pullback taken in $\dstack_{\LS}$. By assumption, each of the maps $(X\times_{Y}\spec\,A)_{\LS}\rightarrow\spec\,A$ is stacky submersive. Invoking $(3)$, we deduce that the natural transformation $\{(X\times_{Y}\spec\,A)_{\LS}\rightarrow\spec\,A\}_{\spec\,A\in \LS_{/Y}}$ is also a Cartesian transformation in $\dstack$. Now we conclude by invoking $(1)$. 
\end{proof}
Using the ideas of this section, we can now give a criterion for a map $X\rightarrow Y$ among stacks locally having the property $\LS$ being an open inclusion (\'{e}tale map, submersive map) as a map of derived $\cinfty$-stacks. 
\begin{prop}\label{prop:criterion}
Let $\LS$ be a property of affine derived $\cinfty$-schemes stable under open subspaces. Let $f:X\rightarrow Y$ be a map of derived $\cinfty$-stacks that locally have the property $\LS$, then the following hold true.
\begin{enumerate}[$(1)$]
\item The map $f$ is an open inclusion as a map of derived $\cinfty$-stacks if and only if for each representable $\spec\,A$ having the property $\LS$, the map $(\spec\,A\times_YX)_{\LS}\rightarrow\spec\,A$ \emph{where the pullback is taken in $\dstack_{\LS}$} is an open inclusion.
\item Suppose that $\LS$ is a local property. The map $f$ is an \'{e}tale map of derived $\cinfty$-stacks if and only if for each representable $\spec\,A$ having the property $\LS$, the map $(\spec\,A\times_YX)_{\LS}\rightarrow\spec\,A$ \emph{where the pullback is taken in $\dstack_{\LS}$} is \'{e}tale.
\item Suppose that $\LS$ is a local property and stable under finite products. The map $f$ is a submersive map of derived $\cinfty$-stacks if and only if for each representable $\spec\,A$ having the property $\LS$, the map $(\spec\,A\times_YX)_{\LS}\rightarrow\spec\,A$ \emph{where the pullback is taken in $\dstack_{\LS}$} is submersive.
\end{enumerate}
\end{prop}
\begin{proof}
In case $f$ is an open inclusion (\'{e}tale map, submersive map) of derived $\cinfty$-stacks, then $f$ is in particular stacky \'{e}tale (stacky submersive), so the `only if' direction of the statements above follow from Proposition \ref{prop:pbstackyetalepreserve}, $(3)$ of Proposition \ref{prop:stackysubmersive} and Corollary \ref{cor:localprop}. We prove the converse statements. It follows from Proposition \ref{prop:stackyetale} that in case $f$ satisfies the condition in $(1)$ then $f$ is stacky \'{e}tale, and similarly for condition $(2)$. It follows from $(4)$ of Proposition \ref{prop:stackysubmersive} that in case $f$ satisfies the condition in $(3)$, then $f$ is stacky submersive. We have to verify that for any $\spec\,B\in \daff$ and any map $\spec\,B\rightarrow Y$, the map $\spec\,B\times_YX\rightarrow\spec\,B$ is an open inclusion (\'{e}tale, submersive). Write $Y$ as a canonical colimit $\colim_{\spec\,A\in \LS_{/Y}}\spec\,A$, then, since being an open inclusion (\'{e}tale, submersive) is a local property of morphisms of $\dstack$ and the collection $\{\spec\,B\times_Y\spec\,A\rightarrow \spec\,B\}_{\spec\,A\in \LS_{/Y}}$ determines an effective epimorphism to $\spec\,B$, it suffices to check that for every $\spec\,A\rightarrow Y$ in $\dstack_{\LS}$, the map $\spec\,B\times_YX\rightarrow \spec\,B\times_Y\spec\,A$ is an open inclusion (\'{e}tale map, submersive map) of derived $\cinfty$-stacks. This map is a pullback of the map $\spec\,A\times_YX \rightarrow \spec\,A$. Since $f$ is stacky \'{e}tale (stacky submersive), we may take the pullback in $\dstack_{\LS}$ (Proposition \ref{prop:pbstackyetalepreserve} or $(3)$ of Proposition \ref{prop:stackysubmersive}), in which case it follows from $(1)$ ($(2)$, $(3)$).
\end{proof}
\begin{cor}
A map $f:X\rightarrow S$ of smooth stacks is an open inclusion (\'{e}tale, submersive) \emph{as a map of smooth stacks} if and only if $f$ is an open inclusion (\'{e}tale, submersive) \emph{as a map of derived $\cinfty$-stacks}.
\end{cor}
In particular, $f$ is an $S$-family of manifolds if and only if $f$ is submersive as a map of derived stacks.
\begin{rmk}
In this work, we are primarily concerned with \'{e}tale and submersive maps but the arguments in this section readily imply the following: let $\LS\subset \LS'$ be properties of affine derived $\cinfty$-schemes that are stable under open subspaces and local. Let $P'$ be a property of morphisms of $\LS'$ that is stable under pullback and local on the target. Suppose that the inclusion $\LS\subset \LS'$ preserves pullbacks along $P'$-morphisms, so that $P'$ also determines a property $P$ of morphisms of $\LS$ stable under pullbacks and local on the target. Then the following holds.
\begin{enumerate}
\item[$(*)$] Let $f:X\rightarrow Y$ be a map of derived $\cinfty$-stacks locally having the property $\LS$. Suppose that $f$ has the property $P$, that is, $f$ is representable (as a map in $\dstack_{\LS}$) and for any $\spec\,A\in \LS$ and any map $\spec\,A\rightarrow Y$, the map $\spec\,A\times_YX\rightarrow\spec\,A$ (with pullback taken in $\dstack_{\LS}$) in $\LS$ has the property $P$. Then $f$ has the property $P'$ as map of derived $\cinfty$-stacks locally having the property $\LS'$. 
\end{enumerate}
In the situation above, $f$ is in particular representable as map in $\dstack_{\LS'}$; note that it is not the case in general that $f$ being representable as a map in $\dstack_{\LS}$ implies that it is representable as map in $\dstack_{\LS'}$.
\end{rmk}

\section{Differential and elliptic moduli problems}
In this section we introduce $S$-families of differential moduli problems and prove the relative elliptic representability theorem for proper families. Roughly speaking, a differential moduli problem takes, for $Y_1\rightarrow Y$ and $Y_2\rightarrow X$ two $S$-families of submersions, the \emph{zeroes of an $S$-family of nonlinear differential operators acting between the $S$-family of sections of $Y_1$ and the $S$-family of sections of $Y_2$}. To make this precise, there are several challenges to overcome.
\begin{enumerate}[$(1)$]
\item For $Y\rightarrow X$ an $S$-family of submersions, we need to construct a derived $\cinfty$-stack $\mathsf{Sections}_{/S}(Y\rightarrow X)$ over $S$, which is suitably functorial in $S$.
\item We need to be able to recognize maps of the form $\mathsf{Sections}_{/S}(Y_1\rightarrow X)\rightarrow \mathsf{Sections}_{/S}(Y_2\rightarrow X)$ as $S$-families of partial differential equations. This amounts to the construction of $S$-families of jet bundles and $S$-families of jet prolongation maps. 
\item To prove the representability theorem, we have to appeal to some basic techniques and results from nonlinear global analysis of elliptic PDEs. For this, we will at some point need to represent mapping stacks as infinite dimensional manifolds modelled on some class of well behaved locally convex topological vector spaces. 
\end{enumerate}
We address $(1)$ in section 3.2 below, where we also develop some tools to understand the local structure of parametrized mapping stacks of sections. We achieve the construction of relative jet stacks for families of maps of stacks $X\rightarrow Y$ in section 3.3 by adhering to the algebro-geometric point of view on jets in the affine setting and applying the machinery of descent of diagrams exposed in the appendix to handle the general case. First, in section 3.1 just below, we compare stacky and functional analytic approaches to infinite dimensional manifolds and show how the \emph{convenient calculus} of Fr\"{o}licher-Kriegl-Michor \cite{frolicherkriegl,krieglmichor} can be treated within the framework of derived $\cinfty$-geometry. 
\subsection{Convenient manifolds}
Our first order of business is to recall several notions of infinite dimensional manifolds and show that under certain circumstances, they represent smooth mapping stacks. Let $E$ be a locally convex (Hausdorff topological vector) space, then recall that a curve $c:\R\rightarrow E$ is \emph{differentiable} if the derivative 
\[ \underset{s\rightarrow 0}{\lim} \frac{c(t+s)-c(t)}{s}  \]
exists for all $t \in \R$, and \emph{smooth} if all iterated derivatives of $c$ exist. The final topology on $E$ induced by all smooth curves $\R\rightarrow E$ is called the \emph{$c^{\infty}$-topology} (also known as the \emph{Mackey-closure} topology). Let $U\subset E$ be a $c^{\infty}$-open subset of a locally convex space $E$, then a map $f:(U\subset E)\rightarrow F$ to another locally convex space $F$ is \emph{smooth} if for every smooth curve $c:\R\rightarrow U$, the composition $f\circ c$ is a smooth curve in $F$.
\begin{rmk}
The collection of smooth curves into a locally convex space $E$ does not change if we replace the locally convex topology on $E$ with its bornologification (the finest locally convex topology on the vector space $E$ that has the same bounded sets). Moreover, a \emph{linear} map $E\rightarrow F$ among locally convex spaces is smooth if and only if it is bounded. It follows that the category $\mathsf{LCVS}_{c^{\infty}}$ whose objects are locally convex spaces and whose morphisms are \emph{smooth} linear maps is equivalent to the category $\mathsf{bLCVS}$ of bornological locally convex spaces. Since we will only be interested in smooth maps, we could thus without cost replace every locally convex space with its bornologification. 
\end{rmk}
\begin{defn}[Convenient spaces]\label{defn:convspace}
Let $E$ be a locally convex topological vector space, then $E$ is \emph{convenient} if any one of the following equivalent conditions are satisfied.
\begin{enumerate}[$(1)$]
\item Any smooth curve $c:\R\rightarrow E$ admits an antiderivative (a curve $\gamma:\R\rightarrow E$ such that $\gamma'=c$).
\item A curve $c:\R\rightarrow E$ is smooth if and only if $l\circ c:\R\rightarrow\R$ is smooth for any continuous linear functional $l\in E^{\vee}$.
\item Let $E\subset F$ be a topological linear subspace, then $E$ is closed in the $c^{\infty}$-topology on $F$.
\item $E$ is \emph{locally complete}: for any absolutely convex closed bounded set $B\subset E$, the auxiliary normed space $E_B$ spanned by $B$ and normed by means of the Minkowski functional associated with $B$ is complete (i.e. a Banach space).
\end{enumerate}    
We let $\mathsf{ConVS}\subset \mathsf{LCVS}_{c^{\infty}}$ be the full subcategory spanned by convenient vector spaces (so that morphisms of convenient vector spaces are smooth linear maps).
\end{defn}
\begin{rmk}
The characterization $(4)$ above shows that $E$ being convenient is also a property of the bounded sets of $E$, that is, being convenient depends only on the bornologification of $E$. Thus, the category of convenient vector spaces and smooth linear maps among them is equivalent to the category of bornological convenient vector spaces and continuous linear maps among them. 
\end{rmk}
\begin{ex}
Since completeness implies local completeness, Fr\'{e}chet spaces and in particular Banach spaces are convenient. Since being metrizable implies being bornological, the standard categories of Fre\'{e}chet and Banach spaces and continuous maps among them are full subcategories of the category $\mathsf{ConVS}$ of convenient vector spaces and smooth linear maps among them.  
\end{ex}
\begin{rmk}
On arbitrary locally convex spaces, there exist a plethora of notions of differentiability, many of which collapse to a single reasonable definition for well behaved spaces like Fr\'{e}chet and Banach spaces when one passes to \emph{infinitely} differentiable, that is, smooth maps. First, recall that the standard notion of differentiability for finite dimensional spaces extends without issues to Banach spaces: let $E,F$ be Banach spaces and let $U\subset E$ be an open set, then a function $f:(U\subset E)\rightarrow F$ is \emph{Fr\'{e}chet differentiable} if the directional derivative 
\[  Df:(U\subset E)\times E\rightarrow F,\quad\quad  Df(x,h):=\underset{s\rightarrow 0}{\lim} \frac{f(x+sh)-f(x)}{s}\]
exists, the function $Df(x,\_)$ is a bounded linear map for all $x\in U$ and the induced map $Df:U\rightarrow B(E,F)$ to the space of bounded linear maps is continuous for the operator norm. It can be shown that this condition is equivalent to the demand that for each $x\in U$, there is a bounded linear map $Df_x:E\rightarrow F$ such that $\lim_{||h||\rightarrow 0}\frac{||f(x+h)-f(x)-Df_xh||}{||h||}=0$. \\
In situations where norms are not available, one can take recourse to differentiability in the sense of \emph{Michal-Bastiani}: let $E,F$ be locally convex spaces, then a function $f:(U\subset E)\rightarrow F$ is \emph{Michal-Bastiani differentiable} if the directional derivative 
\[  Df:(U\subset E)\times E\rightarrow F,\quad\quad  Df(x,h):=\underset{s\rightarrow 0}{\lim} \frac{f(x+sh)-f(x)}{s}\]
exists, is linear in its second variable and is \emph{jointly} continuous (in fact, it is not hard to show that joint continuity implies linearity in the second variable). If $E$ and $F$ are Banach spaces, Fr\'{e}chet differentiability is stronger than Michal-Bastiani differentiability. Like Fr\'{e}chet differentiability, MB differentiability satisfies the chain rule; it turns out there is a precise sense in which MB differentiability asks the weakest continuity condition on the map $U\times E\rightarrow F$ such that the chain rule is satisfied. Let us say that a map $f:(U\subset E)\rightarrow F$ from an open subset of a locally convex space $E$ to a locally convex space $F$ is \emph{MB $C^n$} if for all $k\leq n$, the (inductively defined) directional G\^{a}teaux derivative
\[  Df_k:(U\subset E)\times E\times\ldots \times E \longrightarrow F,\quad\quad Df_k(x,y_1,\ldots,y_k):=\underset{s\rightarrow 0}{\lim} \frac{Df_{k-1}(x+sy_k,y_1,\ldots,y_{k-1})-Df_{k-1}(x,y_1,\ldots,y_{k-1})}{s} \]  
exists and is \emph{jointly} continuous (again it is not hard to show that the joint continuity of this map implies that it is multilinear in its last $k$ variables). It is this notion of smoothness that is used, for instance in Hamilton's seminal review article on the Nash-Moser inverse function theorem \cite{Hamilton}. Similarly, if $E,F$ are Banach spaces, we say that $f$ is \emph{Fr\'{e}chet $C^n$} if the directional G\^{a}teaux derivative $Df_k$ exists and determines a continuous map 
\[ Df_k: (U\subset E)\longrightarrow B(E\times\ldots \times E,F) \]
for $k\leq n$ where the space $B(E\times\ldots \times E,F)$ of bounded linear maps is equipped with the operator norm. It can be shown that (in case one works with Banach spaces) being a MB $C^{n+1}$ map implies being a Fr\'{e}chet $C^n$ map \cite{Kellerdiff}. It follows that the two notions possible notions of smoothness coincide.
\end{rmk}
\begin{ex}
Let $E$ be a Fr\'{e}chet space, then $E$ is convenient. Furthermore, the locally convex topology on $E$ coincides with the $c^{\infty}$-topology (this is another consequence of metrizability, see \cite[Theorem 4.11]{krieglmichor}). This property implies that the category of Fr\'{e}chet spaces and MB smooth maps coincides with the category of Fr\'{e}chet spaces and smooth maps (this is false for non-metrizable locally convex spaces). Indeed, the fact that MB differentiability satisfies the chain rule immediately implies that smoothness is weaker than MB smoothness. Conversely, smoothness of $f$ mostly formally implies that the directional derivatives 
\[  Df_k:(U\subset E)\times E\times\ldots \times E \longrightarrow F   \]
exist and are \emph{smooth} maps \cite[Theorem 3.18]{krieglmichor}. As Fr\'{e}chet spaces carry the $c^{\infty}$-topology, smooth maps among them are continuous; applying this observation to the smooth map $Df_k$ yields MB smoothness. 
\end{ex}
As in ordinary differential geometry, a \emph{convenient $\cinfty$-atlas} on a topological space $M$ is a cover $\{U_i\subset M\}$ (the \emph{charts}) together with homeomorphisms $\phi_i:V_i\rightarrow U_i$ (\emph{chart maps}) with each $V_i\subset E_i$ a $c^{\infty}$-open subset of a convenient vector space such that the transition functions are smooth. Two atlases on a topological space $M$ are \emph{equivalent} if their union is again an atlas. This determines an equivalence relation on the set of convenient $\cinfty$-atlases on $M$ whose equivalences classes are convenient $\cinfty$-structures; a \emph{convenient manifold} is a paracompact Hausdorff topological space together with a choice of a with a convenient $\cinfty$-structure. A map $f:M\rightarrow N$ is declared smooth if it restricts to a smooth map on charts; this is equivalent to demanding that $f$ carries smooth curves in $M$ to smooth curves in $N$. We let $\mathsf{ConMfd}$ denote the category of convenient manifolds and let $\iota_{\mathsf{Con}}:\mathsf{Mfd}\subset\mathsf{ConMfd}$ denote the subcategory spanned by ordinary finite dimensional manifolds. We have full subcategories $\mathsf{BanMfd}\subset \mathsf{FrMfd}\subset \mathsf{ConMfd}$ spanned by convenient manifolds modelled on Banach spaces and Fr\'{e}chet spaces respectively.

\begin{prop}
The inclusion $\mathsf{Mfd}\subset\mathsf{ConMfd}$ is dense.  \end{prop}
\begin{proof}
Let $j_{\mathsf{Con}}$ denote the restricted Yoneda embedding
\[ \mathsf{ConMfd}\overset{j}{\hooklongrightarrow} \pshv(\mathsf{ConMfd}) \overset{\iota_{\mathsf{Con}}^*}{\longrightarrow}\pshv(\mathsf{Mfd}), \]
then we wish to show that the map 
\[ \Hom_{\mathsf{ConMfd}}(M,N)\longrightarrow  \Hom_{\pshv(\mathsf{Mfd})}(j_{\mathsf{Con}}(M),j_{\mathsf{Con}}(N)) \]
is a bijection. To see it is injective, let $f,g:M\rightarrow N$ be smooth maps that are equalized by any smooth map $Q\rightarrow M$ from a finite dimensional manifold. Letting $Q\rightarrow M$ range over the points $*\rightarrow M$, we deduce that $f=g$. To see it is surjective, let $f:j_{\mathsf{Con}}(M)\rightarrow j_{\mathsf{Con}}(N)$ be given and define a map of sets $\tilde{f}:M\rightarrow N$ by the map 
\[ j_{\mathsf{Con}}(M)(*)\overset{f(*)}{\longrightarrow} j_{\mathsf{Con}}(N)(*) .\]
To see that this is smooth, it suffices to check that for any smooth map $h:Q\rightarrow M$ from a finite dimensional manifold, the composition $Q\rightarrow M\overset{\tilde{f}}{\rightarrow} N$ is smooth. For each point $q\in Q$ we have a commuting diagram 
\[
\begin{tikzcd}
    j_{\mathsf{Con}}(M)(Q)\ar[d]\ar[r,"f(Q)"] & j_{\mathsf{Con}}(N)(Q) \ar[d] \\
    j_{\mathsf{Con}}(M)(*) \ar[r,"\tilde{f}"] & j_{\mathsf{Con}}(N)(*)
\end{tikzcd}
\]
so $f(Q)(h):Q\rightarrow N$ carries each point $q\in Q$ to $\tilde{f}(h(q))$, so that $f(Q)(h)=\tilde{f}\circ h$ which is smooth. The same argument shows that the maps $f$ and $j_{\mathsf{Con}}(\tilde{f})$ coincide. 
\end{proof}
We will also denote by $j_{\mathsf{Con}}$ the fully faithful functor $\mathsf{ConMfd}\hookrightarrow\shv(\mathsf{Mfd})\simeq \smst$. Note that this functor preserves all limits that exist in $\mathsf{ConMfd}$ since the Yoneda embedding and the restriction functor $\iota_{\mathsf{Con}}^*$ have this property. Let $\icat\subset\icatd\subset\icate$ be fully faithful functors, then in case $\icat$ is dense in $\icate$, the other two inclusions are dense as well; thus, the inclusions $\mathsf{Mfd} \subset\mathsf{BanMfd}\subset \mathsf{FrMfd} \subset \mathsf{ConMfd}$ are all dense. It follows that the induced fully faithful functors 
\[ \mathsf{BanMfd} \subset  \mathsf{FrMfd} \subset \mathsf{ConMfd}\overset{j_{\mathsf{Con}}}{\hooklongrightarrow} \smst  \]
all preserve limits; in particular, the inclusions $\mathsf{BanMfd} \subset  \mathsf{FrMfd} \subset \mathsf{ConMfd}$ preserve limits. 
\begin{warn}
Let $M$ be a convenient manifold, then the sheaf of smooth $\R$-valued functions on $M$ has a natural structure of a sheaf of $\cinfty$-rings which makes the pair $(M,\cinfty_M)$ a (derived) locally $\cinfty$-ringed space. If this pair is $\cinfty$-regular and $M$ is second countable, then $(M,\cinfty_M)$ is an affine (derived) $\cinfty$-scheme (this is the case, for instance, if $M$ is a paracompact convenient manifold modelled on nuclear Fr\'{e}chet spaces; for example $M$ may be a manifold of mappings $\cinfty(K,N)$ between finite dimensional manifolds with compact source). Beware however that the affine derived $\cinfty$-scheme $(M,\cinfty_M)$ usually \emph{does not} represent the derived $\cinfty$-stack $j_{\mathsf{Con}}(M)$, except when $M$ is finite dimensional. 
\end{warn}
The power of the convenient calculus comes from the fact that it allows for efficient comparison of infinite dimensional manifolds of smooth mappings modelled on well behaved topological vector spaces and mapping sheaves. Let $f:Y\rightarrow X$ be a map of smooth stacks, then we define a smooth stack $\map_X(X,Y)_{\mathsf{Sm}}$, the \emph{smooth mapping stack of sections of $f$} together with a map $\ev:\map_X(X,Y)\times X\rightarrow Y$ such that the following universal property is satisfied: for any smooth stack $Z$, composition with $\ev$ induces a homotopy equivalence of Kan complexes
\[ \Hom_{\mathsf{SmSt}}(Z,\map_X(X,Y)_{\mathsf{Sm}}) \overset{\simeq}{\longrightarrow} \Hom_{\mathsf{SmSt}}(Z\times X,Y). \]
The mapping stack of sections $\map_X(X,Y)$ fits into a pullback diagram
\[
\begin{tikzcd}
\map_X(X,Y)_{\mathsf{Sm}} \ar[d]\ar[r] & \map(X,Y)_{\mathsf{Sm}}\ar[d,"f\circ \_"] \\
*\ar[r,"\mathrm{id}_X"] & \map(X,X)_{\mathsf{Sm}}
\end{tikzcd}
\] 
where $\map(X,\_)_{\mathsf{Sm}}$ denotes the internal hom in $\mathsf{SmSt}$, the right adjoint to the functor $X \times\_:\mathsf{SmSt}\rightarrow\mathsf{SmSt}$. For $E\rightarrow M$ a vector bundle over a manifold, let $\Gamma(E;M)$ denote the convenient vector space of sections of $M$. This is a nuclear Fr\'{e}chet space; moreover, the exponential law holds: a curve $\R\rightarrow \Gamma(E;M)$ is smooth if and only if the induced horizontal map in the commuting diagram
\[  \begin{tikzcd}\R\times E \ar[rr]\ar[dr]&& E\ar[dl] \\&M \end{tikzcd} \]
is smooth \cite[Lemma 30.4,Lemma 30.8]{krieglmichor}. 
\begin{prop}[Exponential Law]
Let $M$ be a manifold and let $E\rightarrow M$ be a finite rank vector bundle over $M$. Let $\Gamma(M;E)$ be the nuclear Fr\'{e}chet space of smooth sections of $E$, viewed as a convenient vector space. Then the map
\[j_{\mathsf{Con}}(\Gamma(M;E))\longrightarrow \map_M(M.E)_{\mathsf{Sm}}  \]
induced by the evaluation mapping $\Gamma(M;E)\times M\rightarrow E$ over $M$ is an equivalence of smooth stacks.
\end{prop}
\begin{proof}
This follows at once from the exponential law of the Fr\"{o}licher-Kriegl-Michor calculus.
\end{proof}
\begin{defn}
Let $f:M\rightarrow N$ be a smooth map between convenient manifolds.
\begin{enumerate}[$(1)$]
\item $f$ is an \emph{open embedding} if $f$ is an open topological embedding. 
\item $f$ is a \emph{local diffeomorphism} if $f$ is locally on the source an open embedding: there exists an open covering $\{U_i\hookrightarrow M\}_i$ such that the composition $U_i\rightarrow N$ is an open embedding for all $i$.
\item $f$ is a \emph{submersion} if there exists a small collection of commuting diagrams
\[
\begin{tikzcd}
  U_i\times V_i\ar[d]\ar[r] & M\ar[d] \\
  U_i\ar[r] & N
\end{tikzcd}
\]
that determines an open cover $\{U_i\times V_i\rightarrow M\}_i$ where $V_i\subset E_i$ a $c^{\infty}$-open subset of a convenient vector space and the left vertical map is the projection onto the first factor.  
\end{enumerate}
\end{defn}
The following result is useful for working with convenient manifolds incarnated as derived $\cinfty$-stacks.
\begin{prop}\label{prop:convmfdetale}
Let $f:M\rightarrow N$ be a smooth map between convenient manifolds, then the following hold true.
\begin{enumerate}[$(1)$]
\item $f$ is an open embedding if and only if $j_{\mathsf{Con}}(f)$ is an open substack inclusion of derived $\cinfty$-stacks.
\item $f$ is a local diffeomorphism if and only if $j_{\mathsf{Con}}(f)$ is an \'{e}tale morhpism of derived $\cinfty$-stacks. 
\item If $f$ is a submersion, then $j_{\mathsf{Con}}(f)$ is stacky submersive. 
\end{enumerate}
\end{prop}
\begin{proof}
For th `only if' direction of $(1)$ and $(2)$, we will use the criterion of Proposition \ref{prop:criterion}. We show that for each manifold $Q$ and each map $Q\rightarrow j_{\mathsf{Con}}(N)$, the map $Q\times_{j_{\mathsf{Con}}(N)}j_{\mathsf{Con}}(M)
\rightarrow Q$ where the pullback is taken in $\smst$ is an open inclusion (\'{e}tale map). Since $j_{\mathsf{Con}}$ preserves limits, it suffices to show that the map $Q\times_NM\rightarrow Q$ of convenient manifolds becomes an open inclusion (\'{e}tale map) after applying $j_{\mathsf{Con}}$, but this is obvious as $Q\times_NM\rightarrow Q$ is an open inclusion (\'{e}tale map) of ordinary manifolds. For the `if' direction of $(1)$ and $(2)$, it suffices to observe that being an open embedding or local diffeomorphism of convenient manifolds can be tested by pulling back to an ordinary manifold. We proceed with $(3)$. Using $(4)$ of Proposition \ref{prop:stackysubmersive} and the fact that $j_{\mathsf{Con}}$ preserves limits, we show that for each manifold $Q$ and each map $Q\rightarrow N$, the map $Q\times_NM\rightarrow Q$ of convenient manifolds becomes stacky submersive after applying $j_{\mathsf{Con}}$. Since submersions of convenient manifolds are \'{e}tale locally projections away from a convenient manifold, this follows from $(2)$.
\end{proof}
\begin{prop}\label{prop:sourceconvenient}
Let $N$ be a convenient manifold and let $f:X\rightarrow j_{\mathsf{Con}}(N)$ be a map of derived $\cinfty$-stacks. If $f$ is an open substack inclusion, then $f$ is representable by an open embedding $V\rightarrow N$ of convenient manifolds.
\end{prop}
\begin{proof}
Let $X\rightarrow j_{\mathsf{Con}}(N)$ be an open substack, then the collection of all points $*\rightarrow N$ for which $*\times_{j_{\mathsf{Con}}(N)}X$ is not initial determines a subset $V\subset N$ and testing against all maps $\R\rightarrow N$ shows that $V$ is open. It follows that $j_{\mathsf{Con}}(V)\rightarrow j_{\mathsf{Con}}(N)$ is an open substack inclusion, so it suffices to show that the map $j_{\mathsf{Con}}(V) \times_{j_{\mathsf{Con}}(N)} X \rightarrow j_{\mathsf{Con}}(V)$ is an equivalence. Since this map is an open substack, it is enough to argue that for each point $*\rightarrow V$, the pullback $*\times_{j_{\mathsf{Con}}(V)}j_{\mathsf{Con}}(V) \times_{j_{\mathsf{Con}}(N)} X \simeq *\times_{j_{\mathsf{Con}}(N)}X$ is not initial which is the case by construction.
\end{proof}
\begin{rmk}\label{rmk:gluingconmfd}
As we can glue affine derived $\cinfty$-schemes in the \inftop of derived $\cinfty$-stacks, so can we glue convenient vector spaces in this \inftopt. Indeed, the essential image of the fully faithful functor $j_{\mathsf{Con}}:\mathsf{ConMfd}\hookrightarrow\dstack$ admits the following characterization. Let $X$ be a derived $\cinfty$-stack, then $X$ is representable by a convenient manifold if and only if the following conditions are satisfied.
\begin{enumerate}[$(1)$]
\item There exists a collection of $c^{\infty}$-open subsets of convenient vector spaces $\{V_i\subset E_i\}_i$ and a collection of open substack inclusions $j_{\mathsf{Con}}(V_i)\rightarrow X$. 
\item Let $U_{ij}\hookrightarrow V_i$ and $U_{ij}\hookrightarrow V_j$ be the open embeddings corresponding to the maps $j_{\mathsf{Con}}(V_i)\times_Xj_{\mathsf{Con}}(V_j)\rightarrow j_{\mathsf{Con}}(V_i)$ and $j_{\mathsf{Con}}(V_i)\times_Xj_{\mathsf{Con}}(V_j)\rightarrow j_{\mathsf{Con}}(V_j)$ (Proposition \ref{prop:sourceconvenient}). Then the topological space $\coprod_i V_i/\sim$ equipped with the identification topology induced by the embeddings $U_{ij}\hookrightarrow V_i$ and $U_{ij}\hookrightarrow V_j$ is paracompact Hausdorff. 
\end{enumerate}
The proof of this result rests on arguments formally identical to those used to prove Proposition \ref{prop:representablebyscheme}.
\end{rmk}
\begin{rmk}
For the purposes of this article -namely the construction of PDE moduli spaces- we will only need to work with the category of $c^{\infty}$-open subsets of convenient vector spaces and smooth maps between them; we introduce convenient manifolds in this section mostly for the sake of completeness. In fact, we need only open subsets of Fr\'{e}chet spaces, and because the Fr\'{e}chet spaces of interest to us are $\cinfty$-regular (either because they are nuclear or because they are separable Hilbert spaces) it suffices to work only with Fr\'{e}chet spaces and smooth maps between them. We may avoid working with convenient \emph{manifolds} because the mapping spaces introduced in the next section on which PDEs will act will be defined via universal properties in the \infcat $\dstack$, and we will only need to establish that such objects are \emph{locally} equivalent to a convenient vector space (which we hasten to add holds only for mapping spaces with compact source). In fact, once local representability has been established, it follows from the general gluing principles of Remark \ref{rmk:gluingconmfd} that these mapping spaces are representable by convenient manifolds (a fact which we will not use in this paper). Reasoning like this leads to an alternative proof of, for instance, \cite[Theorem 42.1]{krieglmichor} on the existence of convenient Fr\'{e}chet manifold structures on mapping spaces with compact source, one that does not require checking that transition functions on overlaps are smooth. This is in line with the paradigm to which we adhere in this series \cite{cinftyI,cinftyII}: once a `smooth' space of any kind has been given the structure of a derived $\cinfty$-stack, its representability (by a derived $\cinfty$-scheme, a derived Artin $\cinfty$-stack, a convenient manifold...) is purely a local question, and the existence of a presentation of this space as a gluing of local models with transition functions preserving the smooth structure is a formal consequence of scheme theoretic principles.
\end{rmk}

\subsection{Parametrized mapping stacks: Weil restriction}
Let $S$ be a smooth stack and let $M\rightarrow S$ be a locally trivial $S$-family of manifolds. Suppose we are given a differential moduli problem $\icate=(X,Y_1,Y_2,P_1,P_2)$ over $M$, then to define the derived solution stack as in the introduction, we need to understand \emph{$S$-families of derived stacks of sections}. Moreover, to prove the representability theorem, we have to leverage classical results from geometric analysis and elliptic theory; to do so, we will need to demonstrate that these families of derived stacks can (locally) be represented by infinite dimensional manifolds of mappings. We have seen that \emph{smooth} stacks of mappings with compact source are representable by infinite dimensional manifolds modelled on nuclear Fr\'{e}chet spaces. Thus, our goals in this section are as follows.
\begin{enumerate}
    \item[$(G1)$] Construct for an $S$-family of derived $\cinfty$-stacks $X\rightarrow S$ and a map $Y\rightarrow X$ a stack $\mathsf{Sections}_{/S}(Y\rightarrow X)$ equipped with a map to $S$ such that for each point $s:*\rightarrow S$ we have a pullback diagram
    \[ \begin{tikzcd}
     \mathsf{Sections}(Y_s\rightarrow X_s)\ar[d]\ar[r] &  \mathsf{Sections}_{/S}(Y\rightarrow X)\ar[d] \\  
     *\ar[r,"s"] & S,   
    \end{tikzcd} \]
    where $X_s=*\times_SX$ and $Y_s=*\times_SY$.
    \item[$(G2)$] Show that under reasonable assumptions, the stack $\mathsf{Sections}_{/S}(Y\rightarrow X)$ is smooth. 
\end{enumerate}
We achieve $G1$ via the formation of \emph{Weil restrictions}, using that \inftopoi are locally Cartesian closed.
\begin{defn}[Weil restriction]\label{defn:weilrestriction}
Let 
\[
\begin{tikzcd}
Y\ar[dr]\ar[rr] && X\ar[dl,"q"] \\
& S
\end{tikzcd}
\]
be a commuting diagram of derived $\cinfty$-stacks. We define a stack $\mathsf{Res}_{X/S}(Y)\rightarrow S$ together with a map
\[ \begin{tikzcd}
    \mathsf{Res}_{X/S}(Y)\times_S X \ar[dr]\ar[rr,"\ev"] && Y \ar[dl] \\
    & X
\end{tikzcd} \]
over $X$ such that the following universal property is satisfied: for any derived $\cinfty$-stack $Z\rightarrow S$ over $S$, composition with $\ev$ induces a homotopy equivalence of Kan complexes
\[  \Hom_{\dstack_{/S}}(Z,\mathsf{Res}_{X/S}(Y)) \overset{\simeq}{\longrightarrow} \Hom_{\dstack_{/X}}(Z\times_SX,Y).  \]
In other words, the map $\ev:\mathsf{Res}_{X/S}(Y)\times_SX\rightarrow Y$ over $X$ exhibits a counit transformation at $Y$ for the functor 
\[ \_\times_S X :\dstack_{/S}\longrightarrow \dstack_{/X}.   \]
We will call the object $\mathsf{Res}_{X/S}(Y)\in \dstack_{/S}$ a \emph{Weil restriction of $Y$ along $q$}.
\end{defn}
\begin{lem}\label{lem:weilrespb}
Let 
\[
\begin{tikzcd}
Y\ar[dr]\ar[rr] && X\ar[dl,"q"] \\
& S
\end{tikzcd}
\]
be a commuting diagram of derived $\cinfty$-stacks. Then the Weil restriction $\mathsf{Res}_{X/S}(Y)$ fits into a pullback diagram 
\[
\begin{tikzcd}
\mathsf{Res}_{X/S}(Y)\ar[d]\ar[r]& \map_{/S}(X,Y)\ar[d] \\
S\ar[r] & \map_{/S}(X,X)
\end{tikzcd}
\]
($\map_{/S}(\_,\_)$ denotes the internal mapping object in $\dstack_{/S}$), where the right vertical map is induced by $Y\rightarrow X$ and the lower horizontal map is obtained by adjunction from the diagram 
\[
\begin{tikzcd}
    X\simeq S\times_SX\ar[dr,"q"]\ar[rr,"\mathrm{id}"] && X\ar[dl,"q"] \\
    & S.
\end{tikzcd}
\]
\end{lem}
\begin{proof}
The functor $\_\times_SX:\dstack_{/S}\rightarrow\dstack_{/X}$ is induced by the functor $\dstack_{/S}\rightarrow\dstack_{/S}$ taking products with $X$ over $S$; it follows from \cite[Proposition 5.2.5.1]{HTT} that the right adjoint to $\_\times_SX$ factors as a composition 
\[ \dstack_{/X}{\longrightarrow}\dstack_{/\map_{/S}(X,X)}\longrightarrow \dstack_{/S} \]
where the first functor is induced by $\map_{/S}(X,\_)$ and the second functor is induced by pullback along the unit map $S\rightarrow\map_{/S}(X,X)$.
\end{proof}
\begin{ex}[Mapping stacks of sections]
Let $S=*$ a point, then the Weil restriction $\mathsf{Res}_{X/*}(Y)$ fits into a pullback diagram
\[
\begin{tikzcd}
\mathsf{Res}_{X/*}(Y)\ar[d]\ar[r]& \map(X,Y)\ar[d] \\
*\ar[r] & \map(X,X),
\end{tikzcd}
\]
that is, $\mathsf{Res}_{X/*}(Y)$ is the mapping stack of sections $\map_X(X,Y)$.
\end{ex}
\begin{rmk}\label{rmk:weilrestrictionpullback}
Let $\sigma:\Delta^1\times\Delta^1\rightarrow\dstack$ given by
\[
\begin{tikzcd}
X\ar[r]\ar[d] & Y \ar[d] \\
X' \ar[r] & Y' 
\end{tikzcd}
\]
be a pullback diagram of derived $\cinfty$-stacks, then the diagram 
\[
\begin{tikzcd}
\dstack_{/Y'}\ar[r]\ar[d] & \dstack_{/Y} \ar[d] \\
\dstack_{/X'} \ar[r] & \dstack_{/X} 
\end{tikzcd}
\]
of \infcats obtained by pulling back along the maps in the diagram, is vertically right adjointable; it follows that for $S'\rightarrow S$ any map of stacks and $Y\rightarrow X$ a map over $S$, the canonical map 
\[ \mathsf{Res}_{X\times_SS'/S'}(Y\times_SS')\longrightarrow \mathsf{Res}_{X/S}(Y)\times_SS' \]
over $S'$ is an equivalence. To see this, it suffices to show that the diagram is horizontally left adjointable. Recall that the square diagram of \infcats above is obtained by the composition $\Delta^1\times\Delta^1\overset{\sigma}{\rightarrow}\dstack\rightarrow\catinfh^{op}$ where the second map is the unstraightening of the Cartesian fibration $\ev_1:\fun(\Delta^1,\dstack)\rightarrow\dstack$. Now suppose we are given a map $f':X''\rightarrow X'$ and an $\ev_1$-coCartesian lift $\alpha':f'\rightarrow g'$ of the map $X'\rightarrow Y'$, then it suffices to show that given a diagram 
\[
\begin{tikzcd}
f\ar[r,"\alpha"] \ar[d]& g\ar[d] \\
f'\ar[r,"\alpha'"] & g'
\end{tikzcd}
\]
where the left vertical map is an $\ev_1$-Cartesian lift of the map $X\rightarrow X'$ and the right vertical map is an $\ev_1$-Cartesian lift of the map $Y\rightarrow Y'$, the transformation $\alpha$ is $\ev_1$-coCartesian. This amounts to the assertion that in the commuting cube
\[
\begin{tikzcd}
& X \arrow[rr] \arrow[dd] & & Y \arrow[dd] \\
 X\times_{X'}X'' \arrow[ur ]\arrow[rr, crossing over,"\beta"{xshift=12pt}] \arrow[dd] & & Y\times_{Y'}X''\arrow[ur]  \\
& X'  \arrow[rr] & & Y'\\
X''\arrow[rr,equal] \arrow[ur]& & X''\arrow[ur] \arrow[from=uu, crossing over]\\
\end{tikzcd}
\]
 where the squares on the side are Cartesian, the indicated map $\beta$ is an equivalence. Since the square $\sigma$ is assumed Cartesian, the front square is also Cartesian which guarantees that $\beta$ is indeed an equivalence.
\end{rmk}
We conclude from the previous remarks that functor of Weil restriction has the properties we desire of the relative mapping stack of sections. Accordingly, for an $S$-family of submersions $f:Y\rightarrow M$, we will define the stack of $S$-parametrized sections of $f$ as the Weil restriction $\mathsf{Res}_{M/S}(Y)$. In the next section, we will apply the formation of Weil restrictions to construct parametrized jet spaces. 
\begin{rmk}
Let $\LS$ be a property of affine derived $\cinfty$-schemes stable under open subspaces. Let $S$ be derived $\cinfty$-stack that locally has the property $\LS$, and let $f:Y\rightarrow X$ be a map of derived $\cinfty$-stacks that locally have the property $\LS$. We will write $\mathsf{Res}_{X/S}(Y)_{\LS}$ for the Weil restriction of $f$ along $X\rightarrow S$ taken in $\LS$. Clearly, the stack $\mathsf{Res}_{X/S}(Y)_{\LS}$ is simply the restriction of the sheaf $\mathsf{Res}_{X/S}(Y)$ to $\LS\subset \daff$.
\end{rmk}
We achieve our second goal $G2$ with the following result.
\begin{prop}\label{prop:weilressmooth}
Let $S$ be a smooth stack, let $q:M\rightarrow S$ be a proper $S$-family of manifolds and let $f:Y\rightarrow M$ be an $S$-family of submersions. Then the counit transformation
\[  \iota_{\mathsf{Sm}!}\mathsf{Res}_{M/S}(Y)_{\mathsf{Sm}}\simeq\iota_{\mathsf{Sm}!}\iota_{\mathsf{Sm}}^{*}\mathsf{Res}_{M/S}(Y)\longrightarrow\mathsf{Res}_{M/S}(Y) \]
is an equivalence.
\end{prop}
In light of this result, we will, for proper $S$-families $M\rightarrow S$ and $S$-families of submersions $Y\rightarrow M$, denote the smooth Weil restriction $\mathsf{Res}_{M/S}(Y)_{\mathsf{Sm}}$ simply by $\mathsf{Res}_{M/S}(Y)$. 
\begin{rmk}
Proposition \ref{prop:weilressmooth} holds more generally with the condition that $q$ be a proper family replaced with the condition that $q$ is a locally trivial family, but we will not need this result in this paper.    
\end{rmk}
The proof of Proposition \ref{prop:weilressmooth} requires some preparation.
\begin{lem}\label{lem:weilresttrivbundle}
Let $S$ be a derived $\cinfty$-stack and let $Y\rightarrow X\times S$ be any map. Then the Weil restriction $\mathsf{Res}_{X\times S/S}(S)$ fits into a pullback diagram
\[
\begin{tikzcd}
\mathsf{Res}_{X\times S/S}(Y) \ar[d]\ar[r] & S\times \map(X,Y)\ar[d] \\
S\ar[r] &S\times \map(X,X\times S)
\end{tikzcd}
\]
in $\dstack_{/S}$.  
\end{lem}
\begin{proof}
The functor $\_\times_{S}S\times X:\dstack_{/S}\rightarrow\dstack_{/X\times S}$ is induced by the functor
\[ \ \dstack_{/S}\longrightarrow\dstack \overset{\_\times X}{\longrightarrow}\dstack \]
where the first functor is the obvious projection. This functor admits a right adjoint carrying $Z\in \dstack$ to $S\times\map(X,Z)$. It follows from \cite[Proposition 5.2.5.1]{HTT} that the right adjoint to $\_\times_{S}S\times X$ factors as a composition
\[ \dstack_{/X\times S}\longrightarrow \dstack_{/S\times \map(X,X\times S)}\longrightarrow\dstack_{/S}  \]
where the first functor is induced by $S\times\map(X,\_)$ and the second functor is induced by pulling back along the unit map $S\rightarrow S\times\map(X,X\times S)$.
\end{proof}
We now establish that the formation of mapping stacks with compact source preserves open embeddings; this is a formal consequence of the tube lemma.
\begin{lem}\label{lem:mappingstackopen}
Let $M$ be a compact manifold and let $f:Q\rightarrow N$ be an open embedding of manifolds, then the map $\map(M,Q)\rightarrow \map(M,N)$ is an open substack inclusion.
\end{lem}
\begin{proof}
Let $(X,\Of_X)=\spec\,A$ be an affine derived $\cinfty$ scheme and suppose we are given a map $(X,\Of_X)\times M\rightarrow N$ determining a map $f:X\times M\rightarrow N$ of topological spaces. We have a pullback diagram 
\[
\begin{tikzcd}
\dstack_{/\map(M,Q)\rightarrow\map(M,N)} \ar[r]\ar[d] &\dstack_{/\map(M,N)}\ar[d] \\
\dstack_{/Q\rightarrow N} \ar[r] & \dstack_{/N}
\end{tikzcd}
\]
of \infcats where the vertical functors take products with $M$. Given the map $(X,\Of_X)\rightarrow\map(M,N)$ determining a map $(X,\Of_X)\times M\rightarrow N$, the functor $\dstack_{/(X,\Of_X)} \rightarrow \dstack_{/N}$ factors as 
\[  \dstack_{/(X,\Of_X)} \overset{\_\times M}{\longrightarrow} \dstack_{/(X,\Of_X)\times M} \longrightarrow \dstack_{/N}. \]
It suffices to show that that the right adjoint to this functor carries $Q\hookrightarrow N$ to an open substack of $(X,\Of_X)$. For formal reasons, it suffices to show that the functor
\[  \daff_{/(X,\Of_X)} \overset{\_\times M}{\longrightarrow} \daff_{/(X,\Of_X)\times M} \longrightarrow \daff_{/N}. \]
admits a counit transformation at $Q$ which determines an open embedding of $(X,\Of_X)$. Let $U\subset X$ be the set $\{x\in X;\, f(x,m)\in Q\,\forall m\in M\}$. This set is open in $X$: if $x\in U$, then the compact set $\{x\}\times M$ is contained in $f^{-1}(Q)$. Choose a finite collection of opens $\{V_i\times W_i\subset f^{-1}(Q)\subset X\times M\}$, then the open set $\cap_i V_i$ is an open neighbourhood of $x$ in $U$. Note that a map $Z\rightarrow X$ of topological spaces factors through $U$ if and only if the induced map $Z\times M\rightarrow N$ factors through $Q$; that is, $U\rightarrow X$ together with the map $U\times N\rightarrow Q$ is a counit at $Q$ for the functor 
\[ \mathsf{Top}_{/X}\overset{\_\times M}{\longrightarrow}\topo_{/X\times M}\longrightarrow \topo_{/N}, \]
so that we have a homotopy pullback diagram 
\[
\begin{tikzcd}
\topo_{/U}\ar[d]\ar[r] &\topo_{/X}\ar[d] \\
\topo_{/Q}\ar[r] & \topo_{/N}
\end{tikzcd}
\]
Since $q_{\sring}:\mathsf{Top}^{\mathrm{loc}}(\sring)\rightarrow\mathsf{Top}$ is a Cartesian fibration and $Q\rightarrow N$ and $(U,\Of_X|_U)\rightarrow (X,\Of_X)$ are $q_{\sring}$-Cartesian morphisms, we have homotopy pullback diagrams 
\[
\begin{tikzcd}
\daff_{/(U,\Of_X|_U)}\ar[d]\ar[r] &\daff_{/(X,\Of_X)}\ar[d] \\
\topo_{/U}\ar[r] & \topo_{/X},
\end{tikzcd}\quad\quad\quad 
\begin{tikzcd}
\daff_{/Q}\ar[d]\ar[r] &\daff_{/N}\ar[d] \\
\topo_{/Q}\ar[r] & \topo_{/N},
\end{tikzcd}
\]
and we conclude that $\daff_{/(U,\Of_X|_U)}$ fits into a homotopy pullback diagram
\[
\begin{tikzcd}
\daff_{/(U,\Of_X|_U)}\ar[d]\ar[r] &\daff_{/(X,\Of_X)}\ar[d] \\
\daff_{/Q}\ar[r] & \daff_{/N},
\end{tikzcd}
\]
as desired.
\end{proof}
\begin{rmk}
Note that this proof continues to hold more generally when $M$ is a compact derived $\cinfty$-scheme and $Q\hookrightarrow N$ an open embedding of derived $\cinfty$-schemes.
\end{rmk}
To proceed, we will improve our understanding of the local structure of mapping stacks $\map(M,N)$ for $M$ and $N$ manifolds with $M$ compact, and more generally parametrized versions of these. This is the same problem one faces when one attempts to give the set of smooth maps $\cinfty(M,N)$ the structure of a smooth (convenient) manifold. A local chart for the manifold of mappings $\cinfty(M,N)$ at some $g:M\rightarrow N$ is constructed by means of a \emph{local addition}. This is a choice of an open neighbourhood $V_g$ of the zero section of $g^*TN$ for which the projection $V_g\rightarrow M$ admits an equivalence to the vector bundle $g^*TN\rightarrow M$, together with an open embedding $V_g\hookrightarrow M\times N$ onto a neighbourhood of the graph of $g$ fitting into a commuting diagram 
\[
\begin{tikzcd}
M \ar[d,"0",hook] \ar[r,"\mathrm{id}\times g",hook]   & M\times N\ar[d] \\
V_g \ar[r]\ar[ur,"\phi",hook] & M
\end{tikzcd}
\]
The map $V_g\hookrightarrow M\times N$ determines an open embedding $(\Gamma(V_g;M)\subset \Gamma(g^*TN;M))\rightarrow\cinfty(M,N)$ (where both sets are given the Whitney $\cinfty$-topology), this construction determines a chart modelled on an open subset of the convenient vector space $\Gamma(g^*TN;M)$, and the collection of all such charts clearly covers $\cinfty(M,N)$. The existence of local additions is a straightforward tubular neighbourhood argument using the exponential map for a connection; see for instance \cite[Theorem 42.1]{krieglmichor} or \cite[Theorem 12.10]{Palais}. Local additions can similarly be used to understand the local structure of derived mapping stacks of manifolds, and more generally of Weil restrictions. Below, we will also want to describe the local structure of maps of mapping stacks induced by maps of manifolds, This will require the construction of local additions with some additional properties, for which we will have to do a bit of work. 
\begin{lem}\label{lem:manifoldsmoothmapping}
Let $M$ be a manifold and let $F\rightarrow M$ be a finite rank vector bundle, then $\map_M(M,F)$ is smooth. Let $M$ and $N$ be manifolds and suppose that $M$ is compact, then the mapping stack $\map(M,N)$ is smooth.    
\end{lem}
\begin{proof}
We will first demonstrate this for trivial vector bundles with fibre $\R^n$. Consider the functor
\[ \sring\longrightarrow \spa,\quad\quad A\longmapsto \Hom_{\sring}(\cinfty(\R^n),A\oinfty\cinfty(M)).  \]
It suffices to show that 
\begin{enumerate}[$(1)$]
\item The functor $\Hom_{\sring}(\cinfty(\R^n),\cinfty(M)\oinfty\_):\sring\rightarrow \spa$ given by composing the functor $\_\oinfty \cinfty(M)$ with the functor corepresented by $\cinfty(\R^n)$ is a left Kan extension along the inclusion $\cartsp^{op}\subset \sring$.
\item The mapping stack functor 
\[ \map(M,\R^n):\daff^{op}\longrightarrow \spa,\quad\quad \spec\,A\longmapsto \Hom_{\daff}(M\times \spec\,A,\R^n)  \]
coincides with the functor
\[ \daff^{op}\longrightarrow \spa,\quad\quad \spec\,A\longmapsto \Hom_{\sring}(\cinfty(\R^n),\cinfty(M)\oinfty A), \]
 the restriction to $\daff^{op}\simeq \sring_{\gmt}$ of the functor $\Hom_{\sring}(\cinfty(\R^n),\cinfty(M)\oinfty\_)$. 
\item The mapping stack functor \[ \map(M,\R^n)_{\mathsf{Sm}}:\mathsf{Mfd}^{op}\longrightarrow \spa,\quad\quad W\longmapsto \Hom_{\mathsf{Mfd}}(M\times W,\R^n)  \]
is a restriction of the functor $\Hom_{\sring}(\cinfty(\R^n),\cinfty(M)\oinfty\_)$ to $\mathsf{Mfd}^{op}$
\end{enumerate}
Indeed, by transitivity of left Kan extensions (\cite[Proposition 4.3.2.8]{HTT}), it follows that the sheaf $\map(M,\R^n)$ on $\daff^{op}$ is a left Kan extension of its restriction to $\mathsf{Mfd}^{op}$. Since this operation coincides with the functor $\iota_{\mathsf{Sm}!}:\shv(\mathsf{Mfd})\rightarrow\dstack$, this implies that $\map(M,\R^n)$ is smooth. Note that $(1)$ follows from \cite[Proposition 5.5.8.15]{HTT} and the fact that $\cinfty(\R^n)$ is a compact projective object of $\sring$, and the assertions $(2)$ and $(3)$ are straightforward, using that $\cinfty(\R^ n)$ is a geometric derived $\cinfty$-ring. By taking retracts, we find that $\map_M(M,F)$ is also smooth for any finite rank vector bundle $F\rightarrow M$. Now for the case of a general manifold $N$, choose for each map $f:M\rightarrow N$ an open neighbourhood $V_f$ of zero section of $f^*TN$ and a diagram 
\[
\begin{tikzcd}
V_f\ar[dr] \ar[rr,hook,"\phi"] && M\times N\ar[dl] \\
& M
\end{tikzcd}
\]
where the map $\phi$ is an open embedding onto a subset containing of the graph of $f$ (we will produce such data in a more general setting in Construction \ref{cons:famlocaladdition}), then we have a map of stacks 
\[ \map_M(M,V_f) \longrightarrow \map(M,N)  \] 
which is pulled back from the map of stacks
\[ \map(M,V_f) \overset{\phi_!}{\longrightarrow} \map(M,M\times N).  \]
It follows from Lemma \ref{lem:mappingstackopen} that the map $\phi_!$ is an open substack inclusion. By the same argument, the map of stacks $\map_M(M,V_f)\rightarrow \map_M(M,f^*TN)$ is an open substack inclusion so that $\map_M(M,V_f)$ is smooth since $\map_M(M,f^*TN)$ is, as just verified. Consider the map \[h:\coprod_{f:M\rightarrow N}\map_M(M,V_f)\longrightarrow \map(M,N),\] then for any affine derived $\cinfty$-scheme $\spec\,A$, the collection $\{\spec\,A\times_{\map(M,N)}\map_M(M,V_f)\rightarrow\spec\,A\}_{f:M\rightarrow N}$ is (the image under Yoneda of) an open covering of $\spec\,A$ so that $h$ is an effective epimorphism. Since open substack inclusions are stacky \'{e}tale, it follows from Lemma \ref{lem:stackyetaleproperty} that each object in the \v{C}ech nerve of $h$ is smooth. Since $\smst\subset\dstack$ is stable under colimits, we conclude. 
\end{proof}
Our next result ensures that the formation of mapping stacks of manifolds with compact source preserves \'{e}tale maps and submersive maps.
\begin{lem}\label{lem:etaletarget}
Let $M$ be a compact manifold and let $f:Q\rightarrow N$ be a map between manifolds. Then the following hold true. 
\begin{enumerate}[$(1)$]
    \item If $f$ is a local diffeomorphism, then the map $\map(M,Q)\rightarrow \map(M,N)$ is \'{e}tale.
    \item If $f$ is a submersion, then $\map(M,Q)\rightarrow \map(M,N)$ is stacky submersive.
\end{enumerate}
\end{lem}
The proof of this lemma will use an auxiliary result asserting that given a submersion $f:X\rightarrow Y$, there exist local additions on $X$ and $Y$ compatible with $f$ in a suitable sense. To achieve this, we first establish the existence of \emph{families} of local additions. We start with some basic definitions. 
\begin{defn}
Let $X\rightarrow S$ be a submersion and let $f:S\rightarrow X$ be a section. A \emph{vector bundle neighbourhood} of $f(S)$ is an open neighbourhood $f(S)\subset V$ such that the map $V\rightarrow S$ admits the structure of vector bundle for which $f:S\rightarrow V$ is the zero section. 
\end{defn}
Here are some basic properties of vector bundle neighbourhoods. 
\begin{lem}\label{lem:vbnbhd}
Let $q:X\rightarrow S$ be a submersion and let $f:S\rightarrow X$ be a section.
\begin{enumerate}[$(1)$]
\item Let $f(S)\subset V\subset X$ be a vector bundle neighbourhood, then there is a canonical equivalence $f^*TX/S\simeq V$ of vector bundles over $S$.
\item Let $f(S)\subset V\subset X$ and $f(S)\subset W\subset X$ be vector bundle neighbourhoods, then there is a canonical equivalence $V\simeq W$ of vector bundles over $S$.
\item Suppose that $q:X\rightarrow S$ is a vector bundle and $f(S)\subset V\subset X$ a vector bundle neighbourhood, then there is a canonical equivalence $X\simeq V$ of vector bundles over $S$.
\item Suppose that $q:X\rightarrow S$ is a vector bundle. For any open neighbourhood $W$ of the zero section $S_0\subset X$, there exists a vector bundle neighbourhood $S_0\subset V\subset W$.
\end{enumerate}
\end{lem}
\begin{proof}
Note that $(2)$ follows from $(1)$ and $(3)$ follows from $(2)$. For $(1)$, note that since $V\subset X$ is open, the canonical map $TV/S\rightarrow V\times_XTX/S$ is an equivalence of vector bundles over $V$, and because $V\rightarrow S$ is a vector bundle, there is a canonical equivalence $V\times V\simeq  TV/S$ of vector bundles over $V$ so that pulling back along $f:S\hookrightarrow V$ yields an equivalence $V\simeq S\times_VTV/S$. Composing these, we have an equivalence $V\simeq f^*TX/S$. For $(4)$, choose a Riemannian metric $g$ on the vector bundle $X$ and choose a smooth function $\epsilon:S\rightarrow \R_{>0}$ such that the open set 
\[  V:= \{(s,e_s)\in X;\, g(e_s,e_s)<\epsilon(s) \}\subset X \]
is contained in $W$ (the existence of such a function is clear in case $X\rightarrow S$ is a trivial bundle and the general case follows using a partition of unity). Then 
\[  X\longrightarrow V,\quad\quad (s,e_s)\longmapsto (s,(\epsilon(s)/\sqrt{1+g(e_s,e_s)})e_s) \]
is a diffeomorphism over $S$ preserving the zero section.
\end{proof}
We will strengthen $(4)$ considerably below the next construction.
\begin{defn}
Let $q:X\rightarrow S$ be a submersion of manifolds. A \emph{local addition on $q$} is a pair of an open subset $V\subset TX/S$ in the relative tangent containing the zero section together with an open embedding $\phi:V\hookrightarrow X\times_SX$ fitting into a commuting diagram
\[
\begin{tikzcd}
X\ar[d,"0",hook]\ar[r,"\Delta",hook]& X\times_SX\ar[d,"\pi_2"] \\
V\ar[r,"p"] \ar[ur,"\phi",hook]& X
\end{tikzcd}
\]
where $\pi_2$ denotes the projection onto the second coordinate and $p$ is the projection $V\subset TX/S\rightarrow X$. We will write $(V,\phi)$ for a local addition on $q$, leaving the required commutativity implicit.
\end{defn}
\begin{cons}\label{cons:famlocaladdition}
Let $q:X\rightarrow S$ be a submersion and let $\nabla$ be a connection on the relative tangent bundle $TX/S$. We will construct a local addition on $q$ using the induced smoothly varying family of exponential maps in the fibres of $q$. This is a special case of the exponential map associated to an integrable Lie algebroid (see \cite[Section 2.3]{CrainicFernandez}), as the Lie algebroid $TX/S$ with tautological injective anchor map $TX/S\subset TX$ is the Lie algebroid associated to the Lie groupoid 
\[ \begin{tikzcd}  X\times_SX\ar[r,shift left]\ar[r,shift right] & X. \end{tikzcd}  \]
More precisely, let $p:TX/S\rightarrow X$ be the projection, then the connnection $\nabla$ is a vector bundle splitting $p^*TX\rightarrow TTX/S$ of the differential $TTX/S\rightarrow p^*TX$ of $p$ (the connection need not be linear, that is, the induced map $d_{\nabla}:\Gamma(TX/S;X)\rightarrow\Omega^1(TX/S;X)$ need not satisfy the Leibniz rule); by pulling back the connection $\nabla$ along the differential of the fibre inclusion $i_s:X_s\subset X$, we have for each $s\in q(S)$ a connection $\nabla_s=i_s^*\nabla$ on $TX_s=i_s^*TX/S$ such that the diagram
\[
\begin{tikzcd}
TTX_s \ar[d,"TTi_s"] & p_s^*TX_s \ar[l,"\nabla_s"'] \ar[d,"Ti_s"] \\
TTX & p^*TX\ar[l,"\nabla"']
\end{tikzcd}
\]
commutes, where $p_s$ is the projection $TX_s\rightarrow X_s$. The subbundle inclusion $TX/S\subset TX$ determines a section $F:TX/S\rightarrow p^*TX$ and therefore a horizontal \emph{geodesic} vector field $G=\nabla\circ F$, the vector field that carries $v_x\in T_xX/S\simeq T_xX_s$ to $(0,v_x)\in T_{(x,v_x)}(TX/S)/X \oplus T_xX/S$ using the splitting $T_{(x,v_x)}TX/S\simeq T_{(x,v_x)}(TX/S)/X \oplus T_xX/S$ induced by the connection. For each $s\in q(X)$, we similarly have a horizontal geodesic field $G_s:TX_s\rightarrow TTX_s$ and the commutativity of the diagram above guarantees that for each $s\in q(X)$, the vector fields $G_s$ and $G$ are $Ti_s$-related. Since there exists an integral curve for $G_s$ at each point in $TX_s$ for all $s \in q(S)$, integral curves are unique and we have a bijection of sets $\coprod_{s\in q(X)}TX_s\cong TX/S$, we deduce that the integral curves of $G$ are precisely the integral curves of all the vector fields $G_s$ as $s$ varies over $S$, that is, the integral curves of $G$ are the geodesics in the fibres associated with the connections $\nabla_s$. Let $\mathcal{D}\subset \R\times TX/S$ be the maximal flow domain of $G$ on which the flow $\mathrm{Fl}_G:\mathcal{D}\rightarrow TX/S$ is defined, then the map $(p\circ \mathrm{Fl}_G)\times p:\mathcal{D}\rightarrow X\times X$ factors through the closed submanifold $X\times_SX$ since every integral curve of $G$ lies wholly in some fibre. Let $V:=\mathcal{D}\times_{\R\times TX/S}\{1\}\times TX/S$, then $V\subset TX/S$ is an open neighbourhood of the zero section and the flow restricts to the exponential map $\phi:V\rightarrow X\times_SX$, which fits into a commuting diagram 
\[
\begin{tikzcd}
X\ar[d,"0",hook]\ar[r,"\Delta",hook]& X\times_SX\ar[d,"\pi_2"] \\
V\ar[r] \ar[ur,"\phi",hook]& X
\end{tikzcd}
\]
The commutativity of the right triangle is immediate and that of the left triangle is simply the fact that geodesics with initial velocity 0 are constant paths. At each fibre over $x$, $\phi$ restricts to the exponential map $(V_x\subset TX_{q(x)})\rightarrow X_{q(x)}$, the derivative of which at 0 is the identity. It follows that after shrinking $V$, we may assume that $\phi_x=\phi|_{V_x}:V_x\rightarrow X_{q(x)}$ is an open embedding, so that $V\rightarrow X\times_SX$ is an open embedding. By shrinking $V$ further and using Lemma \ref{lem:vbnbhd}, we may assume that $V\rightarrow X$ is the vector bundle $TX/S$ itself.
\end{cons}
\begin{cor}
Let $q:X\rightarrow S$ be a submersion, let $f:S\rightarrow X$ be a section and let $f(S)\subset U\subset X$ be an open neighbourhood of image of $f$. Then there exists a vector bundle neighbourhood $f(S)\subset W\subset U$. 
\end{cor}
\begin{proof}
Apply Construction \ref{cons:famlocaladdition} to the submersion $U\rightarrow S$ and pull back the resulting map $\phi:TX/S\simeq V\rightarrow U\times_SU$ over $U$ along the map $f:S\rightarrow U$.
\end{proof}
The preceding corollary suffices for showing that mapping stacks of sections of submersions with compact base are locally mapping stacks of sections of vector bundles (and are therefore smooth). For more general Weil restrictions, we will have to work a bit harder. Our next result, which was proposed by Andrew Stacey \cite[Theorem 5.1]{staceylocaladdition}, guarantees the existence of local additions compatible with a submersion. 
\begin{lem}\label{lem:staceyroberts}
Let $f:Y\rightarrow X$ be a submersion of manifolds and let $\nabla_X$ be an affine connection on $X$. Let $(W,\phi_X)$ be the associated local addition provided by Construction \ref{cons:famlocaladdition}. Then there exists an open neighbourhood $V_X\subset f^*TX$ of the zero section and a local addition $(V,\phi_Y)$ defined on the open neighbourhood $V=TY\times_{f^*TX} V_X\subset TY$ of the zero section of $TY$ such that the differential $Tf:TY\rightarrow TX$ carries $V$ into $W$ and the diagram
\[
\begin{tikzcd}
V\ar[d] \ar[r,"\phi_Y",hook] & Y\times Y\ar[d] \\
W\ar[r,"\phi_X",hook] & X\times X
\end{tikzcd}
\]
commutes.
\end{lem}
An informal explanation of the geometric idea behind the proof below runs as follows. A connection on $f:Y\rightarrow X$ splits the tangent $TY$ into pairs $(e^v_y,e^h_y)\in T_yY_{f(x)}\oplus T_{f(y)}X$ of vertical and horizontal vectors over any $y\in Y$, and we have to assign a point in $Y$ to any such pair in a neighbourhood of 0. Let $x=f(y)\in X$, so that $e^h_y\in T_xX$. A choice of connection $\nabla_X$ on $X$ yields for small times a unique geodesic in $X$ passing through $x$ with initial velocity $e^h_y$, so we can (smoothly on $TX$) assign to $e^h_y$ a point $x'=\exp_{X,x}(e^h_y)$ by following this geodesic until time 1 (if $e^h_y$ is close enough to 0). We can pull back the connection $\nabla_{X}$ to the bundle $f^*TX$ which yields integral curves on $f^*TX$ that project to the geodesics of $X$. Thus, we can assign (smoothly on $f^*TX$) to the vector $e^h_y$ also a point $y'$ in $Y$ such that $p(y')=x'$. To achieve the commuting diagram above, we have to modify the point $y'$ to some $y''$ so as to also take the vertical vector $e^v_y\in T_yY_x$ into account, subject to the requirement that the underlying point in $X$ remains the same, that is, we require that $p(y')=p(y'')=x'$. We achieve this by choosing a complete connection $\nabla_{Y/X}$ on $TY/X$ to first parallel transport $e^v_y$ along the path used to define $y'$ to some $e^v_{y'}\in T_{y'}Y_{x'}$, and then we use the local addition of Construction \ref{cons:famlocaladdition} to exponentiate $e^v_{y'}$ to some point $y''$. Since this local addition respects the fibres of $f$, we indeed have $f(y'')=x'$. 
\begin{proof}[Proof of Lemma \ref{lem:staceyroberts}]
Choose 
\begin{enumerate}[$(1)$]
\item a \emph{linear} connection $\nabla_{Y/X}$ on $TY/X$, which yields by Construction \ref{cons:famlocaladdition} and Lemma \ref{lem:vbnbhd} a local addition $\phi_{Y/X}:TY/X \hookrightarrow Y\times_XY$ over $Y$.
\item a connection on $Y \rightarrow X$; that is, a splitting of the short exact sequence 
\[0\longrightarrow TY/X\longrightarrow TY\longrightarrow f^*TX \longrightarrow 0  \]
of vector bundles over $Y$. 
\end{enumerate}
We will first establish the following assertion.
\begin{enumerate}
\item[$(*)$] Let $G_X\in \Gamma(TTX;TX)$ be the geodesic field associated to the connection $\nabla_X$ (as in Construction \ref{cons:famlocaladdition}). There exists a vector field $\widetilde{G}_X$ on $f^*TX$ that is $f^*$-related to $G_X$ where $f^*$ is the map $f^*TX\rightarrow TX$.  
\end{enumerate}
Let $p_Y:f^*TX\rightarrow Y$ denote the projection and let $f^*\nabla_{X}:p_Y^*TY\rightarrow Tf^*TX$ be the pullback connection on $f^*TX$; this determines a splitting $Tf^*TX\simeq T(f^*TX)/Y\oplus p_Y^*TY$ of vector bundles over $f^*TX$ such that the diagram
\[
\begin{tikzcd}
Tf^*TX \ar[d,"Tf^*"] &[2em] p_Y^*TY \ar[l,"f^*\nabla_{X}"'] \ar[d,"Tf"] \\
TTX  &[2em]  p_X^*TX \ar[l,"\nabla_X"']
\end{tikzcd}
\]
commutes, where $p_X$ is the projection $p_X:TX\rightarrow X$. To produce a horizontal vector field on $f^*TX$ is to give a section $s:f^*TX\rightarrow p_Y^*TY$, that is, a map $F:f^*TX\rightarrow TY$ over $Y$, but such a map is exactly what the connection $(2)$ provides. Because $G_X$ is also a horizontal vector field and $Tf^*$ preserves the horizontal subbundles by the diagram above, the assertion that $\widetilde{G}_X$ and $G_X$ are $f^*$-related amounts to the commutativity of the diagram 
\[
\begin{tikzcd}
f^*TX\ar[dr,"f^*"]\ar[rr,"F"] && TY\ar[dl,"Tf"] \\
 & TX
\end{tikzcd}
\]
which is manifest. We will refer to integral curves for the vector field $\widetilde{G}_X$ as \emph{$X$-geodesics}. Let $\exp_X:W\rightarrow X$ be the composition $\pi_1\circ \phi_X$ of the local addition with the first projection. Let $V_X\subset f^*TX$ be an open neighbourhood of the zero section for which the time 1 flow $\mathrm{Fl}^1_{\widetilde{G}_X}$ of the vector field $\widetilde{G}_X$ exists. Since $f^*$ is an open map, we may shrink $V_X$ so that $f^*(V_X)\subset W$. Define
\[\widetilde{\exp}_{X}:V_X\overset{\mathrm{Fl}^1_{\widetilde{G}_X}}\longrightarrow f^*TX \longrightarrow Y,\]
then $f\circ \widetilde{\exp}_{X}=\exp_X\circ f$ by $(*)$, so the diagram 
\[
\begin{tikzcd}
V_X \ar[d]\ar[r,"\widetilde{\exp}_{X}\times p_Y"] &[2em] Y\times Y\ar[d] \\
W\ar[r,"\phi_X",hook] &[2em] X\times X
\end{tikzcd}
\] 
commutes. Now use the splitting $(2)$ again to realize an equivalence $TY\simeq TY/X\times_Y f^*TX$ and consider the \emph{parallel transport map} 
\[ \mathsf{Pal}:  (TY/X\times_Y V_X \subset TY) \longrightarrow TY/X \]
that takes a pair $(e^{v}_y,e^h_y)\in T_yY_{f(y)}\times T_{f(y)}X$ to the vector $\mathsf{Pal}_{e^h_y}(e^v_y)\in T_{\widetilde{\exp}_{X}(e^h_y)}Y_{\exp_X(e^h_y)}$ defined by parallel transporting $e^v_y$ along the $X$-geodesic with initial position $y$ and initial velocity $e^h_y$ using the connection $\nabla_{Y/X}$; since $\nabla_{Y/X}$ is a linear connection, it is complete so that the parallel transport exists for all curves in $Y$. From smooth dependence of ODE solutions on initial conditions, it follows in a standard manner that the map $\mathsf{Pal}$ is smooth (see for instance \cite[Section 37.5]{krieglmichor}, \cite[Theorem 17.8]{michordiffgeo}), and by construction, the diagram 
\[
\begin{tikzcd}
TY/X\times_Y  V_X  \ar[d]\ar[r,"\mathsf{Pal}\times p"] &[2em] TY/X\times Y \ar[d] \\
W \ar[r,"\phi_X",hook] &[2em] X\times X 
\end{tikzcd}
\]
commutes, where $p$ is the projection $TY\rightarrow Y$. Now consider the composition 
\[
\begin{tikzcd} TY/X\times Y\ar[r,"\phi_{Y/X}\times\mathrm{id}"]&[2em] Y\times_XY\times Y\ar[r] & Y\times Y   \end{tikzcd}
\]
where the first map applies the local addition $\phi_{Y/X}$ and the second projects away the middle coordinate. Since the local addition $\phi_{Y/X}$ restricts to a local addition in all the fibres of $f$, this composition commutes with the projection to $X\times X$. Composing $\mathsf{Pal}$ with the composition above yields a commuting diagram 
\[
\begin{tikzcd}
TY/X\times_Y V_X\ar[d]\ar[r,"\phi_{Y}"] & Y\times Y \ar[d] \\
W \ar[r,hook,"\phi_X"] & X\times X 
\end{tikzcd}
\]
which we claim is the desired local addition. Using that $\phi_{Y/X}$ is a local addition on $f$ and that parallel transport along an $X$-geodesic with initial velocity 0 is the identity, the composition $\phi_Y\circ 0:Y\rightarrow Y\times Y$ is the diagonal embedding, and it is immediate from the construction that $\phi_{Y}$ commutes with both projections $TY/X\times_Y V_X\rightarrow Y$ and $\pi_2:Y\times Y\rightarrow Y$. To see that $\phi_Y$ is an open embedding, it suffices to argue that $\phi_Y$ restricts to an open embedding after pulling back to the fibres of these projections to $Y$ (since both are submersions). For each $y\in Y$, restricting along $\{y\}\hookrightarrow Y$ yields a commuting diagram 
\[
\begin{tikzcd}
T_yY/X \times (V_X\cap T_{f(y)}X) \ar[dr]\ar[rr,"\phi_{Y,y}"] && Y \ar[dl,"f"] \\
& X
\end{tikzcd}
\]
where the left diagonal map factors as 
\[\begin{tikzcd}T_yY/X \times (V_X\cap T_{f(y)}X)\ar[r] &V_X\cap T_{f(y)}X \subset W\cap T_{f(y)}X \ar[r,"\exp_{X,f(y)}"]&[2em] X. \end{tikzcd}\]
Since the diagonal maps are submersion, it suffices to show that $\phi_{Y,y}$ is an open embedding restricted to the fibre of each $x\in X$ that lies in the image of the left diagonal map. Each such $x$ is the image under $\exp_{X,f(y)}$ of a unique $e^h_{y}\in V_X\cap T_{f(y)}X$ and restricting along $\{x\}\hookrightarrow X$ yields the map 
\[\begin{tikzcd} T_yY/X \ar[r,"\mathsf{Pal}_{e^h_y}"] & T_{\widetilde{\exp}_X(e^h_y)}Y_{\exp_X(e^h_y)} \ar[r,"\phi_{Y/X,\widetilde{\exp}_X(e^h_y)}"] &[4em] Y_{\exp_X(e^h_y)} \end{tikzcd}  \]
and we conclude by observing that the first map is a linear isomorphism and the second is an open embedding.
\end{proof}
\begin{proof}[Proof of Lemma \ref{lem:etaletarget}]
Choose local additions $(V,\phi_Q)$ and $(W,\phi_N)$ satisfying the conclusion of Lemma \ref{lem:staceyroberts}, then we deduce for each map $g:M\rightarrow Q$ a commuting diagram 
\[
\begin{tikzcd}
\map_M(M,g^*V) \ar[d,"Tf_!"]\ar[r] & \map(M,Q)\ar[d,"f_!"]\\
\map_M(M,(f\circ g)^*W) \ar[r] &\map(M,N)
\end{tikzcd}
\]
wherein the horizontal maps are open substack inclusions by Lemma \ref{lem:mappingstackopen}. If $f$ is a local diffeomorphism, then the left vertical map is an open substack inclusion as well. Varying $g$ over all smooth maps, we have an effective epimorphism $\coprod_{g:M\rightarrow Q}\map_M(M,g^*V)\rightarrow \map(M,Q)$, so we conclude that $f$ is \'{e}tale. If $f$ is a submersion, then choosing a connection on $f$ determines an equivalence $g^*V\simeq g^*TQ/N\times_M g^*V_X$ with $V_X$ as in Lemma \ref{lem:staceyroberts} over $(f\circ g)^*W$, which in turn yields an equivalence $\map_M(M,g^*V)\simeq \map_M(M,g^*V_X)\times \map_M(M, g^*TQ/N)$. It follows from Lemma \ref{lem:manifoldsmoothmapping} that $f$ is stacky submersive. 
\end{proof}

\begin{proof}[Proof of Proposition \ref{prop:weilressmooth}]
By assumption, for each morphism $S'\rightarrow S$ where $S'$ is a manifold, the map $Y\times_{S}S'\rightarrow X\times_SS'$ is an $S'$-family of submersions between manifolds. We invoke the universality of colimits in $\dstack$, Remark \ref{rmk:weilrestrictionpullback} and the fact that the collection of smooth objects is stable under colimits in $\dstack$ to suppose that $S$ is a manifold and $Y\rightarrow X$ an $S$-family of submersions. Since $M\rightarrow S$ is a locally trivial family of manifolds, we can choose a (countable) open cover $\{U_i\rightarrow M\}$ and suppose that $S$ is a countable disjoint union of manifolds and $M\rightarrow S$ equivalent to a trivial bundle $N\times S\rightarrow S$ (using that coproducts are disjoint and colimits are universal in $\dstack$, we may suppose that $S$ is a Cartesian space, but this is not required for this argument). It follows from Lemma \ref{lem:weilresttrivbundle} that the Weil restriction fits into a pullback diagram 
\[
\begin{tikzcd}
\mathsf{Res}_{M/S}(Y) \ar[d] \ar[r] & S\times \map(N,Y) \ar[d] \\
S \ar[r] & S\times \map(N,M).    
\end{tikzcd}
\]
Since $Y\rightarrow M$ is a submersion, the right vertical map is stacky submersive by Proposition \ref{lem:etaletarget}. By Lemma \ref{lem:manifoldsmoothmapping}, the stacks $ S\times \map(N,Y)$ and $S\times \map(N,M)$ are smooth, so we conclude by invoking $(3)$ of Proposition \ref{prop:stackysubmersive}.
\end{proof}
We now generalize Lemmas \ref{lem:mappingstackopen} and \ref{lem:etaletarget} to Weil restrictions.
\begin{prop}\label{prop:etaletarget}
Let $S$ be a smooth stack, let $M\rightarrow S$ be a proper $S$-family of manifolds, let $Y\rightarrow M$ be an $S$-family of submersions and let $f:Q\rightarrow Y$ be a map of stacks. 
\begin{enumerate}[$(1)$]
\item If $f$ is an open substack inclusion, then the induced map $\mathsf{Res}_{M/S}(Q)\rightarrow \mathsf{Res}_{M/S}(Y)$ is an open substack inclusion.
\item If $f$ is \'{e}tale, then $\mathsf{Res}_{M/S}(Q)\rightarrow \mathsf{Res}_{M/S}(Y)$ is \'{e}tale.
\item If $f$ is submersive, then $\mathsf{Res}_{M/S}(Q)\rightarrow \mathsf{Res}_{M/S}(Y)$ is stacky submersive.
\end{enumerate}
\end{prop}
\begin{proof}
Using Remark \ref{rmk:weilrestrictionpullback} and the fact that the property of being an open substack inclusion (\'{e}tale, stacky submersive) is a local property of morphisms, we may suppose that $S$ is a manifold, $M\rightarrow S$ is a trivial bundle $N\times S\rightarrow S$ with $N$ compact, $Y\rightarrow M$ is a submersion and $Q\rightarrow Y$ is an open embedding (\'{e}tale, submersive). It follows from Lemma \ref{lem:weilresttrivbundle} that there is a pullback diagram
\[
\begin{tikzcd}
\mathsf{Res}_{M/S}(Q) \ar[d] \ar[r] & S\times \map(N,Q) \ar[d] \\
\mathsf{Res}_{M/S}(Y)  \ar[r] & S\times \map(N,Y). 
\end{tikzcd}
\]
Now we conclude by invoking Lemmas \ref{lem:mappingstackopen} and \ref{lem:etaletarget}.
\end{proof}
We now give a result on the local structure of the Weil restrictions of interest to us that we will use later.
\begin{lem}\label{lem:localstructureweil}
Let $S$ be a manifold, let $q:M\simeq N\times S\rightarrow S$ be a trivial $S$-family of manifolds with fibre $N$ and let $f:Y\rightarrow M$ be an $S$-family of submersions. Write $i_s:Y_s\hookrightarrow Y$ for the inclusion of the fibre. For each $s\in q(S)$ and each section $\sigma:N\simeq M_s\rightarrow Y_s$ of $f_s$, there exists
\begin{enumerate}[$(1)$]
    \item An open neighbourhood $s\in S'\subset S$.
    \item A vector bundle neighbourhood $V_{s,\sigma}$ of the zero section of $\sigma^*TY_s/N$.
    \item An open embedding $S'\times V_{s,\sigma}\hookrightarrow Y$ fitting into a commuting diagram 
\[
\begin{tikzcd}
N \ar[r,hook,"\iota_s \circ \sigma"]\ar[d,"\{s\}\times 0"'] & Y\ar[d] \\
 S'\times V_{s,\sigma}\ar[ur]\ar[r]  & S\times N
\end{tikzcd}
\]
\end{enumerate}
\end{lem}
\begin{proof}
Choose connections $\nabla_S$ and $\nabla_N$ on $TS$ and $TN$ respectively and let $\phi_S:(W_S\subset TS)\rightarrow S\times S$ and $\phi_N:(W_N\subset TN)\rightarrow N\times N$ be the associated local additions, then the connection $\nabla_S\oplus \nabla_N$ on $S\times N$ has associated local addition $\phi_N\times\phi_N :W_S\times W_N\rightarrow M\times M$. Using Lemma \ref{lem:staceyroberts}, we find some neighbourhood $V \subset TY$ of the zero section which the differential $Tf$ carries into $W_S\times W_N$ and a commuting diagram 
\[
\begin{tikzcd}
(i_s \circ \sigma)^*V \ar[d]\ar[r] & Y\times N\ar[d] \\
W_{S,s} \times W_N \ar[r] & S\times N \times N
\end{tikzcd}
\]
where $W_{S,s}=W_S\cap T_sS\subset T_sS$ and the lower horizontal map is the product of the exponential map $\exp_{S,s}:W_{S,s}\rightarrow S$ at $s$ associated with the connection $\nabla_S$ with the local addition $\phi_X$. Note that this exponential map $\exp_{S,s}$ is an open embedding whose image contains $s$. Now we pull back the square above along the diagonal maps in the diagram
\[
\begin{tikzcd}
& W_{S,s} \times N\ar[dl,"\mathrm{id}\times 0"'] \ar[dr,"\mathrm{id}\times\Delta"] \\
W_{S,s}\times W_N \ar[rr] && S\times N\times N 
\end{tikzcd}
\]
to obtain open sets $W'_{S,s}\subset W_{S,s}$ and $V_{s,\sigma}\subset \sigma^*TY/N$ containing the zero section and open embedding $W'_{S,s}\times V_{s,\sigma}\rightarrow Y\times_SW_{S,s}$ over $W_{S,s}\times N$. We complete the proof by shrinking $V_{s,\sigma}$ to a vector bundle neighbourhood of the zero section and letting $S'$ be the image of $W'_{S,s}\hookrightarrow S$.
\end{proof}
\begin{lem}\label{lem:weilreslocal2}
Let $S$ be a derived $\cinfty$-stack and let $X\rightarrow Y$ be any map of derived $\cinfty$-stacks. Then the Weil restriction $\mathsf{Res}_{X\times S/S}(Y\times S)$ fits into a pullback diagram    
\[
\begin{tikzcd}
\mathsf{Res}_{X\times S/S}(Y\times S) \ar[d]\ar[r] & S\times \map(X,Y)\ar[d] \\
S\ar[r] &S\times \map(X,X)
\end{tikzcd}
\]
in $\dstack_{/S}$, that is, we have a canonical equivalence $\mathsf{Res}_{X\times S/S}(Y\times S) \simeq S\times\map_X(X,Y)$. 
\end{lem}
\begin{proof}
It follows from Lemma \ref{lem:weilresttrivbundle} that we have a pullback diagram 
\[
\begin{tikzcd}
\mathsf{Res}_{X\times S/S}(Y\times S) \ar[d]\ar[r] & S\times \map(X,Y)\times\map(X,S)\ar[d] \\
S\ar[r] &S\times \map(X,X)\times \map(X, S)
\end{tikzcd}
\]
and we conclude after projecting away $\map(X, S)$.
\end{proof}
\begin{prop}
Let $S$ be a manifold and let $q:N\times S\rightarrow S$ be a trivial $S$-family of manifolds with fibre $N$ and let $f:Y\rightarrow N$ be an $S$-family of submersions. Let $\sigma:N\rightarrow Y_s$ be a section at $s\in S$. Suppose that $N$ is compact, then composition with an open embedding $S'\times V_{s,\sigma}\hookrightarrow Y$ provided by Lemma \ref{lem:localstructureweil} induces an open substack inclusion 
\[\begin{tikzcd} S'\times\map_N(N,TY_s/N) \ar[rr,hook] \ar[dr] && \mathsf{Res}_{M/S}(Y) \ar[dl] \\
& S
\end{tikzcd} \]
whose image contains the section $\sigma$.
\end{prop}
\begin{proof}
Pull back the Weil restriction along the open embedding $S'\subset S$ and combine Lemmas \ref{lem:localstructureweil}, \ref{lem:weilreslocal2} and Proposition \ref{prop:etaletarget}.
\end{proof}

\subsection{Jet spaces}
We will define elliptic moduli problems as a subclass of \emph{differential} moduli problems. For $S=*$, this notion is not hard to grasp: recall that a map 
\[ P:\Gamma(Y_1;M)\longrightarrow\Gamma(Y_2;M) \]
between sections of submersions over a manifold $M$ is a $k$'th order nonlinear partial differential equation just in case there exists a smooth map $\mathcal{P}:J^k_M(Y_1)\rightarrow Y_2$ from the $k$'th jet space of the map $Y_1\rightarrow M$ fitting as the top horizontal map in a commuting diagram 
\[
\begin{tikzcd}
   J^k_M(Y_1)\ar[rr,"\mathcal{P}"] \ar[dr] && Y_2 \ar[dl]\\
   & M
\end{tikzcd}
\]
such that $P=\mathcal{P}\circ j^k$, where $j^k$ is the $k$'th jet prolongation $\Gamma(Y_1;M)\rightarrow \Gamma(J^k_M(Y_1);M)$. For our purposes, we will need to understand jet bundles smoothly parametrized by smooth stacks. Below, we will first treat the case of $S$ a manifold and then generalize to arbitrary smooth stacks by descent.\\
Let $S$ be a smooth manifold and let 
\[
\begin{tikzcd}
   Y\ar[dr]\ar[rr,"f"] && X\ar[dl] \\
   & S
\end{tikzcd}
\]
be an $S$-family of submersions. For $k\geq 1$ an integer, the \emph{relative $k$'th order jet space of sections of $f$} is the manifold $J^k_{X/S}(Y)$ whose points are triples $(s,x,\sigma_{x})$ where $s\in S$, $x$ is a point in the fibre $X_s$ and $\sigma_{x}$ is the $k$'th order jet of a section of $f_s:Y_s\rightarrow X_s$ at $x$. What follows is a rigorous construction of the manifold $J^k_{X/S}(Y)$.
\begin{cons}[Infinitesimal diagonal neighbourhood]
Let $q:X\rightarrow S$ be a submersion of manifolds. Let the ideal $I_{\Delta}\subset\cinfty(X\times_SX)$ be the kernel of the map 
\[ \Delta^*:\cinfty(X\times_SX)\longrightarrow \cinfty(X) \]
of $\cinfty$-rings induced by the diagonal $\Delta:X\rightarrow X\times_SX$. We will write $q^{(k)}$ or $(X/S)^{(k)}$ for the spectrum (Theorem \ref{thm:spectrumglobalsections}) of the $\cinfty$-ring $\cinfty(X\times_SX)/{I^{k+1}_{\Delta}}$, where $I^r_{\Delta}$ is the $r$'th power of the ideal $I_{\Delta}$. We have a commuting diagram
\[
\begin{tikzcd}
 (X/S)^{(k)}\ar[d,"\pi_2"]\ar[r,"\pi_1"] & X\ar[d] \\
 X \ar[r] & S
\end{tikzcd}
\]
of affine (derived) $\cinfty$-schemes. When the base $S$ is a point, we will write $X^{(k)}$ for the $k$'th order infinitesimal diagonal neighbourhood $(X/*)^{(k)}$. Viewing the square 
\[
\begin{tikzcd}
 (X/S)^{(k)}\ar[d,"\pi_2"]\ar[r,"\pi_1"] & X\ar[d] \\
 X \ar[r] & S
\end{tikzcd}
\]
as a square of derived $\cinfty$-stacks, the $k$'th order infinitesimal diagonal neighbourhood determines a functor 
\[  \fun(\Delta^1,\mathsf{Mfd})^{\mathsf{Sub}} \longrightarrow \fun(\Delta^1\times\Delta^1,\dstack).  \]
\end{cons}
\begin{rmk}
Since the kernel $I_{\Delta}$ of the map $\cinfty(X\times_SX)\rightarrow \cinfty(X)$ is finitely generated, its powers are finitely generated as well. Since finitely generated ideals of finitely presented $\cinfty$-rings are germ determined, the $\cinfty$-rings $\cinfty(X\times_SX)/I_{\Delta}^k$ are geometric as derived $\cinfty$-rings. Consequently, by the equivalence $\daff\simeq \sring_{\gmt}^{op}$ the functor
\[  \mathsf{Mfd}\longrightarrow \daff,\quad\quad  (X\rightarrow S)\longmapsto (X/S)^{(k)}\]
coincides with the functor $(X\rightarrow S)\mapsto \cinfty(X\times_SX)/I_{\Delta}^{k+1}$. 
\end{rmk}
Before we introduce the jet space functor we have to recall some concepts from commutative algebra. Recall that a map $f:\tilde{A}\rightarrow A$ in the category $\calg^{\heartsuit}_{B/}$ of commutative algebras over a commutative ring $B$ is a \emph{square-zero extension} if it is a surjection and in the ideal $\ker f$ the product of any two elements vanishes. In this case, one defines an $A$-module structure on $\ker f$ by  letting an element $a\in A$ act on $I$ by any element in the preimage $f^{-1}(a)$. A square-zero extension $f:\tilde{A}\rightarrow A$ in $\calg^{\heartsuit}_{B/}$ is \emph{trivial} if $f$ admits a section (in the category of commutative $B$-algebras) in which case the inclusion $I\subset\tilde{A}$ and the section $A\rightarrow\tilde{A}$ determine an isomorphism $A\oplus I\cong \tilde{A}$ of $A$-algebras where the $A$-module $A\oplus I$ is an algebra with multiplication defined via the formula $(a,m)\cdot (a',m'):=(aa',am'+a'm)$. In the tower of commutative algebras
\[ \ldots\overset{f_{k+1}}{\longrightarrow} \cinfty(X\times_SX)/{I^{k+1}_{\Delta}}\overset{f_{k}}{\longrightarrow} \cinfty(X\times_SX)/{I^{k}_{\Delta}} \overset{f_{k-1}}{\longrightarrow}\ldots \overset{f_{2}}{\longrightarrow} \cinfty(X\times_SX)/{I^{2}_{\Delta}} \overset{f_{1}}{\longrightarrow}\cinfty(X\times_SX)/{I_{\Delta}}\cong \cinfty(X)  \]
every consecutive map is a square-zero extension: the kernel of the map $\cinfty(X\times_SX)/{I^{K+1}_{\Delta}}\rightarrow \cinfty(X\times_SX)/{I^{K}_{\Delta}} $ is the ideal $I^{k}_{\Delta}/I^{k+1}_{\Delta}$ which obviously has the property that the product of any two elements vanishes. Just as for square-zero extensions, for each $k\geq 0$ the $\cinfty(X\times_SX)$-module structure on $I^{k+1}_{\Delta}/I_{\Delta}^k$ comes from an $\cinfty(X)$-module structure by letting $g\in \cinfty(X)$ act by any element in the preimage of $g$ by the restriction $\Delta^*:\cinfty(X\times_SX)\rightarrow(X)$ along the diagonal. This $\cinfty(X)$-module structure coincides with the one induced by the restriction $\pi_i^*:\cinfty(X)\rightarrow\cinfty(X\times_SX)/I^{k+1}$ along either projection $\{\pi_i\}_{i=1,2}$. It follows that in the abelian category $\Mod_{\cinfty(X)}^{\heartsuit}$, we have a short exact sequence 
\begin{align}\label{eq:jetbundleseq}  0\longrightarrow I^{k}_{\Delta}/I^{k+1}_{\Delta} \longrightarrow \cinfty(X\times_SX)/{I^{k+1}_{\Delta}}\longrightarrow \cinfty(X\times_SX)/{I^{k}_{\Delta}} \longrightarrow 0.    \end{align}
\begin{lem}\label{lem:jebundleseq}
The short exact sequence \eqref{eq:jetbundleseq} is a sequence of finitely generated projective $\cinfty(X)$-modules. Moreover, there is a canonical equivalence $I^{k}_{\Delta}/I^{k+1}_{\Delta} \simeq \sym_{\cinfty(X)}^k(\Gamma(T^{\vee}X/S))$, the $k$'th symmetric power of the global sections of the vertical cotangent bundle of $X$ with respect to the projection $X\rightarrow S$.
\end{lem}
\begin{proof}
We proceed by induction on $k$. For $k=1$, the short exact sequence \eqref{eq:jetbundleseq} yields the sequence 
\[0\longrightarrow I_{\Delta}/I^{2}_{\Delta} \longrightarrow \cinfty(X\times_SX)/{I^{k+1}_{\Delta}}\longrightarrow \cinfty(X) \longrightarrow 0. \]
Since $\cinfty(X)$ is a free $\cinfty(X)$-module of rank 1, it suffices to show that there is a canonical equivalence of $ I_{\Delta}/I^{2}_{\Delta}$ with the global sections of the vertical cotangent sheaf of the submersion $X\rightarrow S$, as the category of finitely projective modules is stable under extensions. Consider the map 
\begin{align}\label{eq:conormal}
I\subset\cinfty(X\times_SX)\overset{d_{\dR}}{\longrightarrow}\Gamma(T^{\vee}(X\times_SX)) \longrightarrow \Gamma(X\times_{X\times_SX}T^{\vee}(X\times_SX))
\end{align}   
that takes a function $g$ in $I_{\Delta}$ to the restriction to the diagonal of the 1-form $d_{\dR}g$. One readily verifies using that $X$ is a closed submanifold of $X\times_SX$ that the kernel of the map \eqref{eq:conormal} coincides with $I^{2}_{\Delta}$ and that the sequence 
\begin{align}\label{eq:con1} 0\longrightarrow I_{\Delta}/I_{\Delta}^2\longrightarrow \Gamma(X\times_{X\times_SX}T^{\vee}(X\times_SX)) \overset{\Delta^*}{\longrightarrow}\Gamma(T^{\vee}X)\longrightarrow 0  \end{align}
is exact (this is nothing but a differential geometric analogue of the conormal sequence from algebraic geometry). The cotangent of $X\times_SX$ also fits into a short exact sequence 
\[0\longrightarrow X\times_SX\times_X T^{\vee}X \longrightarrow T^{\vee}(X\times_SX) \longrightarrow X\times_SX\times_X T^{\vee}X/S \longrightarrow 0     \]
of vector bundles over $X\times_SX$. Pulling back along the diagonal $X\hookrightarrow X\times_SX$ and taking global sections, we have a short exact sequence 
\begin{align}\label{eq:con2} 0\longrightarrow \Gamma(T^{\vee}X) \longrightarrow \Gamma(X\times_{X\times_SX}T^{\vee}(X\times_SX))\longrightarrow \Gamma(T^{\vee}X/S)\longrightarrow 0  \end{align}
of $\cinfty(X)$-modules. The first map is split by the map $\Delta^*$, so it follows that the composition 
\[  I_{\Delta}/I^2_{\Delta}\longrightarrow \Gamma(X\times_{X\times_SX}T^{\vee}(X\times_SX))\longrightarrow \Gamma(T^{\vee}X/S) \]
where the first map comes from the exact sequence \eqref{eq:con1} and the second from the exact sequence \eqref{eq:con2} is an isomorphism. For the inductive step, it suffices to observe an equivalence $ \sym_{\cinfty(X)}^k(\Gamma(T^{\vee}X/S))\simeq I^{k}_{\Delta}/I^{k+1}_{\Delta}$ of $\cinfty(X)$-modules and invoke the fact that finitely generated projective modules are stable under retracts again.
\end{proof}
\begin{rmk}\label{rmk:splitjetspace}
The map $\pi_2^*:\cinfty(X)\rightarrow \cinfty(X\times_SX)/I^{2}_{\Delta}$ induced by the second projection provides a section of the square-zero extension $f_1:\cinfty(X\times_SX)/I^{2}_{\Delta}\rightarrow \cinfty(X)$, so $\cinfty(X\times_SX)/I^{2}_{\Delta}$ can be identified with the trivial square zero extension $\cinfty(X)\oplus I_{\Delta}/I^{2}_{\Delta}$ in the category of commutative $\R$-algebras. It follows that for $k=1$, the short exact sequence \eqref{eq:jetbundleseq} splits \emph{canonically} via the section $\cinfty(X)\rightarrow \cinfty(X\times_SX)/I^{2}_{\Delta}$. For $k>1$, the square-zero extension $\cinfty(X\times_SX)/{I^{k+1}_{\Delta}}{\rightarrow} \cinfty(X\times_SX)/{I^{k}_{\Delta}}$ is in general not trivial and splits only (noncanonically) at the level of $\cinfty(X)$-modules. Also note that while the square-zero extension $f_1$ is trivialized via the section $\pi_2^*$ in the category of commutative $\R$-algebras, the square-zero extension $f_1:\cinfty(X\times_SX)/I^{2}_{\Delta}\rightarrow
\cinfty(X)$ viewed as a map of commutative $\cinfty(X)$-algebras via the first projection $\pi_1^*$ is \emph{not} trivial, as the map $\pi_2^*$ does not yield a section for $f_1$ over $\cinfty(X)$. In fact one readily verifies that, using the section $\pi_2$ to trivialize the square-zero extension $f_1$, the map 
\[  \cinfty(X)\overset{\pi_1^*}{\longrightarrow} \cinfty(X)\oplus T^{\vee}X/S \]
carries $f$ to $(f,d_{\dR}f|_{TX/S})$. 
\end{rmk}
\begin{lem}\label{lem:ntruncinfnbhd}
Let $X\rightarrow S$ be a submersion of manifolds and let $M$ be a $\cinfty(X\times_SX)$-module for some $n\geq 0$. Then for all $k\geq 0$ the following holds: should $M\otimes_{\cinfty(X\times_SX)}\cinfty(X)$ be $n$-truncated for some $n\geq 0$, then $M\otimes_{\cinfty(X\times_SX)}\cinfty(X\times_SX)/I^{k+1}$ is $n$-truncated as well. 
\end{lem}
\begin{proof}
We proceed by induction on $k$. For $k=0$, there is nothing to prove. We have a short exact sequence
\[0\longrightarrow \sym^{k}_{\cinfty(X)}(\Gamma(T^{\vee}{X/S})) \longrightarrow \cinfty(X\times_SX)/{I^{k+1}_{\Delta}}\longrightarrow \cinfty(X\times_SX)/{I^{k}_{\Delta}} \longrightarrow 0\]
of discrete $\cinfty(X\times_SX)$-modules resulting in a fibre sequence 
\[  M\otimes_{\cinfty(X\times_SX)}\sym^{k}_{\cinfty(X)}(\Gamma(T^{\vee}{X/S})) \longrightarrow M\otimes_{\cinfty(X\times_SX)}\cinfty(X\times_SX)/I_{\Delta}^{k+1}\longrightarrow M\otimes_{\cinfty(X\times_SX)}\cinfty(X\times_SX)/I_{\Delta}^{k}  \]
of $\cinfty(X\times_SX)$-modules. By the inductive hypothesis it suffices to show that $M\otimes_{\cinfty(X\times_SX)}\sym^{k}_{\cinfty(X)}(\Gamma(T^{\vee}{X/S}))$ is $n$-truncated. The $\cinfty(X\times_SX)$-module structure on $\sym^{k}_{\cinfty(X)}(\Gamma(T^{\vee}{X/S}))$ is obtained via the functor \[\Mod_{\cinfty(X\times_SX)}\longrightarrow\Mod_{\cinfty(X)}\]
induced by the map $\cinfty(X\times_SX)\rightarrow \cinfty(X)$. Since $\sym^{k}_{\cinfty(X)}(\Gamma(T^{\vee}{X/S}))$ is finitely generated and projective as a $\cinfty(X)$-module, it is a retract of a finitely generated and free $\cinfty(X)$-module (as $\Q\subset\R$); it follows that also in the \infcat of $\cinfty(X\times_SX)$-modules, $\sym^{k}_{\cinfty(X)}(\Gamma(T^{\vee}{X/S}))$ is a retract of a finite sum of the module $\cinfty(X)$ and we reduce to the assumption that $M\otimes_{\cinfty(X\times_SX)}\cinfty(X)$ is $n$-truncated.
\end{proof}
\begin{lem}\label{lem:infnbhpb}
Suppose we are given a pullback diagram
\[
\begin{tikzcd}
  X'\ar[d,"q'"]\ar[r] &  X\ar[d,"q"] \\
  S'\ar[r] & S
\end{tikzcd}
\]
of manifolds where $q$ (and $q'$) is a submersion (thus by transversality this is a pullback of affine derived $\cinfty$-schemes). Then the induced diagram
\[
\begin{tikzcd}
  (X'/S')^{(k)}\ar[d]\ar[r] &  (X/S)^{(k)}\ar[d] \\
  S'\ar[r] & S
\end{tikzcd}
\]
is a pullback of affine derived $\cinfty$-schemes as well.
\end{lem}
\begin{proof}
We will show the stronger assertion that the diagram
\[
\begin{tikzcd}
\cinfty(S)\ar[d]\ar[r] & \cinfty(X\times_SX)/I^{k+1}_{\Delta}\ar[d] \\
\cinfty(S') \ar[r] & \cinfty(X'\times_{S'}X')/I^{k+1}_{\Delta}
\end{tikzcd}
\]
is a pushout of derived $\cinfty$-rings (we abuse notation by letting $I^{k+1}_{\Delta}$ denote both the $k$'th power of the kernel of $\cinfty(X\times_SX)\rightarrow \cinfty(X)$ and of $\cinfty(X'\times_{S'}X')\rightarrow \cinfty(X')$). Consider the larger diagram
\[
\begin{tikzcd}
\cinfty(S)\ar[d]\ar[r] &\cinfty(X\times_SX) \ar[d]\ar[r]&\cinfty(X\times_SX)/I^{k+1}_{\Delta}\ar[r]\ar[d] & 
\cinfty(X)\ar[d] \\
\cinfty(S') \ar[r] &\cinfty(X'\times_{S'}X')\ar[r] &\cinfty(X'\times_{S'}X')/I^{k+1}_{\Delta} \ar[r]& \cinfty(X')
\end{tikzcd}
\]
then it suffices to argue that the left and middle square are both pushouts. The left square is obviously a pushout. For the square on the middle, we invoke \cite[Corollary 3.1.8]{cinftyI} to reduce to the assertion that the underlying square of derived commutative $\R$-algebras is a pushout. Since the left square is a pushout and the total rectangle is a pushout, we deduce that the right rectangle obtained by composing the middle and right square is a pushout. Invoking \cite[Corollary 3.1.8]{cinftyI} again, we deduce that that we have an equivalence $
\cinfty(X'\times_{S'}X')\otimes_{\cinfty(X\times_SX)}\cinfty(X)\simeq \cinfty(X')$; in particular, the derived tensor product is 0-truncated. It follows from Lemma \ref{lem:ntruncinfnbhd} that it suffices to argue that the middle square is a pushout of \emph{ordinary} commutative $\R$-algebras. This is elementary commutative algebra, using that the rectangle consisting of the middle and right square is a pushout and that the ideals $I_{\Delta}$ are kernels of the horizontal maps in this rectangle.
\end{proof}
It follows that the $k$'th order infinitesimal diagonal neighbourhood
\[
\begin{tikzcd}
\fun(\Delta^1,\mathsf{Mfd})^{\mathsf{Sub}}\ar[rr]\ar[dr,"\ev_1"] && \fun(\Delta^1\times\Delta^1,\dstack)\times_{\dstack}\mathsf{Mfd} \ar[dl,"\ev_{\infty}"] \\
& \mathsf{Mfd}
\end{tikzcd}
\]
carries $\ev_1$-Cartesian edges to $\ev_{\infty}$-Cartesian edges, where the right diagonal functor evaluates at the cone point of $\Delta^1\times\Delta^1\cong (\Lambda^2_0)^{\rhd}$. Since both diagonal maps unstraighten to sheaves, we can expect a natural transformation carrying every submersive map $q:X\rightarrow S$ of stacks to a square diagram of derived $\cinfty$-stacks we can interpret as the $k$'th order infinitesimal diagonal neighbourhood of $q$. The only obstruction to carrying out this argument is the fact that the fibres of $\ev_{\infty}$ are not small, but this is easily fixed by imposing a cardinality bound. Let $\kappa$ be an uncountable regular cardinal and let $\Of^{\Lambda^2_0,\kappa} \subset \fun(\Delta^1\times\Delta^1,\dstack)$ be the full subcategory spanned by those diagrams for which the morphism $\Delta^1\cong (\Delta^0)^{\rhd}\hookrightarrow (\Lambda^2_0)^{\rhd}$ is \emph{relatively $\kappa$-compact} in the sense of \cite[Definition 6.1.6.4]{HTT}. Since being a relatively $\kappa$-compact morphism is a local property of morphisms of an \inftop for $\kappa$ sufficiently large, we deduce from Remark \ref{rmk:localdiagram} that the property determined by $\Of^{\Lambda^2_0,\kappa}$ is a local property of $\Delta^1\times\Delta^1$-diagrams. Now choose a regular cardinal $\kappa$ large enough to contain the essential image of upper horizontal functor in the diagram above, then we have diagram
\[
\begin{tikzcd}
\fun(\Delta^1,\mathsf{Mfd})^{\mathsf{Sub}}\ar[rr]\ar[dr,"\ev_1"] && \Of^{\Lambda^2_0,\kappa}\times_{\dstack}\mathsf{Mfd} \ar[dl,"\ev_{\infty}"] \\
& \mathsf{Mfd}
\end{tikzcd}
\]
Passing to subcategories spanned by Cartesian edges, we have a functor of right fibrations with essentially small fibres. Both these functors determine via unstraightening sheaves for the \'{e}tale topology on $\mathsf{Mfd}$. It follows that we have an induced functor  
\[
\begin{tikzcd}
\Of^{(\mathsf{Sub})} \ar[dr,"\ev_1"]\ar[rr] && \Of^{(\Lambda^ 2_0,\kappa)}\ar[dl,"\ev_{\infty}"] \\
&\mathsf{SmSt}
\end{tikzcd}
\]
of right fibrations over $\mathsf{SmSt}$ and thus a functor $\Of^{(\mathsf{Sub})}\rightarrow \Of^{(\Lambda^2_0,\kappa)}$ over the fully faithful inclusion $\mathsf{SmSt}\subset\dstack$. We will also refer to this functor as the \emph{$k$'th order infinitesimal diagonal neighbourhood}. It carries a submersive map of smooth stacks $X\rightarrow S$ to a square 
\[
\begin{tikzcd}
 (X/S)^{(k)}\ar[d,"\pi_2"]\ar[r,"\pi_1"] & X\ar[d] \\
 X \ar[r] & S
\end{tikzcd}
\]
in such a manner that for each manifold $M$ and each map $M\rightarrow S$ determining a submersion $M\times_SX\rightarrow M$, there is a map $(M\times_SX/ M)^{(k)}\rightarrow (X/S)^{(k)}$ fitting into a pullback diagram 
\[
\begin{tikzcd}
 (M\times_SX/M)^{(k)}\ar[d]\ar[r] & (X/S)^{(k)}\ar[d] \\
 M \ar[r] & S
\end{tikzcd}
\]
of derived $\cinfty$-stacks.
\begin{defn}
Let $S$ be a smooth stack, let $q:X\rightarrow S$ be an $S$-family of manifolds and let $f:Y\rightarrow X$ be an arbitrary map of derived $\cinfty$-stack. The \emph{relative $k$'th order jet space of sections of $f$}, denoted $J^k_{X/S}(Y)$, is the Weil restriction $\mathsf{Res}_{(X/S)^{(k)}/X}(Y\times_{X}(X/S)^{(k)})\in\dstack$, where the pullback $Y\times_{X}(X/S)^{(k)}$ is taken along $\pi_1$ and the Weil restriction along $\pi_2$. For $X\rightarrow S$ an $S$-family of manifolds, we let 
\[ J^k_{(X/S)}(\_):\dstack_{/X} \longrightarrow \dstack_{/X}  \]
denote the relative $k$'th order jet space functor. For $S=*$ a final object (so that $X$ is a manifold) and $Y\rightarrow X$ a map of derived $\cinfty$-stacks, we write $J^k_X(Y)$ for the derived $\cinfty$-stack $J^k_{X/*}(Y)$.
\end{defn}
\begin{rmk}
Unwinding the definition of Weil restriction, the relative $k$'th  order jet space of sections of $f:Y\rightarrow X$ over $S$ is defined by the following universal property: for every map $\spec\,A\rightarrow X$ from an affine derived $\cinfty$-scheme, there is a canonical equivalence between the space $\Hom_{\dstack_{/X}}(\spec\,A,J^k_{X/S}(Y))$ and the space of commuting diagrams
\[
\begin{tikzcd}
  \spec\,A\times_X (X/S)^{(k)}\ar[r]\ar[d] & Y\ar[d,"f"] \\
  (X/S)^{(k)}\ar[r,"\pi_1"] & X.
\end{tikzcd}
\]
where the pullback in the upper left corner is taken along $\pi_2$.
\end{rmk}
We have a tower of relative jet spaces
\[ \ldots\longrightarrow  J^{k+1}_{X/S}(Y)\longrightarrow J^{k}_{X/S}(Y)\longrightarrow \ldots J^1_{X/S}(Y) \longrightarrow Y   \]
over $X$, where the map $J^{k+1}_{X/S}(Y)\rightarrow J^{k-1}_{X/S}(Y)$ is induced by the composition $J^{k+1}_{X/S}(Y)\times_{X}(X/S)^{(k)}\rightarrow J^{k+1}_{X/S}(Y)\times_{X}(X/S)^{(k+1)}\rightarrow Y$. For $k>l$ will denote the map $J^k_{X/S}(Y)\rightarrow J^l_{X/S}(Y)$ by $p_{k,l}$, the map $J^k_{X/S}(Y)\rightarrow Y$ by $p_{k,0}$ and the map $J^k_{X/S}(Y)\rightarrow X$ by $p$. We collect some facts about the relative jet space construction.
\begin{prop}\label{prop:jetbundleproperties}
Let $X\rightarrow S$ be an $S$-family of manifolds over a smooth stack and $Y\rightarrow X$ a map of derived $\cinfty$-stacks. Then the following hold true.
\begin{enumerate}[$(1)$]
\item The relative jet space functor $J_{X/S}(\_):\dstack_{/X}\rightarrow\dstack_{/X}$ preserves limits. 
\item Let $S'\rightarrow S$ be a map of smooth stacks, then the canonical morphism
\[ J^k_{X\times_SS'/S}(Y\times_SS')\longrightarrow J^k_{X/S}(Y)\times_SS'  \]
of derived $\cinfty$-stacks over $X\times_SS'$ is an equivalence.
\item Let $U\rightarrow Y$ be an open substack inclusion. Then the diagram
\[
\begin{tikzcd}
J^k_{X/S}(U)\ar[d]\ar[r] &J^k_{X/S}(Y)\ar[d] \\
U\ar[r] & Y
\end{tikzcd}
\] 
is a pullback diagram. 
\item Let $S$ be (representable by) a manifold and $Y\rightarrow X$ a vector bundle of rank $r$, then the derived $\cinfty$-stack $J^k_{X/S}(Y)$ is (representable by) a manifold and the map $J^k_{X/S}(Y)\rightarrow X$ is a vector bundle over $X$ of rank . Moreover, for $k\geq 1$, there is canonical map of vector bundles $\sym^k(T^{\vee}{X/S})\rightarrow J^k_{X/S}(Y)$ where $T^{\vee}X/S$ is the vertical cotangent bundle of $X$ with respect to the submersion $X\rightarrow S$ which fits into a short exact sequence 
\[   0\longrightarrow \sym^k(T^{\vee}X/S)\otimes Y \longrightarrow J^k_{X/S}(Y) \longrightarrow J^{k-1}_{X/S}(Y)\longrightarrow 0  \]
of vector bundles over $X$.
\item Let $Y\rightarrow X$ be an $S$-family of submersions, then the jet space $J^k_{X/S}(Y)$ is smooth and the map $J^k_{X/S}(Y)\rightarrow J^{k-1}_{X/S}(Y)$ is submersive.
\end{enumerate}
\end{prop}
We will need the following lemma.
\begin{lem}\label{lem:infnbhdtensor}
Let $X\rightarrow S$ be a submersion of manifolds, then for any $k\geq 0$, the natural transformation
\[ \cinfty(X\times_SX)/I^{k+1}_{\Delta}\otimes_{\cinfty(X)}(\_)^{\rmalg} \overset{\alpha}{\longrightarrow} (\cinfty(X\times_SX)/I^{k+1}_{\Delta}\oinfty_{\cinfty(X)}(\_))^{\rmalg}  \]
of functors $\sring_{\cinfty(X)/}\rightarrow \mathsf{d}\calg_{\R}$ is an equivalence.
\end{lem}
\begin{proof}
We proceed by induction on $k$, the case $k=0$ being trivial. Since both functor preserve sifted colimits, it suffices to show that $\alpha$ is an equivalence on the full subcategory spanned by finitely generated free derived $\cinfty$-rings over $\cinfty(X)$ (see \cite[Proposition 2.4.7]{cinftyI}). We will first show that for any $n\geq 0$, the derived $\cinfty$-ring $\cinfty(X\times_SX)/I^{k+1}_{\Delta}\oinfty_{\cinfty(X)}\cinfty(X\times\R^n)\simeq \cinfty(X\times_SX)/I^{k+1}_{\Delta}\oinfty\cinfty(\R^n)$ is $0$-truncated. The coproduct fits into a pushout diagram 
\[
\begin{tikzcd}
\cinfty(X\times_SX)\ar[d]\ar[r] & \cinfty(X\times_SX)/I^{k+1}_{\Delta} \ar[d] \\
\cinfty(X\times_SX\times\R^n)\ar[r] &  \cinfty(X\times_SX)/I^{k+1}_{\Delta}\oinfty\cinfty(\R^n)
\end{tikzcd}
\]
of derived $\cinfty$-rings. Since the upper horizontal map is an effective epimorphism, the underlying diagram of derived commutative $\R$-algebras is also a pushout (\cite[Corollary 3.1.8]{cinftyI}), so we deduce from Lemma \ref{lem:ntruncinfnbhd} that $ \cinfty(X\times_SX)/I^{k+1}_{\Delta}\oinfty\cinfty(\R^n)$ is 0-truncated. It follows that we can identify the coproduct $\cinfty(X\times_SX)/I^{k+1}_{\Delta}\oinfty\cinfty(\R^n)$ with the $\cinfty$-ring $\cinfty(X\times_SX\times\R^n)/J^{k+1}$ where $J$ is the kernel of the map $\cinfty(X\times_SX\times\R^n)\rightarrow \cinfty(X\times\R^n)$. We have a commuting diagram 
\[
\begin{tikzcd}
\cinfty(X\times_SX)/I^{k+1}_{\Delta}\otimes_{\R}\cinfty(\R^n)\ar[d]\ar[r] & \cinfty(X\times_SX\times\R^n)/J^{k+1}\ar[d] \\
\cinfty(X\times_SX)/I^{k}_{\Delta}\otimes_{\R}\cinfty(\R^n)\ar[r] & \cinfty(X\times_SX\times\R^n)/J^{k}
\end{tikzcd}
\]
of commutative $\R$-algebras; by what we have just shown and by the inductive hypothesis, it suffices to show that the upper horizontal map induces an equivalence on the kernels of the vertical maps. The kernel of the left vertical map is identified with $I^k_{\Delta}/I^{k+1}_{\Delta}\otimes\cinfty(\R^n)$ and the same argument as the one employed in Lemma \ref{lem:jebundleseq} shows that the map 
\[J/J^2\longrightarrow \Gamma(T^{\vee}(X\times_SX\times\R^n)\times_{X\times_SX\times\R^n}X\times\R^n)\longrightarrow \Gamma(T^{\vee}X/S)\otimes_{\R} \cinfty(\R^n)\]
 carrying a function $g$ on $X\times_SX\times\R^n$ to the restriction of $d_{\dR}g$ to $X\times\R^n$ is an isomorphism and induces an isomorphism $J^k/J^{k+1}\cong \sym_{\cinfty(X)}^k(\Gamma(T^{\vee}X/S))\otimes_{\R}\cinfty(\R)$ for all $k\geq 1$. It is straightforward to verify that the map 
\[  I^k_{\Delta}/I^{k+1}_{\Delta}\otimes_{\R}\cinfty(\R^n)\longrightarrow J^k/J^{k+1}\cong \sym_{\cinfty(X)}^k(\Gamma(T^{\vee}X/S))\otimes_{\R}\cinfty(\R) \]
coincides with the isomorphism $I^k_{\Delta}/I^{k+1}_{\Delta}\otimes_{\R}\cinfty(\R^n)\cong  \sym_{\cinfty(X)}^k(\Gamma(T^{\vee}X/S))\otimes_{\R}\cinfty(\R)$ induced by the isomorphism $I^k_{\Delta}/I^{k+1}_{\Delta}\cong \sym_{\cinfty(X)}^k(\Gamma(T^{\vee}X/S))$.
\end{proof}
Let $X\rightarrow S$ be an $S$-family of manifolds over a smooth stack and let $Y\rightarrow X$ be an arbitrary map of derived $\cinfty$-stacks. By definition, to lift a map $Z\rightarrow X$ to a map $Z\rightarrow J^k_{X/S}(Y)$, one must provide a map $Z\times_X(X/S)^{(k)}\rightarrow Y$ over $X$, characterizing $J^k_{X/S}(Y)$ by a universal property. We will now establish a similar universal property for the jet space $J^k_{X/S}(Y)$ as an object \emph{over $J^{k-1}_{X/S}(Y)$}. For this, we will use the following categorical result.
\begin{prop}\label{prop:mate}
Let $f_0,f_1:\icat\rightarrow \icatd$ be functors admitting right adjoints $g_0$ and $g_1$ and suppose we have a natural transformation $\alpha:f_0\rightarrow f_1$. The transformation $\alpha$ has a \emph{mate} $\beta:g_1\rightarrow g_0$ defined as 
\[  g_1\longrightarrow g_0f_0g_1\overset{g_0\alpha g_1}{\longrightarrow} g_0f_1g_1\longrightarrow g_0   \]
where the first and last map are induced by the unit of $f_0\adj g_0$ and the counit of $f_1\adj g_1$ respectively. Suppose that for each map $D\rightarrow D'$ of $\icatd$, there is a pullback diagram
\[ 
\begin{tikzcd}
G(D\rightarrow D') \ar[d]\ar[r] & g_1(D')\ar[d,"\beta(D')"] \\
g_0(D)\ar[r] & g_0(D')
\end{tikzcd}
\]
in $\icat$. Let $F:\icat\rightarrow\fun(\Delta^1,\icatd)$ be the functor determined by the transformation $\alpha$, then $F$ admits a right adjoint $G$ which carries a map $D\rightarrow D'$ of $\icatd$ to $G(D\rightarrow D')$. Moreover, there is a commuting diagram
\[
\begin{tikzcd}
f_0(G(D\rightarrow D')) \ar[r,"\alpha"] \ar[d]& f_1(G(D\rightarrow D'))\ar[d] \\
f_0g_0(D) \ar[d] & f_1g_1(D')\ar[d] \\
D\ar[r] & D'  
\end{tikzcd}
\]
in $\icatd$ where the lower vertical maps are counits of $f_0\adj g_0$ and $f_1\adj g_1$ respectively, which exhibits a counit transformation for the object $D\rightarrow D'$.
\end{prop} 
\begin{proof}
To show the existence of a functor right adjoint to $F$, we are required to show that for every map $D\rightarrow D'$ in $\icatd$, the right fibration $\icat\times_{\fun(\Delta^1,\icatd)}\fun(\Delta^1,\icatd)_{/D\rightarrow D'}$ is representable. The domain projection $\ev_0:\fun(\Delta^1,\icatd)\rightarrow\icatd$ is a Cartesian fibration, so the commuting square 
\[
\begin{tikzcd}
D\ar[d] \ar[r] & D'\ar[d] \\
D'\ar[r,equal] & D'
\end{tikzcd}
\]
induces a homotopy pullback square
\[
\begin{tikzcd}
\fun(\Delta^1,\icatd)_{/D\rightarrow D'}\ar[d] \ar[r] & \fun(\Delta^1,\icatd)_{/D'\rightarrow D'}\ar[d] \\ \icatd_{/D} \ar[r] & \icatd_{/D'}
\end{tikzcd}
\]
of \infcatst. Pulling back along the map $F$, we deduce a homotopy pullback square
\[
\begin{tikzcd}
\icat\times_{\fun(\Delta^1,\icatd)}\fun(\Delta^1,\icatd)_{/D\rightarrow D'}\ar[d] \ar[r] & \icat\times_{\fun(\Delta^1,\icatd)}\fun(\Delta^1,\icatd)_{/D'\rightarrow D'}\ar[d] \\ \icat\times_{\icatd}\icatd_{/D} \ar[r] & \icat\times_{\icatd}\icatd_{/D'}
\end{tikzcd}
\]
of right fibrations over $\icat$, where the lower pullbacks are taken along the composition $\icat\overset{F}{\longrightarrow}\fun(\Delta^1,\icatd)\overset{\ev_0}{\longrightarrow}\icatd$ which coincides with $f_0$. Since $f_0$ admits a right adjoint $g_0$, we have equivalences $\icat_{/g_0(D)}\simeq \icat\times_{\icatd}\icatd_{/D}$ and $\icat_{/g_0(D')}\simeq \icat\times_{\icatd}\icatd_{/D'}$ and we may identify the resulting functor $\icat_{/g_0(D)}\rightarrow \icat_{/g_0(D')}$ as the one induced by the map $g_0(D)\rightarrow g_0(D')$. The upper right \infcat fits into a diagram 
\[
\begin{tikzcd}
\icat\times_{\fun(\Delta^1,\icatd)}\fun(\Delta^1,\icatd)_{/D'\rightarrow D'} \ar[d] \ar[r] & \fun(\Delta^1,\icatd)_{/D'\rightarrow D'}\ar[d] \ar[r] & \icatd_{/D'} \ar[d]\\ \icat \ar[r,"F"] & \fun(\Delta^1,\icatd)\ar[r,"\ev_1"] & \icatd
\end{tikzcd}
\]
Both squares are homotopy pullbacks (as the diagonal functor $\icatd\rightarrow\fun(\Delta^1,\icatd)$ is right adjoint to $\ev_1$) so we have an equivalence $\icat\times_{\fun(\Delta^1,\icatd)}\fun(\Delta^1,\icatd)_{/D'\rightarrow D'}\simeq \icat_{/g_1(D')}$ of right fibrations over $\icat$. Unwinding the definitions, the induced map of right fibrations $\icat_{/g_1(D')}\rightarrow \icat_{/g_0(D')}$ corresponds to the mate $\beta(D')$. As the Yoneda embedding $j:\icat\hookrightarrow\mathsf{RFib}_{\icat}$ preserves all limits that exist in $\icat$, we deduce that, under our assumption on $\icat$, the \infcat $\icat\times_{\fun(\Delta^1,\icatd)}\fun(\Delta^1,\icatd)_{/D\rightarrow D'}$ admits a final object which is carried to $G(D\rightarrow D')$ by the projection to $\icat$. By construction, we have commuting diagrams
\[
\begin{tikzcd}
\icat_{/G(D\rightarrow D')} \ar[d]\ar[r] & \fun(\Delta^1,\icatd)_{/D\rightarrow D'}\ar[d,"\ev_0"]\ar[d] \\
\icat_{/g_0(D)} \ar[r,"f_0"] &\icatd_{/D},
\end{tikzcd} \quad\quad\quad
\begin{tikzcd}
\icat_{/G(D\rightarrow D')} \ar[d]\ar[r] & \fun(\Delta^1,\icatd)_{/D\rightarrow D'}\ar[d,"\ev_1"]\ar[d] \\
\icat_{/g_1(D')} \ar[r,"f_1"] &\icatd_{/D'}
\end{tikzcd}
\]
of \infcats which imply the assertion regarding the counit of $F\adj G$.
\end{proof}
\begin{rmk}
Let $k\geq 0$ and apply Proposition \ref{prop:mate} to the natural transformation $\alpha_{k}:\_\times_X(X/S)^{(k)}\rightarrow \_\times_X(X/S)^{(k+1)}$ of functors $\dstack_{/X}\rightarrow \dstack_{/X}$, then we deduce that the functor $F:\dstack_{/X}\rightarrow \fun(\Delta^1,\dstack_{/X})$ determined by $\alpha_k$ admits a right adjoint $G$ which carries the map $Y\rightarrow X$ to $J^k_{X/S}(Y)\times_{J^k_{X/S}(X)}J^{k+1}_{X/S}(X)\simeq J^k_{X/S}(Y)$ (as $X$ is final in $\dstack_{/X}$ and the jet functors are right adjoints) and the map $Y\rightarrow Y$ to $J^k_{X/S}(Y)\times_{J^k_{X/S}(Y)}J^{k+1}_{X/S}(Y)\simeq J^{k+1}_{X/S}(Y)$. It follows that the commuting diagram
\[
\begin{tikzcd}
Y\ar[d] \ar[r,equal] & Y\ar[d] \\
Y\ar[r] & X
\end{tikzcd}
\]
determines a map $\beta_k(Y):J^{k+1}_{X/S}(Y)\rightarrow J^k_{X/S}(Y)$, the mate of the transformation $\alpha_k$. Applying the counit of the adjunction $F\adj G$ yields a commuting diagram 
\[
\begin{tikzcd}
J^{k+1}_{X/S}(Y)\times_X(X/S)^{(k)} \ar[d]\ar[r] & J^{k}_{X/S}(Y)\times_X(X/S)^{(k)}\ar[d] \\
J^{k+1}_{X/S}(Y)\times_X(X/S)^{(k+1)} \ar[r] & Y.
\end{tikzcd}
\]
Since the map $p_{k+1,k}:J^{k+1}_{X/S}(Y)\rightarrow J^k_{X/S}(Y)$ defined above is characterized up to a contractible space of choices by the existence of this diagram, the maps $p_{k+1,k}$ and $\beta_k$ are equivalent. From the induced adjunction 
\[ \dstack_{/J^k_{X/S}(Y)} \longrightarrow \fun(\Delta^1,\dstack_{/X})_{/X\rightarrow Y}  \]
we deduce an equivalence of spaces $\Hom_{\dstack_{/J^k_{X/S}}}(Z,J_{X/S}^{k+1}(Y))\simeq\Hom_{\dstack_{Z\times_X(X/S)^{(k)}//X}}(Z\times_X(X/S)^{(k+1)},Y)$. 
\end{rmk}
The last preliminary we require before we can prove Proposition \ref{prop:jetbundleproperties} is an observation concerning vector bundles and modules of derived $\cinfty$-rings. Let $A$ be a derived $\cinfty$-ring and consider the functor
\[ (\_)^{\rmalg}:\sring_{A/}\longrightarrow (\scring_{\R})_{A^{\rmalg}/}  \]
induced by $(\_)^{\rmalg}$. Since $(\_)^{\rmalg}:\sring\rightarrow \scring_{\R}$ has a left adjoint $F^{\cinfty}$ taking free derived $\cinfty$-rings, the forgetful functor on derived $\cinfty$-rings under $A$ also admits a left adjoint. Denote this left adjoint by $F_A^{\cinfty}$. According to (the proof of) \cite[Proposition 5.2.5.1]{HTT}, for $B$ a derived commutative $\R$-algebra under $A^{\rmalg}$, the unit map $B\rightarrow F^{\cinfty}_A(B)^{\rmalg}$ is the right vertical composition in the diagram 
\[
\begin{tikzcd}
A^{\rmalg} \ar[d] \ar[r] & B\ar[d] \\
F^{\cinfty}(A^{\rmalg})^{\rmalg} \ar[r] \ar[d,"\epsilon^{\rmalg}"]& F^{\cinfty}(B)^{\rmalg} \ar[d] \\
A^{\rmalg}\ar[r] &  F^{\cinfty}_A(B)^{\rmalg}
\end{tikzcd}
\]
where the upper vertical maps are unit transformations for the adjunction $(F^{\infty}\adj (\_)^{\rmalg})$ and the lower square is the image under $(\_)^{\rmalg}$ of a pushout diagram in $\sring$ where $\epsilon$ is the counit map at $A$. We have a composition of adjunctions
\[ \begin{tikzcd} \Mod^{\geq0}_A \ar[r,shift left,"\sym^{\bullet}_A"] &[2em] (\mathsf{dCAlg}_{\R})_{A^{\rmalg}/} \ar[l,shift left]\ar[r,shift left,"F^{\cinfty}_A"] &[2em] \sring_{A/}.\ar[l,shift left,"(\_)^{\rmalg}"] \end{tikzcd}  \]
Let $M$ be a manifold, then the category of vector bundles and vector bundle maps among them is equivalent to the full subcategory of $\Mod_{\cinfty(M)}$ spanned by finitely presented and projective $\cinfty(M)$-modules. This equivalence is implemented by the composition of left adjoints above: if $E\rightarrow M$ is a vector bundle, then the affine derived $\cinfty$-scheme of finite presentation $(E,\cinfty_E)$ is obtained by applying $\spec\,$ to the (derived) $\cinfty$-ring $F^{\cinfty}_{\cinfty(M)}\circ \sym^{\bullet}_{\cinfty(M)}(\Gamma(E;N)^{\vee})$ (note we take the $\cinfty(M)$-linear \emph{dual} of the module of global sections). We will employ this adjunction in the proof below.
\begin{proof}[Proof of Proposition \ref{prop:jetbundleproperties}]
Since the functor $J_{X/S}(\_)$ is a composition of the functors $\_\times_X(X/S)^{(k)}$ and $\mathsf{Res}_{(X/S)^{(k)}/X}(\_)$ which are both right adjoints, $(1)$ is immediate. Note that $(2)$ is an immediate consequence of the construction of the $k$'th order infinitesimal neighbourhood and Remark \ref{rmk:weilrestrictionpullback}. By descent, it suffices to show that $(3)$ holds for $X\rightarrow S$ a submersion of manifolds and $Y$ a representable stack. We will show that for any map $\spec\,A\rightarrow X$, the diagram 
\[
\begin{tikzcd}
\Hom_{\dstack_{/X}}(\spec\,A,J^k_{X/S}(U))\ar[d]\ar[r] & \Hom_{\dstack_{/X}}(\spec\,A,J^k_{X/S}(Y))\ar[d] \\
\Hom_{\dstack_{/X}}(\spec\,A,U) \ar[r] & \Hom_{\dstack_{/X}}(\spec\,A,Y) 
\end{tikzcd}
\]
of spaces is a pullback. Apply Proposition \ref{prop:mate} to the natural transformation $\mathrm{id}\rightarrow \_\times_X(X/S)^{(k)}$ induced by the map $X\rightarrow (X/S)^{(k)}$, then suffices to show that the diagram 
\[
\begin{tikzcd}
\Hom_{\dstack_{/X}}(\spec\,A\times_X(X/S)^{(k)},U)\ar[d]\ar[r] & \Hom_{\dstack_{/X}}(\spec\,A\times_X(X/S)^{(k)},Y)\ar[d] \\
\Hom_{\dstack_{/X}}(\spec\,A,U) \ar[r] & \Hom_{\dstack_{/X}}(\spec\,A,Y) 
\end{tikzcd}
\]
is a pullback, where the horizontal maps are induced by composition with $U\rightarrow Y$ and the vertical maps are induced by composition with $X\rightarrow (X/S)^{(k)}$. Both horizontal maps are inclusions onto the connected components spanned by the maps $\spec\,A\times_X(X/S)^{(k)}\rightarrow Y$ and $\spec\,A\rightarrow Y$ whose underlying map of topological spaces factors through $q_{\sring}(U)$ (as $U\rightarrow Y$ is represented by a $q_{\sring}$-Cartesian morphism in $\dsch$). The result thus follows from the observation that the map $X\rightarrow (X/S)^{(k)}$ induces a homeomorphism on the underlying topological spaces and that $q_{\sring}$ preserves limits. To prove the remaining assertions, we consider the composition of functors
\[  \sring_{\cinfty(X)/} \longrightarrow \sring_{\cinfty(X\times_SX)/I^{k+1}_{\Delta}/} \longrightarrow \sring_{\cinfty(X)/},\quad\quad A\longmapsto A\oinfty_{\cinfty(X)}\cinfty(X\times_SX)/I^{k+1}_{\Delta},   \]
where the first functor base changes along $\pi_2^*$ and the second functor restricts along $\pi_1^*$. Suppose for a geometric derived $\cinfty$-ring $B$ over $\cinfty(X)$, we are given a geometric derived $\cinfty$-ring $A$ and a unit transformation $ B\rightarrow A\oinfty_{\cinfty(X)}\cinfty(X\times_SX)/I^{k+1}_{\Delta}$ over $\cinfty(X)$, then it is formal that the corresponding map 
\[  \spec\,A\times_X(X/S)^{(k)} \longrightarrow \spec\,B  \]
of affine derived $\cinfty$-schemes $X$ exhibits the object $\spec\,A$ as the relative $k$'th order jet space of the map $\spec\,B\rightarrow X$. It follows from Lemma \ref{lem:infnbhdtensor} that we have a commuting diagram 
\[
\begin{tikzcd}
\sring_{\cinfty(X)/} \ar[d,"(\_)^{\rmalg}"]\ar[r,"\_\oinfty_{\cinfty(X)}\cinfty(X\times_SX)/{I^{k+1}_{\Delta}}"]&[8em] \sring_{\cinfty(X)/} \ar[d,"(\_)^{\rmalg}"]\\
(\scring_{\R})_{\cinfty(X)/} \ar[r,"\_\otimes_{\cinfty(X)}\cinfty(X\times_SX)/{I^{k+1}_{\Delta}}"] \ar[d]&[8em] (\scring_{\R})_{\cinfty(X)/}\ar[d] \\
\Mod^{\geq 0}_{\cinfty(X)} \ar[r,"\_\otimes_{\cinfty(X)}\cinfty(X\times_SX)/{I^{k+1}_{\Delta}}"] &[8em] \Mod^{\geq 0}_{\cinfty(X)} 
\end{tikzcd}
\]
of \infcatst. We will first show that the lower horizontal map admits a left adjoint $L$. Since the \infcat $\sring_{\cinfty(X)/}$ is generated under colimits by the essential image of the functor $F^{\cinfty}_{\cinfty(X)}\circ \sym^{\bullet}_{\cinfty(X)}$ left adjoint to the right vertical composition, this will demonstrate that the upper horizontal functor admits a left adjoint characterized by the property that for any connective $\cinfty(X)$-module $M$, the derived $\cinfty$-ring $F^{\cinfty}_{\cinfty(X)}\sym^{\bullet}_{\cinfty(X)}(M)$ is carried to the derived $\cinfty$-ring $F^{\cinfty}_{\cinfty(X)} \sym^{\bullet}_{\cinfty(X)}L(M)$. The lower horizontal functor in the diagram above may be described as follows: the two maps $\{\pi_i^*:\cinfty(X)\rightarrow \cinfty(X\times_SX)/{I^{k+1}_{\Delta}}\}_{i=1,2}$ furnish on $\cinfty(X\times_SX)/{I^{k+1}_{\Delta}}$ the structure of a $\cinfty(X)$-$\cinfty(X)$-bimodule and the lower horizontal functor above is the functor
\[\begin{tikzcd} \Mod^{\geq0}_{\cinfty(X)} \ar[r,"\mathrm{id}\times \{\cinfty(X\times_SX)/{I^{k+1}_{\Delta}}\}"] &[8em]  \Mod^{\geq0}_{\cinfty(X)} \times {}_{\cinfty(X)}\mathsf{BMod}^{\geq0}_{\cinfty(X)} \ar[r,"\_\otimes_{\cinfty(X)}\_"] &[3em] \Mod^{\geq0}_{\cinfty(X)} \end{tikzcd}\] 
where the second functor is the relative tensor product. Since $\cinfty(X\times_SX)/I^{k+1}_{\Delta}$ is a finitely generated projective $\cinfty(X)$-module by Lemma \ref{lem:jebundleseq}, it is dualizable as a connective $\cinfty(X)$-module, that is, the functor $\_\otimes_{\cinfty(X)} \cinfty(X\times_SX)/{I^{k+1}_{\Delta}}$ admits a left adjoint given by the functor 
\[\begin{tikzcd} \Mod^{\geq 0}_{\cinfty(X)} \ar[r,"\mathrm{id}\times \{(\cinfty(X\times_SX)/{I^{k+1}_{\Delta}})^{\vee}\}"] &[8em]  \Mod^{\geq 0}_{\cinfty(X)} \times {}_{\cinfty(X)}\mathsf{BMod}^{\geq 0}_{\cinfty(X)} \ar[r,"\_\otimes_{\cinfty(X)}\_"] &[3em] \Mod^{\geq 0}_{\cinfty(X)} \end{tikzcd}\] 
where we view the dual $(\cinfty(X\times_SX)/{I^{k+1}_{\Delta}})^{\vee}$ as a $\cinfty(X)$-$\cinfty(X)$-bimodule where the actions have switched sides (\cite[Proposition 4.6.2.1]{HA}). We deduce that in case $Y\rightarrow X$ is a vector bundle (so that we have an equivalence $Y\simeq \spec\,F^{\cinfty}_{\cinfty(X)}\sym^{\bullet}_{\cinfty(X)}(\Gamma(Y)^{\vee})$ over $X$), then the jet space $J^k_{X/S}(Y)$ may be identified with the vector bundle over $X$ whose global sections are given the module $\Gamma(Y)\otimes_{\cinfty(X)}\cinfty(X\times_SX)/I^{k+1}_{\Delta}$. It follows from Proposition \ref{prop:mate} that the map $J^k_{X/S}(Y)\rightarrow J^{k-1}_{X/S}(Y)$ is induced by the mate of the transformation 
\[ \_\otimes_{\cinfty(X)}\cinfty(X\times_SX)/{I^{k+1}_{\Delta}} \longrightarrow \_\otimes_{\cinfty(X)}\cinfty(X\times_SX)/{I^{k}_{\Delta}},  \]
which is the transformation  
\[ \_\otimes_{\cinfty(X)}(\cinfty(X\times_SX)/{I^{k}_{\Delta}})^{\vee} \longrightarrow \_\otimes_{\cinfty(X)}(\cinfty(X\times_SX)/{I^{k+1}_{\Delta}})^{\vee}  \]
that takes the tensor product with the dual map $(\cinfty(X\times_SX)/{I^{k}_{\Delta}})^{\vee}\rightarrow (\cinfty(X\times_SX)/{I^{k+1}_{\Delta}})^{\vee} $. From Lemma \ref{lem:jebundleseq} we deduce a fibre sequence 
\[   \sym^k(T^{\vee}X/S)\otimes Y \longrightarrow J^k_{X/S}(Y) \longrightarrow J^{k-1}_{X/S}(Y)\]
of vector bundles. We proceed with the proof of $(5)$. By descent, we may suppose that $S$ is a manifold and $X\rightarrow S$ and $Y\rightarrow X$ submersions of manifolds. Observing for each open set $U\subset X$ an equivalence $(X/S)^{(k)}\times_XU\simeq (U/S)^{(k)}$ of affine derived $\cinfty$-schemes, it follows from Remark \ref{rmk:weilrestrictionpullback} that we may choose an open cover for $X$ and replace $X$ by an element of that cover. It follows from $(3)$ that we may also choose an open cover of $Y$ and replace $Y$ by an element of that cover. Thus we may suppose that $Y\rightarrow X$ is a projection $\R^n\times Z\rightarrow Z$, in which case the assertion follows from $(4)$.
\end{proof}
\begin{rmk}
Applying $(2)$ of Proposition \ref{prop:jetbundleproperties} to an arbitrary map $Y\rightarrow X$ of derived $\cinfty$-stacks with target a manifold and an arbitrary smooth stack $S$ shows that the canonical morphism
\[ J^k_{X\times S/S}(Y\times S) \longrightarrow J^k_X(Y)\times S  \]
of derived $\cinfty$-stacks over $X\times S$ is an equivalence.
\end{rmk}
\begin{rmk}
Suppose that $X\rightarrow S$ a submersion of ordinary manifolds and that $Y\rightarrow  X$ is a submersion. It can be shown that for $k\geq 0$ the manifold $J^{k+1}_{X/S}(Y)$ admits a canonical structure of a \emph{torsor} for the vector bundle $p^*_{k,0}TY/X\otimes p^*\sym^{k+1} (T^{\vee}X/S)$ on $J^k(X/S)(Y)$, that is, $J^{k+1}_{X/S}(Y)$ is an affine bundle modelled on the vector bundle $p^*_{k,0}TY/X\otimes p^*\sym^{k+1} (T^{\vee}X/S)$. This is a standard fact about jet spaces; we will indicate how it follows from the geometry of infinitesimal neighbourhoods (we do not wish to develop all technology necessary for a full proof here; see \cite{cinftyII}). As a square-zero extension, the map $\cinfty(X\times_SX)/{I_{\Delta}^{k+2}}\rightarrow \cinfty(X\times_SX)/{I_{\Delta}^{k+1}}$ is a torsor for the trivial square zero extension $\cinfty(X\times_SX)/{I_{\Delta}^{k+1}}\oplus \sym^{k+1}(\Gamma(T^{\vee}X/S))\rightarrow \cinfty(X\times_SX)/{I_{\Delta}^{k+1}}$ in the \infcat $\sring_{\cinfty(X)/}$. Let $(X/S)^{(k)}\oplus \sym^{k+1}(T^{\vee}X/S)$ denote the spectrum of $\cinfty(X\times_SX)/I_{\Delta}^{k+1}\oplus \sym^{k+1}(\Gamma(T^{\vee}X/S))$, then the map 
\[(X/S)^{(k)}\longrightarrow (X/S)^{(k+1)}\]
has a canonical structure of a \emph{cotorsor} for the map 
\[(X/S)^{(k)}\longrightarrow (X/S)^{(k)}\oplus \sym^{k+1}(T^{\vee}X/S)\]
in the \infcat $\dstack_{/X}$. It then follows that the transformation $\alpha_{k}:\_\times_X(X/S)^{(k)}\rightarrow \_\times_X(X/S)^{(k+1)}$ viewed as a functor \[\dstack_{/J^k_{X/S}(Y)}\longrightarrow\fun(\Delta^1,\dstack_{/X})_{/Y\rightarrow X}\] admits the structure of a cotorsor for the transformation $\_\times_X(X/S)^{(k)}\rightarrow \_\times_X (X/S)^{(k)}\oplus \sym^{k+1}(T^{\vee}X/S)$. For any two \infcats $\icat,\icatd$ passing to adjoints extends to yield a canonical equivalence $\fun^{\mathrm{L}}(\icat,\icatd)\simeq \fun^{\mathrm{R}}(\icatd,\icat)^{op}$, so we conclude that the jet space $J^{k+1}_{X/S}(Y)$ admits the structure of a torsor for the stack obtained by applying the right adjoint to the functor $\_\times_X (X/S)^{(k)}\oplus \sym^{k+1}(T^{\vee}X/S)$ on the identity map $Y\rightarrow Y$. Now we use the following fact.
\begin{enumerate}
\item[$(*)$] Suppose that $\cinfty(X\times_SX)/I^{k+1}\oplus M\rightarrow \cinfty(X\times_SX)/I^{k+1}$ is a trivial square-zero extension by an (ordinary) $\cinfty(X\times_SX)/I^{k+1}$-module $M$, then the left adjoint to the functor
\[  \sring_{/\cinfty(X)}\longrightarrow \fun(\Delta^1,\sring_{\cinfty(X)/}) \]
taking the $\cinfty$-tensor product with the map $\cinfty(X\times_SX)/I^{k+1}\oplus M\rightarrow \cinfty(X\times_SX)/I^{k+1}$ carries the identity map $\cinfty(Y)\rightarrow \cinfty(Y)$ to the free $\cinfty$-ring under $\cinfty(J^k_{X/S}(Y))$ on the dual of the $\cinfty(J^k_{X/S}(Y))$-module $\Gamma(p_{k,0}^*TY/X)\otimes p^*M$.
\end{enumerate} 
We do not wish to develop the theory of square-zero extensions of derived $\cinfty$-rings in this work so we do not prove this.
\end{rmk}
\begin{rmk}
Suppose that $X\rightarrow S$ a submersion of ordinary manifolds and that $Y\rightarrow  X$ is a submersion. It can be shown that the manifold $J^1_{X/S}(Y)$ fits into a pullback diagram 
\[
\begin{tikzcd}
J^1_{X/S}(Y)\ar[d]\ar[r] & \mathsf{Hom}(TY/X,\pi^*_{1}TX/S)\ar[d] \\
Y \ar[r,"\mathrm{id}"] & \mathsf{End}(\pi^*_{1}TX/S)
\end{tikzcd}
\]
of submersions over $Y$, where $\mathsf{Hom}(\_,\_)$ and $\mathsf{End}(\_)$ denote the bundle of homomorphisms between two vector bundles and the bundle of endomorphisms of a vector bundle respectively, and the lower horizontal map is the identity section. In other words, the first jet space may be identified with the bundle over $Y$ whose sections are $S$-parametrized connections on the map $Y\rightarrow X$. This is a standard fact about the first jet space and it can be deduced from a variant of the assertion $(*)$ of the previous remark and the fact that the map $\cinfty(X\times_SX)/I^2_{\Delta}\rightarrow\cinfty(X)$ is a trivial square-zero extension in the category of $\R$-algebras but not in the category of $\cinfty(X)$-algebras (see Remark \ref{rmk:splitjetspace}). Alternatively, since we have already shown that $J^1_{X/S}(Y)$ is a manifold by $(4)$ of Proposition \ref{prop:jetbundleproperties}, one only has to verify that the pullback above satisfies the universal property of the jet space in the ordinary category of $\cinfty$-rings, which is not difficult.
\end{rmk}
For $Y\rightarrow X$ a map of derived $\cinfty$-stacks, we let $j^k$ denote the top horizontal map in the commuting diagram
\[  \begin{tikzcd}
    \mathsf{Res}_{X/S}(Y) \ar[rr,"j^k"] \ar[dr]&& \mathsf{Res}_{X/S}(J^k_{X/S}(Y))\ar[dl]\\
    & S
\end{tikzcd}  \]
defined by the commuting diagram 
\[  \begin{tikzcd}
    \mathsf{Res}_{X/S}(Y)\times_S (X/S)^{(k)} \ar[rr] \ar[dr]&& Y\ar[dl]\\
    & X
\end{tikzcd}  \]
wherein the top horizontal map is the composition 
\[  \begin{tikzcd} \mathsf{Res}_{X/S}(Y)\times_S (X/S)^{(k)} \ar[r,"\mathrm{id}\times \pi_1"] &[1.5em]\mathsf{Res}_{X/S}(Y)\times_S X\ar[r,"\ev"]& Y.\end{tikzcd}\]
The map $j^k$ is the \emph{$k$'th jet prolongation map}. 
\subsection{The elliptic representability theorem}
In this subsection we state and prove the main result of this paper.
\begin{defn}
Let $S$ be a smooth stack and let $Y_1\rightarrow M$ and $Y_2\rightarrow M$ be $S$-families of submersions. An \emph{$S$-family of $k$'th order nonlinear partial differential equations (PDEs) from $Y_1$ to $Y_2$} is a map $J^k_{M/S}(Y_1)\rightarrow Y_2$ over $M$.  
\end{defn}
Recall that a nonlinear PDE $\mathcal{P}:J^k_M(Y_1)\rightarrow Y_2$ between submersions over a manifold $M$ is \emph{elliptic} if for each $\sigma\in\Gamma(Y_1)$, the linear differential operator
\[ T\mathcal{P}_{\sigma}:J_M^k(\sigma^*TY_1/M)\cong j^k\sigma^*TJ^k_M(Y_1)/M\longrightarrow \mathcal{P}(j^k\sigma)^*TY_2/M  \]
induced by the differential
\[ T\mathcal{P}|_{TY_1/M}:TJ^k_{M}(Y_1)/M\longrightarrow TY_2/M \]
is elliptic. 
\begin{defn}
Let $S$ be a smooth stack and let $M\rightarrow S$ be an $S$-family of manifolds. A \emph{differential moduli problem over $M$} is a quintuple $(X\rightarrow M,Y_1\rightarrow M,Y_2\rightarrow M,\mathcal{P}_1,\mathcal{P}_2)$ where
\begin{enumerate}[$(1)$]
    \item The maps $X\rightarrow M$, $Y_1\rightarrow M$ and $Y_2\rightarrow M$ are $S$-families of submersions.
    \item $\mathcal{P}_1$ and $\mathcal{P}_2$ are $S$-families of $k$'th order nonlinear partial differential equations $J^k_{M/S}(Y_1)\rightarrow X$ and $J^k_{M/S}(Y_2)\rightarrow X$.
\end{enumerate}
We let $P_1$ and $P_2$ denote the compositions \[P_1:\mathsf{Res}_{M/S}(Y_1)\overset{j^k}{\longrightarrow}\mathsf{Res}_{M/S}(J^k_{M/S}(Y_1))\overset{\mathcal{P}_1}{\longrightarrow }\mathsf{Res}_{M/S}(X)\]
and 
\[P_2:\mathsf{Res}_{M/S}(Y_2)\overset{j^k}{\longrightarrow}\mathsf{Res}_{M/S}(J^k_{M/S}(Y_2))\overset{\mathcal{P}_2}{\longrightarrow }\mathsf{Res}_{M/S}(X).\]
The differential moduli problem is \emph{elliptic} if the data above satisfies the following condition.
\begin{enumerate}
    \item[$(Ell)$] For each point $s:*\rightarrow S$ and each pair of sections $(\sigma_1,\sigma_2):*\rightarrow \map_{M_s}(M_s,Y_{1s})\times \map_{M_s}(M_s,Y_{2s})$ for which $P_1(\sigma_1)$ and $P_2(\sigma_2)$ determine the same section $\sigma_X:M_s\rightarrow X_s$, the linearized PDE
    \[   J^k_M(\sigma_1^*TY_{1s}/M_s)\times_{M_s}J^k_M(\sigma_2^*TY_{2s}/M_s)\longrightarrow \sigma_X^*TX_s/M_s\times_{M_s} \sigma_X^*TX_s/M_s\longrightarrow \sigma_X^*TX_s/M_s   \]
    at $s$ and the pair $(\sigma_1,\sigma_2)$, where the first map is induced by $\mathcal{P}_1$ and $\mathcal{P}_2$ and the second is the map of vector bundles $(v,w)\mapsto v-w$, is a linear elliptic PDE.
\end{enumerate}
\end{defn}
Let $M\rightarrow S$ be an $S$-family of manifolds over a smooth stack $S$ and let $\icate=(X,Y_1,Y_2,\mathcal{P}_1,\mathcal{P}_2)$ be an elliptic moduli problem over $M$, then we will write $\mathsf{Sol}(\icate)$ for the cone in the pullback diagram 
\[
\begin{tikzcd}
\mathsf{Sol}(\icate)\ar[d]\ar[r] & \mathsf{Res}_{M/S}(Y_1)\ar[d,"P_1"] \\
\mathsf{Res}_{M/S}(Y_2) \ar[r,"P_2"] & \mathsf{Res}_{M/S}(X) 
\end{tikzcd}
\]
among derived $\cinfty$-stacks over $S$.
\begin{thm}[Relative elliptic representability]\label{thm:relellipticrepresentability}
Let $S$ be a smooth stack and $M\rightarrow S$ a proper $S$-family of manifolds. Let $\icate=(X,Y_1,Y_2,\mathcal{P}_1,\mathcal{P}_2)$ be an $S$-family of elliptic moduli problems over $M$. Then the solution stack $\mathsf{Sol}(\icate)\rightarrow S$ is representable by quasi-smooth derived $\cinfty$-schemes locally of finite presentation.
\end{thm}
We will use the following straightforward lemma.
\begin{lem}\label{lem:pullbackopeninclusion}
Let $\icat$ be an \infcat that admits finite limits. Let $P$ be a property of morphisms of $\icat$ that satisfies the following conditions.
\begin{enumerate}[$(a)$]
\item The property $P$ is stable under pullbacks.
\item The property $P$ is stable under products of morphisms: if $f$ and $g$ have the property $P$, then $f\times g$ has the property $P$.
\item In a commuting triangle
\[
\begin{tikzcd}
A\ar[rr,"f"]\ar[dr,"h"]&& B\ar[dl,"g"] \\
& C,
\end{tikzcd}
\]
should $g$ have the property $P$, then $f$ has the property $P$ if and only if $h$ has the property $P$.
\end{enumerate}
Suppose we are given a transformation $\alpha:\Lambda^2_2\times\Delta^1\rightarrow\icat$ from a diagram 
\[
\begin{tikzcd}
& X\ar[d] \\
Y\ar[r] & Z
\end{tikzcd}
\]
to a diagram 
\[
\begin{tikzcd}
& X'\ar[d] \\
Y'\ar[r] & Z'.
\end{tikzcd}
\]
If each of the maps $X\rightarrow X'$, $Y\rightarrow Y'$ and $Z\rightarrow Z'$ have the property $P$, then the induced map $X\times_ZY\rightarrow X'\times_{Z'}Y'$ has the property $P$.
\end{lem}
\begin{proof}
Consider the commuting diagram 
\[
\begin{tikzcd}
&&& X'\times_{Z'}Y' \arrow[rr] \arrow[dd] & & X'\times Y'\arrow[dd] \\
X\times_YZ \ar[dd ]\ar[rr,"f"] \ar[urrr,"h",bend left] && W  \arrow[ur,"g"]\arrow[rr, crossing over]  & & X\times Y\arrow[ur]  \\
 &&& Z'  \arrow[rr] & & Z'\times Z'\\
Z \ar[rr,"\alpha"]\ar[urrr,"\gamma",bend left]&& V \arrow[rr]\ \arrow[ur,"\beta"]\arrow[from=uu,crossing over]& & Z\times Z\arrow[ur] \arrow[from=uu, crossing over]\\
\end{tikzcd}
\]
where the bottom face and the front face of the cube are defined to be pullbacks. Since the back face is a pullback, the top face is a pullback as well. By $(a)$, $(b)$ and the assumption that $X\rightarrow X'$ and $Y\rightarrow Y'$ have the property $P$, the map $g$ has the property $P$. By $(c)$, it suffices to show that the map $f$ has the property $P$. Because the front rectangle is a pullback, the front square on the left is also a pullback, so it suffices to show that the map $\alpha$ has the property $P$. The map $\gamma$ has the property $P$ by assumption, and since the bottom face is a pullback, the map $\beta$ has the property $P$ as well by $(b)$, so we conclude using $(c)$.
\end{proof}
The preceding lemma applies in particular to the class of open substack inclusions and the classes of \'{e}tale and stacky \'{e}tale maps of $\dstack$ (it also applies, for example, to the class of $n$-truncated maps in any \infcat that admits finite limits). Before we give the proof of Theorem \ref{thm:relellipticrepresentability}, we recall the following basic facts about $L^2$ Sobolev spaces on manifolds, which are all consequences of the Sobolev embedding and multiplication theorems.
\begin{lem}[Functorial Sobolev scale]\label{thm:sobolevscale}
Let $M$ be a compact manifold and let $\mathsf{Vb}(M)\subset\mathsf{Mfd}_{/M}$ be the full subcategory spanned by finite rank vector bundles (so that morphisms are not necessarily vector bundle maps). Then there is a functor 
\[ H^{\bullet}(\_;M) :\mathsf{Vb}(M) \longrightarrow 
\fun((\Z^{op}_{>\mathrm{dim}\,M/2})^{\lhd},\mathsf{FrMfd})\times_{\fun((\Z^{op}_{>\mathrm{dim}\,M/2}),\mathsf{FrMfd})}\fun((\Z^{op}_{>\mathrm{dim}\,M/2}),\mathsf{BanMfd}) \]
with the following properties.
\begin{enumerate}[$(1)$]
\item The functor $H^{\bullet}(\_;M)$ carries each vector bundle $E\rightarrow M$ to the tower
\[  \Gamma(E;N) \longrightarrow \ldots \longrightarrow H^{l+1}(E;N)\longrightarrow H^{l}(E;N)\longrightarrow \ldots, \]
where $H^{l}(E;N)$ for $l>\mathrm{dim}\,N/2$ is the Sobolev Hilbert space of $L^2$ sections of $E$ that are $l$-times weakly differentiable and the value at the cone point $\Gamma(E;N)$ is the Fr\'{e}chet space of smooth sections of $E$. Each map in this diagram is an injective compact linear operator with dense image. Moreover, the tower is a limit diagram in the category of Fr\'{e}chet manifolds.
\item The evaluation of $H^{\bullet}(\_;M)$ at the cone point is the functor 
\[  \mathsf{Vb}(M) \longrightarrow \mathsf{FrMfd},\quad\quad E\longmapsto \Gamma(E;N) \]
that carries a smooth map $f:E\rightarrow F$ over $M$ to the composition $f \circ \_$.
\item Let $C(\_;M):\mathsf{Vb}(M)\rightarrow\mathsf{FrMan}$ be the functor that carries a vector bundle to the Banach space of continuous sections equipped with the compact open topology (induced by any norm coming from a Riemannian metric). Regard this functor as a constant diagram $(\Z^{op}_{>\mathrm{dim}\,M/2})^{\lhd}\rightarrow\mathsf{FrMan}$, then the functor $H^{\bullet}$ admits a lift to 
\[  \fun((\Z^{op}_{>\mathrm{dim}\,M/2})^{\lhd},\mathsf{FrMfd})_{/C(\_;M)}\times_{\fun((\Z^{op}_{>\mathrm{dim}\,M/2}),\mathsf{FrMfd})
_{/C(\_;M)}}\fun((\Z^{op}_{>\mathrm{dim}\,M/2}),\mathsf{BanMfd})_{/C(\_;M)}    \]
such that evaluation at the cone point yields the obvious natural transformation $\Gamma(\_;M)\subset C(\_;M)$ of smooth sections embedded into continuous sections.
\item For $E\rightarrow F$ an open embedding over $M$, the diagram 
\[
\begin{tikzcd}
H^l(E;N) \ar[d] \ar[r] & C(E;N) \ar[d] \\
H^l(F;N) \ar[r] & C(F;N)
\end{tikzcd}
\]
provided by the lift of $(3)$ is a pullback for all $l>\mathrm{dim}\,N/2$. In particular, the map $H^l(E;N)\rightarrow H^l(F,N)$ is an open embedding of Banach manifolds (as $C(E;N)$ and $C(F;N)$ carry the compact open topology). 
\item Let $k\geq 0$ and let $H^{\bullet+k}(\_;M)$ be the restriction to $(\Z^{op}_{>\mathrm{dim}N/2})^{\lhd}$ of the composition of $H^{\bullet}(\_;M)$ with the shift functor 
\[(\Z^{op}_{>\mathrm{dim}N/2-k})^{\lhd}\longrightarrow (\Z^{op}_{>\mathrm{dim}N/2})^{\lhd},\quad\quad n\longmapsto n+k. \]
Then for each $k\geq 0$, there exists a natural transformation $j^k:H^{\bullet+k}(\_;M)\rightarrow H^{\bullet}(J^k_M(\_);M)$, the \emph{jet prolongation}. At the cone point, this natural transformation restricts to the jet prolongation transformation $j^k:\Gamma(\_;M)\rightarrow \Gamma(J^k_M(\_);M)$ on smooth sections.
\end{enumerate}
\end{lem}
The construction of the Sobolev Hilbert scale functor is done carefully by Palais \cite{Palais}. The last preliminary we require is the following (elementary) version of elliptic bootstrapping.
\begin{lem}[Regularizing property of nonlinear elliptic PDEs]
Let $F,E\rightarrow M$ be vector bundles over a compact manifold, let $\sigma\in \Gamma(F;M)$ be a section and let $V\subset J^k_M(F)$ be vector bundle neighbourhood of the image of the prolongation $j^k(\sigma)$. Let $\mathcal{P}:V\rightarrow E$ be a smooth map over $M$ and let 
\[Q:= \{\sigma\in  \Gamma(F;M);\,j^k(\sigma)(M)\subset V\}\subset \Gamma(F;M)
\] 
be the (open) subset of those sections $\sigma$ whose $k$'th prolongation carries $M$ into $V$, and define $Q_l\subset H^l(F;M)$ for $l>\mathrm{dim}\,M/2$ similarly using the jet prolongation map of Sobolev sections of Lemma \ref{thm:sobolevscale}. Let 
\[P:Q\longrightarrow \Gamma(V;M)\overset{\mathcal{P}}{\longrightarrow}\Gamma(E,M)\]
 be the \emph{partially defined} nonlinear PDE (acting on sections of $F$ that are close to $\sigma$ in the Whitney $C^k$-topology on $\Gamma(E;M)$), and define $P_l:Q_{l+k}\rightarrow H^l(E;M)$ for $l>\mathrm{dim}\,M/2$ similarly using the functoriality of the Sobolev space construction of Lemma \ref{thm:sobolevscale}. Suppose that for all $\sigma \in Q$, the linearization $T\mathcal{P}_{\sigma}:J^k_M(F)\rightarrow E$ of $\mathcal{P}$ at $\sigma$ is a linear elliptic operator. Then there exists an $l'>\mathrm{dim}\,M/2$ depending only on $\mathrm{dim}\,M$ and $k$ (thus independent of $\mathcal{P}$) such that for all $l>l'$ the following holds: let $\tau\in \Gamma(E;M)$ and suppose that $s\in Q_{l+k}$ satisfies $P_l(s)=\tau$. Then $s$ lies in the image of the map $Q \hookrightarrow Q_{l+k}$.  
\end{lem}
This is an easy exercise using parametrices for elliptic operators with limited regularity as in \cite[Theorem 1.2.D]{taylorpseudopde}.
\begin{proof}[Proof of Theorem \ref{thm:relellipticrepresentability}]
By descent, we may suppose that $S$ is a manifold and that $Y_1\rightarrow M, Y_2\rightarrow M$ and $X\rightarrow M$ are $S$-families of submersions and $M\rightarrow S$ is a trivial $S$-family with compact fibre $N$. We will show that the conditions of Proposition \ref{prop:representablebyscheme} are satisfied for $\mathsf{Sol}(\icate)$. We will first demonstrate that we may assume that all three submersions $Y_1\rightarrow M$, $Y_2\rightarrow M$ and $X\rightarrow M$ are trivial $S$-families of vector bundles (note that the $S$-families are trivial, not the vector bundles) while $P_1$ and $P_2$ are partially defined on sections close in the Whitney $C^k$-topology to the respective zero sections. Choose a solution $t:*\rightarrow \mathsf{Sol}(\icate)$ determining a point $s\in S$ and sections $\sigma_1:N\cong M_s\rightarrow Y_{1s}$, $\sigma_2:N\rightarrow Y_{2s}$ and $\sigma_X:N\rightarrow X_s$. Applying Lemma \ref{lem:localstructureweil} (three times), we can choose 
\begin{enumerate}[$(1)$]
    \item an open neighbourhood $s\in S'\subset S$.
    \item vector bundle neighbourhoods 
    \[V_{s,\sigma_X}\subset \sigma_X^*TX_s/N,\quad\quad
    V_{x,\sigma_1}\subset\sigma_{1}^*TY_{1s}/N,\quad\quad
    V_{x,\sigma_2}\subset\sigma_{2}^*TY_{2s}/N\]
     of the respective zero sections. 
    \item commuting diagrams
\[
\begin{tikzcd}
N \ar[r,hook,"\iota_{s} \circ \sigma_X"]\ar[d,"\{\sigma_X\}\times 0"'] & X\ar[d] \\
 S'\times V_{s,\sigma_X}\ar[ur,hook]\ar[r]  & S\times N,
\end{tikzcd}\quad\quad 
\begin{tikzcd}
N \ar[r,hook,"\iota_{s} \circ \sigma_1"]\ar[d,"\{\sigma_1\}\times 0"'] & Y_1\ar[d] \\
 S'\times V_{s,\sigma_1}\ar[ur,hook]\ar[r]  & S\times N,
\end{tikzcd}
\quad\quad 
\begin{tikzcd}
N \ar[r,hook,"\iota_{s} \circ \sigma_2"]\ar[d,"\{\sigma_2\}\times 0"'] & Y_2\ar[d] \\
 S'\times V_{s,\sigma_2}\ar[ur,hook]\ar[r]  & S\times N,
\end{tikzcd}
\]
where the diagonal maps are open embeddings.
\end{enumerate}
Pulling back the elliptic moduli problem back along the open embedding $S'\subset S$, we may assume that $S=S'$. By Proposition \ref{prop:jetbundleproperties}, the data above provides open embeddings $S\times J^k_N(\sigma^*_1TY_{1s}/N)\rightarrow J^k_{M/S}(Y_1)$ and $S\times J^k_N(\sigma^*_2TY_{2s}/N)\rightarrow J^k_{M/S}(Y_2)$ over $S\times N$. Shrinking $S$ again if necessary, we may find vector bundle neighbourhoods $W_1\subset \sigma^*_1TY_{1s}/N$ and $W_2\subset  \sigma^*_2TY_{2s}/N$ of the respective zero sections fitting into commuting diagrams
\[
\begin{tikzcd}
    S\times W_1\ar[d]\ar[r,hook] & S\times J^k_N(\sigma^*_1TY_{1s}/N)\ar[r,hook] & J^k_{M/S}(Y_1) \ar[d,"P_1"] \\
    S\times V_{s,\sigma_X} \ar[rr,hook] && X,
\end{tikzcd}\quad\quad\quad 
\begin{tikzcd}
    S\times W_2\ar[d]\ar[r,hook] &S\times J^k_N(\sigma^*_2TY_{2s}/N) \ar[r,hook]& J^k_{M/S}(Y_2) \ar[d,"P_2"] \\
    S\times V_{s,\sigma_X} \ar[rr,hook] && X.
\end{tikzcd}
\]
Let $Q_1$ and $Q_2$ be defined by the pullbacks
\[
\begin{tikzcd}
Q_1\ar[d]\ar[r] & \map_N(N,\sigma^*_1TY_{1s}/N) \ar[d,"j^k"] \\
\map_N(N,W_1) \ar[r] & \map_N(N,J^k_N(\sigma^*_1TY_{1s}/N)),
\end{tikzcd}\quad\quad\quad
\begin{tikzcd}
Q_2\ar[d]\ar[r] & \map_N(N,\sigma^*_2TY_{2s}/N) \ar[d,"j^k"] \\
\map_N(N,W_2) \ar[r] & \map_N(N,J^k_N(\sigma^*_2TY_{2s}/N)),
\end{tikzcd}
\]
among derived $\cinfty$-stacks, then we have by construction commuting diagrams 
\[
\begin{tikzcd}
S\times Q_1\ar[d]\ar[r] & \mathsf{Res}_{M/S}(Y_1) \ar[d,"P_1"] \\
 S\times \map_N(N,\sigma_X^*TX_s/N) \ar[r] & \mathsf{Res}_{M/S}(X),
\end{tikzcd}\quad\quad\quad
\begin{tikzcd}
S\times Q_2\ar[d]\ar[r] & \mathsf{Res}_{M/S}(Y_2) \ar[d,"P_1"] \\
 S\times \map_N(N,\sigma_X^*TX_s/N) \ar[r] & \mathsf{Res}_{M/S}(X),
\end{tikzcd}
\]
where the horizontal maps are open substack inclusions by Proposition \ref{prop:etaletarget}. Let $\mathsf{Sol}(\icate)_t$ be the cone in the pullback diagram 
\[
\begin{tikzcd}
    \mathsf{Sol}(\icate)_t\ar[d]\ar[r] & S \times Q_1\ar[d] \\
    S \times Q_2 \ar[r] & S\times\map_N(N,\sigma_X^*TX_s/N)
\end{tikzcd}
\]
among derived $\cinfty$-stacks, then the map $\mathsf{Sol}(\icate)_t\rightarrow \mathsf{Sol}(\icate)$ is an open substack inclusion whose image contains the section $t$, by Lemma \ref{lem:pullbackopeninclusion}. Varying $t$ over all solutions, we have an effective epimorphism
\[ \coprod_{t\in \mathsf{Sol}(\icate)}\mathsf{Sol}(\icate)_t\longrightarrow \mathsf{Sol}(\icate)\]
so invoking Proposition \ref{prop:representablebyscheme}, it suffices to argue that $\mathsf{Sol}(\icate)_t$ is representable by a quasi-smooth derived $\cinfty$-scheme locally of finite presentation. Consequently, we may suppose that we are in the following situation: we have vector bundles $E\rightarrow N$, $F_1\rightarrow N$ and $F_2\rightarrow N$, and we are given $S$-families of nonlinear PDEs 
\[\mathcal{P}_1:S\times(W_1\subset J^k_N(F_1))\longrightarrow E,\quad\quad \mathcal{P}_2:S\times(W_1\subset J^k_N(F_2))\longrightarrow E\]
over $N$ defined on (vector bundle) neighbourhoods of the respective zero sections. Define the open substack $Q_1\hookrightarrow\map_N(N,F_1)$ as the pullback $Q_1=\map_N(N,F_1)\times_{\map_N(N,J^k_N(F_1))}\map_N(N,W_1)$ and define $Q_2$ similarly, then we have a commuting diagram
\[
\begin{tikzcd}
    \mathsf{Sol}(\icate)\ar[d]\ar[r] & S \times Q_1\times Q_2\ar[d] \\
    S \times \map_N(N,E)\ar[r,"\mathrm{id}_S\times \Delta"]\ar[d] & S\times\map_N(N,E\times_NE) \ar[d]\\
    S\ar[r,"\mathrm{id}_S\times 0"]  & S\times\map_N(N,E)
\end{tikzcd}
\]
in which both squares are pullbacks, where the lower horizontal map is the ($S$-parametrized) zero section of $E$ and the lower right vertical map is induced by the map $E\times_NE\rightarrow E$ carrying $(v,w)$ to $v-w$. Write $F$ for the vector bundle $F_1\times_NF_2$, $W$ for $W_1\times_NW_2\subset J^k_N(F)$ and $Q$ for $Q_1\times Q_2$ (so that $Q\hookrightarrow \map_N(N,F)$ is the open substack defined by the pullback $\map_N(N,F)\times_{\map_N(N,J^k_N(F))}\map_N(N,W)$). By Proposition \ref{prop:sourceconvenient}, the open substack $Q\subset \map_M(M,F)$ is representable by the open subset of the convenient vector space $\Gamma(F;M)$ of those sections whose $k$'th prolongation lies in $W$; we will abusively also denote this open set $Q$. The right vertical composition in the diagram above is then an $S$-family of PDEs representable by the map 
\[ P:S\times Q\longrightarrow S\times\Gamma(E;N) \]
of Fr\'{e}chet manifolds (which are simply open subsets of Fr\'{e}chet spaces) induced by the map $\mathcal{P}:S\times (W\subset J^k(F))\rightarrow S\times E$ obtained from $\mathcal{P}_1$ and $\mathcal{P}_2$. We will let $P_s$ denote the PDE
\[  P_s:\{s\}\times Q\subset S\times Q\overset{P}{\longrightarrow} S\times \Gamma(E;N)\longrightarrow \Gamma(E;N)  \]
for all $s\in S$, where the last map projects away $S$ (note that the composition of the first two maps already takes values in $\{s\}\times\Gamma(E;N)$). Let $Q_{k+l}\subset H^{k+l}(F;N)$ be the pullback $H^{k+l}(F;N)\times_{H^l(J^k_N(F);N)}H^l(W;N)$ along the jet prolongation $H^{k+l}(F;N)\rightarrow H^l(J^k(F;N))$ provided by Lemma \ref{thm:sobolevscale}, then we have a commuting tower 
\[
\begin{tikzcd}
   Q \ar[d,hook]\ar[r] & \ldots\ar[r] &  Q_{k+l+1}\ar[d,hook]\ar[r] & Q_{k+l}\ar[r]\ar[d,hook] & \ldots \\
  \Gamma(F;N) \ar[r] &\ldots \ar[r] &  H^{l+1}(F;N)\ar[r] &  H^l(F;N)\ar[r] & \ldots
\end{tikzcd}
\]
in the category of Fr\'{e}chet manifolds and smooth maps among them. It follows from Lemma \ref{thm:sobolevscale} that the lower tower is a limit diagram and that each square in the tower is a pullback, so that the upper tower is also a limit. As a minor variation of $(1)$ and $(2)$ of Lemma \ref{thm:sobolevscale}, Sobolev's embedding theorem guarantees that for $l\in\Z_{> \mathrm{dim}\,N/2}$, the composition 
\[S\times Q\longrightarrow  S\times\Gamma(E;N)\longrightarrow S\times H^l(E;N)  \]
admits for all $k'\geq k$ a unique smooth extension 
\[ S\times Q_{k'+l}\longrightarrow S\times H^l(E;N) \]
over $S$. Write $P_l$ for the map $S\times Q_{k+l}\rightarrow S\times H^l(E;N)$ and write $P_{s,l}$ for the composition 
\[   P_{s,l}:\{s\}\times Q_{k+l}\subset S\times Q_{k+l}\overset{P_l}{\longrightarrow} S\times H^k(E;N)\longrightarrow H^k(E;N)     \]
for all $s\in S$, where the last map projects away $S$ (and the composition of the first two maps already takes values in $\{s\}\times H^k(E;N)$), so that $P_{s,l}$ is the unique smooth extension of $P_s:Q\rightarrow\Gamma(E;N)$ provided by Lemma \ref{thm:sobolevscale}. We have a commuting tower 
\[
\begin{tikzcd}
    S\times Q \ar[d,"P"]\ar[r] & \ldots\ar[r] &  S\times Q_{k+l+1}\ar[d,"P_{l+1}"]\ar[r] & S\times Q_{k+l}\ar[d,"P_l"]\ar[r] & \ldots \\
    S\times\Gamma(E;N) \ar[r] &\ldots \ar[r] & S\times H^{l+1}(E;N)\ar[r] & S\times H^l(E;N)\ar[r] & \ldots
\end{tikzcd}
\]
in the category of Fr\'{e}chet manifolds and smooth maps among them, which is a limit in the category of convenient manifolds, by Lemma \ref{thm:sobolevscale}. Since $j_{\mathsf{Con}}$ preserves limits, the tower 
\[
\begin{tikzcd}
    S\times Q\ar[d,"P"]\ar[r] & \ldots\ar[r] &  S\times j_{\mathsf{Con}}(Q_{k+l+1})\ar[d,"P_{l+1}"]\ar[r] & S\times j_{\mathsf{Con}}(Q_{k+l})\ar[d,"P_l"]\ar[r] & \ldots \\
    S\times \map_N(N,E) \ar[r] &\ldots \ar[r] & S\times j_{\mathsf{Con}}(H^{l+1}(E;N))\ar[r] & S\times j_{\mathsf{Con}}(H^l(E;N))\ar[r] & \ldots
\end{tikzcd}
\]
is a limit in the \infcat $\smst$ (note however that this tower \emph{need not be} a limit in the \infcat $\dstack$, the inclusion $\smst\subset \dstack$ does not preserve limits of towers!). By assumption, for each solution $t$ of $P$ determining a point $(s,\sigma)\in S\times Q$, the linearization of $P_s:Q\rightarrow \Gamma(E;N)$ at $\sigma$ is a linear elliptic PDE. We will use the following standard Fredholm and linear elliptic theory.
\begin{enumerate}[$(1)$]
\item For all $l\in \Z_{\geq 0}$, a linear PDE $P:\Gamma(F;N)\rightarrow\Gamma(E;N)$ is elliptic if and only if its continuous extension $P_l:H^{k+l}(F;N)\rightarrow H^l(E;N)$ is a Fredholm operator \cite[Theorem 19.5.1, Theorem 19.5.2]{HormanderIII}. 
\item Let $B_0,B_1$ Banach spaces and let $B(B_0,B_1)$ the Banach space of bounded linear operators among them equipped with the operator norm, then the space of Fredholm operators $\mathsf{Fred}(B_0,B_1)\subset B(B_0,B_1)$ is open.
\item For $B_0,B_1$ Banach spaces, the assignment
\[  \mathsf{Fred}(B_0,B_1) \longrightarrow \Z_{\geq 0},\quad\quad P\longmapsto \mathrm{dim}\,\mathrm{coker}\,P \]
is upper semicontinuous. 
\end{enumerate}
For each $l>\mathrm{dim}\,N/2$ consider the smooth map 
\[   S\times Q_{k+l}l\times H^{k+l}(F;N)\overset{0\times\mathrm{id}}{\hooklongrightarrow} TS\times Q_{k+l} \times  H^{k+l}(F;N) \overset{TP_l}{\longrightarrow} TS\times H^l(E;N)    \longrightarrow H^l(E;N)\]
where the first map embeds into the zero section of $TS$, the second maps is the differential of $P_l$ and the last map projects away $TS$. This composition is adjoint to a smooth (and thus continuous since Banach spaces carry the $c^{\infty}$-topology) map 
\[  TP_{l}^0: S\times Q_{k+l}\longrightarrow B(H^{k+l}(F;N),H^l(E;N)), \]
which carries a point $(s',\sigma')$ to the differential $T_{\sigma'}P_{s',l}$ of the map $P_{s',l}:Q_{k+l}\rightarrow H^l(E;N)$ at $\sigma'$. Let $O_l\subset S\times Q_l$ be the set obtained by taking the inverse image of the set of Fredholm operators along the map $TP_{l}^0$; this is open by $(2)$ above. For all $(s',\sigma')\in S\times Q_{k+l}$, the diagram 
\[
\begin{tikzcd}
H^{k+l}(F;N)\ar[d,"0\times\mathrm{id}"]\ar[r,"T_{\sigma'}P_{s',l}"] &[2em] H^k(E;N) \ar[d,"0\times\mathrm{id}"] \\
T_{s'}S\times H^{k+l}(F;N)\ar[r,"T_{(s',\sigma')}P_l"] &[2em] T_{s'}S\times H^k(E;N)
\end{tikzcd}
\]
of Banach spaces and linear maps between them identifies the kernel and cokernel of $T_{\sigma'}P_{s',l}$ with the kernel and cokernel of $T_{(s',\sigma')}P_l$, and the same holds for $l=\infty$, that is, the kernel and cokernel of $T_{\sigma'}P_{s'}$ are identified with the kernel and cokernel of $T_{(s',\sigma')}P$. By this observation and $(1)$ and $(2)$ above, the sets $\{O_{l+k}\}_{l>\mathrm{dim}\,N/2}$ have the following properties.
\begin{enumerate}[$(i)$]
\item For every $l> \mathrm{dim}\,N/2$, the subset $O_{l+k}\subset S\times Q_{l+k}$ consists of those points $(s',\sigma')$ for which the differential 
\[ T_{(s',\sigma')}P_{l} : T_{s'}S \times H^{k+l}(F;N) \longrightarrow T_{s'}S\times H^l(E;N)  \]
is a Fredholm operator. 
\item For every $l> \mathrm{dim}\,N/2$ and every $(s',\sigma')\in S\times Q \cap O_{l+k}$, the diagram 
\[
\begin{tikzcd}
T_{s'}S \times \Gamma(F;N) \ar[r,"T_{(s',\sigma')}P"] \ar[d]&[2em] T_{s'}S \times \Gamma(E;N)
 \ar[d]\\ T_{s'}S \times H^{k+l}(F;N)  \ar[r,"T_{(s',\sigma')}P_{l}"] &[2em] T_{s'}S\times H^l(E;N)  
\end{tikzcd}
\]
identifies the kernel and cokernel of $T_{(s',\sigma')}P$ with the kernel and cokernel of $T_{(s',\sigma')}P_l$ (in particular, the kernel and cokernel of the latter linear operator consist of smooth sections).
\item For any $l>\mathrm{dim}\,N/2$, we have the equality of open subsets
\[ S\times Q \cap O_{l+k}= \{(s',\sigma')\in S\times Q;\, T_{\sigma'}P_{s'}:\Gamma(F;N)\rightarrow \Gamma(E;N)\text{ is elliptic}\} \]
of $S\times Q$. In particular this set, that we will denote $O\subset S\times Q$, is independent of $l$.
\end{enumerate}
Thus, to study solutions of $P$, it suffices to restrict to $O\subset S\times Q$. We will deduce the local representability of the solution stack of $P$ by way of an implicit function theorem that allows us to replace the maps $P_l$ with maps among finite dimensional manifolds. To implement this strategy, we add `obstruction bundles'  to the spaces $O_{l+k}$. Let $t$ be a solution of $P$ determining a pair $(s,\sigma)\in O$. Choose a continuous linear section $\iota$ for the projection $\Gamma(E;N)\rightarrow \mathrm{coker}\,T_{(s,\sigma)}P$ and consider for any $l>\mathrm{dim}\,N/2$ the smooth maps 
\[ P_l+\iota : O_{l+k}\times\mathrm{coker}\,T_{(s,\sigma)}P\longrightarrow H^l(E;N).  \]
These maps and sets have the following properties.
\begin{enumerate}[$(i')$]
\item For every $l>\mathrm{dim}\,N/2$ and every $(s',\sigma',e)\in O_{l+k}\times\mathrm{coker}\,T_{(s,\sigma)}P$, the differential $T_{(s',\sigma',e)}(P_l+\iota)=T_{(s',\sigma')}+\iota:$
\[ T_{s'}S\times H^{k+l}(F;N)\times\mathrm{coker}\,T_{(s,\sigma)}P \longrightarrow  T_{s'}S\times H^{k+l}(E;N) \]
is a Fredholm operator.
\item For every $l> \mathrm{dim}\,N/2$ and every $(s',\sigma',e)\in O \times\mathrm{coker}\,T_{(s,\sigma)}P \cap O_{l+k}\times\mathrm{coker}\,T_{(s,\sigma)}P $, the diagram 
\[
\begin{tikzcd}
T_{s'}S \times \Gamma(F;N) \times \mathrm{coker}\,T_{(s,\sigma)}P \ar[r,"T_{(s',\sigma')}P+\iota"] \ar[d]&[3em] T_{s'}S \times \Gamma(E;N)
 \ar[d]\\ T_{s'}S \times H^{k+l}(F;N) \times \mathrm{coker}\,T_{(s,\sigma)}P  \ar[r,"T_{(s',\sigma')}P_{l}+\iota"] &[3em] T_{s'}S\times H^l(E;N)  
\end{tikzcd}
\]
identifies the kernel and cokernel of $T_{(s',\sigma')}P+\iota$ with the kernel and cokernel of $T_{(s',\sigma')}P_l+\iota$ (in particular, the kernel and cokernel of the latter linear operator consist of smooth sections).
\end{enumerate}
Using $(ii')$ and $(3)$ above, the subset 
\[  U_{s,\sigma,l}:=\{(s',\sigma',e)\in O_{l+k}\times \mathrm{coker}\,T_{(s,\sigma)}P;\, T_{(s',\sigma',e)}(P_l+\iota)\text{ is onto}\}\subset O_{l+k}\times\mathrm{coker}\,T_{(s,\sigma)}P  \] 
is an open neighbourhood of $\{(s,\sigma)\}\times\mathrm{coker}\,T_{(s,\sigma)}P\subset  O_{l+k}\times\mathrm{coker}\,T_{(s,\sigma)}P$ and we have an equality of subsets 
\[ O\times\mathrm{coker}\,T_{(s,\sigma)}P \cap U_{s,\sigma,l} = \{(s',\sigma',e)\in O\times \mathrm{coker}\,T_{(s,\sigma)}P;\, T_{(s',\sigma',e)}(P+\iota)\text{ is onto}\} \]
of $O\times\mathrm{coker}\,T_{(s,\sigma)}P$. In particular this set, that we will denote $U_{s,\sigma}$, is independent of $l$ and is an open neighbourhood of $\{(s,\sigma)\}\times\mathrm{coker}\,T_{(s,\sigma)}P \subset O\times \mathrm{coker}\,T_{(s,\sigma)}P$. Inductively redefining $U_{s,\sigma,l}$ as $U_{s,\sigma,l}\cap U_{s,\sigma,l-1}$ (starting at the first integer larger than $\mathrm{dim}\,N/2$), we may assume that $U_{s,\sigma,l}\subset U_{s,\sigma,l-1}$. We now have a limit diagram 
\[
\begin{tikzcd}
    U_{s,\sigma}\ar[d,"P+\iota|_{U_{s{,}\sigma}}"]\ar[r] & \ldots\ar[r] & U_{s,\sigma,l+1}\ar[d,"P_{l+1}+\iota|_{U_{s{,}t{,}l+1}}"]\ar[r] &  U_{s,\sigma,l}\ar[d,"P_l+\iota|_{U_{s{,}t{,}l}}"]\ar[r] & \ldots \\
    S\times\Gamma(E;N) \ar[r] &\ldots \ar[r] & S\times H^{l+1}(E;N)\ar[r] & S\times H^l(E;N)\ar[r] & \ldots
\end{tikzcd}
\]
of convenient manifolds indexed by $(\Z^{\rhd}_{>\mathrm{dim}\,N/2})^{op}$ where each vertical map has everywhere surjective differential. The core of the argument is the following claim.
\begin{enumerate}
    \item[$(*)$] The map $P+\iota|_{U_{s,\sigma}}: j_{\mathsf{Con}}(U_{s,\sigma})\rightarrow S\times j_{\mathsf{Con}}(\Gamma(E;N))$ is a submersive map in $\dstack$.  
\end{enumerate}
To prove this, we invoke the criterion of Proposition \ref{prop:criterion} to conclude that we have to show that for any smooth manifold $Z$ and any map $Z\rightarrow S\times j_{\mathsf{Con}}(\Gamma(E;N))$ of smooth stacks, the map
\[ g:Z\times_{S\times j_{\mathsf{Con}}(\Gamma(E;N))}  j_{\mathsf{Con}}(U_{s,\sigma})\longrightarrow Z\]
where the pullback is \emph{taken in $\smst$}, is (representable by) a submersion of ordinary manifolds. Since the functor $j_{\mathsf{Con}}:\mathsf{ConMfd}\hookrightarrow\smst$ preserves limits, it suffices to show that the limit $Z\times_{S\times \Gamma(E;N)} U_{s,\sigma}$ exists in the category of convenient manifolds and that the map $ Z\times_{S\times \Gamma(E;N)}  U_{s,\sigma}\rightarrow Z$ is a submersion of ordinary manifolds. Since for each $l>l'$, the map $P_l+\iota|_{U_{s,\sigma,l}}$ is a Fredholm submersion of Banach manifolds, applying the implicit function theorem for Banach manifolds shows that the limit $Z\times_{S\times H^l(E;N)} U_{s,\sigma,l}$ exists in the category of convenient manifolds and that the map $Z\times_{S\times H^l(E;N)}U_{s,\sigma,l}\rightarrow Z$ is a submersion of ordinary manifolds. We thus have a diagram 
\[
\begin{tikzcd}
\ldots\ar[d]\ar[r] & Z\times_{S\times H^{l+1}(E;N)} U_{s,\sigma,l+1}\ar[r] \ar[d]& 
    Z\times_{S\times H^{l}(E;N)} U_{s,\sigma,l}\ar[r] \ar[d] &\ldots \ar[d] \\
\ldots \ar[r,equal] & Z \ar[r,equal] & Z\ar[r,equal] & \ldots    
\end{tikzcd}
\]
of submersions in the category $\mathsf{Mfd}$ whose limit in $\smst$ after applying the Yoneda embedding (which coincides with $j_{\mathsf{Con}}$ on the full subcategory $\mathsf{Mfd}\subset\mathsf{ConMfd}$) is equivalent to $ Z\times_{S\times j_{\mathsf{Con}}(\Gamma(E;N))}  j_{\mathsf{Con}}(U_{s,\sigma})$, so we are reduced to proving that the limit of the diagram above exists in the category of manifolds and that the resulting map $Z'\rightarrow Z$ is a submersion. It will suffice to show that the diagram is eventually essentially constant. Write $Z'_l=Z\times_{S\times H^{l}(E;N)} U_{s,\sigma,l}$, then it follows immediately from the regularizing property of elliptic PDEs and the assumption that the maps $Z\rightarrow S\times H^l(E;N)$ factor through the smooth sections as $Z\rightarrow S\times \Gamma(E;N)\rightarrow S\times H^l(E;N)$ that the maps $Z'_{l+1}\rightarrow Z'_{l}$ are eventually bijections on the underlying sets. It now suffices to argue that the maps $Z'_{l+1}\rightarrow Z'_l$ are eventually local diffeomorphisms. Since $Z'_l\rightarrow Z$ is a submersion for each $l>l'$, it suffices to show that the map $Z'_{l+1}\rightarrow Z'_l$ induces an isomorphism on each tangent space of each fibre over $Z$. As we have just argued, we may assume that $Z'_{l+1}$ and $Z'_l$ consist entirely of \emph{smooth} points, that is, the maps $Z'_l\rightarrow U_{s,\sigma,l}$ factors set-theoretically through $U_{s,\sigma}$, and similarly for $Z'_{l+1}$. For all $l>\mathrm{dim}\,N/2$, the map $Z'_l\rightarrow U_{s,\sigma,l}$ identifies the tangent relative to $Z$ with the kernel of the tangent map $T(P_l+\iota|_{U_{s,\sigma,l}})$. By ellipticity (more precisely, by $(ii')$ above), the maps of the Sobolev scale become isomorphisms on this kernel at each smooth point (points in the image of $U_{s,\sigma}\subset U_{s,\sigma,l}$), which concludes the proof of $(*)$\footnote{For an alternative argument, it is possible to apply the Nash-Moser inverse function theorem directly to the map $P+\iota|_{U_{s,\sigma}}$}. The map $ j_{\mathsf{Con}}(U_{s,\sigma})\rightarrow  S\times Q \times\mathrm{coker}\,T_{(s,\sigma)}P$ is an open substack inclusion by Proposition \ref{prop:convmfdetale}. Let $H_{s,\sigma}\hookrightarrow Q\times S$ be the open substack defined as the cone in the pullback diagram 
\[
\begin{tikzcd}
H_{s,\sigma} \ar[d] \ar[r] & S\times Q \ar[d,"\mathrm{id}\times 0"]\\
j_{\mathsf{Con}}(U_{s,\sigma})\ar[r]&S\times Q \times\mathrm{coker}\,T_{(s,\sigma)}P.
\end{tikzcd}\]
Now consider the diagram
\[
\begin{tikzcd}
Z'_t\ar[d]\ar[r] & H_{s,\sigma}\ar[d]\ar[r] & S\ar[d,"\mathrm{id}\times 0"] \\
Z_t\ar[d]\ar[r] & j_{\mathsf{Con}}(U_{s,\sigma}) \ar[d,"P+\iota|_{U_{s{,}t}}"] \ar[r,"\pi"]&S\times \mathrm{coker}\,T_{(s,\sigma)}P \\  
S\ar[r,"\mathrm{id}\times 0"] & S\times \map_N(N,E).
\end{tikzcd}
\]
where the two left squares are defined to be pullbacks in the \infcat $\dstack$ and the upper right square is a pullback by construction. By assertion $(*)$, the object $Z_t$ is (representable by) an ordinary manifold. Since the upper rectangle is a pullback, it follows that $Z'_t$ is representable by a quasi-smooth affine derived $\cinfty$-scheme of finite presentation (the lower composition in this rectangle is the \emph{Kuranishi map}). However, by Lemma \ref{lem:pullbackopeninclusion} there is an open substack inclusion $Z'_t\rightarrow \mathsf{Sol}(\icate)$ containing the section $t$. Letting $t$ vary over all solutions, we have an effective epimorphism 
\[  \coprod_{t\in\mathsf{Sol(\icate)}}Z'_t\longrightarrow \mathsf{Sol}(\icate)\]
so we conclude by invoking Proposition \ref{prop:representablebyscheme} once more.
\end{proof}
\begin{rmk}
Most of the arguments above are standard techniques in nonlinear Fredholm analysis (passing to Sobolev Banach spaces to have the inverse function theorem available, adding an obstruction bundle and using (nonlinear) elliptic regularity to recover smoothness), reinterpreted to fit our derived and stacky setup. We emphasize the interpretation of the Sobolev scale as a \emph{limit diagram} to deduce results at the level of universal solution stacks, independent of any choice of Banach space completion. There is but one technical problem with this strategy obstructing the most straightforward implementation of the aforementioned analytic methods in the derived setting: the Sobolev scale is a limit diagram of smooth stacks \emph{but not of derived $\cinfty$-stacks}. One might attempt to circumvent this problem by simply `truncating the regularity', that is, forgoing to work with the universal solution stack in favour of Sobolev spaces with some fixed regularity. This will not suffice for most nontrivial application however. Aside from aesthetic considerations (going this route would mean defining the solution stack as a pullback of arbitrarily chosen Sobolev completions), this crutch suffers serious technical disadvantages: the diffeomorphism group of a manifold $M$ of dimension $>0$ does not act smoothly on $H^l(F;M)$ for any $l$ (the same is true for other types of completions like H\"{o}lder spaces), which entails that we cannot straightforwardly formulate the notion of a \emph{smooth family} of moduli problems at finite regularity. The example concerning pseudo-holomorphic curves in the next subsection for instance would be out of reach. This analytic nuisance is well documented in symplectic topology where working up to holomorphic automorphisms cannot be avoided. The problem of describing families of moduli problems at (varying but) finite regularity was systematically addressed relatively recently by Hofer-Wysocki-Zehnder's \emph{scale calculus} \cite{HWZbook}, which introduces a weaker notion of differentiability on a scale of Banach spaces for which diffeomorphism groups do act smoothly on the source. We have no need of such sophisticated analytic technology since we never glue Banach mapping spaces of finite regularity together along automorphisms on the source; by design, gluing in our setup always happens at infinite regularity. Indeed, the global $\cinfty$-structure of the solution stack is evident as a pullback of stacks of sections, and Banach spaces only come in to describe its local properties. On the other hand, we are faced with the issue that the \emph{derived} mapping stack is not a limit of the Sobolev scale. The main insight at the interface of derived geometry and nonlinear analysis we advance in this work that solves this problem (and also solves the same problem for any extension of the basic representability theorem of this paper to more complicated and interesting geometric situations) is that basic features of topos theory (those exposed in Section 2.2) allow one to perform all the nonlinear analysis necessary to reduce to a finite dimensional situation in the \infcat of smooth stacks (in the proof above specifically, to see that a map of smooth stacks is submersive as a map of \emph{derived} stacks, it suffices to show it is submersive as a map of smooth stacks), after which may perform a (possibly nontransverse) intersection of finite dimensional manifolds in the \inftop $\dstack$ to obtain the desired local model for the solution stack.
\end{rmk}

\section{Application: nonsingular pseudo-holomorphic curves}
Armed with the elliptic representability theorem, we can give a simple proof of the representability of the moduli space of $J$-pseudo-holomorphic curves in a symplectic manifold.
\begin{cons}
Define a category $\Of_{\mathsf{Surf}_{\C}}$ as follows.
\begin{enumerate}
\item[$(O)$] Objects are tuples $(p:C\rightarrow S,j)$ with $p:C\rightarrow S$ a proper submersion with 2-dimensional fibres and $j$ a section of the endomorphism bundle $\mathsf{End}(TC/S)\rightarrow S$ of the vertical tangent bundle with respect to $p$ that satisfies $j^2=-\mathrm{id}$. For each $s\in S$, we let $j_s:TC_s\rightarrow TC_s$ be the restriction of $j$ to the fibre $C_s:=p^{-1}(s)$; this is a complex structure on the compact surface $C_s$.
\item[$(M)$] Morphisms between tuples $(p':C'\rightarrow S',j')$ and $(p:C\rightarrow S,j)$ are commuting diagrams
\[
\begin{tikzcd}
C'\ar[d]\ar[r] & C\ar[d] \\
S'\ar[r,"f"] & S
\end{tikzcd}
\]
such that 
\begin{enumerate}[$(i)$]
\item The diagram is a pullback. Note that given this condition, the diagram above provides for each $s\in S'$ a diffeomorphism $C_s'\rightarrow C_{f(s)}$ and an isomorphism $TC'_s\rightarrow TC_{f(s)}$ of vector bundles over the map $C_s'\rightarrow C_{f(s)}$.
\item For each $s\in S'$, the diagram 
\[
\begin{tikzcd}
TC'_s\ar[d,"\cong"]\ar[r,"j'_s"]& TC'_{s} \ar[d,"\cong"] \\
 TC_{f(s)}\ar[r,"j_{f(s)}"] & TC_{f(s)}
\end{tikzcd}
\]
commutes.
\end{enumerate}
\end{enumerate}
Let $\Of_{\mathsf{Surf}}\subset \fun(\Delta^1,\mathsf{Mfd})$ be the full subcategory spanned by proper submersions with 2-dimensional fibres. We have an obvious functor $\Of_{\mathsf{Surf}_{\C}}\rightarrow \Of_{\mathsf{Surf}}$ which is a map of right fibrations (with small fibres) over the category $\mathsf{Mfd}$. We let $\mathsf{Surf}_{\C}$ and $\mathsf{Surf}$ be the presheaves on $\mathsf{Mfd}$ obtained from the right fibration $\Of_{\mathsf{Surf}_{\C}}$ respectively the right fibration $\Of_{\mathsf{Surf}}$ via unstraightening. For each proper submersion $C\rightarrow S$, the fibre of $\mathsf{Surf}_{\C}(S)\rightarrow \mathsf{Surf}(S)$ is equivalent to the set $\{j\in \mathsf{End}(TC/S);\,j^2=-1\}\subset \mathsf{End}(TC/S)$. As the presheaf $\mathsf{Surf}$ is a sheaf (since being a proper submersion with 2-dimensional fibres is a local property) and the assignment 
\[  (U\subset S)\longmapsto \{j\in \mathsf{End}(T(C\times_SU)/U);\,j^2=-1\}\subset \mathsf{End}(T(C\times_SU)/U))  \] 
is a sheaf on $S$, we deduce that $\mathsf{Surf}_{\C}$ is a sheaf. Finally, consider the universal proper submersive map with surface fibres $\widetilde{\mathsf{Surf}}\rightarrow\mathsf{Surf}$, then we define a proper submersive map with surface fibres $\widetilde{\mathsf{Surf}}_{\C}\rightarrow \mathsf{Surf}_{\C}$ as the pullback $\widetilde{\mathsf{Surf}}\times_{\mathsf{Surf}}\mathsf{Surf}_{\C}$. More informally, the map $\widetilde{\mathsf{Surf}}_{\C}\rightarrow \mathsf{Surf}_{\C}$ is characterized by the property that for every proper submersion with surface fibres $C\rightarrow S$ equipped with a smoothly varying fibrewise complex structure with base a manifold $S$ determining a map $S\rightarrow \mathsf{Surf}_{\C}$, there is a map $C\rightarrow \widetilde{\mathsf{Surf}}_{\C}$ which is unique up to contractible ambiguity fitting into a pullback diagram 
\[
\begin{tikzcd}
C \ar[d] \ar[r]& \widetilde{\mathsf{Surf}}_{\C}\ar[d] \\
S\ar[r] &  \mathsf{Surf}_{\C}
\end{tikzcd}
\]
of smooth stacks (and also of derived $\cinfty$-stacks by $(3)$ of Proposition \ref{prop:stackysubmersive}).
\end{cons}
Let $(X,\omega)$ be a symplectic manifold equipped with an $\omega$-tame almost complex structure $J$ and consider the surjective submersive map $X\times\widetilde{\mathsf{Surf}}_{\C}\rightarrow \widetilde{\mathsf{Surf}}_{\C}$ of smooth stacks. To define the elliptic moduli problem whose solution stack parametrizes $J$-pseudo-holomorphic curves, we define the \emph{Cauchy-Riemann map}, a $\mathsf{Surf}_{\C}$-family of nonlinear PDEs from $X\times \widetilde{\mathsf{Surf}}_{\C}$ to $J^1_{\widetilde{\mathsf{Surf}}_{\C}/\mathsf{Surf}_{\C}}(X\times \widetilde{\mathsf{Surf}}_{\C})$, that is, a map 
\[  CR:J^1_{\widetilde{\mathsf{Surf}}_{\C}/\mathsf{Surf}_{\C}}(X\times \widetilde{\mathsf{Surf}}_{\C}) \longrightarrow J^1_{\widetilde{\mathsf{Surf}}_{\C}/\mathsf{Surf}_{\C}}(X\times \widetilde{\mathsf{Surf}}_{\C})    \]
over $\widetilde{\mathsf{Surf}}_{\C}$ as follows. The functor $\Of_{\mathsf{Surf}_{\C}}\rightarrow \mathsf{Mfd}\hookrightarrow \smst$ has colimit $\mathsf{Surf}_{\C}$ so we have a diagram $\Of_{\mathsf{Surf}_{\C}}^{\rhd}\rightarrow\mathsf{SmSt}$ carrying the cone point to $\mathsf{Surf}_{\C}$. By descent for \inftopoit, the composition $\Of_{\mathsf{Surf}_{\C}}^{\rhd}\rightarrow\mathsf{SmSt}\rightarrow \catinfh^{op}$ where the second functor is the unstraightening of the codomain fibration, is a colimit diagram. It follows that both functors in the diagram
\[  \fun_{\smst}^{\mathrm{Cart}}(\Of_{\mathsf{Surf}_{\C}},\fun(\Delta^1,\smst)) \longleftarrow  \fun^{\mathrm{Cart}}_{\smst}(\Of_{\mathsf{Surf}_{\C}}^{\rhd},\fun(\Delta^1,\smst)) \overset{\ev_{\infty}}{\longrightarrow} \smst_{/\mathsf{Surf}_{\C}}  \]
are equivalences, so that the map $J^1_{\widetilde{\mathsf{Surf}}_{\C}/\mathsf{Surf}_{\C}}(X\times \widetilde{\mathsf{Surf}}_{\C})\rightarrow \widetilde{\mathsf{Surf}}_{\C}$ is determined by a natural transformation $\alpha$ of Cartesian sections $\Of_{\mathsf{Surf}_{\C}}\rightarrow\fun(\Delta^1,\mathsf{SmSt})$. It follows from $(2)$ of Proposition \ref{prop:jetbundleproperties} that $\alpha$ carries a tuple $(p:C\rightarrow S,j)\in \Of_{\mathsf{Surf}_{\C}}$ to the map 
\[ J^1_{C/S}(X\times C) \longrightarrow  C  \]
over $S$. By $(4)$ of Proposition \ref{prop:jetbundleproperties}, the natural transformation $\alpha$ lies in the full subcategory $\fun_{\mathsf{Mfd}}(\Of_{\mathsf{Surf}_{\C}},\fun(\Delta^1,\mathsf{Mfd}))$. We have an equivalence 
\[J^1_{C/S}(X\times C)\simeq \mathsf{Hom}(T(X\times C)/S,T(X\times C)/C)\times_{\mathsf{End}(T(X\times C)/C)}X\times C\]
over $X\times C$. This is the bundle over $X\times C$ whose fibre at $(x,c)$ is the space of linear maps $T_{c}C_{p(c)}\rightarrow T_xX$. Under this identification, the map $J^1_{C'/S'}(X\times C')\rightarrow J^1_{C/S}(X\times C)$ induced by a pullback diagram
\[
\begin{tikzcd}
C'\ar[d,"p'"] \ar[r,"g"] & C\ar[d,"p"] \\
S'\ar[r,"f"] & S
\end{tikzcd}
\]
is simply the map that carries a linear map $T_{c}C'_{p'(c)}\rightarrow T_xX$ in the fibre over $(x,c)$ to the linear map $T_{g(c)}C_{fp'(c)}\rightarrow T_xX$ induced by the linear isomorphism $T_{c}C'_{p'(c)}\rightarrow T_{g(c)}C_{fp'(c)}$; to see this, it suffices to observe that this map fits as the upper horizontal map in a pullback diagram
\[
\begin{tikzcd}
 J^1_{C'/S'}(X\times C') \ar[d]\ar[r] &  J^1_{C/S}(X\times C) \ar[d] \\
 S'\ar[r] & S
\end{tikzcd}
\]
of manifolds. The \infcat $\fun_{\mathsf{Mfd}}(\Of_{\mathsf{Surf}_{\C}},\fun(\Delta^1,\mathsf{Mfd}))$ is (the nerve of) an ordinary category so to produce the Cauchy-Riemann map it suffices to provide for each $(C\rightarrow S,j)$ a map 
\[ CR:J^1_{C/S}(X\times C)\longrightarrow J^1_{C/S}(X\times C)  \]
such that the diagram 
\[
\begin{tikzcd}
J^1_{C/S}(X\times C)\ar[dr]\ar[rr,"CR"] && J^1_{C/S}(X\times C)\ar[dl]\\
& C
\end{tikzcd}
\]
commutes, and for each morphism $(C'\rightarrow S',j')\rightarrow (C\rightarrow S,j)$ in the category $\Of_{\mathsf{Surf}_{\C}}$, the diagram 
\[
\begin{tikzcd}
J^1_{C'/S'}(X\times C')\ar[d]\ar[r,"CR"] & J^1_{C'/S'}(X\times C') \ar[d]\\
J^1_{C/S}(X\times C)\ar[r,"CR"] & J^1_{C/S}(X\times C)
\end{tikzcd}
\]
commutes. Under the identifications above, we define the Cauchy-Riemann map as the map that carries a linear map $L:T_{c}C_{p(c)}\rightarrow T_xX$ over a pair $(x,c)$ to the linear map 
\[1/2(L +J|_{T_xX}\circ L\circ j_{p(c)}|_{T_cC_{p(c)}}):T_{c}C_{p(c)}\rightarrow T_xX\]
It is trivial to check that the two diagrams above commute. Similarly, we may define the \emph{zero map} as the map that carries a linear map $L:T_{c}C_{p(c)}\rightarrow T_xX$ over a pair $(x,c)$ to the zero map $0:T_{c}C_{p(c)}\rightarrow T_xX$. Let us now  define a differential moduli problem over $\mathsf{Surf}_{\C}$ via the quintuple
\[ \icate:=(J^1_{\widetilde{\mathsf{Surf}}_{\C}/\mathsf{Surf}_{\C}}(X\times \widetilde{\mathsf{Surf}}_{\C}),X\times \widetilde{\mathsf{Surf}}_{\C},X\times \widetilde{\mathsf{Surf}}_{\C},CR,0).  \]
It is standard that this is an elliptic moduli problem (recall that this is a pointwise condition on $\mathsf{Surf}_{\C}$). Since $\widetilde{\mathsf{Surf}}_{\C}\rightarrow\mathsf{Surf}_{\C}$ is proper, the elliptic representability theorem guarantees that the derived $\cinfty$-stack $\mathcal{M}(X,J):=\mathsf{Sol}(\icate)$ is relatively representable by quasi-smooth derived $\cinfty$-schemes over $\mathsf{Surf}_{\C}$.
\begin{rmk}
It is straightforward to incorporate families of almost complex structures on $X$, or even families of almost complex manifolds by replacing $\mathsf{Surf}_{\C}$ with a suitable moduli stack $S$ that has the property that a map $M\rightarrow S$ corresponds to a pair of an $S$-family of compact complex surfaces and an $S$-family of almost complex manifolds. We can also add marked points by adding to objects $(C\rightarrow S,j)$ of $\Of_{\mathsf{Surf}_{\C}}$ a number of sections $S\rightarrow C$ that the morphisms are required to intertwine with.
\end{rmk}
\begin{rmk}
It follows from Kodaira-Spencer theory \cite{KodairaSpencer} that $\mathsf{Surf}_{\C}$ is a smooth Artin $\cinfty$-stack. As a consequence $\mathcal{M}(X,J)$ is a quasi-smooth derived Artin $\cinfty$-stack.
\end{rmk}

\appendix

\section{Local properties of morphisms and diagrams}
Several key arguments in this work use in a crucial manner that certain properties of morphisms (of smooth or derived stacks) are \emph{local (on the target)}. This appendix offers some technical tools surrounding locality (in particular pertaining to the notion of a \emph{Cartesian transformation}) that come in handy in the main text. We summarizes various characterizations of local properties of morphisms in general \inftopoi and finally generalize to local properties of \emph{diagrams}, which feature in our construction of the jet stack functor of Section 3.3.
\begin{defn}(\cite[Definition 6.1.3.8]{HTT})
Let $\xtop$ be an \inftop and let $P$ be a property of morphisms of $\xtop$ spanning a full subcategory $\Of^P\subset \fun(\Delta^1,\xtop)$. We say that the property $P$ is a \emph{local} property if the following conditions are satisfied.
\begin{enumerate}[$(1)$]
    \item The property $P$ is stable under pullbacks by arbitrary morphisms of $\xtop$ so that the composition 
    \[\Of^P\subset \fun(\Delta^1,\xtop)\overset{\ev_1}{\longrightarrow} \xtop \]
    is a Cartesian fibration. 
    \item The functor 
    \[  \Of^P:\xtop^{op} \longrightarrow \catinfh \]
    obtained as the straightening of the fibration $\ev_1|_{\Of^P}$ preserves small limits. 
\end{enumerate}
\end{defn}
\begin{rmk}
Note that a for a property $P$ of morphisms stable under pullbacks, should $f$ have the property $P$, then any map equivalent to $f$ has the property $P$ as well.
\end{rmk}
Up to technical boundedness assumptions, the class of local properties of morphisms is precisely the class of properties that admit a \emph{classifying object} in $\xtop$. For the following definition, we will write
\[ \Of^{(P)}\subset \Of^{P} \]
for the subcategory on the Cartesian morphisms of $\Of^P$.
\begin{defn}
Let $\xtop$ be an \inftop and let $P$ be a property of morphisms of $\xtop$ stable under pullback. We say that an object $\Omega(P)$ together with a morphism $f:U(P)\rightarrow \Omega(P)$ having the property $P$ is a \emph{classifying object for $P$} if $f$ is a final object in the \infcat $\Of^{(P)}$.    
\end{defn}
\begin{rmk}
Unraveling the definition, a morphism $U(P)\rightarrow \Omega(P)$ is a classifying object for the property $P$ if for each $X\rightarrow Y$ that has the property $P$, there is a unique (up to a contractible space of choices) pullback diagram 
\[
\begin{tikzcd}
X\ar[d]\ar[r] & U(P)\ar[d] \\
Y\ar[r] & \Omega(P).
\end{tikzcd}
\]
Moreover, for each object $Y\in \xtop$, the functors 
\[  
\begin{tikzcd}
\Of^{(P)}\ar[dr,"q"'] & \Of^{(P)}_{/U(P)\rightarrow \Omega(P)}\ar[d]\ar[l,"\simeq"] \ar[r,"\simeq"] & \xtop_{/\Omega(P)} \ar[dl]\\
& \xtop    
\end{tikzcd}
\]
over $\xtop$ induces an equivalence $q^{-1}(Y)\simeq \Hom_{\xtop}(Y,\Omega(P))$. Note that the upper right functor is an equivalence because $\Of^{(P)}$ is a right fibration with final object $U(P)\rightarrow \Omega(P)$ \cite[Proposition 4.4.4.5]{HTT}.
\end{rmk}
\begin{prop}\label{prop:localmorphismclassifyingobject}
Let $P$ be a property of morphisms of an \inftop $\xtop$ stable under pullbacks, then $P$ admits a classifying object if and only if $P$ is a local property and for each $X\in\xtop$, the full subcategory of $\xtop_{/X}$ spanned by morphisms having the property $P$ is essentially small.
\end{prop}
\begin{rmk}
We will say that a property $P$ of morphisms in an \inftop is \emph{small} if for each $X\in\xtop$, the full subcategory $\Of^P_X\subset \xtop_{/X}$ spanned by morphisms having the property $P$ is essentially small.    
\end{rmk}
We have introduced local properties of morphisms in terms of a sheaf condition on the functor 
\[ \xtop^{op}\longrightarrow \catinfh,\quad\quad Y\longmapsto \{\text{Morphisms }X\rightarrow Y\text{ having the property }P\}.  \]
A more pedestrian way to define local properties is to demand that the verification that a morphism has the required property can be done `locally', that is, $X\rightarrow Y$ is a $P$-morphism precisely if we can find a cover $\{U_i\rightarrow Y\}_i$ such that $U_i\times_YX\rightarrow U_i$ is a $P$-morphism. We verify that this latter definition coincides with our initial one. For this, we need the following important notion.
\begin{defn}\label{defn:cartesiantrans}
Let $\icat$ be an \infcatt, $K$ a simplicial set and let $\alpha:f\rightarrow g$ be a natural transformation of diagrams $K\rightarrow \icat$. The natural transformation $\alpha$ is \emph{Cartesian} if for each arrow $k\rightarrow k'$, the diagram
\[
\begin{tikzcd}
f(k)\ar[d]\ar[r,"\alpha(k)"] &g(k) \ar[d] \\
f(k') \ar[r,"\alpha(k')"] & g(k')
\end{tikzcd}
\]   
is a pullback square.
\end{defn}
\begin{rmk}\label{rmk:cartesian}
The class of \inftopoi is characterized as the class of presentable \infcats satisfying the following condition: for each small simplicial set $K$ and each natural transformation $\alpha:f\rightarrow g$ of diagrams $f,g:K^{\rhd}\rightarrow \xtop$, in case $\alpha|_{K}$ is a Cartesian transformation and $g$ is a colimit diagram, then $\alpha$ is a Cartesian transformation if and only if $f$ is a colimit diagram (\cite[Theorem 6.1.3.9, Proposition 6.1.3.10]{HTT}).
\end{rmk}

\begin{prop}[Descent]\label{prop:localpropertycharacterize} 
Let $\xtop$ be an \inftop and let $P$ be a property of morphisms of $\xtop$ stable under pullbacks.
Then the following are equivalent.
\begin{enumerate}[$(1)$]
    \item The property $P$ is a local property.
    \item For each small simplicial set $K$ and each natural transformation $\alpha:f\rightarrow g$ of colimit diagrams $f,g:K^{\rhd}\rightarrow \xtop$, in case $\alpha|_{K}$ is a Cartesian transformation and $\alpha(k)\in \Of^P$ for each $k\in K$, then $\alpha(\infty)\in \Of^P$ (note that then $\alpha$ is a Cartesian transformation by the previous remark). 
    \item For each map $X\rightarrow Y$, should there exist a collection of maps $\{U_i\rightarrow Y\}_{i\in I}$ determining an effective epimorphism $\coprod_iU_i\rightarrow Y$ such that for all $i\in I$, the map $X\times_YU_i\rightarrow U_i$ has the property $P$, then $X\rightarrow Y$ has the property $P$. 
\end{enumerate}
\end{prop}
\begin{rmk}
Notice that $(3)$ implies that a local property of morphisms is stable under small coproducts.    
\end{rmk}
Before we give the proof, we provide the following useful characterization of Cartesian transformations on cones, which we employ a few times in this work.
\begin{lem}\label{lem:cartesiantransfkanext}
Let $\icat$ be an \infcatt, let $K$ be a small \infcatt. Let $\ev_{\infty}:\fun(K^{\rhd},\icat)\rightarrow\icat$ be the functor evaluating at the cone point, then the following hold true.
\begin{enumerate}[$(1)$]
\item let $\alpha:f\rightarrow g$ be a natural transformation of functors $f,g:K^{\rhd}\rightarrow\icat$, then $\alpha$ is a Cartesian transformation if and only if $\alpha$ is $\ev_{\infty}$-Cartesian.
\item Suppose that $\icat$ admits pullbacks. Let $g:K^{\rhd}\rightarrow \icat$ be a diagram and let $X\rightarrow g(\infty)$ be morphism in $\icat$ to the value of $g$ on the cone point. Then there exists a diagram $f:K^{\rhd}\rightarrow \icat$ and a  Cartesian transformation $\alpha:f\rightarrow g$ such that $\alpha(\infty)$ is the map $X\rightarrow g(\infty)$ which is unique up to a contractible space of choices. 
\item Let $\alpha:f\rightarrow g$ be a natural transformation between augmented simplicial objects $\simpopplus\rightarrow\icat$. Suppose that $g$ is a \v{C}ech nerve, then $\alpha$ is Cartesian if and only if $f$ is a \v{C}ech nerve and the diagram
\[
\begin{tikzcd}
f(0)\ar[d]\ar[r] & f(-1)\ar[d]\\
g(0)\ar[r] & g(-1)
\end{tikzcd}
\] 
is a pullback.
\end{enumerate}
\end{lem}
\begin{proof}
The categorical equivalence $K\times\Delta^1\coprod_{K\times\{1\}}\Delta^0\overset{\simeq}{\rightarrow}K^{\rhd}$ shows that the functor $\ev_{\infty}$ fits as the left vertical map in a homotopy pullback diagram 
\[
\begin{tikzcd}
\fun(K^{\rhd},\icat)\ar[d]\ar[r] & \fun(\Delta^1,\fun(K,\icat))\ar[d,"\ev_1"] \\
\icat \ar[r] & \fun(K,\icat)
\end{tikzcd}
\]
among \infcatst. It follows from \cite[Lemma 6.1.1.1]{HTT} that a morphism $f\rightarrow g$ in $\fun(K^{\rhd},\icat)$ is $\ev_{\infty}$-Cartesian precisely if for all $k\in K$, the diagram 
\[
\begin{tikzcd}
f(k)\ar[d]\ar[r,"\alpha(k)"] &g(k) \ar[d] \\
f(\infty) \ar[r,"\alpha(k)"] & g(\infty)
\end{tikzcd}
\]   
is a pullback square, so $(1)$ follows from the pasting property of pullback squares. Should $\icat$ admit pullbacks, then $\fun(K,\icat)$ admits pullbacks taken objectwise. It follows from \cite[Lemma 6.1.1.1]{HTT} that the functor $\ev_{\infty}$ is a Cartesian fibration, so that the map $X\rightarrow g(\infty)$ has a Cartesian lift terminating at $g$ unique up to contractible ambiguity which by $(1)$ yields the desired Cartesian transformation. We now prove $(3)$. Suppose that $\alpha:f\rightarrow g$ is a Cartesian transformation and that $g:\simpopplus\rightarrow\icat$ is a \v{C}ech nerve, then one readily verifies using pasting of pullback squares that $f|_{\simpop}$ is a groupoid object and that 
\[
\begin{tikzcd}
f(1)\ar[r]\ar[d] & f(0)\ar[d] \\
f(0)\ar[r] & f(-1)
\end{tikzcd}
\]
is a pullback square, that is $f$ is a \v{C}ech nerve. For the converse, we first assume that $\icat$ admits finite limits. Then by $(2)$ we may choose a Cartesian transformation $\alpha':f'\rightarrow g$ which fits into a diagram 
\[
\begin{tikzcd}
& f'\ar[dr,"\alpha'"] \\
f\ar[rr,"\alpha"]\ar[ur,"\beta"] && g.
\end{tikzcd}
\]  
Since $\alpha'$ is Cartesian and we have assumed that $\alpha$ induces an equivalence $f(0)\simeq g(0)\times_{g(-1)}f(-1)$, $\beta$ induces an equivalence between $\alpha'|_{(\simpopplus)^{\leq 0}\times\Delta^1}$ and $\alpha|_{(\simpopplus)^{\leq 0}\times\Delta^1}$. By assumption, $f$ is a \v{C}ech nerve and we have just shown that $f'$ is a \v{C}ech nerve. Since \v{C}ech nerves are right Kan extensions of their restriction to $(\simpopplus)^{\leq 0}$, we see that $\beta$ is an equivalence. The case of a general \infcat follows by embedding $\icat$ in $\pshv(\icat)$.
\end{proof}
\begin{proof}[Proof of Proposition \ref{prop:localpropertycharacterize}]
The equivalence of $(1)$ and $(2)$ is the content of \cite[Lemma 6.1.3.5]{HTT} so we are reduced to proving that $(2)$ and $(3)$ are equivalent. Suppose $(2)$ holds. First, we apply $(2)$ to the collection of small \emph{discrete} simplicial sets to deduce that the collection of morphisms having the property $P$ is stable under small coproducts. Now let $X\rightarrow Y$ be any map and suppose we are given a small collection $\{U_i\rightarrow Y\}_i$ determining an effective epimorphism $h:\coprod_iU_i\rightarrow Y$ such that $U_i\times_YX\rightarrow U_i$ has the property $P$ for all $i$. Let $\check{C}(h)_{\bullet}:\simpop_{+}\rightarrow\xtop$ be a \v{C}ech nerve of $h$, then by using Lemma \ref{lem:cartesiantransfkanext}, we can construct a Cartesian transformation 
\[  \alpha:\simpop_{+}\times\Delta^1\longrightarrow\xtop \]
such that $\alpha|_{\{\infty\}\times\Delta^1}$ is the morphism $X\rightarrow Y$. Since $\xtop$ is an \inftop and $\alpha$ is a Cartesian transformation, the map $\alpha:\simpop_{+}\rightarrow\fun(\Delta^1,\xtop)$ is a colimit diagram. For each $[n]\in \simp$, the map $\alpha|_{\{[n]\}\times\Delta^1}$ is a pullback of the map
\[  \coprod_i X\times_YU_i\longrightarrow \coprod_i U_i,  \]
which has the property $P$ as every cofactor has the property $P$ by assumption and the property $P$ is stable under small coproducts as we have just verified. Since the property $P$ is stable under pullbacks, we deduce that $\alpha$ is a diagram satisfying the conditions specified in $(2)$ so we conclude. For the converse, we suppose we are given a natural transformation $\alpha:f\rightarrow g$ of diagrams $f,g:K^{\rhd}\rightarrow \xtop$ as in $(2)$, then we wish to prove that the map $\colim_Kf\rightarrow \colim_Kg$ has the property $P$. The result now follows from $(3)$ since \cite[Lemma 6.2.3.13]{HTT} guarantees that the collection $\{g(k)\rightarrow \colim_Kg\}_{k\in K}$ determines an effective epimorphism $\coprod_{k\in K}g(k)\rightarrow \colim_Kg$, and the assumption that $\alpha$ is a Cartesian transformation ensures that the map $g(k)\times_{\colim_{K}g}\colim_{K}f\rightarrow g(k)$ can be identified with the map $f(k)\rightarrow g(k)$, which has the property $P$ by assumption.
\end{proof}
Any \inftop $\xtop$ is a left exact localization of an \infcat of presheaves $\pshv(\icat)$. We now investigate how properties of morphisms of $\xtop$ relate to presheaves on $\icat$.
\begin{prop}\label{prop:denselocal}
Let $\xtop$ be an \inftop and let $\icat$ be a small \infcat equipped with a fully faithful dense embedding $f:\icat\hookrightarrow \xtop$, so that the induced functor $\xtop\rightarrow\pshv(\icat)$ is fully faithful (for instance, if $\kappa$ is a regular cardinal for which $\xtop$ is $\kappa$ compactly generated, then $\icat\subset\xtop$ can be the full subcategory spanned by $\kappa$-compact objects). Let $P$ be a property of morphisms of $\xtop$ stable under pullbacks and let $F_P$ be the (possibly large) presheaf on $\icat$ associated to the right fibration $\Of^{(P)}\times_{\xtop}\icat$. Then the following are equivalent.
\begin{enumerate}[$(1)$]
\item The property $P$ is a local property. 
\item For each map $X\rightarrow Y$, should there exist a collection of morphisms $\{Z_i\rightarrow Y\}_i$ with $Z_i$ in the essential image of $f$ determining an effective epimorphism $\coprod_iZ_i\rightarrow Y$ such that $Z_i\times_YX\rightarrow Z_i$ has the property $P$ for all $i$, then $X\rightarrow Y$ has the property $P$.
\end{enumerate}
If either of these conditions are satisfied and $P$ is small, then the object $F_P$ lies in the essential image of the functor $\xtop\hookrightarrow\pshv(\icat)$ and is a classifying object for the property $P$.  
\end{prop}
\begin{proof}
Replacing $\icat$ with a minimal model of the essential image of $f$, we may assume that $\icat$ is a full subcategory. The implication $(1)\Rightarrow (2)$ is obvious from Proposition \ref{prop:localpropertycharacterize}. Conversely, given a map $X\rightarrow Y$ in $\xtop$ and an effective epimorphism $\coprod_iU_i\rightarrow Y$ such that $U_i\times_YX\rightarrow U_i$ lies in $\Of^P$ for all $i$, then choosing by density of $\icat$ in $\xtop$ for each $i$ an effective epimorphism $\coprod_jZ_{ij}\rightarrow U_i$ with $Z_{ij}\in \icat$ and invoking $(2)$ yields the result. Since unstraightening is functorial, there is a commuting diagram 
\[
\begin{tikzcd}
\mathsf{RFib}'_{\xtop}\ar[r,hook] \ar[d,"\simeq"]& \mathsf{RFib}_{\xtop}\ar[d,"\simeq"]\ar[r,"\_\times_{\xtop}\icat"] & \mathsf{RFib}_{\icat} \ar[d,"\simeq"] \\
\xtop\ar[r,hook,"j"]&\pshv(\xtop) \ar[r,"f^*"] & \pshv(\icat)
\end{tikzcd}
\]
where $\mathsf{RFib}'_{\xtop}\subset\mathsf{RFib}_{\xtop}$ is the full subcategory spanned by representable right fibrations. Since $P$ is local and small if and only if the right fibration $\Of^{(P)}$ is representable, we deduce that $F_P$ is a classifying object for $P$.
\end{proof}

Let $P$ be a property of morphisms of an \inftop stable under pullbacks. If $P$ is not local, we can consider the smallest property $\widehat{P}$ stable under pullbacks containing $P$ that is local: let us say that a morphism $X\rightarrow Y$ \emph{locally has the property $P$} if there is a covering $\{U_i\rightarrow Y\}_i$ such that $U_i\times_YX\rightarrow U_i$ has the property $P$. As it turns out, the property $\widehat{P}$ of locally having the property $P$ viewed as a sheaf $\Of^{(\widehat{P})}$ on $\xtop$ is universal among sheaves admitting a map from $\Of^{(P)}$.  
\begin{prop}\label{prop:sheafificationproperty}
Let $\xtop$ be an \inftop and let $P$ be a small property of morphisms stable under pullbacks. Let $\widehat{P}$ be the following property of morphisms of $\xtop$.
\begin{enumerate}
    \item[$(\widehat{P})$] A morphism $X\rightarrow Y$ belongs to $\widehat{P}$ just in case there exists a small family $\{U_i\rightarrow Y\}_{i\in I}$ of morphisms determining an effective epimorphism $\coprod_i U_i\rightarrow Y$ such that for each $i\in I$, the morphism $U_i\times_YX\rightarrow U_i$ has the property $P$.
\end{enumerate}
Then the property $\widehat{P}$ is small and local and the morphism 
\[  \Of^{(P)} \longrightarrow \Of^{(\widehat{P})}  \]
of functors $\xtop^{op}\rightarrow \spa$ exhibits $\Of^{(\widehat{P})}$ as a sheafification of $\Of^{(P)}$. 
\end{prop}
\begin{proof}
First, we show that $\widehat{P}$ is local. For stability under pullbacks, suppose we are given a morphism $X\rightarrow Y$ having the property $P$ and $Z\rightarrow Y$ any morphism. Choose a collection $\{U_i\rightarrow Y\}_i$ determining an effective epimorphism $\coprod_iU_i\rightarrow Y$ such that $U_i\times_{Y}X\rightarrow U_i$ has the property $P$, then the collection $\{U_i\times_YZ\rightarrow Z\}_i$ determines an effective epimorphism $\coprod_iU_i\times_YZ\rightarrow Z$ and for each $i$, the map $U_i\times_YZ\times_Z (Z\times_YX)\rightarrow U_i\times_YZ$ is a pullback of the map $U_i\times_{Y}X\rightarrow U_i$ and thus has the property $P$. The locality is an obvious consequence of $(3)$ of Proposition \ref{prop:localpropertycharacterize}. To see that the property $\widehat{P}$ is small, we first choose a regular cardinal $\kappa$ such that $\xtop$ is $\kappa$-compactly generated. It follows from Proposition \ref{prop:denselocal} that for any map $X\rightarrow Y$ having the property $P$, we can choose a pullback diagram 
\[
\begin{tikzcd}
\coprod_{i\in I} X\times_YU_i\ar[d]\ar[r] &\coprod_{i\in I}U_i\ar[d] \\
X\ar[r] & Y
\end{tikzcd}
\] 
where both vertical maps are effective epimorphisms, $U_i\in \xtop^{\kappa}$, the full subcategory spanned by $\kappa$-compact objects for each $i\in I$, and $X\times_YU_i\rightarrow U_i$ lies in $P$ for each $i\in I$. Since $X\rightarrow Y$ is a colimit of the \v{C}ech nerve of the diagram above, this map is determined up to equivalence by the maps $\{U_i\rightarrow Y\}_{i\in I}$ and $\{X\times_YU_i\rightarrow U_i\}_{i\in I}$. We may assume that $I$ is a subset of the set $S:=\pi_0(\xtop^{\kappa}_{/Y})$, which is a small set. Since $P$ is small, there is for each $Z\in \xtop^{\kappa}$ up to equivalence a bounded number of $P$-morphisms $Z'\rightarrow Z$. It follows that the cardinality of the set of equivalence classes of $P$-morphisms to $Y$ is bounded by that of the set $\prod_{I\in P(S)}\prod_{Z\rightarrow Y\in I}\pi_0(\Of^P_{Z})$ where $P(S)$ is the power set of $S$, which is small.\\
Now we prove that the map $f:\Of^{(P)}\rightarrow \Of^{(\widehat{P})}$ is a sheafification map. We are required to show that for any limit preserving functor $G:\xtop^{op}\rightarrow \spa$, composition with $f$ induces a homotopy equivalence
\[  \Hom_{\pshv(\xtop)}(\Of^{(\widehat{P})},G)\longrightarrow   \Hom_{\pshv(\xtop)}(\Of^{(P)},G). \]
Via the straightening-unstraightening equivalence, we may alternatively show that the map
\[  \Hom_{\mathsf{RFib}_{\xtop}}(\Of^{(\widehat{P})},\icat_G)\longrightarrow   \Hom_{\mathsf{RFib}{\xtop}}(\Of^{(P)},\icat_G) \]
is an equivalence, where $\icat_G\rightarrow \xtop$ is a right fibration associated to $G$ via unstraightening. Since the limit preserving functors $G:\xtop^{op}\rightarrow \spa$ correspond precisely to representable right fibrations, we may suppose that $\icat_{G}\rightarrow \xtop$ is of the form $\xtop_{/Z}\rightarrow \xtop$ for some $Z\in \xtop$. We complete the proof by showing that the functor
\[ \fun_{\xtop}(\Of^{(\widehat{P})},\xtop_{/Z})\longrightarrow \fun_{\xtop}(\Of^{(P)},\xtop_{/Z}) \]
is a trivial Kan fibration. In view of \cite[Proposition 4.3.2.15]{HTT}, it suffices to argue that every functor $\Of^{(P)}\rightarrow\xtop_{/Z}$ admits a $q$-left Kan extension along the full subcategory inclusion $\Of^{(P)}\subset \Of^{(\widehat{P})}$, and that every functor $\Of^{(\widehat{P})}\rightarrow \xtop_{/Z}$ is a $q$-left Kan extension of its restriction to $\Of^{(P)}$. Since the projection $q:\xtop_{/Z}\rightarrow\xtop$ preserves and reflects all relative colimit diagrams, we are reduced to proving that for every object $X\rightarrow Y$ of $\Of^{(\widehat{P})}$, the induced functor $(\Of^{(P)}_{/X\rightarrow Y})^{\rhd}\rightarrow \xtop$ carrying the cone point to $Y$ is a colimit diagram. Using pasting of pullback squares, we can identify $\Of^{(P)}_{/X\rightarrow Y}$ with the \emph{full} subcategory of $\fun(\Delta^1,\xtop)_{/X\rightarrow Y}$ spanned by pullback squares
\[
\begin{tikzcd}
V\ar[d]\ar[r] & U\ar[d]\\
X \ar[r] & Y
\end{tikzcd}
\]
for which $V\rightarrow U$ has the property $P$. It follows from \cite[Proposition 4.3.2.15]{HTT} that the forgetful functor $\Of^{(P)}_{/X\rightarrow Y}\rightarrow\xtop_{/Y}$ induces a trivial Kan fibration onto the full subcategory $\xtop^0_{/Y}\subset\xtop_{/Y}$ spanned by objects $U\rightarrow Y$ for which $U\times_YX\rightarrow U$ has the property $P$. In view of \cite[Lemma 4.3.2.7]{HTT} and the fact that the inclusion of the final object $\{Y\}\hookrightarrow \xtop_{/Y}$ is left cofinal, it suffices to argue that the projection $\xtop_{/Y}\rightarrow \xtop$ is a left Kan extension of the composition $\xtop^0_{/Y}\subset\xtop_{/Y}\rightarrow \xtop$ along the inclusion $\xtop^0_{/Y}\subset\xtop_{/Y}$. Since the projection $\xtop_{/Y}\rightarrow\xtop$ preserves and reflects colimits, it suffices to show that the identity functor $\xtop_{/Y}\rightarrow\xtop_{/Y}$ is left Kan extension of the inclusion $\xtop^0_{/Y}\subset\xtop_{/Y}$ along the inclusion $\xtop^0_{/Y}\subset\xtop_{/Y}$, that is, the inclusion $\xtop^0_{/Y}\subset\xtop_{/Y}$ is \emph{dense}. For notational convenience, write $\ytop$ for the slice \inftop $\xtop_{/Y}$, so that we may suppose that $Y$ is a final object and $\ytop^0=\xtop^0_{/Y}$ the full subcategory spanned by those $U\in \ytop$ such that $U\times X\rightarrow U$ has the property $P$. Let $\ytop^0\subset\ytop^1\subset\ytop$ be the largest full subcategory containing $\ytop^0$ for which the inclusion $\ytop^0\subset \ytop^1$ is dense. It suffices to argue the following.
\begin{enumerate}[$(1)$]
    \item The inclusion $\ytop^1\subset\ytop$ is stable under colimits.
    \item The inclusion $\ytop^0\subset\ytop$ generates $\ytop$ under colimits.
\end{enumerate}
We first prove $(2)$. It suffices to show that $\ytop^0$ generates the final object of $\ytop$. By assumption on the map $X\rightarrow Y$, there exists a small collection  $\{U_i\rightarrow Y\}_i$ determining an effective epimorphism $h:\coprod_i U_i\rightarrow 1_{\ytop}$. The \v{C}ech nerve $\check{C}(h)_{\bullet}$ of this map is in each level a small coproduct of objects in $\ytop^0$ (note that $\ytop^0\subset\ytop$ is stable under pullbacks) so we conclude. For the proof of $(1)$, consider a colimit diagram $\mathcal{J}:K^{\rhd}\rightarrow \ytop$ with colimit $T$, then we wish to show that the map $\colim_{Z\in \ytop^0_{/T}}Z\rightarrow T$ from the canonical colimit is an equivalence. Since being an equivalence is a local property, it suffices to show that for each $k\in K$, the map $\colim_{Z\in \ytop^0_{/T}}Z\times_T\mathcal{J}(k)\rightarrow \mathcal{J}(k)$ is an equivalence, that is, that the functor
\[  \ytop_{/T}^0\subset \ytop_{/T}\overset{\_\times_T\mathcal{J}(k)}{\longrightarrow}\ytop_{/\mathcal{J}(k)}   \]
determines a colimit diagram $(\ytop_{/T}^0)^{\rhd}\rightarrow \ytop$. Since the property $P$ is stable under pullbacks, this functor factors as 
\[ \ytop_{/T}^0 \overset{\_\times_T\mathcal{J}(k)}{\longrightarrow} \ytop^0_{/\mathcal{J}(k)}\subset  \ytop_{/\mathcal{J}(k)}.\]
The first functor in this composition has a left adjoint given by composition with $\mathcal{J}(k)\rightarrow T$ and is thus left cofinal, and the second functor determines a colimit diagram $(\ytop^0_{/\mathcal{J}(k)})^{\rhd}\rightarrow \ytop$ since $\mathcal{J}(k)\in \ytop^1$.
\end{proof}
\begin{defn}
For $P$ a small property of morphisms stable under pullbacks in an \inftopt, we will call the property $\widehat{P}$ defined in Proposition \ref{prop:sheafificationproperty} the \emph{sheafification} of the property $P$.     
\end{defn}
\begin{rmk}
The preceding result also holds for properties that are not small, but proving this requires switching to \inftopoi in a larger universe.    
\end{rmk}
\begin{cor}
Let $P$ be a small property of morphisms of an \inftop $\xtop$ stable under pullbacks and let $\widehat{P}$ be the sheafification of $P$. Then for each sheaf $F:\xtop^{op}\rightarrow\spa$ and each map $\Of^P\rightarrow F$, there exist a commuting triangle 
\[
\begin{tikzcd}
& \Of^{\widehat{P}} \ar[dr] \\
\Of^{P}\ar[rr] \ar[ur]  && F   
\end{tikzcd}
\]
of presheaves on $\xtop$ unique up to contractible ambiguity.
\end{cor}

We will now briefly consider properties of arbitrary small diagrams $L^{\rhd}\rightarrow \xtop$ in \inftopoi (the case $L=\Delta^0$ being the one considered until now). We will call diagrams with domain $L^{\rhd}$ in an \inftop $\xtop$ \emph{$L^{\rhd}$-diagrams in $\xtop$}. For $P$ a property of $L^{\rhd}$-diagrams, we will denote by $\Of^{L,P}\subset \fun(L^{\rhd},\xtop)$ the full subcategory determined by $P$. We say that a property $P$ of $L^{\rhd}$-diagrams is \emph{stable under pullback} if for every natural transformation $\alpha:f\rightarrow g$ for which $g$ has the property $P$ and $\alpha$ is a Cartesian transformation, the diagram $f$ also has the property $P$. It follows from Lemma \ref{lem:cartesiantransfkanext} that in case $P$ is stable under pullback, evaluation at the cone point determines a Cartesian fibration $\Of^{L,P}\rightarrow \xtop$. We will write $\Of^{(L,P)}$ for the subcategory on the Cartesian morphisms. 
\begin{prop}\label{prop:diagramproperty}
Let $P$ be a property of $L^{\rhd}$-diagrams of an \inftop $\xtop$, then the following are equivalent.
\begin{enumerate}[$(1)$]
\item The functor $\Of^{L,P}:\xtop^{op}\rightarrow\catinfh$ preserves small limits.
\item For each small simplicial set $K$ and each bifunctor $F:K^{\rhd}\times L^{\rhd}\rightarrow \xtop$ determining a colimit diagram $K^{\rhd}\rightarrow \fun(L^{\rhd},\xtop)$, in case the restriction $F|_{K\times L^{\rhd}}$ is Cartesian and $F|_{\{k\}\times L^{\rhd}}\rightarrow \xtop$ has the property $P$, then $F$ is Cartesian and $F_{\{\infty\}\times L^{\rhd}}$ has the property $P$.
\item For $f:L^{\rhd}\rightarrow \xtop$ with $f(\infty)=Y$, should there exist an effective epimorphism $\coprod_iU_i\rightarrow Y$ such that $U_i\times_Yf$ has the property $P$, then $f$ has the property $P$, where $U_i\times_Yf$ denotes the pullback of the diagram $f$ along $U_i\rightarrow Y$ (obtained from an $\ev_{\infty}$-Cartesian lift of $U_i\rightarrow Y$ terminating at $f$).
\end{enumerate} 
\end{prop}
\begin{rmk}\label{rmk:localdiagram}
From the characterization $(3)$ we immediately deduce the following: let $\xtop$ be an \inftopt, let $L$ be small simplicial set and suppose we are given a collection $\{P_l\}_{l\in L}$ of local properties of morphisms of $\xtop$ indexed by the vertices of $L$, then the property of $L^{\rhd}$-diagrams that for any vertex $l:\Delta^0\rightarrow L$, the morphism 
\[\Delta^1\cong (\Delta^0)^{\rhd}\overset{l}{\hooklongrightarrow} L^{\rhd}\longrightarrow\xtop\]
has the property $P_l$ is a local property. Suppose we are also given for each edge $e:\Delta^1\rightarrow L$ a local property $P_e$, then the property of $L^{\rhd}$-diagrams that for each $e:\Delta^1\rightarrow L$, the morphism 
\[\Delta^1\overset{e}{\longrightarrow} L\subset L^{\rhd}\longrightarrow\xtop\] has the property $P_e$ is a local property of $L^{\rhd}$-diagrams.
\end{rmk}
We will use the following terminology: let $F:K\times L\rightarrow\xtop$ be a bifunctor, then we say that $F$ is \emph{Cartesian} if for every edge $\Delta^1\rightarrow K$, the natural transformation $F|_{\Delta^1\times L}$ is Cartesian. Note that this condition is symmetric in $K$ and $L$. 
\begin{lem}\label{lem:cartesianbifunctor}
Let $\xtop$ be an \inftopt. For each small simplicial set $K$ and each bifunctor $F:K^{\rhd}\times L^{\rhd}\rightarrow \xtop$ such that the restriction $F|_{K\times L^{\rhd}}$ is Cartesian and $F|_{K^{\rhd}\times \{\infty\}}$ is a colimit diagram, the restriction $F|_{K^{\rhd}\times \{l\}}$ is a colimit diagram for all $l\in L$ if and only if $F$ is a Cartesian bifunctor.
\end{lem}

\begin{proof}
Consider the composition 
\[\widehat{F}:K^{\rhd}\longrightarrow \fun(L^{\rhd},\xtop)\longrightarrow\fun(L\times\Delta^1,\xtop)\cong\fun(\Delta^1,\fun(L,\xtop)).\]
Since $\fun(L^{\rhd},\xtop)$ fits into a homotopy pullback diagram
\[
\begin{tikzcd}
\fun(L^{\rhd},\xtop)\ar[d]\ar[r] & \fun(\Delta^1,\fun(L,\xtop))\ar[d,"\ev_1"] \\
\xtop \ar[r] & \fun(L,\xtop)
\end{tikzcd}
\]
it follows that $F$ is a Cartesian bifunctor if and only if and only if $\widehat{F}$ carries every edge to an $\ev_1$-Cartesian edge of $\fun(\Delta^1,\fun(L,\xtop))$, and similarly for the restriction $\widehat{F}|_{K}$. Also, $F|_{K^{\rhd}\times \{\infty\}}$ is a colimit diagram if and only if $\ev_1\circ \widehat{F}:K\rightarrow \fun(L,\xtop)$ is a colimit diagram and $F|_{K^{\rhd}\times \{l\}}$ is a colimit diagram for all $l\in L$ if and only if $\ev_0\circ\widehat{F}:K\rightarrow\fun(L,\xtop)$ is a colimit diagram. We conclude using Remark \ref{rmk:cartesian}, since $\fun(L,\xtop)$ is an \inftopt.
\end{proof}
\begin{proof}[Proof of Proposition \ref{prop:diagramproperty}]
The condition $(1)$ is equivalent to the assertion that the restriction
\[   \fun^{\mathrm{Cart}}_{\xtop}(K^{\rhd},\Of^{L,P}) \longrightarrow  \fun^{\mathrm{Cart}}_{\xtop}(K,\Of^{L,P}) \]
is an equivalence by \cite[Proposition 3.3.3.1]{HTT} (here `Cart' denotes we are taking full subcategories spanned by Cartesian sections). Let $\fun^{\mathrm{Cart}}_{\xtop}(K^{\rhd},\fun(\Delta^1,\xtop))_0\subset \fun_{\xtop}(K^{\rhd},\fun(\Delta^1,\xtop))$ be the full subcategory spanned by functors $F$ such that
\begin{enumerate}[$(a)$]
\item the restriction to $K$ determines a Cartesian bifunctor $K\times L^{\rhd} \rightarrow \Of^{L,P}$.
\item the functor $F:K^{\rhd}\rightarrow\fun(L^{\rhd},\xtop)$ is a colimit diagram.
\end{enumerate}
By Lemma \ref{lem:cartesianbifunctor} and the assumption that $K^{\rhd}\rightarrow\xtop$ is a colimit diagram, we have an inclusion $\fun^{\mathrm{Cart}}_{\xtop}(K^{\rhd},\Of^{L,P})\subset \fun^{\mathrm{Cart}}_{\xtop}(K^{\rhd},\fun(\Delta^1,\xtop))_0$  of full subcategories. A functor $K^{\rhd}\rightarrow \fun(L^{\rhd},\xtop)$ satisfies condition $(b)$ above if and only if it is a $\ev_{\infty}$-colimit since $K^{\rhd}\rightarrow\xtop$ is a colimit (\cite[Proposition 4.3.1.5]{HTT}) so the restriction 
\[\fun^{\mathrm{Cart}}_{\xtop}(K^{\rhd},\fun(\Delta^1,\xtop))_0 \longrightarrow \fun^{\mathrm{Cart}}_{\xtop}(K,\Of^{L,P})\]
is a trivial fibration (\cite[Proposition 4.3.2.15]{HTT}). It follows that $(1)$ is equivalent to the condition that the inclusion $\fun^{\mathrm{Cart}}_{\xtop}(K^{\rhd},\Of^{L,P})\subset \fun^{\mathrm{Cart}}_{\xtop}(K^{\rhd},\fun(\Delta^1,\xtop))_0$ is an equality, which, in view of Lemma \ref{lem:cartesianbifunctor}, is a reformulation of $(2)$. We now show $(2)\Rightarrow (3)$. We first apply $(2)$ to the collection of small discrete simplicial set to conclude that the property $P$ is stable under small coproducts. Now suppose we are given a diagram $f:L^{\rhd}\rightarrow\xtop$ with $f(\infty)=Y$ and an effective epimorphism $h:\coprod_iU_i\rightarrow Y$. Let $\check{C}(h)_{\bullet}^+:\simpopplus\rightarrow \xtop$ be a \v{C}ech nerve of $h$. Since the inclusion of the cone point $\{\infty\}\subset \simpopplus$ is Cartesian marked anodyne (where we mark all morphisms of $\simpopplus$), it follows from Lemma \ref{lem:cartesiantransfkanext} that we may choose a Cartesian bifunctor $F:\simpopplus\times L^{\rhd}\rightarrow\xtop$ that restricts to the diagram $f$ on $\{\infty\}\times L^{\rhd}$ and to $\check{C}(h)_{\bullet}$ on $\simpopplus\times\{\infty\}$. By Lemma \ref{lem:cartesianbifunctor}, the functor $\simpopplus\rightarrow\fun(L^{\rhd},\xtop)$ is a colimit diagram. For every $[n]\in \simp$, the restriction $F|_{\Delta^1 \times L}$ to any map $[0]\rightarrow [n]$ of $\simpop$ is a Cartesian transformation, so we conclude using $(2)$ that $F_{\{[n]\}\times L^{\rhd}}$ has the property $P$ as $F|_{\{[0]\} \times L^{\rhd}}$ is equivalent to $\coprod_iU_i\times_Yf$, a coproduct of $L^{\rhd}$-diagrams that have the property $P$. For the implication $(3)\Rightarrow (2)$, we are given a Cartesian bifunctor $K^{\rhd}\times L^{\rhd}\rightarrow\xtop$ determining a colimit diagram $K^{\rhd}\rightarrow \fun(L^{\rhd},\xtop)$ such that $F|_{K\times L^{\rhd}}$ is Cartesian and $F_{\{k\}\times L^{\rhd}}$ has the property $P$. Then the collection $\{F(k,\infty)\rightarrow F(\infty,\infty)\}_{k\in K}$ determines an effective epimorphism $\coprod_{k\in K}F(k,\infty)\rightarrow F(\infty,\infty)$ (\cite[Lemma 6.2.3.13]{HTT}) satisfying the condition of $(3)$, so $F|_{\{\infty\}\times L^{\rhd}}$ has the property $P$.
\end{proof}

Proposition \ref{prop:sheafificationproperty} admits a straightforward generalization to properties of $L^{\rhd}$-diagrams, whose formulation and proof we leave to the reader.
\newpage

\printbibliography
\end{document}